\documentclass[10pt]{article}
\usepackage{amssymb,amsmath,amsfonts,amsthm,enumerate,color}
\usepackage[toc,page,title,titletoc,header]{appendix}
\usepackage{graphicx}
\usepackage{indentfirst}
\usepackage{multicol}
\usepackage{booktabs}

\numberwithin{equation}{section}

\newtheorem{Theorem}{Theorem}[section]
\newtheorem{Lemma}[Theorem]{Lemma}

\newtheorem{Corollary}[Theorem]{Corollary}

\numberwithin{equation}{section}

 \def\p{\partial} 
\def \Vh0{\stackrel{\circ}{V}_h}

\newcommand{\q}{\quad}    
\def\l{\label}    

\def\m{\mbox}   
  \def\ss{\smallskip}

\newcommand{\lc}
{\mathrel{\raise2pt\hbox{${\mathop<\limits_{\raise1pt\hbox
{\mbox{$\sim$}}}}$}}}

\newcommand{\gc}
{\mathrel{\raise2pt\hbox{${\mathop>\limits_{\raise1pt\hbox{\mbox{$\sim$}}}}$}}}

\newcommand{\ec}
{\mathrel{\raise2pt\hbox{${\mathop=\limits_{\raise1pt\hbox{\mbox{$\sim$}}}}$}}}

\def\bb{\begin{equation}} \def\ee{\end{equation}}

\def\beqn{\begin{eqnarray}}  \def\eqn{\end{eqnarray}}

\def\beqnx{\begin{eqnarray*}} \def\eqnx{\end{eqnarray*}}

\def\bn{\begin{enumerate}} \def\en{\end{enumerate}}
\def\i{\item}
\def\bd{\begin{description}} \def\ed{\end{description}}

\makeatletter

\newenvironment{figurehere}
  {\def\@captype{figure}}
  {}
\makeatother

\title{Analysis on Non-negative Factorizations and Applications}

\begin{document}
\author{
Yat Tin Chow \footnote{Department of Mathematics, The Chinese University of Hong Kong, Shatin, Hong Kong. (ytchow@math.cuhk.edu.hk, zou@math.cuhk.edu.hk).}
\and Kazufumi Ito\footnote{Department of Mathematics and Center for Research in Scientific Computation, North Carolina State University, Raleigh, North Carolina (kito@unity.ncsu.edu).}
\and Jun Zou\footnotemark[1]
}

\date{}
\maketitle

\begin{abstract}
In this work we perform some mathematical analysis on non-negative matrix factorizations (NMF) and apply NMF to some imaging and inverse problems.
We will propose a sparse low-rank approximation of big positive data and images in terms of tensor products of positive vectors, and investigate its effectiveness in terms of the number of tensor products to be used in the approximation.
A new concept of multi-level analysis (MLA) framework is also suggested
to extract major components in the matrix representing structures of different resolutions, but still preserving
the positivity of the basis and sparsity of the approximation.
 We will also propose a semi-smooth Newton method based on primal-dual active sets
 for the non-negative factorization. Numerical results are given to demonstrate the effectiveness of the proposed method to capture features in images and structures of inverse problems under no \textit{a-priori} assumption on the data structure, as well as to provide a sparse low-rank representation of the data.
\end{abstract}

\noindent { \footnotesize {\bf Mathematics Subject Classification
(MSC2000)}:
    15A23, 
    65F22, 
    65F30, 
    65F50, 
    78M25. 
}

\noindent { \footnotesize {\bf Keywords}:
Non-negative matrix factorization, Clustering,  Feature extraction, Multi-level analysis, Inverse problems
}

\section{Introduction to Non-negative Factorizations} \label{sec1}
Non-negative factorization (NMF) has attracted a great deal of attention in the last decade,
in an attempt to tackle k-clustering problems and structural analysis of big data. It is very effective in extraction of principle components, features and similarities inside a large set of data or image.
NMF was studied as early as in 1994 \cite{o27}, and used for machine learning and data mining \cite{o18, o19}.
The concept of NMF as k-means clustering for principle component analysis has been widely studied
theoretically and numerically in literature, see, e.g., \cite{o3, o5, clustering, o9, o15, o20, o27, o32};
and the concept of tri-factorization was used as a concurrent column and row clustering (a.k.a. co-clustering) of data in
\cite{tri}.
In order to extract desired features as well as to reduce memory complexity, sparsity is often
imposed in NMF using $l_{0}$ or $l_{1}$ regularization.
Effective NMF toolboxes have been also developed to provide different choices of regularizers and constraints, e.g.,
the non-negative matrix factorization toolbox in MATLAB \cite{matlab}.
A convex model for NMF was suggested in \cite{Osher}, where the convex $l_{1,\infty}$-norm is used
as the regularizer to enforce row sparsity.  In an application of this convex model to hyper-spectral end-members selections,
the NMF succeeded to provide abundance maps of end-members representing different structures inside an image, e.g.,
roofs, trees, grass, soil and road.

In general, an NMF of a given matrix $Y\in\mathbb{R}^{N\times M}$ is of the form
\beqn
Y \approx A P\,, \q A\in\mathbb{R}^{N\times k}, ~P\in\mathbb{R}^{k\times M}\,,
\eqn
where matrix $P$ is non-negative component-wise.
In most applications, we may require dimension $k$ to be much smaller than the dimension of $Y$, i.e. $k << \min(N,M)$. $P$ is regarded as a basis of the information contained in matrix $Y$.  We may further impose $P$ to be nearly orthornormal, i.e., $P P^T \approx I$. In this case, it is similar to a partition of unity in the underlying space and the vectors in $P$ are similar to some indicator functions. In order to reduce memory complexity in storing the basis $P$, one may further add a sparsity constraint on $P$. The matrix $A$ is an assignment matrix, which gives some special weighting to the corresponding vectors in $P$.  It is our aim to obtain a sparse matrix $A$ which has a very small number of non-zero entries.
Therefore, $A$ can be interpreted as some sparse assignments of linear combinations of basis vectors in $P$.
If matrix $Y$ is also non-negative component-wise, we may further require $A \geq 0$ component-wise.  This constraint
may be infeasible if $Y$ is not nonnegative, and in this case
we shall relax and drop the non-negativity condition for $A$.
The sparsity constraint on $A$ ensures more concise information extraction.
Moreover, we may also have a post-process to sort vectors of $A$ in descending order in terms of
magnitude, which can yield the most important bases of matrix $P$.
Using a standard $l_1$ regularization to impose sparsity for $A$ and $P$ and near-orthogonality for $P$,
the problem of NMF for a non-negative matrix $Y$ can be reformulated as the following minimization problem:
\begin{equation}
\min_{A\geq 0, P \geq 0}||Y-AP||_{F,2}^2+\alpha||A||_{F,1} + \nu||P||_{F,1}+ \gamma ||PP^T-I||_{F,1}
\label{mini}
\end{equation}
over nonnegative matrices $A\in\mathbb{R}^{N\times k}$ and $P\in\mathbb{R}^{k\times M}$, where $||X||_{F,2}:= \sqrt{\sum_{i,j} |X_{ij}|^2}$ is the Frobenius norm, $||X||_{F,1}:= \sum_{i,j} |X_{ij}|$ and $\alpha,\nu, \gamma$ are regularization parameters.

A natural approach for matrix factorization is the singular value decomposition (SVD), which helps obtain the best low-rank approximation of a matrix in $l_2$ sense and extracts the most important components of the matrix based on the magnitude of their corresponding singular values.  The factorization of SVD is of the form
\beqn
Y = U \Sigma V^T
\eqn
where we can interpret the matrices $U,V$ as bases of information, $\Sigma$ as a weighting representing the importance of the corresponding basis vectors in $U$ and $V$.
Although this approach gives the best low-rank approximation of matrix $Y$ in $l_2$ norm after a truncation of $\Sigma$, the SVD factorization is unstructured and usually does not respect positivity, often with the basis vectors of $U$ and $V$
being rather oscillatory. Especially for a matrix $Y$ which represents an image or a probability density function, such an SVD factorization does not give us much useful information of the underlying structures that $Y$ represents, e.g., identifying regions of high probability, locating objects inside the image, etc.
Therefore we shall turn to NMF to obtain a more structural decomposition of the matrix that shall respect more the positivity of the basis.
Now, combining the non-negativity constraints and the SVD gives rise to the idea of non-negative matrix tri-factorization,
which was studied in \cite{tri}.
In this work, we suggest and investigate the following version of non-negative matrix tri-factorization for non-negative matrix $Y$ using $l_1$ regularisation:
\begin{equation}
\min_{U \geq 0, \Sigma \geq 0, V \geq 0}||Y-U \Sigma V^T ||_{F,2}^2+ \alpha||\Sigma ||_{F,1}+ \nu||U ||_{F,1} + \nu||V||_{F,1} +\gamma ||U U^T-I||_{F,1} + \gamma||V V^T-I||_{F,1}.
\end{equation}

Similarly, we may interpret the matrices $U \in M_{M \times p}$, $V \in M_{p \times N}$ as basis of information, $\Sigma \in M_{p \times p}$ as a generalized singular matrix.  We emphasize that the matrix $\Sigma$ is not required to be diagonal
in our setting here, but to be sparse only.

We shall propose the application of the aforementioned model of non-negative matrix tri-factorization to big data and large images
to extract their major components, which may represent some special structures or features, and obtain an approximation of the data with low memory complexity when the rank $p$ is small, even when the original data and images do not attain any sparsity structure.  This shall be quite effective, considering the fact that the factorization gives a low rank sparse approximation of the matrix in term of the tensor products of column and row vectors of $U$ and $V$.  The fact that $p$ is small requires the storage of only a few columns and rows in the matrices $U$ and $V$, therefore greatly reduces memory complexity.  The sparsity of $\Sigma$ is also very important for the reduction of memory complexity because we only need to store the respective columns and rows of the matrices $U$ and $V$, e.g., $u_i$ and $v_j$,
where the corresponding entry $\sigma_{ij}$ in the singular matrix $\Sigma$ is significant. The sparsity of $U$ and $V$ are equally important because $u_i$ and $v_j$ will then have a few number of non-zero entries and are inexpensive to store.
These reasons suggest us to apply the above NMF model to big data and imaging.
To effectively implement the NMF, we utilize the well-known semi-smooth Newton method based on primal-dual active sets\cite{semismooth} for the optimization process.
It may be more advantageous than some classical methods \cite{clustering} \cite{tri}.
Using the result of NMF from the Newton method, we propose a dissection of the image into levels by its order of importance.

We then proceed to propose a new concept of multi-level analysis (MLA) framework of the images based on the NMF, which
aims to extract major components inside the matrix $Y$ representing structures of different resolutions and obtain
sparse low-rank approximations of different levels with positive basis.
For each ine level, we hope to extract and represent features of up to a finer resolution
with sparse approximation by positive basis.
Our MLA framework is partially motivated by, though different from, the multi-resolution analysis (MRA) in wavelet analysis, e.g. in \cite{wavelet}.  The MRA framework is well-established to provide successive approximations of increasing resolutions of a function by a shifting and scaling of a mother wavelet. However, it has the property that the basis functions generated from the mother wavelet always do not have the same sign of the whole space.  This is a very undesirable feature in our context.
Hence, we introduce a new MLA framework, which shall respect the positivity of the basis for function/matrix approximation,
and on the other hand provide a similar multi-resolution property as in MRA.
In our MLA framework, we introduce a nested sequence of linear spaces $H_s$ each of which
represents a level of fineness, and define
interpolation operators among these spaces of coarser and finer levels.  The NMF is then performed on each level to obtain a positive sparse approximation.
We would like to emphasize that the main purpose of either our NMF model or the newly proposed MLA framework is only to identify and represent structures (of different scales) in the images or big data, and
we are neither hoping to reconstruct the data in full entity nor aiming at very high-quality compression of image to defeat any available well-developed compression techniques, e.g. wavelet/curvelet compression, JPEG etc.
Numerical experiments show acceptable resolution of images and data can be achieved by this sparse approximation using the MLA framework of the NMF model, as well as extracting the major features and components in the images and data without any \textit{a-priori} assumption of their structures, such as sparsity and specific patterns.

This paper is organized as follows. In section \ref{sec2} the general mathematical framework of non-negative matrix tri-factorization using $l_1$ regularization is clearly stated, and an optimal choice of the dimension of generalized singular matrix is investigated.
An MLA framework using NMF is introduced in section \ref{sec3} and a semi-smooth Newton method based on primal-dual active sets for NMF is presented in section \ref{sec4}.
Applications of our framework to imaging and inverse problems are provided in section \ref{sec5}, providing numerical evidence for some successful feature extractions and sparse low-rank representation of the data.

\section{A non-negative matrix tri-factorization using $l_1$ regularization} \label{sec2}

In this section we shall clearly state the type of matrix tri-factorizations for our subsequent consideration.
For the purpose, we often write $M_{M \times N}$ for the set of $M \times N$ matrices and $M_{M \times p}^+ \subset M_{M \times N}$
for those with positive entries.
Given a matrix $Y \in M_{M \times N}^+$, we define a functional $\mathcal{J}_{p}^{\alpha, \nu, \gamma}:
M_{M \times p}^+ \times M_{p \times p}^+ \times  M_{p \times N}^+  \rightarrow  \mathbb{R}$
for a fixed set of parameters $p, \alpha, \gamma$:
\bb
\mathcal{J}_{p}^{\alpha, \nu, \gamma} (U,\Sigma, V) := ||Y-U \Sigma V^T ||_{F,2}^2+ \gamma ||\Sigma ||_{F,1} + \nu || U ||_{F,1} + \nu || V ||_{F,1}+ \alpha ||U U^T-I||_{F,1} + \alpha ||V V^T-I||_{F,1}.
\label{functional}
\ee 
Let $[\tilde{U}_p,\tilde{\Sigma}_p, \tilde{V}_p]$ be a minimizer of the functional, then we
define an operator $\mathcal{I}_{p}^{\alpha, \nu, \gamma}$: $M_{M \times N}^+ \rightarrow   M_{M \times N}^+$ by
\begin{eqnarray}
\mathcal{I}_{p}^{\alpha, \nu,\gamma}(Y) &:=& \tilde{U}_p \tilde{\Sigma}_p \tilde{V}_p = \sum_{i,j} \sigma_{ij} (\tilde{u}_p)_i \otimes (\tilde{v}_p)_j
\label{approximation}
\end{eqnarray}
where $(\tilde{u}_p)_i, (\tilde{v}_p)_j$ denote the column and row vectors of $\tilde{U}_p $ and $\tilde{V}_p$ respectively and $\sigma_{ij}$ is the $(i,j)$-th entry of the matrix $\tilde \Sigma_p$. This non-negative matrix tri-factorization can be regarded as a non-negative version of the SVD,
with matrix $\tilde \Sigma_p$ being the generalized singular matrix, which is not restricted to be diagonal as in the standard SVD.

It is easy to note that with a smaller $p$, the memory of storing the matrix triple $[\tilde{U}_p,\tilde{\Sigma}_p, \tilde{V}_p] $ is
less.
If $\tilde{\Sigma}_p$ is a sparse matrix, the memory complexity can be further reduced, as we only need to store the vectors $(\tilde{u}_p)_i$ and $ (\tilde{v}_p)_j$ when $\sigma_{ij}$ is non-zero.
In fact, for a generic matrix $Y$, if $p$ can be chosen to be small and yet $||Y- \mathcal{I}_{p}^{\alpha, \nu, \gamma}(Y)||_{F,2}^2$ can still be maintained to be a small quantity, then $[\tilde{U}_p,\tilde{\Sigma}_p, \tilde{V}_p] $ may serve as our desired sparse low-rank approximation of $Y$.
However, it is obvious that the smaller the value of $p$ is, the worse the approximation of $Y$ by $\mathcal{I}_{p}^{\alpha, \nu, \gamma}(Y)$ will be. With a smaller $p$, the error $||Y- \mathcal{I}_{p}^{\alpha, \nu, \gamma}(Y)||_F^2$ and also $\mathcal{J}_{p}^{\alpha, \nu,  \gamma} (\tilde{U}_p,\tilde{\Sigma}_p, \tilde{V}_p)$, will be larger. Therefore, in practice, it is an interesting question to ask how we should choose the number $p$ as $N,M$ grow large.

\subsection{An Optimal choice of $p$}

In what follows, we aim to find an optimal choice of $p$ with respect to $N,M$ by means of a probabilistic argument. We first obtain a lower bound in terms of $p,N,M,\delta$ of the probability that there exists a triple $[U,\Sigma, V]$ such that $\mathcal{J}_{p}^{\alpha, \nu, \gamma} (U,\Sigma, V) < \delta$. From this lower bound, we suggest an optimal choice of $p$ to maximize this probability.  The value $\mathcal{J}_{p}^{\alpha, \nu, \gamma} (U_p,\Sigma_p, V_p)$ reflects the derivations of matrices $U_p$, $V_p$ from being orthogonal, the sparsity of $\tilde{U}_p,\tilde{\Sigma}_p, \tilde{V}_p$ and the error of the approximation of $Y$ by $\mathcal{I}_{p}^{\alpha, \nu, \gamma} (Y)$.
In particular, if for some $[U,\Sigma, V]$, we have $\mathcal{J}_{p}^{\alpha, \nu, \gamma} (U,\Sigma, V) < \delta$, then
\beqn
 || Y- \mathcal{I}_{p}^{\alpha, \nu, \gamma} (Y) ||_{F,2}^2 \leq \mathcal{J}_{p}^{\alpha, \nu, \gamma} (\tilde U_p,\tilde \Sigma_p, \tilde V_p) \leq \mathcal{J}_{p}^{\alpha, \nu, \gamma} (U,\Sigma, V) < \delta \,.
 \notag
\eqn

We begin by showing the following lemmas concerning a set of i.i.d. random vectors.
Consider a set of i.i.d random vectors $\{ X_i\}_{i=1}^{N} \in [0,1]^d$, where the probability distribution $d \, \mathbb{P} _{X} = f d x$ with $dx$ denoting the standard Lebesgue measure and $ 0 <C_1 <f <C_2 < \infty$. Then it is direct to see that the random variables $\{ \omega_i := X_i/||X_i||_2\}_{i=1}^{N} \in \mathbb{S}^{d-1}$ has a probability density $d \, \mathbb{P} _{\omega} = g d \omega$ where $d \omega $ is the standard surface measure and $\frac{C_1}{||\omega||_{\infty}} \leq g \leq \frac{C_2}{||\omega||_{\infty}}$ for some other constants $0 <C_1,C_2 <\infty$. From this, we can derive the following important results for our subsequent analysis.

\begin{Lemma}
Consider a set of i.i.d random vectors $\{ X_i\}_{i=1}^{N} \in [0,1]^d$, where the probability distribution $d \, \mathbb{P} _{X} = f d x$ with $dx$ denoting the standard Lebesgue measure and $ 0 <C_1 <f <C_2 < \infty$. Then the probability of the vectors $\omega_i := X_i/||X_i||_2$ that can be approximated by $p$ points $\{ P_i\}_{i=1}^{p} \in \mathbb{S}^{d-1} \bigcup [0,1]^d $ within an error of small $\varepsilon>0$ can be estimated by
\beqn
p^N (C_3 \varepsilon)^{(d-1)N} \leq \mathbb{P} \left(\exists \{ P_i\}_{i=1}^{p} \text{ s.t. } \{ \omega_i\}_{i=1}^{N} \subset \bigcup_{1 \leq i\leq P} B_{\varepsilon} (P_i)\right) \leq p^N (C_4 \varepsilon)^{(d-1) N}
\eqn
for two positive constants $C_3$ and $C_4$.
\end{Lemma}

\begin{proof}
Using the fact that for small $\varepsilon>0$, $C \varepsilon <\sin \varepsilon < \varepsilon$  for some $C>0$,
we can actually observe from the assumption of the i.i.d. random vectors and the binomial theorem that
\beqn
& & \mathbb{P} \left(\exists \{ P_i\}_{i=1}^{p} \text{ s.t. } \{ \omega_i\}_{i=1}^{N} \subset \bigcup_{1 \leq i\leq P} B_{\varepsilon} (P_i)\right)  \notag \\
&=& \sum_{\sum_{i =1}^{p} N_i = N} \frac{N!}{ \prod_i N_i !} \frac{1}{|\mathbb{S}^{d-1}\bigcap[0,1]^d| }\prod_{i} \int_{\mathbb{S}^{d-1}\bigcap[0,1]^d} \mathbb{P} ( || \omega_i - K ||_2 < \varepsilon )^{N_i} dK \notag \\
&\geq& \sum_{\sum_{i =1}^{p} N_i = N} \frac{N!}{ \prod_i N_i !} (C_3 \varepsilon)^{(d-1) \sum_i N_i} \notag \\
&\geq& p^N (C_3 \varepsilon)^{(d-1) N} \notag
\eqn
for some $C_3 > 0$. The other estimate is similar.
\end{proof}

\begin{Lemma}
Consider a set of i.i.d random vectors $\{ P_i\}_{i=1}^{p} \in [0,1]^d$, where the probability distribution $d \, \mathbb{P}_{\omega} = f d \omega$ with $d \omega$ denoting the standard surface measure and $ 0 <C_1 <f <C_2 < \infty$. Then
for $p \leq d$, the probability of the set of vectors $P_i$ being almost mutually orthogonal within an error of small $\varepsilon>0$ can be estimated by
\beqn
p ! \, d \, (C_3 \varepsilon)^{\frac{(p)(p-1)}{2} + (d-1)} \leq \mathbb{P} \left( | \langle  P_i,P_j \rangle - \delta_{ij} | < \varepsilon \, \forall i,j \, \right) \leq p ! \, d \, (C_43 \varepsilon)^{\frac{(p)(p-1)}{2} + (d-1)}
\eqn
for two positive constants $C_3$ and $C_4$.
\end{Lemma}

\begin{proof}
By direct counting, the fact that $||P_i-P_j ||^2 = 2- 2 \langle  P_i,P_j \rangle$ and along
with the half angle formula, we have for $p \leq d$ that
\beqn
& & \mathbb{P} \left(  \langle | P_i,P_j \rangle - \delta_{ij} | < \varepsilon \, \forall i,j \right) \notag \\
&\geq& p! \, d \, (C_3 \varepsilon )^{d-1 } \prod_{ 1 \leq i \leq p }(C_3 \varepsilon )^{i} |(\mathbb{B}^{i}_1\times \mathbb{B}^{n-i}_1 ) \bigcap[0,1]^d|   \notag \\
&\geq& p ! \, d \, (C_3 \varepsilon)^{\frac{(p)(p-1)}{2} + (d-1)} \notag
\eqn
for some $C_3 > 0$. The other estimate is similar.
\end{proof}

\begin{Lemma}
\label{lemma3}
Consider a set of i.i.d random vectors $\{ X_i\}_{i=1}^{N} \in [0,1]^d$, where the probability distribution $d \mathbb{P} _{X} = f d x$ with $dx$ denoting the standard Lebesgue measure and $ 0 <C_1 <f <C_2 < \infty$. Then for  $p \leq N$,
the probability of the event $E_{p,\varepsilon}$
representing  the existence of $\{ P_i\}_{i=1}^{p} \text{ such that } \{ \omega_i\}_{i=1}^{N} \subset \bigcup_{1 \leq i\leq P} B_{\varepsilon} (P_i) $ and $| \langle  P_i,P_j \rangle - \delta_{ij} | < \varepsilon \, \forall i,j$ for a small $\varepsilon>0$
can be estimated by
\beqn
 \left( p^N - (p-1)^N   \right) pl ! \, d \, (C_3 \varepsilon)^{\frac{ p (p-1)}{2} + (d-1)(N+1)}  \leq \mathbb{P} \left( E_{p,\varepsilon} \backslash E_{p-1,\varepsilon} \right) \leq \left( p^N - (p-1)^N   \right) pl ! \, d \, (C_4 \varepsilon)^{\frac{ p (p-1)}{2} + (d-1)(N+1)}
 \notag
\eqn
for two positive constants $C_3$ and $C_4$, and therefore
\beqn
\sum_{l = 1}^p  \left( l^N - (l-1)^N   \right) l ! \, d \, (C_3 \varepsilon)^{\frac{ l (l-1)}{2} + (d-1)(N+1)}  \leq \mathbb{P} \left( E_{p,\varepsilon} \right) \leq \sum_{l = 1}^p  \left( l^N - (l-1)^N   \right) l ! \, d \, (C_4 \varepsilon)^{\frac{ l (l-1)}{2} + (d-1)(N+1)} \,.
\notag
\eqn
Moreover, we have the following lower bound estimate
\beqn
\mathbb{P} \left( E_{p,\varepsilon} \right) \geq d p^N (C_3 \varepsilon)^{(d-1)(N+1) + \frac{(p)(p-1)}{2}} \,.
\eqn
\end{Lemma}

\begin{proof}
The following inequality follows directly from the argument of the above two lemmas
\beqn
 & &\sum_{\sum_{i =1}^{p} N_i = N \, ,  \, N_i > 0} \frac{N!}{ \prod_i N_i !}  p ! \, d \, (C_3 \varepsilon)^{\frac{ p (p-1)}{2} + (d-1)(N+1)} \notag \\
 & \leq&  \mathbb{P}
\left( E_{p,\varepsilon} \backslash E_{p-1,\varepsilon} \right) \notag\\
&\leq& \sum_{\sum_{i =1}^{p} N_i = N \, ,  \, N_i > 0} \frac{N!}{ \prod_i N_i !}  p ! \, d \, (C_4 \varepsilon)^{\frac{ p (p-1)}{2} + (d-1)(N+1)} \,.
\notag
\eqn
Now since the last term can be simplified as follows:
\beqn
& &  p ! \, d \, (C_3 \varepsilon)^{\frac{ p (p-1)}{2} + (d-1)(N+1)} \sum_{\sum_{i =1}^{p} N_i = N \, ,  \, N_i > 0} \frac{N!}{ \prod_i N_i !} \notag \\
&=&  p ! \, d \, (C_3 \varepsilon)^{\frac{ p (p-1)}{2} + (d-1)(N+1)} \left( \sum_{\sum_{i =1}^{p} N_i = N } \frac{N!}{ \prod_i N_i !} - \sum_{\sum_{i =1}^{p-1} N_i = N} \frac{N!}{ \prod_i N_i !}   \right) \notag \\
&=& \left( p^N - (p-1)^N   \right) p ! \, d \, (C_3 \varepsilon)^{\frac{ p (p-1)}{2} + (d-1)(N+1)} \,,  \notag
\eqn
we directly have
\beqn
\sum_{l = 1}^p  \left( l^N - (l-1)^N   \right) l ! \, d \, (C_3 \varepsilon)^{\frac{ l (l-1)}{2} + (d-1)(N+1)}  \leq \mathbb{P} \left( E_{p,\varepsilon} \right) \leq \sum_{l = 1}^p  \left( l^N - (l-1)^N   \right) l ! \, d \, (C_4 \varepsilon)^{\frac{ l (l-1)}{2} + (d-1)(N+1)}  \,.\notag
\eqn
%
The last inequality comes readily from
\beqn
\mathbb{P} \left( E_{p,\varepsilon} \right) \geq
\sum_{l = 1}^p  \left( l^N - (l-1)^N   \right) \, d \, (C_3 \varepsilon)^{\frac{ p (p-1)}{2} + (d-1)(N+1)} =   d p^N (C_3 \varepsilon)^{(d-1)(N+1) + \frac{(p)(p-1)}{2}} \,. \notag
\eqn
\end{proof}

Now we consider a general image or large data $Y = \sum_{i,j} Y_{ij} \, e_i \otimes e_j $ comprised of non-negative entries. Without loss of generality, we may assume $\max_{i,j}|Y_{ij}| = 1$.
If we write $Y_i := \sum_{j} Y_{ij} \, e_j$, and $ \omega_i = Y_i / ||Y_i||_2 $, then $Y = \sum_i ||Y_i||_2 \, e_i \otimes \omega_i$.
If there exists a set of $ \{ P_i\}_{i=1}^{p} \text{ such that } \{ \omega_i\}_{i=1}^{N} \subset \bigcup_{1 \leq i\leq P} B_{\varepsilon} (P_i) $ and $| \langle  P_i,P_j \rangle - \delta_{ij} | < \varepsilon \, \, \forall i,j$,
we can write $\{\omega_{k_j}\}_{j=1}^{K_j} \in B_{\varepsilon} (P_j) $ for some $K_j$ with $1\leq j\leq p$. Then intuitively, we have
\beqn
I= \sum_i ||Y_i||_2 e_i \otimes \omega_i \approx  \sum_{j = 1}^p \sum_{k_j=1}^{K_j}||Y_{k_j}||_2 \, e_{k_j} \otimes P_{j} \,. \notag
\eqn
Writing $ Q_j := (\sum_{k_j=1}^{K_j}||Y_{k_j}||_2 e_{k_j} ) / \sqrt{ \sum_{k_j=1}^{K_j}||Y_{k_j}||_2}  $
and denoting $\sigma_{ij} = \delta_{ij} \sqrt{ \sum_{k_j=1}^{K_j}||Y_{k_j}||_2} $, then
\beqn
I \approx \sum_i\sigma_{ij} \, Q_i \otimes P_{j} \notag
\eqn
where $|\langle  P_i,P_j \rangle - \delta_{ij} | < \varepsilon$ and $| \langle  Q_i,Q_j \rangle - \delta_{ij} | =0$ for any $i,j$.
By setting $\Sigma = (\sigma_{ij})$, $P = (P_i )^T$, $Q = (Q_j)$, we derive directly that
\beqn
|| I - \sum_i\sigma_{ij} \,  Q_i \otimes P_{j}||_{F_2} \leq
\sum_{j = 1}^p \sum_{k_j=1}^{K_j}||Y_{k_j}||_2 | \omega_{k_j} - P_{j} | \leq ||I||_{F,2} \varepsilon \leq NM \varepsilon \,, \notag
\eqn
hence
\beqn
\mathcal{J}_{p}^{\alpha, \nu, \gamma} (\tilde{U}_p,\tilde{\Sigma}_p, \tilde{V}_p )
&\leq&
\mathcal{J}_{p}^{\alpha, \nu, \gamma} (Q,\Sigma, P) \notag\\
&\leq& ||I||_{F,2} \varepsilon + \gamma \sum_{j}  \sqrt{ \sum_{k_j=1}^{K_j}||Y_{k_j}||_2} + \nu \sum_{j} ||Q_j||_1+ \nu \sum_{i} ||P_i||_1 +\alpha p (p-1) \varepsilon  \notag\\
&\leq& NM \varepsilon + NM (\gamma + 2 \nu )  + \alpha \, p (p-1) \varepsilon\,. \notag
\eqn
The probability of the event $E_{p,\varepsilon}$ such that the above estimate holds
can be bounded below by
\beqn
\mathbb{P} \left( E_{p,\varepsilon} \right) &\geq &
\sum_{l = 1}^p  \left( l^N - (l-1)^N   \right) l ! \, M \, (C_3 \varepsilon)^{\frac{ l (l-1)}{2} + (M-1)(N+1)}
\notag \\
& \geq & M \, p^N (C_3 \varepsilon)^{(M-1)(N+1) + \frac{(p)(p-1)}{2}}\,. \notag
\eqn
Similarly, switching the columns and rows of the image, we may follow the above argument and analysis to conclude the same with $N,M$ swapped. Combining the above two statements, we come to
\beqn
& & \mathbb{P} \left(
\mathcal{J}_{p}^{\alpha,\nu, \gamma} (\tilde{U}_p,\tilde{\Sigma}_p, \tilde{V}_p )  <  NM \varepsilon + NM (\gamma + 2 \nu)+ \alpha \, p(p-1) \varepsilon
 \right) \notag \\
 &\geq &
 \sum_{l = 1}^p  \left( l^{\max(N,M)} - (l-1)^{\max(N,M)}   \right) l ! \, \min(N,M) \, (C_3 \varepsilon)^{ \mu(N,M,l)}
\notag \\
 &\geq & \min(N,M) p^{\max(N,M)} (C_3 \varepsilon)^{ \mu(N,M,p) } \,, \notag
\eqn
where the function $\mu(\,\cdot\,,\,\cdot\,, \,\cdot\, )$ is defined for all $N,M,l \in \mathbb{N}$ by
\beqn
\mu(N,M,l) := \frac{ l (l-1)}{2} + MN - |N-M| -1 \,.
\label{imp_coeff}
\eqn
If we further choose the parameter $\gamma + 2 \nu \leq (K-1) \varepsilon$ for some $K> 1$, then we deduce the following
lemma.
\begin{Lemma}
For any small $\varepsilon>0$ and for all $N,M \in \mathbb{N}$, it holds
\beqn
& &\mathbb{P} \left(
\mathcal{J}_{p}^{\alpha,\nu,  \gamma } (\tilde{U}_p,\tilde{\Sigma}_p, \tilde{V}_p )  <  \left( K NM + \min(N,M)^2 \right)\varepsilon
 \right) \notag \\
 &\geq &
 \sum_{l = 1}^p  \left( l^{\max(N,M)} - (l-1)^{\max(N,M)}   \right) l ! \, \min(N,M) \, (C_3 \varepsilon)^{ \mu(N,M,l)} \label{sum1}
 \\
 &\geq & \min(N,M) p^{\max(N,M)} (C_3 \varepsilon)^{ \mu(N,M,p) }  \label{sum2}
\eqn
where the function $\mu(\,\cdot\,,\,\cdot\,, \,\cdot\, )$ is defined as in \eqref{imp_coeff} and $\gamma $ is such that $\gamma + 2 \nu \leq (K-1) \varepsilon$ for some $K> 1$.
\end{Lemma}

Before we derive a sharp bound of an optimal choice for $p$, let us consider a rough lower bound introduced in the last inequality \eqref{sum2}. Clearly if we consider the function $$F(p):= \min(N,M) p^{\max(N,M)}\, (C_3 \varepsilon)^{  \mu(N,M,p) } $$ for $p \geq 1$, then
it is easy to see
\beqn
F'(p) = \frac{F(p)}{p} \left( \max(N,M) + \frac{| \log (C_3 \varepsilon)|}{16}  - | \log (C_3  \varepsilon)| (p - \frac{3}{4})^2\right) 
\begin{Bmatrix} >  \\ = \\ < \end{Bmatrix} 0 \, \notag \,.
\eqn
namely
$$
p \begin{Bmatrix} <  \\ = \\ > \end{Bmatrix} \frac{3}{4} + \sqrt{\frac{1}{16} + \frac{\max(M,N)}{|\log (C_3 \varepsilon)|}}\,.
$$
Therefore we can propose a primitive optimal choice of $p$ to maximize the lower bound of the possibility \\ $\mathbb{P} \left(
\mathcal{J}_{p}^{\alpha, \nu, \gamma} ([\tilde{U}_p,\tilde{\Sigma}_p, \tilde{V}_p] )  <  \left( K NM + \min(M,N)^2 \right)\varepsilon
 \right)$, i.e. to choose
 \beqn
  p = \sqrt{ \frac{\max(M,N)}{|\log (C_3  \varepsilon)|}}
  \label{beforesub}
 \eqn
for large $N,M$.
Following some basic substitutions, we obtain the following theorem.

\begin{Theorem}
\label{theoremhaha}
For any small $\delta>0$, we have
\beqn
\mathbb{P} \left(
\min_{p} \mathcal{J}_{p}^{\alpha, \nu, \gamma} (\tilde{U}_p,\tilde{\Sigma}_p, \tilde{V}_p )  < \delta
 \right) \geq \min(N,M)\, p_{N,M,\delta}^{\max(N,M)} \left( C_3 \varepsilon \right)^{ \mu\left(N,M,p_{N,M,\delta}\right) }
\eqn
whenever $\gamma + 2 \nu \leq (K-1) \,\varepsilon$, where $ \varepsilon := \delta \left( K NM + \min(M,N)^2 \right)^{-1} $ for some $K>1$, the function $\mu(\,\cdot\,,\,\cdot\,, \,\cdot\, )$ is defined as in \eqref{imp_coeff}, and
 $p_{N,M,\delta}$ stands for the following constant
\beqn
  p_{N,M,\delta} := \sqrt{ \frac{\max(M,N)}{|\log (C_3  \varepsilon)|}}  = \sqrt{ \frac{\max(M,N)}{ \log ( K NM + \min(M,N)^2 ) - |\log \delta | - \log C_3 }} \, .
  \label{optimal2}
\eqn
\end{Theorem}
When $M=N$, it is obvious that the above optimal choice of $p$ for a fixed $\delta >0$ is of the form
\beqn
  p = p_{N,N,\delta} = \sqrt{ \frac{N}{ 2 \log  N  - |\log \delta | -  \log C_3 + \log ( K + 1 )}} \sim \sqrt{ \frac{N}{ 2 \log  N }}\,
  \label{optimal}
\eqn
as $N$ goes to infinity.
The last asymptotic relation actually gives a precise approximation
and
\beqn
  \sqrt{ \frac{N}{ 2 \log  N }} \leq p_{N,N,\delta} \leq \sqrt{ \frac{N}{ \log  N }}\,
  \label{optimalbound}
\eqn
if $N$ is large enough such that $  N  >   C_3 \delta^{-1} $.  Hence \eqref{optimal} serves as an optimal choice of $p$ for large $N$. Furthermore, with this choice of $p$, the memory complexity is asymptotically $\sqrt{ \frac{ 2 N^3}{ \log  N }} $ as $N$ goes to infinity.

However, we note that the optimal choice of $p$ obtained above is only based on a rough lower bound \eqref{sum2}.
In what follows, we deduce a sharper bound by using \eqref{sum1}. Since \eqref{sum1} always increases with respect to $p$,
we get an optimal choice of $p$ by controlling the increment of \eqref{sum1} with respect to $p$.
In order to do so, we investigate the ratio of the terms
$$
a_l := \left( l^{\max(N,M)} - (l-1)^{\max(N,M)}   \right) l ! \, \min(N,M) \, (C_3 \varepsilon)^{ \mu(N,M,l)}\,,
$$
explicitly given by
\beqnx
\frac{a_{l+1}}{a_{l}} = \frac{(l+1)^{\max(N,M)} - l^{\max(N,M)}}{ l^{\max(N,M)} - (l-1)^{\max(N,M)} } l e^{- |\log(C_3 \varepsilon )| (l+1)} \,.
\eqnx
From the l'Hospital rule, we can directly see that for a fixed pair of $N,M$, the above term $a_{l+1} / a_{l} \rightarrow 0$ as $l\rightarrow \infty$. Therefore, given a small $\eta <1$, there is always a $\hat{p}_{N,M,\eta,\varepsilon}$ such that $a_{l+1} / a_{l} \leq \eta$ whenever $l > \hat{p}_{N,M,\eta,\varepsilon}$. Then for all $p > \hat{p}_{N,M,\eta,\varepsilon}$ we have
{\small
\beqn
& &\mathbb{P} \left(
\mathcal{J}_{p}^{\alpha,\nu, \gamma} (\tilde{U}_p,\tilde{\Sigma}_p, \tilde{V}_p )  <  \left( K NM + \min(N,M)^2 \right)\varepsilon
 \right) \notag \\
 &\geq&
  \sum_{l = 1}^{ \hat{p}_{N,M,\eta,\varepsilon} -1 }  \left( l^{\max(N,M)} - (l-1)^{\max(N,M)}   \right) l ! \, \min(N,M) \, (C_3 \varepsilon)^{\mu(N,M,l)} \notag\\
  & &+ \frac{1}{1-\eta} \left( (\hat{p}_{N,M,\eta,\varepsilon})^{\max(N,M)} - (\hat{p}_{N,M,\eta,\varepsilon}-1)^{\max(N,M)}   \right)(\hat{p}_{N,M,\eta,\varepsilon})! \, \min(N,M) \, (C_3 \varepsilon)^{\mu(N,M,\hat{p}_{N,M,\eta,\varepsilon})} \notag
\eqn
}whenever $\gamma + 2 \nu \leq (K-1) \, \varepsilon $, and that the increment of $p$ from $\hat{p}_{N,M,\eta,\varepsilon}$ onward brings insignificant increment to \eqref{sum1}. Now we aim to find an explicit $\hat{p}_{N,M,\eta,\varepsilon}$ in terms of $N,M$, thus obtaining an optimal choice of $p$.
By H\"older's inequality we readily derive
\beqn
a_{p+1}/a_{p} =
\frac{\sum_{i=0}^{\max(N,M)-1}(1+1/p)^i}{\sum_{i=0}^{\max(N,M)-1}(1-1/p)^i}
p \, e^{- |\log(C_3 \varepsilon )| (p+1)}  \leq \frac{p (p+1)^{\max(N,M)-1}}{(p-1)^{\max(N,M)-1}}
p \, e^{- |\log(C_3 \varepsilon )| (p+1)} \,.
\eqn
Now if we consider the function $$G(N_0,p):=  \frac{p (p+1)^{N_0-1}}{(p-1)^{N_0-1}} p \, e^{- |\log(C_3 \varepsilon )| (p+1)} $$ for $p \geq 2$ and $N_0 \geq 2$, then we see
\beqnx
\frac{\p}{\p p} G (N_0,p) = G(N_0,p) \left(\frac{1}{p} + \frac{N_0-1}{p+1} - \frac{N_0-1}{p-1}  - | \log (C_3  \varepsilon)| \right) 
\begin{Bmatrix} >  \\ = \\ < \end{Bmatrix} 0\,,
\eqnx
that implies
$$
p \begin{Bmatrix} <  \\ = \\ > \end{Bmatrix}  p_0(N_0) \notag
$$
where $p_0(N_0)$ is the unique real zero of $- | \log (C_3 \varepsilon)| p^3 + p^2 + ( | \log (C_3  \varepsilon)| - 2 N_0 +2 ) p -1 = 0$, which can be found explicitly by the Cardano's formula or the Lagrange's method.
Fixing $N_0$, we get that $G(N_0,p_0(N_0))$ is the global maximum of $G(N_0,p)$ on $(2, \infty)$, and from $p_0(N_0)$ onward, the function is decreasing. Together with the fact that $G(N_0,2)=4(5/3)^{N_0-1} (C_3 \varepsilon)^4$, we have that  $ G( N_0 , \cdot)^{-1} : (0,4(C_3 \varepsilon)^4) \rightarrow (p_0 , \infty)$ is a well-defined smooth function and is monotone by the inverse function theorem, and that the implicit function $g: (1, \infty) \rightarrow (1,\infty)$ defined by $G(N_0,g(N_0)) = \eta$ is well-defined and smooth by the implicit function theorem as $g(N_0) = [G( N_0 , \cdot)]^{-1} (\eta)$. Moreover
\beqn
g' = - \frac{ \frac{\p N_0}{\p p}( N_0 , g (N_0)) }{ \frac{\p G}{\p p}( N_0 , g (N_0))} &=& - \log\left(\frac{g+1}{g-1}\right) \left(\frac{1}{g} + \frac{N_0-1}{g+1} - \frac{N_0-1}{g-1}  - | \log (C_3 \varepsilon)| \right)^{-1} \notag \\ &=&  \log\left(\frac{g+1}{g-1}\right) \frac{ g (g+1)(g-1)}{ | \log (C_3  \varepsilon)| g^3 - g^2 - ( | \log (C_3  \varepsilon)| - 2 N_0 +2 ) g +1 }\,. \notag
\eqn
Now noting that $g(N_0) > p_0(N_0) + \hat{\delta} > 1 $ for some $\hat{\delta} > 0$ by our choice of domain, we have $ | \log (C_3  \varepsilon)| p^3 - p^2 - ( | \log (C_3  \varepsilon)| - 2 N_0 +2 ) p +1 > \hat{C} > 0$ for some $\hat{C}$, and $ 0 < g'(N_0) < \infty $ for all $N_0$ as well as $g'(N_0) \rightarrow \infty$  as $N_0 \rightarrow \infty$. Moreover putting these inequalities back into the expression of $g'$,
we have $g'(N_0) \rightarrow 0$  as $N_0 \rightarrow \infty$, and that $g$ satisfies the following differential inequality for large $N_0$,
\beqnx
g' \leq  \log\left(\frac{g+1}{g-1}\right) \frac{2}{| \log (C_3  \varepsilon)|} \leq \frac{4}{(g-1) | \log (C_3  \varepsilon)|}\,.
\eqnx
Now using the Gronwall-Bellman-Bihari's inequality, we directly infer that
\beqn
g \leq H^{-1} ( H(a(\eta)) + N_0) \,
\eqn
for some constant $a(\eta)$ depending only on $\eta$, where the function $H$ is defined as
\beqn
H(s) := \frac{| \log (C_3  \varepsilon)|}{4} \int (s-1) ds = \frac{| \log (C_3  \varepsilon)| (s-1)^2}{8} + K_0(\eta)
\eqn
for some $K_0(\eta)$.  Therefore the following inequality holds for $g$ and some
constants $K_1(\eta), K_2(\eta), K_3(\eta)$:
\beqnx
g \leq \sqrt{ \frac{K_1(\eta) N_0 - K_2(\eta) }{ | \log (C_3  \varepsilon)| } } + K_3(\eta) \,.
\eqnx
Using $p_{N,M,\delta}$ defined in \eqref{optimal2},
we can choose $\hat{p}_{N,M,\eta,\varepsilon}$ such that
\beqn
\hat{p}_{N,M,\eta,\varepsilon} = K_\eta \sqrt{ \frac{ \max(N,M) }{ | \log (C_3  \varepsilon)| } } = K_\eta p_{N,M,\delta}
\eqn
for some $K_\eta$ depending on $\eta$, then for all
$
p> \hat{p}_{N,M,\eta,\varepsilon} \geq g\left(  \max(N,M) \right) = \left[G\left(  \max(N,M)  , \cdot\right)\right]^{-1} (\eta),
$
we have
\beqnx
\frac{p (p+1)^{\max(N,M)-1}}{(p-1)^{\max(N,M)-1}}
p \, e^{- |\log(C_3 \varepsilon )| (p+1)}  < \eta  \,.
\eqnx
Therefore the growth of the probability $\mathbb{P} \left(
\mathcal{J}_{p}^{\alpha, \nu, \gamma} (\tilde{U}_p,\tilde{\Sigma}_p, \tilde{V}_p )  <  \left( K NM + \min(N,M)^2 \right)\varepsilon
 \right)$ with respect to $p$ becomes insignificant from $\hat{p}_{N,M,\eta,\varepsilon}$ onward. This gives another optimal choice of $p$. Surprisingly, we notice that
$
 \hat{p}_{N,M,\eta,\varepsilon} \sim  p_{N,M,\delta}\,,
$
i.e., the two choices of $p$ are of the same order. This leads to the following results.

\begin{Theorem}
The following probability bound holds for any small $\delta>0$:
\beqn
 \mathbb{P} \left(
 \mathcal{J}_{p}^{\alpha, \nu, \gamma} (\tilde{U}_p,\tilde{\Sigma}_p, \tilde{V}_p )  < \delta
 \right)
 \geq
\sum_{l = 1}^p  \left( l^{\max(N,M)} - (l-1)^{\max(N,M)}   \right) l ! \, \min(N,M) \, \left( C_3 \varepsilon \right)^{\mu(N,M, l )}
\eqn
whenever $\gamma + 2 \nu \leq (K-1) \,\varepsilon$, where $ \varepsilon := \delta \left( K NM + \min(M,N)^2 \right)^{-1} $ for some $K>1$ and the function $\mu(\,\cdot\,,\,\cdot\,, \,\cdot\, )$ is defined as in \eqref{imp_coeff}.
For a given small constant $\eta$, the growth of the summation above
with respect to $p$ can be controlled by $\eta$ when $p > K_\eta \, p_{N,M,\delta}$ for some $K_\eta$ depending only on $\eta$, where $p_{N,M,\delta}$  is defined as \eqref{optimal2}.
\end{Theorem}

We can easily see that $||Y - \mathcal{I}_{p}^{\alpha, \nu \gamma}(Y) ||_{F,2}^2 < \delta$ if
$\mathcal{J}_{p}^{\alpha, \nu, \gamma} (\tilde{U}_p,\tilde{\Sigma}_p, \tilde{V}_p )  < \delta$.
Clearly, in the particular case when $M=N$, the following asymptotic order for $p$
\beqn
  p  \sim \sqrt{ \frac{N}{ \log  N }}\,
  \label{optimal22}
\eqn
is basically an optimal choice of $p$, and they are equivalent up to a multiplicative constant whenever $  N  >   C_3 \delta^{-1} $.  Following this optimal choice of $p$, the memory complexity grows in the order $\sqrt{ \frac{N^3}{ \log  N }} $ as $N$ goes to infinity.

\subsection{Effects of magnitudes of entries in generalized singular matrices}

In this subsection, we discuss a further reduction of memory complexity by truncating the generalized singular matrix $\tilde{\Sigma}_p = ( \sigma_{ij} )$.
We aim to remove the less important components $(\tilde{u}_p)_{i} \otimes (\tilde{v}_p)_{j}$ in \eqref{approximation}
in a way that it still serves as a good approximation of the original matrix $Y$.

For doing so, we rearrange $\sigma_{ij}$ from the largest value to the smallest one as $\sigma_{i_1 j_1} \geq \sigma_{i_2 j_2} \geq ..\geq \sigma_{i_{p^2} j_{p^2}}$. We then denote $\tilde{\sigma}_l = \sigma_{i_l j_l} e_l \otimes e_l $, and write $  \tilde{\Sigma}_{p,\tilde{p}} = \sum_{l=1}^{\tilde{p}} \tilde{\sigma}_l $ as the truncated generalized singular matrix for all $\tilde{p} \leq p^2$. The sequence $ \{ \tilde{\sigma}_l \}_{l=1}^{p^2}$ represents the components of $\tilde{\Sigma}_p$ in descending order by the importance of its magnitudes.
With the above definition, we then define an operator $\mathcal{I}_{p,\tilde{p}}^{\alpha, \nu, \gamma}$:
$M_{M \times N}^+ \rightarrow   M_{M \times N}^+$ by
\begin{eqnarray}
\mathcal{I}_{p,\tilde{p}}^{\alpha, \nu, \gamma}(Y) := \tilde{U}_p \tilde{\Sigma}_{p,\tilde{p}} \tilde{V}_p = \sum_{l=1}^{\tilde{p}} \sigma_{i_l j_l} (\tilde{u}_p)_{i_l} \otimes (\tilde{v}_p)_{j_l}
\label{approximation2}
\end{eqnarray}
where $[\tilde{U}_p, \Sigma_{p}, \tilde{V}_p]$ is a minimizer of the functional \eqref{functional} and $\tilde{\Sigma}_{p,\tilde{p}}$ is the truncated generalized singular matrix.

The approximation $Y \approx \mathcal{I}_{p,\tilde{p}}^{\alpha, \nu, \gamma}(Y) = \tilde{U}_p \tilde{\Sigma}_{p,\tilde{p}} \tilde{V}_p$ is a truncation of the approximation \eqref{approximation} of $Y$ up to $\tilde{p}$. This truncated approximation removes the less important components.,
hence we only need to save the vectors $(\tilde{u}_p)_{i_l}$ and $ (\tilde{v}_p)_{j_l}$ for $1 \leq l \leq \tilde{p}$.  This further reduces the memory complexity and serves as our desired sparse low-rank approximation of $Y$.

In what follows, we give a brief analysis for the aforementioned truncated approximation of $Y$. Indeed, from the pigeon-hole principle, we directly obtain that
\beqn
 \mathcal{J}_{p}^{\alpha, \nu,  \gamma} (\tilde{U}_p, \tilde{\Sigma}_{p,\tilde{p}} , \tilde{V}_p )   & < &
  \mathcal{J}_{p}^{\alpha, \nu,  \gamma} (\tilde{U}_p, \tilde{\Sigma}_{p} , \tilde{V}_p ) + C || I ||_1 \sum_{i=0}^{p^2 - \tilde{p}} \frac{1}{p^2 - i}  \notag \\
  & < &
  \mathcal{J}_{p}^{\alpha, \nu, \gamma} (\tilde{U}_p, \tilde{\Sigma}_{p} , \tilde{V}_p ) + C || I ||_1 \int^1_{\frac{\tilde{p}}{p^2}} 1/x dx  \notag\\
  & < &
  \mathcal{J}_{p}^{\alpha, \nu, \gamma} (\tilde{U}_p, \tilde{\Sigma}_{p} , \tilde{V}_p ) + C NM \log \left(\frac{p^2}{\tilde{p}}\right) \notag\\
  & < & \left( ( K+ L C) NM  + \min(N,M)^2 \right) \varepsilon \notag
\eqn
whenever $\mathcal{J}_{p}^{\alpha, \nu, \gamma} (\tilde{U}_p, \tilde{\Sigma}_{p} , \tilde{V}_p )< \left(( K NM  + \min(N,M)^2 \right) \varepsilon $ and $ \tilde{p} > e^{ - L \varepsilon } p^2 $.  Combining this with Theorem \ref{theoremhaha}, the following corollary follows directly.

\begin{Corollary}
Let $\varepsilon := \delta ( (K + C L) NM + \min(M,N)^2 )^{-1}$, then
the following estimate holds for any small $\delta>0$ and $\gamma + 2 \nu \leq (K-1) \, \varepsilon$,
{\small
\beqn
\mathbb{P} \left(
\min_{p} \mathcal{J}_{p}^{\alpha, \nu,  \gamma} (\tilde{U}_p,\tilde{\Sigma}_{p,\tilde{p}}, \tilde{V}_p )  <  \delta
 \right) \geq \min(N,M)\, p_{N,M,\delta}^{\max(N,M)} \left( C_3 \varepsilon  \right)^{ \mu(N,M,p_{N,M,\delta}) }\,
 \notag
\eqn
}where $p_{N,M,\delta}$ is stated in \eqref{optimal2} and $ \tilde{p} > e^{ - L \varepsilon } p^2 $ for some $C, K,L$.
\end{Corollary}

\section{Multi-level analysis (MLA) of non-negative tri-factorizations} \label{sec3}

In this section, we introduce a multi-level analysis (MLA) framework based on the aforementioned tri-factorization.
We notice that, for a matrix $Y$, especially for those representing an image, there are features of different scales in $Y$ which usually represent different objects in the image.  We aim at extracting these features of different scales and represent them in a sparse low-rank approximation in terms of tensor products.  Therefore we introduce a MLA framework to NMF which helps us achieve a sparse representation of the features of multiple scales, ranging from the coarsest scale up to the finest scale in the image $Y$.
This MLA framework aims to identify the major components in the matrix $Y$ which represent structures at multiple scales/levels of the image
so that structures from large scales up to small scales in the image can be separately identified and sparsely represented.
Our MLA framework is partially motivated by the MRA in wavelet analysis, which is widely use to capture different resolution of a function or image as well as for compression purpose.
However, an essential difference of our MLA framework from the MRA lies in our hope
to respect the positivity of the basis for the function/matrix approximation, but still obtain a similar multi-resolution property as in MRA.

The most primitive idea of MRA is to successively approximate an $L^2$-function by dyadic shifts and dilations of a wavelet function $\psi$ (a.k.a. the mother wavelet), which results in multiple resolution of the $L^2$-function concerned.
More precisely, we recall, e.g. in \cite{wavelet}, that an MRA in wavelet analysis consists of a nested linear vector spaces, $\cdots \subset V_{1} \subset V_0 \subset V_{-1} \subset \cdots $, such that their union is dense in $L^2$, and that they satisfy self-similarity conditions in both time and scaling as well as a regularity condition requiring the integer shifts of a piecewise continuous scaling function $\varphi$ with compact support (a.k.a the father wavelet) shall form a frame for the subspace $V_0\subset L^2$.
In the case of integer shifts on $\mathbb{R}$, the above assumptions of nested linear vector spaces implies the following dilation equation, e.g. in \cite{wavelet}: there exists a finite sequence of coefficients $c_k$ with $|k|\leq N$ such that
$$ \varphi(x)=\sum_{k=-N}^{N} c_k \varphi(2x-k) \,. $$
The mother wavelet can then be defined as
$$ \psi(x):=\sum_{k=-N}^{N} (-1)^k c_{1-k} \,\varphi(2x-k)\,,$$
and with this definition, one can easily render that $ V_{l-1} =  V_{l} \bigoplus W_l$ for all $l$, where
$W_l \subset V_{l-1}$ denotes the closed subspace generated by the frame $\{\psi(2^{-l}x-k): k \in \mathbb{N}\}$.
Recursively, we can show that
$$ \{\psi_{l,k}(x)=\sqrt2^{(-l)}\psi(2^{-l}x-k):\;l,k\in\mathbb{Z}\}$$
shall form a complete orthonormal base in $L^2(\mathbb{R})$ and that $L^2(\mathbb{R})= \bigoplus_{l\in\mathbb{Z}}W_l$.
A similar result holds for higher dimension with a similar argument.

However one can directly infer that the mother wavelet $\psi$ has the property \cite{wavelet,int_wavelet} that
$$\int_{-\infty}^{\infty} \psi (x) \,dx= 0, $$ which directly implies that the function $\psi$ can never have the same sign on the whole space. Therefore the approximation of an $L^2$ function by
$$ f = \sum_{l,k\in \mathbb{Z}} c_{l,k} \, \psi_{l,k} \,,$$
though acquiring the multi-resolution property, fails to be a representation of $f$ by positive basis. This observation that $\psi$ does not have the same sign over the whole space is also true for higher dimension. Therefore it may be an undesirable feature if function $f$ is positive, and when we hope to approximate the function with positive basis.  This is the case when the function/matrix represents an image or a probability density function.   This motivates us for a non-negative version of a similar multi-level approximation of the function based on the NMF technique, which we name as the multi-level analysis (MLA), in hope that
each increasingly fine level of approximation of the function by positive basis shall represent an increase of resolution in some sense.

In what follows, we give a mathematical framework for the MLA in NMF.
For the sake of exposition, we introduce the following several operators which are very useful in the subsequent discussion. We first define an interpolation operator $\iota_s : M_{M \times N} \rightarrow M_{\frac{M}{r^s} \times \frac{N}{r^s}}$ as the following averaging operator:
\beqn
 \iota_s (Y) &:=& \sum_{1 \leq i \leq M/r^s , 1 \leq j \leq N/r^s} \frac{1}{r^{2s}} \sum_{k,l \in Q_{I_i,J_j}} Y_{kl} e_i \otimes e_j
 \label{interpolation}
\eqn
where $Q_{I_i,J_j}$ contains the entries $i M/r^s \leq k < (i+1) M/r^s , j M/r^s \leq l < (j+1) M/r^s $.  We note that this interpolation operator gives an interpolation between a fine space $H_0 := M_{M \times N}$ to a coarse space $H_s := M_{\frac{M}{r^s} \times \frac{N}{r^s}}$, and the spaces $H_s$ actually forms a nested sequence of spaces, i.e. $H_s \subset H_{l}$ if $s > l$. One may actually define a more general nested sequence of spaces and
interpolation operators, but for the sake of simplicity, we shall only discuss this averaging operator.
Then we define
$
\mathcal{I}_{s,p}^{\alpha, \nu, \gamma} : M_{M \times N}^+ \rightarrow M_{M \times N}^+ \notag
$
by
\beqn
\mathcal{I}_{s,p}^{\alpha,  \nu,  \gamma} := \iota_s^T \circ \mathcal{I}_{p}^{\alpha, \gamma} \circ \iota_s \,.
\eqn
The approximation $\mathcal{I}_{s,p}^{\alpha,  \nu,  \gamma} (Y)$ represents the approximation of the $(s_{\max}-s)$-th level of the image $Y$ by NMF where $s_{\max} \leq [\log(\min(N,M))/\log(r)]$ and $[ \cdot ]$ denotes the floor function.
Similarly, we define
$
\mathcal{I}_{s,p,\tilde{p}}^{\alpha,  \nu, \gamma} : M_{M \times N}^+ \rightarrow M_{M \times N}^+ \notag
$
by
\beqn
\mathcal{I}_{s,p,\tilde{p}}^{\alpha, \nu,  \gamma} := \iota_s^T \circ \mathcal{I}_{p,\tilde{p}}^{\alpha,  \nu,  \gamma} \circ \iota_s \,,
\eqn
which serves as a truncated approximation of the $(s_{\max}-s)$-th level of $Y$.

Now we are ready to investigate and analyse the error of the approximation given by this MLA framework.
In fact, it is easy to see by combining the arguments in previous sections and the Poincare inequality that
\beqn
 || Y - \iota_s^T \circ \mathcal{I}_{p,\tilde{p}}^{\alpha,  \nu, \gamma} \circ \iota_s  (Y) ||_{F,2}^2 &\leq& r^{2s} || \iota_s ( Y) - \mathcal{I}_{p,\tilde{p}}^{\alpha,  \nu,  \gamma} \circ \iota_s  (Y) ||_{F,2}^2 + \sum_{I,J} ||\nabla_{\delta} Y_{IJ}||_{F,2}^2  \notag \\ &\leq& r^{2s} \mathcal{J}_{p}^{\alpha,  \nu, \gamma} (\tilde{U}_p,\tilde{\Sigma}_{p,\tilde{p}}, \tilde{V}_p )  + \sum_{I,J} ||\nabla_{\delta} Y_{IJ}||_{F,2}^2 \notag \\
  &\leq& r^{2s} \mathcal{J}_{p}^{\alpha,  \nu,  \gamma} (\tilde{U}_p,\tilde{\Sigma}_{p,}, \tilde{V}_p ) + r^{2s} \left(C r^{-2s}NM \log \left(\frac{p^2}{\tilde{p}}\right) \right) + \sum_{I,J} ||\nabla_{\delta} Y_{IJ}||_{F,2}^2 \notag
\eqn
where $\nabla_{\delta}$ is the difference gradient operator defined as $( \nabla_{\delta} Y) _{i,j} = \left(Y_{i+1,j}- Y_{i,j}, Y_{i,j+1}- Y_{i,j} \right)$, the matrix $Y_{IJ}$ are the $(I,J)$-th block of the $Y$ matrices, $[\tilde{U}_p,\tilde{\Sigma}_{p}, \tilde{V}_p]$ is an argument minimum of \eqref{functional} with $Y$ replaced by $\iota_s(Y)$ and $\tilde{\Sigma}_{p,\tilde{p}}$ is the truncation of $\tilde{\Sigma}_{p}$ up to $\tilde{p}$ as stated in the previous section.

Therefore if we can choose $[\tilde{U}_p,\tilde{\Sigma}_{p}, \tilde{V}_p]$ such that $\mathcal{J}_{p}^{\alpha,  \nu,  \gamma} (\tilde{U}_p,\tilde{\Sigma}_{p}, \tilde{V}_p ) < r^{-2s} \left(( K NM  + \min(N,M)^2 \right) \varepsilon $ and $ \tilde{p} > e^{ - L \varepsilon } p^2 $, then
\beqnx
 || Y -\mathcal{I}_{s,p,\tilde{p}}^{\alpha,  \nu, \gamma } (Y) ||_{F,2}^2 \leq \left( ( K+ C L ) NM  + \min(N,M)^2 \right) \varepsilon +\sum_{I,J} ||\nabla_{\delta} Y_{IJ}||_{F,2}^2 \,.
\eqnx
Let $p_{r^{-s} N,r^{-s} M,\delta}$ be defined as in \eqref{optimal}. Then we know from the discussions in the previous section
that the probability of the above event, denoted as $E_{p, \tilde{p},\delta}$, is bounded below by
\beqnx
\mathbb{P} (E_{p, \tilde{p},\delta}) \geq
r^{-s} \min(N,M)\, p_{r^{-s}N, r^{-s}M,\delta}^{r^{-s}\max(N,M)} \left( C_3 \varepsilon \right)^{ \mu(r^{-s}N, r^{-s}M , p_{r^{-s} N,r^{-s} M,\delta} )}
\quad \m{for} ~~\gamma + 2 \nu \leq (K-1) \, \varepsilon\,.
\eqnx

In general, we have no hope that either $ ||\nabla_{\delta} Y||_{F,2}^2 $ or $ \sum_{I,J} ||\nabla_{\delta} Y_{IJ}||_{F,2}^2 $ can be controlled, since we did not impose any regularity conditions for $Y$ in general.  However, if we further assume that $Y$ has some regularity, for instance $ \sum_{I,J} ||\nabla_{\delta} Y_{IJ}||_{F,2}^2 < K_0 MN \varepsilon$, then
\beqn
 || Y - \mathcal{I}_{s,p,\tilde{p}}^{\alpha,  \nu, \gamma} (Y) ||_{F,2}^2 &\leq& \left( ( K+  CL + K_0) NM  + \min(N,M)^2 \right) \varepsilon \,.
 \notag
\eqn
Combining all the previous arguments and theorems then yield the following results.
\begin{Theorem}
Let $\varepsilon := - r^{2s} \delta \left( ( K+  CL + K_0) NM  + \min(N,M)^2 \right)^{-1}$,
and for any small $\delta>0$, $E_{p, \tilde{p},\delta}$ be the event such that the following inequality holds:
\beqn
 || Y -\mathcal{I}_{s,p,\tilde{p}}^{\alpha, \nu,  \gamma} (Y) ||_{F,2}^2 \leq \left( ( K+ C L ) NM  + \min(N,M)^2 \right) \varepsilon +\sum_{I,J} ||\nabla_{\delta} Y_{IJ}||_{F,2}^2 \,, \notag
\eqn
then if $\tilde{p}$ is chosen such that $ \tilde{p} > e^{  L \varepsilon } p^2 $,
we have for any $s$ and $\gamma + 2 \nu \leq (K-1)\, \varepsilon$ that
\beqn
\mathbb{P} \left(\bigcup_{ p } E_{p, \tilde{p},\delta} \right) \geq
r^{-s} \min(N,M)\, p_{r^{-s}N, r^{-s}M,\delta}^{r^{-s}\max(N,M)} \left( C \varepsilon \right)^{ \mu\left(r^{-s}N, r^{-s}M , p_{r^{-s} N,r^{-s} M,\delta} \right)}
\eqn
where the function $\mu( \cdot , \cdot , \cdot )$
and $p_{r^{-s}N, r^{-s}M,\delta}$ are defined as in \eqref{imp_coeff} and \eqref{optimal2} respectively.
For all $s$ and $p < r^{-s}\min(N,M)$, we have
\beqn
 \mathbb{P} \left(  E_{p, \tilde{p},\delta}  \right)
 \geq
\sum_{l = 1}^p  \left( l^{r^{-s}\max(N,M)} - (l-1)^{r^{-s}\max(N,M)}   \right) l ! \, r^{-s} \min(N,M) \, \left(C_3 \varepsilon \right)^{ \mu(r^{-s}N, r^{-s}M ,l) } \,
\eqn
whenever $ \tilde{p} > e^{  L \varepsilon } p^2 $ and $\gamma + 2 \nu \leq (K-1)\, \varepsilon$.
For a given small constant $\eta$, the growth of the summation above with respect to $p$ can be controlled by $\eta$ when $p > K_\eta \, p_{r^{-s} N ,r^{-s}M,\delta}$ for some $K_\eta$ depending only on $\eta$. Furthermore,
when the event $E_{p, \tilde{p},\delta} $ happens and the inequality $\sum_{I,J} ||\nabla_{\delta} Y_{IJ}||_{F,2}^2 < K_0 MN \varepsilon$ holds we have
\beqn
 || Y - \mathcal{I}_{s,p,\tilde{p}}^{\alpha,  \nu,  \gamma } (Y) ||_{F,2}^2 \leq \delta \,.
\eqn
\end{Theorem}

Now we can see from the above theorem that for a given threshold $\delta$ and $M=N$,
 if $Y$ has the regularity such that
 $ ||\nabla_{\delta} Y||_{F,2}^2 < \tilde{K}  \delta$ for some $\tilde{K} < 1$, then
 the lower bound of the probability of $|| Y - \mathcal{I}_{s,p,\tilde{p}}^{\alpha, \gamma}  (Y) ||_{F,2}^2 < \delta$ is higher than that of $|| Y -  \mathcal{I}_{p,\tilde{p}}^{\alpha,  \nu,   \gamma} (Y) ||_{F,2}^2 < \delta$ with an appropriately selected
 $\tilde{p}$.
Furthermore, for each $s$, the optimal choice of $p$ has the same order as $p_{r^{-s}N,r^{-s}M,\delta}$, which behaves
asymptotically like
\beqn
  p \sim r^{-s/2}\sqrt{ \frac{N}{ \log  N - 2 s \log r}}\,,
  \label{optimal111}
\eqn
with the memory complexity of $\mathcal{I}_{s,p,\tilde{p}}^{\alpha,  \nu,  \gamma}$ growing in the order $r^{-3 s/2}\sqrt{ \frac{N^3}{  \log  N - 2 s \log r }} $ as $N$ goes to infinity. This tells us that, by increasing $s$, the probability of fine approximation by NMF is increased as well as the memory complexity is decreased.
Moreover, from numerical experiments, we can observe that the resulting approximations $\mathcal{I}_{s,p,\tilde{p}}^{\alpha,  \nu, \gamma}$ from larger values of $s$ capture the coarser features of $Y$, then achieve finer and finer
features as $s$ decreases.

\section{Semi-smooth Newton method for non-negative factorizations} \label{sec4}
In this section, we propose and describe an efficient and cost-effective numerical algorithm
to realise the NMF of the image or big data $Y$ as we discussed
in the previous sections.

Instead of finding the optimal solution $[\tilde{U}_p, \Sigma_{p}, \tilde{V}_p]$ of
the functional \eqref{functional}, we shall propose to perform the following alternative two-stage NMF
to obtain an approximation of $\mathcal{I}_{p}^{\alpha, \nu, \gamma} (Y)$ in two stages:
\beqn
Y \approx A V^T \, , \q A^T \approx \Sigma^T U^T \,, \text{ then combine to get } Y \approx U \Sigma V^T \,.
\eqn
In each of the above two NMFs, we minimize the functional \eqref{mini} via a semi-smooth Newton method based on
primal-dual active sets \cite{semismooth2}, which will be derived below.
The semi-smooth Newton method is more advantageous than some classical methods \cite{clustering}
\cite{tri} and converges faster.
This two-stage process does not yield the optimal solution $[\tilde{U}_p, \Sigma_{p}, \tilde{V}_p]$ of the functional \eqref{functional}, but generates an sufficiently fine approximation of $\mathcal{I}_{p}^{\alpha, \nu, \gamma} (Y)$ as we shall
observe from our numerical experiments.
More importantly, this two-stage process is more user-friendly and less expensive computationally,
since the linearized systems of the functional \eqref{functional} involved
in the semi-smooth Newton iteration is much more convenient to evaluate numerically
than the systems encountered when one minimizes \eqref{functional} directly.

\subsection{Semi-smooth Newton method based on primal-dual active sets for NMF}

Before we present a two-stage NMF for an approximation of $\mathcal{I}_{p}^{\alpha, \nu, \gamma} (Y)$, we first discuss
some mathematical properties of the important non-convex minimisation problem \eqref{mini}.
The semi-smooth Newton method based on primal-dual active sets were proposed earlier
in \cite{semismooth2} to solve either convex or non-convex
non-smooth optimization problems effectively by combining the ideas of active sets and Newton-type update.
In this section, we formulate this method for solving the non-smooth non-convex optimization \eqref{mini}:
\begin{equation}
\min_{A\geq 0, P \geq 0} J(A,P) := ||Y-AP||_{F,2}^2+\alpha||A||_{F,1} +\nu||P||_{F,1} +\gamma||P P^T -I||_{F,1}.
\label{mini111}
\end{equation}

\subsubsection{Complementary Conditions}
We first recall two complementary conditions for the characterization of some constraints conditions from \cite{semismooth2}, which is crucial for the development of the algorithm in the subsequent analysis.
For this purpose, we will need the sub-differential of
the function $| \cdot | : \mathbb{R} \rightarrow \mathbb{R}$, which is
the set-valued signum function defined by
\beqn
\partial | \cdot | (x) =
\begin{cases}
1  & \text{ if } x > 0 \,,\\
[-1,1]  & \text{ if } x = 0 \,,\\
-1  & \text{ if } x < 0 \,.\\
\end{cases}
\eqn
We shall also often require the following complementarity condition \cite{semismooth2} which characterizes the set-valued sub-differential $\partial | \cdot | $ by
\beqn
\lambda = \frac{\lambda + c x}{ \max (1 , | \lambda + c x|)}
\eqn
for any given $c > 0$, based on the following equivalence.
\begin{Lemma} For any given constant $c>0$, it holds that
\label{theorem1}
\beqn
\lambda = \frac{\lambda + c x}{ \max (1 , | \lambda + c x|)} \Leftrightarrow
\lambda \in \partial | \cdot | (x)\,.
\eqn
\end{Lemma}
\begin{proof}
First we assume $\lambda = \frac{\lambda + c x}{ \max (1 , | \lambda + c x|)}$.
If $ | \lambda + c x| \leq 1$, then $\lambda = \lambda + c x$, which gives $x = 0$, hence $|\lambda| \leq 1$ and $\lambda \in \partial | \cdot | (x) $.
For $| \lambda + c x| > 1$, we know $\lambda = \frac{\lambda + c x}{ | \lambda + c x|} = \pm 1$.  If
$\lambda = 1$, then $ | 1 + c x| > 1$, which directly gives $x > 0$,
therefore $\lambda \in \partial | \cdot | (x) $. The case for $\lambda = -1$ is similar.

Now we assume that $\lambda \in \partial | \cdot | (x) $.
If $x=0$, then $|\lambda| \leq 1$,
therefore $\lambda = \frac{\lambda }{ \max (1 , | \lambda |)} = \frac{\lambda + c x}{ \max (1 , | \lambda + c x|)}  $.
Furthermore, if $x > 0$, then $\lambda =  1 $ and $|\lambda + c x| > 1$, therefore $\frac{\lambda + c x}{ \max (1 , | \lambda + c x|)}  = \frac{\lambda + c x}{ | \lambda + c x|}  = \lambda $.
The case for $x <0$ is similar.
\end{proof}
Note that in the above complementary condition, the choice of $c$ is arbitrary. However, in a practical implementation
using the complementary condition, $c$ is often chosen as a fixed constant that acts as a stabilisation parameter.

Now for any matrix $A$, we note that $||A||_{F,1} = \sum_{i,j} |A_{i,j}|$. Then the set-valued sub-differential function $ \partial || \cdot ||_{F,1} (A) $ is given by
\beqn
\left( \, \partial || \cdot ||_{F,1} [A] \, \right)_{i,j} =
\begin{cases}
1  & \text{ if } A_{i,j} > 0 \,,\\
[-1,1]  & \text{ if } A_{i,j} = 0 \,,\\
-1  & \text{ if } A_{i,j} < 0 \,.\\
\end{cases}
\eqn
Using the complementarity condition \eqref{theorem1} for a dual variable $\lambda$, we have
\beqn
\lambda_{i,j} = \frac{\lambda_{i,j} + c A_{i,j}}{ \max (1 , | \lambda_{i,j} + c A_{i,j}|)} \Leftrightarrow
\lambda \in \partial || \cdot ||_{F,1} (A)  \,.
\eqn
We may often write this simply as
$\lambda = \frac{\lambda + c A}{ \max (1 , | \lambda + c A|)} $, where the division, the maximum and the absolute value
are all taken point-wise.

Next we introduce a second complementary condition that is used to characterise an inequality constrain $x \geq 0$
\cite{semismooth2}.  We sketch the argument from \cite{semismooth2} to motivate
our desired complementary condition. For a functional $F : \mathbb{R}^N \rightarrow \mathbb{R}$,
consider the constrained optimization:
\begin{equation}
\min F(x) \quad  \text{ subject to } \quad  x \geq 0 \,. \label{optimality}
\end{equation}
We introduce its following equivalent augmented Lagrangian formulation with the same necessary optimality condition and
a dummy variable $z$ and a Lagrangian variable $\mu$:
\begin{equation}\l{eq:aug2}
\min F(x) + \langle \mu, x- z \rangle + \frac{c}{2} || x - z ||_2^2 \quad \text{ subject to } \quad  x = z \text{ and } z \geq 0 \,.
\end{equation}
This functional is clearly convex in $z$. Minimizing it over $z \geq 0$,
we obtain the following entry-wise necessary and sufficient conditions for $z$:
\beqn
\begin{cases}
    \mu_i   + c  (x_i- z_i)  < 0   & \text{ if } z_i = 0 \, ,\\
    \mu_i   + c  (x_i- z_i)  = 0   & \text{ if } z_i > 0 \,
\end{cases}
\q \m{or} \q
\begin{cases}
    0 > \mu_i  + c  x_i     & \text{ if } z_i = 0 \, ,\\
    z_i = \frac{\mu_i   + c  x_i}{c}  & \text{ if } z_i > 0 \,,
\end{cases}
\eqn
which gives the following unique minimizer for the variable $z$:
\beqn
z_i = \max \Big( 0, \frac{\mu_i   + c  x_i}{c} \Big)\,.
\eqn
Or we will also write it simply as $z = \max ( 0, \frac{\mu   + c  x }{c} )$.
Using this, we can directly compute
\beqn
& & \langle \mu, x- z \rangle + \frac{c}{2} || x - z ||_2^2 \notag \\
&=& \frac{1}{c} \langle \mu,  \min( c x, - \mu ) \rangle + \frac{1}{2c} || \min ( c x, - \mu )  ||_2^2 \notag \\
&=& \frac{1}{2c}\left(  || \min ( c x, - \mu ) + \mu  ||_2^2 - || \mu ||^2_2 \right) \notag \\
&=& \frac{1}{2c}\left(  || \min ( \mu + c x, 0 ) ||_2^2 - || \mu ||^2_2 \right) \,.
\eqn
Substituting this expression into the functional in \eqref{eq:aug2},
we obtain its equivalent minimization:
\begin{equation}
\min F(x) +   \frac{1}{2c}\Big(  || \min ( \mu + c x, 0 ) ||_2^2 - || \mu ||^2_2 \Big)\,,
\end{equation}
whose necessary optimality conditions are given by the following set-valued equations:
\begin{equation}
\begin{cases}
 0 \in \partial F(x) + \min ( \mu + c x, 0 )  \partial \min ( \cdot , 0 ) \left( \min ( \mu + c x, 0 ) \right) \,,\\
 0 \in  - \mu + \min ( \mu + c x, 0 )  \partial \min ( \cdot , 0 ) \left( \min ( \mu + c x, 0 ) \right) \,.
\end{cases}
 \label{opt}
\end{equation}
Equivalently, by a point-wise comparison, we know
$\min ( \mu + c x, 0 )  \partial \min ( \cdot , 0 ) \left( \min ( \mu + c x, 0 ) \right)= 0$ if $\mu + c x \leq 0$.
Then we see from the above necessary optimality condition that
$\mu = 0$ and $\min ( \mu + c x, 0 ) = 0$,
therefore $\mu = \min ( \mu + c x, 0 ) $. On the other hand,
if $\mu + c x > 0$, we obtain that
$\min ( \mu + c x, 0 )  \partial \min ( \cdot , 0 ) \left( \min ( \mu + c x, 0 ) \right) = \min ( \mu + c x, 0 ) $.
This, along with the necessary optimality condition, yields that $\mu = \min ( \mu + c x, 0 )$.
Therefore, by combining the above two cases we arrive at an equivalent optimality condition for $\mu$,
$\mu = \min ( \mu + c x, 0 )$.  This leads us to the following necessary optimality conditions.
\begin{Theorem}
\label{theorem2}
The necessary optimality conditions for the minimization problem \eqref{optimality} are given by
\begin{equation}
 0 \in \partial F(x) + \mu
 \quad \text{ and } \quad  \mu = \min ( \mu + c x, 0 ) \,.
\end{equation}
\end{Theorem}
The condition $\mu = \min ( \mu + c x, 0 )$ for the dual variable $\mu$ is regarded as a complementary condition in \cite{semismooth2},  which serves as a characterization of the constraint $x \geq 0$.  This complementary condition may also be regarded as a project of the solution to the convex set as the epigraph defined by the constraint.

\subsubsection{Necessary optimality conditions for the optimization \eqref{mini111}}
By directly applying Theorem\,\ref{theorem2} and calculating the sub-differentials involved,
we come to the necessary optimality conditions for the optimisation
\eqref{mini111} using the primal-dual variables $( A,P, \mu_A,\mu_P )$ for a given $c_1$:
\begin{equation*}
\begin{cases}
 0 \in \partial_A  J(A,P)  + \mu_A = 2 A P P^T - 2 Y P^T   + \mu_A + \alpha \partial || \cdot ||_{F,1} (A) \,, \\
 \mu_A = \min ( \mu_A + c_1 A, 0 ),\\
 0 \in \partial_P  J(A,P)  + \mu_P = - 2 A^T Y + 2 A^T A P + \mu_P + \nu \partial || \cdot ||_{F,1} (P) +
 \gamma \{ \partial || \cdot ||_{F,1} (P P^T - I) \}\left( P + T \circ P^T \circ T \right),\\
 \mu_P = \min ( \mu_P + c_1 P, 0 )\,
\end{cases}
\end{equation*}
where $T: M_{M \times N} \rightarrow M_{N \times M}$ is the transpose operator that maps $A$ to $A^T $.
Now, applying Lemma\,\ref{theorem1} to the above system and introducing two more variables $R, L$,
we obtain the following optimality conditions.
\begin{Theorem}
The necessary optimality conditions for the optimisation \eqref{mini111} can be given in terms of
the primal-dual variables $( A,P, R,L, \mu_A, \lambda_A, \mu_P, \lambda_P, \lambda_L )$ and
two constants $c_1,c_2$ by
\begin{equation}
\begin{cases}
 0 &= 2 A P P^T - 2 Y P^T   + \mu_A + \alpha \lambda_A  \\
 \lambda_A &= \frac{\lambda_A + c_2 A}{ \max (1 , | \lambda_A + c_2 A|)}  \\
 \mu_A &= \min ( \mu_A + c_1 A, 0 )  \\
 0  &=  - 2 A^T Y + 2 A^T A P + \mu_P + \nu \lambda_P + \gamma \lambda_L R \\
 \lambda_P &= \frac{\lambda_P + c_2 P}{ \max (1 , | \lambda_P + c_2 P|)}  \\
 L  & = P P^T - I \\
 R  &= P \circ T + T \circ P^T \circ T \\
 \lambda_L &= \frac{\lambda_L + c_2 L }{ \max (1 , | \lambda_L + c_2 L |)}  \\
 \mu_P &= \min ( \mu_P + c_1 P, 0 ) \,.
\end{cases}
\label{veryimp2}
\end{equation}
\end{Theorem}

\subsubsection{Semi-smooth Newton strategy}

We derived the necessary optimality conditions for solving the optimization problem \eqref{mini111}
in the previous subsection. We shall now develop a semi-smooth Newton method for solving these optimality systems,
which can be readily shown to be Newton differentiable \cite{semismooth2}.
To further develop our algorithm, we separate the variables $( A,P, R,L, \mu_A, \lambda_A, \mu_P, \lambda_P, \lambda_L )$  into three sets, i.e., $( A, \mu_A, \lambda_A )$, $( P,  \mu_P , \lambda_P )$  and $( L, R, \lambda_L)$,
and solve for each set of variables independently.
Clearly, the separated systems are easier for us to perform active set techniques and greatly reduce
the computational costs, and more importantly, each separated nonlinear system consists of much fewer variables,
and is therefore much more stable when performing semi-smooth Newton iterations.
With these motivations, we separate \eqref{veryimp2} into three sets of equations:

\ss
\noindent (1) For a fixed $P$, solve the system for $( A, \mu_A, \lambda_A )$:
\begin{eqnarray}
\begin{cases}
 0 &= 2 A P P^T - 2 Y P^T   + \mu_A + \alpha \lambda_A\,,  \\
 \lambda_A &= \frac{\lambda_A + c_2 A}{ \max (1 , | \lambda_A + c_2 A|)}\,,  \\
 \mu_A &= \min ( \mu_A + c_1 A, 0 ) \,.
\end{cases}
\label{veryimpsep1}
\end{eqnarray}
(2) For the fixed $A, L, R, \lambda_L$, solve the system for $( P, \mu_P, \lambda_P )$:
\begin{eqnarray}
\begin{cases}
 0  &=  - 2 A^T Y + 2 A^T A P + \mu_P + \nu \lambda_P  + \gamma \lambda_L  R \\
 \lambda_P &= \frac{\lambda_P + c_2 P}{ \max (1 , | \lambda_P + c_2 P|)}  \\
 \mu_P &= \min ( \mu_P + c_1 P, 0 ) \,.
\end{cases}
\label{veryimpsep2}
\end{eqnarray}
(3) For a fixed $P$, solve the system for $( L, R, \lambda_L)$:
\begin{eqnarray}
\begin{cases}
 L  & = P P^T - I \\
 R  &= P \circ T + T \circ P^T \circ T \\
 \lambda_L &= \frac{\lambda_L + c_2 L }{ \max (1 , | \lambda_P + c_2 L |)} \,.
\end{cases}
\label{veryimpsep3}
\end{eqnarray}
Now we introduce the following active and inactive sets:
\beqn
& \mathcal{A}_{A,1} = \{ (i,j)  \, : \,   (\mu_A)_{i,j} + c_1 A_{i,j} > 0  \}\,,
\quad
& \mathcal{I}_{A,1} = \{ (i,j)  \, : \,   (\mu_A)_{i,j} + c_1 A_{i,j} \leq 0  \}\,, \notag \\
&\mathcal{A}_{A,2} = \{ (i,j)  \, : \,   | (\lambda_A)_{i,j} + c_2 A_{i,j}| \leq 1  \}\,,
\quad
&\mathcal{I}_{A,2}  = \{ (i,j)  \, : \,  | (\lambda_A)_{i,j} + c_2 A_{i,j}| > 1 \}\,, \notag \\
&\mathcal{A}_{P,1} = \{ (i,j)  \, : \,   (\mu_P)_{i,j} + c_1 P_{i,j} > 0  \}\,,
\quad
&\mathcal{I}_{P,1} = \{ (i,j)  \, : \,   (\mu_P)_{i,j} + c_1 P_{i,j} \leq 0  \}\,, \notag \\
&\mathcal{A}_{P,2} = \{ (i,j)  \, : \,   | (\lambda_P)_{i,j} + c_2 P_{i,j}| \leq 1  \}\,,
\quad
&\mathcal{I}_{P,2}  = \{ (i,j)  \, : \,  | (\lambda_P)_{i,j} + c_2 P_{i,j}| > 1 \}\,, \notag \\
&\mathcal{A}_{L} = \{ (i,j)  \, : \,   | (\lambda_L)_{i,j} + c_2 L_{i,j} | \leq 1  \}\,,
\quad
&\mathcal{I}_{L} = \{ (i,j)  \, : \,  | (\lambda_L)_{i,j} + c_2 L_{i,j} | > 1 \} \,, \notag
\eqn
then we can further reduce the previous 3 systems into the following much simpler ones
thanks to direct substitutions and point-wise comparisons of the complementary conditions:

\noindent (1) For a fixed $P$, we have $A= 0$ on $\mathcal{A}_{A,1} \bigcup \mathcal{A}_{A,2}$,
while $( A, \lambda_A )$ on $\mathcal{I}_{A,1} \bigcap \mathcal{I}_{A,2}$ satisfies
\begin{eqnarray}
\begin{cases}
 0 &= 2 A P P^T - 2 Y P^T  + \alpha \lambda_A \,, \\
 0 &= \lambda_A | \lambda_A + c_2 A|- (\lambda_A + c_2 A) \,.
\end{cases}
\label{veryimpsep1a}
\end{eqnarray}
(2) For the fixed $A, L, R, \lambda_L$, we have $P= 0$ on $\mathcal{A}_{P,1} \bigcup \mathcal{A}_{P,2}$,
while $(P, \lambda_P)$ on $\mathcal{I}_{P}$ ssatisfies
\begin{eqnarray}
\begin{cases}
 0  &=  - 2 A^T Y + 2 A^T A P + \nu \lambda_P  + \gamma \lambda_L  R\,, \\
 0 &= \lambda_P | \lambda_P + c_2 P|- (\lambda_P + c_2 P) \,.
\end{cases}
\label{veryimpsep2a}
\end{eqnarray}
(3) For a fixed $P$, we have $L= 0$ on $\mathcal{A}_{L}$, while $( L, R, \lambda_L)$
on $\mathcal{I}_{L}$ satisfies
\begin{eqnarray}
\begin{cases}
 L  & = P P^T - I\,, \\
 R  &= P \circ T + T \circ P^T \circ T\,, \\
 0 &= \lambda_L | \lambda_L + c_2 L|- (\lambda_L + c_2 L) \,.
\end{cases}
\label{veryimpsep3a}
\end{eqnarray}

For the nonlinear constraints with $\lambda_A, \lambda_P$ and $\lambda_L$,
we propose a semi-smooth Newton-step update as in \cite{semismooth} to solve the corresponding equations.
One might suggest the explicit Uzawa iteration \cite{semismooth2} instead,
but it is only conditionally stable and converges slowly.
We shall give only a sketch of the derivation of the semi-smooth Newton update,
following the general principle in \cite{semismooth}. We first consider the system \eqref{veryimpsep1a}.
Assume that $(A, \lambda_A)$ are perturbed to $(A^h, \lambda_A^h)$ such that
the increment is of order $O(h)$ and satisfies the second equation in \eqref{veryimpsep1a}.
Then we can derive
\beqn
 \lambda_A^h| \lambda_A + c_2 A|  + \lambda_A \left( \frac{\lambda_A + c_2 A}{| \lambda_A + c_2 A|} (\lambda_A^h + c_2 A^h - \lambda_A - c_2 A) \right) - (\lambda_A^h + c_2 A^h)  = O(h^2)\,,
\notag
\eqn
which gives the following Newton update from $(A, \lambda_A)$  to $(A^+, \lambda_A^+)$:
\beqn
\lambda_A^+ | \lambda_A + c_2 A|  + \lambda_A \frac{\lambda_A + c_2 A}{| \lambda_A + c_2 A|} \left(  \lambda_A^+ + c_2 A^+  \right)
= \lambda_A | \lambda_A + c_2 A| + (\lambda_A^+ + c_2 A^+)\,. \notag
\eqn
Now following \cite{semismooth2}, we suggest the following Newton update involving
damping and regularization:
\beqn
\lambda_A^+ | \lambda_A + c_2 A|  + \frac{\lambda_A + c_2 A}{| \lambda_A + c_2 A|} \left(  \lambda_A^+ + c_2 A^+  \right) \frac{\theta \lambda_A }{\max(1,|\lambda_A|)}
=| \lambda_A + c_2 A| \frac{\theta \lambda_A }{\max(1,|\lambda_A|)} + (\lambda_A^+ + c_2 A^+) \notag
\eqn
where $\theta$ is a stability parameter and the regularizer $\lambda_A / \max(1,|\lambda_A|)$ is set to automatically
restrict $\lambda_A$ to be in $[-1,1]$. Following \cite{semismooth}, we set $\theta \, | \lambda_A + c_2 A| / \left( | \lambda_A + c_2 A| -1 + \theta \frac{\lambda_A (\lambda_A + c_2 A) }{\max(1,|\lambda_A|) | \lambda_A + c_2 A|} \right) =1$, which gives $\theta \leq 1$ to simplify the iteration and leads to the following update after direct substitution:
\beqn
0 = \lambda_A^+ - c_2\frac{1 - a_A b_A }{ d_A - 1} A^+ + a_A \notag
\eqn
where $a_A = \frac{\lambda_A }{\max(1,|\lambda_A|)}$, $b_A = \frac{ \lambda_A + c_2 A}{| \lambda_A + c_2 A|}$ and $d_A = | \lambda_A + c_2 A|$, which is used as the semi-smooth update for the first system
\begin{eqnarray}
\begin{cases}
 0 &= 2 A^+ P P^T - 2 Y P^T  + \alpha \lambda_A^+  \\
 0 &= \lambda_A^+ -  \frac{c_2 }{ d_A - 1} \left( I - a_A b_A^T \right) A^+ + a_A\,.
\end{cases}
\end{eqnarray}
We can linearize the constraints for the variables $\lambda_P$ and $\lambda_L$ similarly.

We may solve the third system \eqref{veryimpsep3a}
for the other two variables $(L, R)$, but it is actually not an easy job.
Although the second equation in \eqref{veryimpsep3a} is linear,
it is computationally expensive as the transpose operator $T$ is involved.
We therefore derive a semi-smooth Newton update for $R$ from $L$ and $P$ instead of a direct substitution.
Assume $(L,R,P)$ are perturbed to $(L^h,R^h,P^h)$ such that the increment is of order $O(h)$ and satisfies $L  = P P^T - I$, we then have
\beqn
(L^h - L) = R^h (P^h - P) + O(h^2) \,, \notag
\eqn
which suggests the following update for $R$:
\beqn
 R^+ (P^+ - P) &= (L^+ - L) \,. \notag
\eqn
Combining this update with the aforementioned strategy for $\lambda_L$, we obtain the following semi-smooth Newton update from $(L,R,P)$ to $(L^+,R^+,P^+)$ for the third system \eqref{veryimpsep3a}:
\begin{eqnarray}
\begin{cases}
 L^+  & = P^+ (P^+)^T - I\,,\\
 R^+ (P^+ - P) &= (L^+ - L)\,, \\
 \lambda_L^+ &=   \frac{c_2 }{ d_L - 1} \left( I - a_L b_L^T \right) L^+ - a_L
\end{cases}
\end{eqnarray}
where $a_L$, $b_L$ and $d_L$ are given by
$a_L = \frac{\lambda_L }{\max(1,|\lambda_L|)}$, $b_L = \frac{ \lambda_L + c_2 L}{| \lambda_L + c_2 L|}$ and $d_L = | \lambda_L + c_2 L|$.
%

\subsubsection{Numerical algorithms}

Combining all the techniques and results from the previous subsections, we are ready to propose
the semi-smooth Newton method based on primal-dual active sets
for solving the optimality system \eqref{veryimp2} to tackle the minimization problem \eqref{mini111}.

\ss
{\bf Semi-smooth Newton Algorithm 1}.
Given two constants $c_1, c_2$; initialize $( A^0,P^0, \mu_A^0, \lambda_A^0, \mu_P^0, \lambda_P^0, \lambda_L^0 )$.

For $k = 0, 1, ..., K$,  do the following steps\,:
\bn
\i   
Compute
$
\mu_A^{(k)} := - 2 A^{(k)} P^{(k)} {P^{(k)}}^T + 2 Y (P^{(k)})^T  - \alpha \lambda_A^{(k)}\,.
$

\i   
Set the active and inactive sets  $\mathcal{A}_{A,i}^k $ and $\mathcal{I}_{A,i}^k$ for $i = 1,2$\,:
\beqn
& \mathcal{A}_{A,1}^{(k)} = \{ (i,j)  \, : \,   (\mu_A)_{i,j}^{(k)} + c_1 A_{i,j}^{(k)} > 0  \}\,,
\quad
& \mathcal{I}_{A,1}^{(k)} = \{ (i,j)  \, : \,   (\mu_A)_{i,j}^{(k)} + c_1 A_{i,j}^{(k)} \leq 0  \}\,, \notag \\
&\mathcal{A}_{A,2}^{(k)} = \{ (i,j)  \, : \,   | (\lambda_A)_{i,j}^{(k)} + c_2 A_{i,j}^{(k)}| \leq 1  \}\,,
\quad
&\mathcal{I}_{A,2}^{(k)}  = \{ (i,j)  \, : \,  | (\lambda_A)_{i,j}^{(k)} + c_2 A_{i,j}^{(k)}| > 1 \} \,. \notag
\eqn

\i 
Compute $ a_A^{(k)}, b_A^{(k)}, d_A^{(k)} $\,:
\beqn
a_A^{(k)} := \frac{\lambda_A^{(k)} }{\max(1,|\lambda_A^{(k)}|)} \, , \quad b_A^{(k)} := \frac{ \lambda_A^{(k)} + c_2 A^{(k)}}{| \lambda_A^{(k)} + c_2 A^{(k)}|},  \quad d_A^{(k)} := | \lambda_A^{(k)} + c_2 A^{(k)}| \,. \notag
\eqn

\i 
Set $A^{(k+1)}:= 0$ on $\mathcal{A}^{(k)}_{A,1} \bigcup \mathcal{A}^{(k)}_{A,2}$;
~solve the system for $( A^{(k+1)}, \lambda_A^{(k+1)} )$  on $\mathcal{I}^{(k)}_{A,1} \bigcap \mathcal{I}^{(k)}_{A,2}$\,:
\begin{eqnarray}
\begin{cases}
0 &= 2 A^{(k+1)} P^{(k)} {P^{(k)}}^T - 2 Y (P^{(k)})^T  + \alpha \lambda_A^{(k+1)}\,,  \\
0 &= \lambda_A^{(k+1)} - \frac{ c_2 }{ d_A^{(k)} - 1}  \left(I - a_A^{(k)} [b_A^{(k)}]^T \right)  A^{(k+1)} + a_A^{(k)}\,.
\end{cases}
\notag
\end{eqnarray}


\i 
Compute
$
\mu_P^{(k)} :=   2 (A^{(k+1)})^T Y - 2(A^{(k+1)})^T A^{(k+1)} P^{(k)} - \nu  \lambda_P^{(k)} - \gamma \lambda_L^{(k)}  R^{(k)} \,.
$

\i 
Set the active and inactive sets $\mathcal{A}_{P,i}^k $ and $\mathcal{I}_{P,i}^k$ for $i = 1,2$\,:
\beqn
& \mathcal{A}_{P,1}^{(k)} = \{ (i,j)  \, : \,   (\mu_P)_{i,j}^{(k)} + c_1 P_{i,j}^{(k)} > 0  \}\,,
\quad
& \mathcal{I}_{P,1}^{(k)} = \{ (i,j)  \, : \,   (\mu_P)_{i,j}^{(k)} + c_1 P_{i,j}^{(k)} \leq 0  \}\,, \notag \\
&\mathcal{A}_{P,2}^{(k)} = \{ (i,j)  \, : \,   | (\lambda_P)_{i,j}^{(k)} + c_2 P_{i,j}^{(k)}| \leq 1  \}\,,
\quad
&\mathcal{I}_{P,2}^{(k)}  = \{ (i,j)  \, : \,  | (\lambda_P)_{i,j}^{(k)} + c_2 P_{i,j}^{(k)}| > 1 \} \,. \notag
\eqn

\i 
Compute $ a_P^{(k)}, b_P^{(k)}, d_P^{(k)} $\,:
\beqn
a_P^{(k)} := \frac{\lambda_P^{(k)} }{\max(1,|\lambda_P^{(k)}|)} \, , \quad b_P^{(k)} := \frac{ \lambda_P^{(k)} + c_2 P^{(k)}}{| \lambda_P^{(k)} + c_2 P^{(k)}|},  \quad d_P^{(k)} := | \lambda_P^{(k)} + c_2 P^{(k)}| \,. \notag
\eqn

\i 
Set $P^{(k+1)} := 0$ on $\mathcal{A}^{(k)}_{P,1} \bigcup \mathcal{A}^{(k)}_{P,2}$\,;
~solve the system for $( P^{(k+1)}, \lambda_P^{(k+1)} )$ on $\mathcal{I}^{(k)}_{P,1} \bigcap \mathcal{I}^{(k)}_{P,2}$\,:
\begin{eqnarray}
\begin{cases}
 0  &=  - 2 (A^{(k+1)})^T Y + 2(A^{(k+1)})^T A^{(k+1)} P^{(k+1)} + \nu \lambda_P^{(k+1)} + \gamma \lambda_L^{(k)}  R^{(k)} \\
 0 &= \lambda_P^{(k+1)} - \frac{ c_2 }{ d_P^{(k)} - 1}  \left(I - a_P^{(k)} [b_P^{(k)}]^T \right)  P^{(k+1)} + a_P^{(k)}\,.
 \end{cases}
 \notag
\end{eqnarray}


\i 
Set the active and inactive sets $\mathcal{A}_{L}^{(k)}$ and $ \mathcal{I}_{L}^{(k)} $\,:
\beqn
&\mathcal{A}_{L}^{(k)} = \{ (i,j)  \, : \,   | (\lambda_L)_{i,j}^{(k)} + c_2 L_{i,j}^{(k)} | \leq 1  \}\,,
\quad
&\mathcal{I}_{L}^{(k)} = \{ (i,j)  \, : \,  | (\lambda_L)_{i,j}^{(k)} + c_2 L_{i,j}^{(k)} | > 1 \} \,. \notag
\eqn

\i 
Compute $ a_L^{(k)}, b_L^{(k)}, d_L^{(k)} $\,:
\beqn
a_L^{(k)} := \frac{\lambda_L^{(k)} }{\max(1,|\lambda_L^{(k)}|)} \, , \quad b_L^{(k)} := \frac{ \lambda_L^{(k)} + c_2 L^{(k)}}{| \lambda_L^{(k)} + c_2 L^{(k)}|}, \quad d_L^{(k)} := | \lambda_L^{(k)} + c_2 L^{(k)}| \,.  \notag
\eqn

\i 
Set $L^{(k+1)}= 0$ on $\mathcal{A}^{(k)}_{L}$\,;
~evaluate $( L^{(k+1)}, R^{(k+1)}, \lambda_L^{(k+1)})$ on $\mathcal{I}^{(k)}_{L}$\,:
\begin{eqnarray}
\begin{cases}
 L^{(k+1)}  & = P^{(k+1)} (P^{(k+1)})^T - I\,, \\
 R^{(k+1)} (P^{(k+1)} - P^{(k)}) &= \left(L^{(k+1)} - L^{(k)} \right) \,,\\
 \lambda_L^{(k+1)} &=   \frac{ c_2 }{ d_L^{(k)} - 1}  \left(I - a_L^{(k)} [b_L^{(k)}]^T \right)  L^{(k+1)} - a_L^{(k)}\,.
\end{cases}
 \notag
\end{eqnarray}
\en

A natural choice of the stopping criterion is based on the changes of the active sets: if the active sets for two consecutive iterations are the same, we may stop the iteration \cite{semismooth2}.
As the iteration goes on, $A,P,L$ become more and more sparse, and the sizes of the linear systems involved
drop drastically, so the inversions of the linear systems are more stable and
less expensive computationally.

Finally, a few remarks are in order for effective implementations of the algorithm\,:
\bn
    \item
        With the enforcement of the constraints $A,P \geq 0$ by the dual variables $\mu_A, \mu_P$, the algorithm ensures naturally
        $A^{(k)} , P^{(k)} \geq 0$ for all $k$ if the initial guesses $A^{(0)}$ and $P^{(0)}$ are set to be non-negative.
        Thus the algorithm can be simplified by setting the dual variables $\lambda_A^{(k)}$ and $\lambda_P^{(k)}$ to be
        $ \lambda_A^{(k)} = \lambda_P^{(k)}=1$ and drop the active/inactive sets $\mathcal{A}_{A,2}^{(k)} $, $\mathcal{I}_{A,2}^{(k)}$, $\mathcal{A}_{P,2}^{(k)} $ and $\mathcal{I}_{P,2}^{(k)}$.
    \item
        In order to further simplify the algorithm, we may normalize the row vectors of $P$ after \textbf{Step 8} so that $PP^T$ has unitary diagonal entries. If this normalisation is added, then $L^{(k)} \geq 0$ for all $k$.  In this case
        $\lambda_L^{(k)}$ can be simply set to be $\lambda_L^{(k)}=1$ while
        $\mathcal{A}_{L}^{(k)}$ and $\mathcal{I}_{L}^{(k)}$ can be dropped.
    \item
        In the development of our algorithm above, we assume $Y\geq 0$ entry-wise, therefore it is natural to enforce
        the constraint $A \geq 0$.  This non-negativity condition for $A$ is however infeasible and shall be dropped if $Y$ is not non-negative entry-wise. In this case, nonetheless, we can still utilize the above algorithm for a non-negative factorization with the following minor modification: drop the dual variable $\mu_A$ and the active/inactive sets $\mathcal{A}_{A,1}^{(k)} $ and $\mathcal{I}_{A,1}^{(k)}$.
\en

\subsection{Non-negative matrix factorization of an image}

With \textbf{Semi-smooth Newton Algorithm 1} to minimize the functional \eqref{mini111}, we are ready to propose
an algorithm
to approximate $\mathcal{I}_{p}^{\alpha, \nu, \gamma} (Y)$ in \eqref{approximation} and $\mathcal{I}_{p,\tilde{p}}^{\alpha, \nu, \gamma} (Y)$ in \eqref{approximation2} for the NMF of an image $Y$\,.

\ss
\textbf{Non-negative Matrix Factorization Algorithm 2}.
Specify $5$ parameters $\alpha$, $\nu$, $\gamma$, $p$, $\tilde{p}$.
\bn
\i 
Apply \textbf{Semi-smooth Newton Algorithm 1} to find
a minimizer $ [A_{0}, V_{0}]$ of the problem\,:
\begin{equation*}
\min_{A\geq 0, V\geq 0}||Y-AV^T||_{F,2}^2+\alpha||A||_{F,1} +\nu||V||_{F,1} +\gamma||V^T V-I||_{F,1}.
\end{equation*}

\i 
Apply \textbf{Semi-smooth Newton Algorithm 1} to find
a minimizer $ [\Sigma_{0}, U_{0}]$ of the problem\,:
\begin{equation*}
\min_{\Sigma \geq 0, U \geq 0}||A_{0}^T-\Sigma^T U^T||_{F,2}^2+\alpha||\Sigma||_{F,1} +\nu||U||_{F,1} +\gamma||U^T U-I||_{F,1}.
\end{equation*}
\i 
Form
$
\mathcal{I}_{p}^{\alpha, \nu,  \gamma} (Y) : = U_{0}  \Sigma_{0}  V_{0}^T \,
$
from $[U_{0} , \Sigma_{0} , V_{0}]$\,.
\i 
Sort the entries of $\Sigma_0$ from the largest to the smallest as $\sigma_{i_1 j_1} \geq \sigma_{i_2 j_2} \geq ..\geq \sigma_{i_{p^2} j_{p^2}}$.

\i 
Compute $\tilde{\sigma}_l := \sigma_{i_l j_l} e_l \otimes e_l $, then form
$
\Sigma_{0, \tilde{p}} := \sum_{l=1}^{\tilde{p}} \tilde{\sigma}_l \,.
$
\i 
Form the factorisation
$
\mathcal{I}_{p,\tilde{p}}^{\alpha, \nu,  \gamma} (Y) := U_{0}  \Sigma_{0, \tilde{p}}  V_{0}^T  \,.
$
\en

\subsection{Multi-level analysis algorithm based on NMF}

Based on the results from a NMF, we can propose a multi-level analysis algorithm.

\ss
{\bf Multi-level Analysis Algorithm 3}. Specify a scaling parameter $r$ and a constant
$s_{max}$ such that

$ s_{max}< \log N / \log r$;  ~set parameters $\alpha$, $\nu$, $\gamma$ and 2 arrays of parameters
$[p(1),... , p(s_{max})]$, $ [\tilde{p}(1),... , \tilde{p}(s_{max})]$.

For $s =1, 2, ..., s_{max}$, do the following steps\,:
%
\bn
\i 
Compute $\iota_s (Y)$ as in \eqref{interpolation}.

\i 
Calculate $ \mathcal{I}_{p(s),\tilde{p}(s)}^{\alpha, \nu, \gamma} [ \iota_s (Y) ]$ by
Non-negative Matrix Factorization Algorithm 2.

\i 
Calculate
$\mathcal{I}_{s,p(s),\tilde{p}(s)}^{\alpha,\nu,  \gamma} (Y) := \iota_s^T \circ \mathcal{I}_{p(s),\tilde{p}(s)}^{\alpha, \nu,  \gamma} \circ \iota_s (Y) $.
\en

\section{Applications to photo images, EIT and DOT images} \label{sec5}
In this section we shall apply both the NMF and the MLA framework of a NMF suggested in
Section\,\ref{sec4} to some photo images and several EIT and DOT images reconstructed by
some direct sampling methods.
We shall investigate two applications, the first one being an MLA for photo images using NMF,
and the second one being an NMF over the images from an inversion algorithm
for a broad class of coefficient determination inverse problems.
In the first application, we aim at capturing features of different scales in an image and obtain a sparse low-rank representation of these features; while in the second application, we hope to identify the principal components in the image, which correspond to the signals coming from the inhomogeneous coefficients to be determined in the corresponding inverse problems, and remove artifacts and noise from the images.

\subsection{Applications to photo images}
We shall now perform an MLA using NMF for several grey-scaled images $Y$.
In view of the fact that an image can be represented by a positive function, and so are the major structures/objects inside these images,
we are naturally motivated to use the NMF to identify the principal components of the image corresponding to these major objects in the figure, and obtain a sparse representation of these objects and structures.
MLA is employed to obtain these corresponding principal components representing structures/objects at multiple scales/levels of the image, so that structures of large scales and small scales in the image can be separately identified and sparsely represented.  We shall also aim to obtain a sparse representation which is robust to noise during transmission of data through channels.
But we would like to emphasize that
we are neither aiming at reconstructing the image in full entity from all the NMF components in terms of tensor products, nor hoping to obtain a very high compression ratio of memory complexity to defeat any well-developed compression techniques, e.g. wavelet/curvelet compression, JPEG etc, since they are surely better candidates for compressions.
Our major purpose is instead to identify and keep structures in the images in a robust manner.

In the subsequent $4$ examples, we shall utilize the Multi-level Analysis Algorithm 3
to approximate $\mathcal{I}_{s,p(s),\tilde{p}(s)}^{\alpha, \nu, \gamma} (Y)$, in which the Non-negative Matrix Factorization Algorithm 2 is used to calculate $ \mathcal{I}_{p(s),\tilde{p}(s)}^{\alpha, \nu, \gamma} [ \iota_s (Y) ]$ and
the Semi-smooth Newton Algorithm 1 is used to minimize \eqref{mini111} for the NMF.
In all the following examples, the parameters in \textbf{Algorithm 3} are set to $$r=2, \quad \alpha=0.2, \quad \nu = 0.02, \quad \gamma = 0.02, $$
whereas $s_{max}$ is set differently in each example.
Considering the theoretical asymptotic order for an optimal choice of $p$ as in \eqref{optimal111},
the array of parameters $p(s)$ is set to
$$p(s)= \left[ K_1 \sqrt{ \frac{\max(N,M)}{ \max(1 , \log  \max(N,M) - 2 s \log r) }} r^{-s/2} \right] $$
in all our examples, where $[\cdot]$ denotes the round-off function and $K_1$ is a given constant.
We observe from numerical experiments that this asymptotic formula \eqref{optimal111} is, on one hand, necessary for good approximation of the desirable structures we hope to identify, and on the other hand, grows fairly slowly as the value $s_{max}-s$ grows and henceforth is a practical choice and very desirable for feature identifications and sparse representation.
To ensure that the fidelity of the most important features in the image can be kept after dropping the less important components from the $\widetilde{\Sigma}_{p,\tilde{p}} $, the parameter $\tilde{p}(s)$ is chosen by a threshold based on the $L^1$-norm of $\widetilde{\Sigma_{p}}$, i.e. as the first integer such that
\beqnx
\sum_{l=1}^{\tilde{p}(s)} \sigma_{i_l j_l} > K_2 \sum_{l=1}^{{p(s)}^2} \sigma_{i_l j_l} \,,
\eqnx
where $K_2$ is a threshold which is smaller than $1$.
In all the following examples, $K_1$ and $K_2$ are always chosen as $K_1 = 3.5$ and $K_2=0.95$.
A quantization process $\mathbb{Q}$ is performed on all the three matrices $[\tilde{U}_{p}, \tilde{\Sigma}_{p, \tilde{p}},\tilde{V}_{p}]$ which we get from \textbf{Algorithm 2} as $\mathbb{Q} (A_{ij}) := \left[\frac{A_{ij}}{0.01}\right]$ for any matrix $(A_{ij})$.
This is to minimize the number of possible choices of values in the matrix entries in order to embrace a possibility for an efficient entropy coding post-processing after the NMF process and minimize memory complexity.
The parameters $c_1,c_2$ in \textbf{Algorithm 1} are always set to $1$.

For the sake of comparisons between feature extraction, sparsity of representation and robustness against noise in the transmission channel, we shall also compare the performance of NMF with the ones by
the SVD and the JPEG compression process.  For any given image $Y$, the SVD with the level parameter $s$, $I_{SVD,s}$, is taken directly as
\beqn
I_{SVD,s} := \iota_S^T (U \Sigma V^T) \,, \quad \text{ where } \iota_s (Y) = U \Sigma V^T \,.
\eqn
Again, the same quantization process $\mathbb{Q}$ is performed on the three matrices $[U, \Sigma, V]$ as
described above to embrace a possibility for efficient entropy coding.
Meanwhile, for the JPEG compression format, we follow the standard routine as in \cite{jpg}.
Namely we first perform a discrete cosine transform (DCT) on $8 \times 8$ pixel-blocks to give the DCT coefficients $(D_{ij})$ on each block, then perform the standard JPEG quantization process $C_{ij} = \left[\frac{D_{ij}}{{(Q_{50})}_{ij}}\right]$ with the given standard JPEG quantization matrix $Q_{50}$ (with quality $Q=50$) \cite{jpg}:
\beqnx
Q_{50} :=
\begin{bmatrix}
16& 11& 10& 16& 24& 40& 51& 61 \\
12& 12& 14& 19& 26& 58& 60& 55 \\
14& 13& 16& 24& 40& 57& 69& 56 \\
14& 17& 22& 29& 51& 87& 80& 62 \\
18& 22& 37& 56& 68& 109& 103& 77\\
24& 35& 55& 64& 81& 104& 113& 92\\
49& 64& 78& 87& 103& 121& 120& 101\\
72& 92& 95& 98& 112& 100& 103& 99
\end{bmatrix}
\eqnx
A level parameter $s$ is introduced to define the image $I_{JPG,s}$ as the reconstruction of the JPEG from
only the first $2^{3-s}$ Fourier coefficients in each $8 \times 8$ pixel-blocks for $s = 0,1,2,3$.
Note that, with this definition, only $4$ levels are available for JPEG.

In order to test the robustness of the algorithms for feature preservation during the transmission process of data through channel,
multiplication noise is added to simulate the scenario of data transmission through a noisy cable for each of the aforementioned algorithms, i.e. NMF, SVD and JPEG.  For the NMF process, multiplicative noise is added to the three matrices $[\tilde{U}_{p}, \tilde{\Sigma}_{p, \tilde{p}}, \tilde{V}_{p} ]$ after quantization as
\beqn
 (\tilde{U}_{p}^\zeta)_{ij} = (\tilde{U}_{p})_{ij} (1 + \sigma \zeta_{ij} ) \,,  \quad ( \tilde{\Sigma}_{p, \tilde{p}}^\zeta )_{ij} = ( \tilde{\Sigma}_{p, \tilde{p}} )_{ij} (1 + \sigma \zeta_{ij} ) \,, \quad  (\tilde{V}_{p})_{ij}^\zeta = (\tilde{V}_{p})_{ij} (1 + \sigma \zeta_{ij} )\,,
\eqn
where $\mathcal{I}_{p(s),\tilde{p}(s)}^{\alpha, \nu, \gamma} [ \iota_s (Y) ] :=  \tilde{U}_{p}  \tilde{\Sigma}_{p, \tilde{p}}  \tilde{V}_{p}^T\,$,
$ \, \mathcal{I}_{s,p(s),\tilde{p}(s)}^{\alpha,\nu,  \gamma} (Y) := \iota_s^T \circ \mathcal{I}_{p(s),\tilde{p}(s)}^{\alpha, \nu,  \gamma} \circ \iota_s (Y) $, $\, \sigma$ is the noise level and $\zeta$ is uniformly distributed between $[-1,1]$.
Noisy reconstruction from the NMF is then given by
\beqn
 \left[\mathcal{I}_{s,p(s),\tilde{p}(s)}^{\alpha,\nu,  \gamma} \right]^{\zeta} (Y) := \iota_s^T \tilde{U}^\zeta_{p}  \tilde{\Sigma}^\zeta_{p, \tilde{p}}  {(\tilde{V}^\zeta_{p})}^T \,.
\eqn
Similarly, for the SVD process, multiplicative noise is added in $[U, \Sigma, V]$ after quantization such that
\beqn
 U_{ij}^\zeta = U_{ij} (1 + \sigma \zeta_{ij} ) \,,  \quad \Sigma^\zeta_{ij} = \Sigma_{ij} (1 + \sigma \zeta_{ij} ) \,, \quad  V_{ij}^\zeta = V_{ij} (1 + \sigma \zeta_{ij} )\,,
\eqn
where $I_{SVD,s} := \iota_s^T (U \Sigma V^T)$ and $\iota_s (Y) := U \Sigma V^T$.
The noisy reconstruction $I_{SVD,s}^\zeta$ is then taken as
\beqn
I_{SVD,s}^\zeta := \iota_S^T (U^\zeta \Sigma^\zeta (V^\zeta)^T).
\eqn
For the JPEG process, multiplicative noise is added in DCT coefficients on each $8\times 8$ pixel block after quantization:
\beqn
C_{ij}^\zeta = C_{ij}(1 + \sigma \zeta_{ij} ) \,,
\eqn
and the noisy reconstruction $I_{JPG,s}^\zeta$ comes as the de-quantization of $C^\zeta$ by multiplication by $Q_{50}$ followed by an inverse DCT.
In all our numerical examples, we always set the noise level to be $\sigma = 25 \%$

The relative error of the reconstruction image $I_{\text{reconst}}$ from each reconstruction method
is quantified in the following manner on the quotient space of $L^2$ after taking an affine equivalence:
\beqnx
 \varepsilon(I_{\text{reconst}}) := \frac{\min_{a,b \in \mathbb{R}} || a I_{\text{reconst}} + b - Y||_{L^2} }{|| Y||_{L^2}}
\eqnx
This measurement of error is adopted because all the reconstructed images are shown such that the color scale gives only the relative contrast of the gray scale, and therefore an affine equivalence is taken for an appropriate measure of relative error.
For each image, we shall also measure the memory complexity ratio of a given method, which is given as the ratio between the memory size of the data after performing the corresponding method and that of the original data.
We would like to remark that the memory complexities for all the three methods (including JPEG) in our examples are computed based on its size before entropy coding; meanwhile, a same entropy coding technique can be applied to all the three methods considering the fact that all of them have undergone a quantization process.

\textbf{Example 1}. In this example, we set $Y$ as the grey-scale image presented in Figure \ref{fig:NMF1}. The parameter $s_{max}$ is chosen as $ s_{max} = [\log(\min(N,M))/\log(r)- 3]$.
The resulting images from MLA without noise are shown in Figure \ref{fig:NMF1a} whereas reconstructions with $25 \%$ noise are given in Figure \ref{fig:NMF1a2}.  The memory complexity ratios for the $(s_{max}-s)$-th level of the three methods and their respective relative $L^2$ errors with and without noise are shown as follows:
\beqnx
\begin{matrix}
s_{max}-s &:& 1 & 2 & 3 &4 & 5 &6 \\
p &:& 20  &  24   & 24  &  28  &  34  &  42\\
\tilde{p} &:&  142 &  177 &  152  & 153  & 195  & 332 \\
\text{memory complexity ratio of NMF} &:&  0.0017 &   0.0033  &  0.0061 &   0.0116  &  0.0271  &  0.0573 \\
\text{memory complexity ratio of SVD} &:&  0.0015  &  0.0036  &  0.0072  &  0.0168 &   0.0409  &  0.1011  \\
\text{memory complexity ratio of JPEG} &:&  \text{NA}  &    \text{NA}  &  0.0154 &   0.0497  &  0.0982 &   0.1048 \\
\text{Relative $L^2$ error in NMF (with $0\%$ noise)} &:&   0.2723  &  0.2567  &  0.2350 &   0.1878 &   0.1630 &   0.1561 \\
\text{Relative $L^2$ error in SVD (with $0\%$ noise)} &:&   0.2733  &  0.2584  &  0.2342  &  0.1875  &  0.1591 &   0.1594  \\
\text{Relative $L^2$ error in JPEG (with $0\%$ noise)} &:&  \text{NA}  &    \text{NA}  &  0.1855  &  0.0974   & 0.0689  &  0.0535 \\
\text{Relative $L^2$ error in NMF (with $25\%$ noise)} &:&     0.2768 &   0.2631  &  0.2462 &   0.2029 &   0.1770  &  0.1689  \\
\text{Relative $L^2$ error in SVD (with $25\%$ noise)} &:&   0.2755  &  0.2629 &   0.2456  &  0.2029  &  0.1754  &  0.1704 \\
\text{Relative $L^2$ error in JPEG (with $25\%$ noise)} &:&   \text{NA}  &  \text{NA}  & 0.1941 &   0.1089  &  0.0711 &   0.0673
\end{matrix}
\eqnx
We can see from Figure \ref{fig:NMF1a} and  \ref{fig:NMF1a2} that
in the absence of noise, although it is true that the NMF does not outperform SVD and JPEG of the same level,
many reasonable details of different scales can already be captured in different levels of NMF, starting from the coarser image of the horse, then finer details and afterwards the clear black-and-white strips on the horse.
In each level, JPEG gives the best image of the three, however, it also needs a relatively high memory complexity in the same level. Meanwhile the NMF provides a representation of a relatively low memory complexity of the same layer. It is especially interesting to note that a memory complexity ratio of about $0.01$ (before entropy coding) at level $4$ can already give us many details of the horse.
With the presence of noise, we can see that although the relative $L^2$ errors of both NMF and SVD are more or less the same,
many coarser layers of SVD are not free from the contamination of noise in the form of vertical and horizontal strips in the background, and that the NMF gives a better shape of the horse.  The NMF layers are affected by noise, but most of the nice details of the horse can still be kept.
The JPEG stays the most robust against the noise, nonetheless, considering the fact that NMF of the same layer usually requires less than half of the memory as JPEG, the performance of NMF is already quite reasonable.

\begin{center}
\begin{figurehere}
\hfill{}\includegraphics[clip,width=0.35\textwidth]{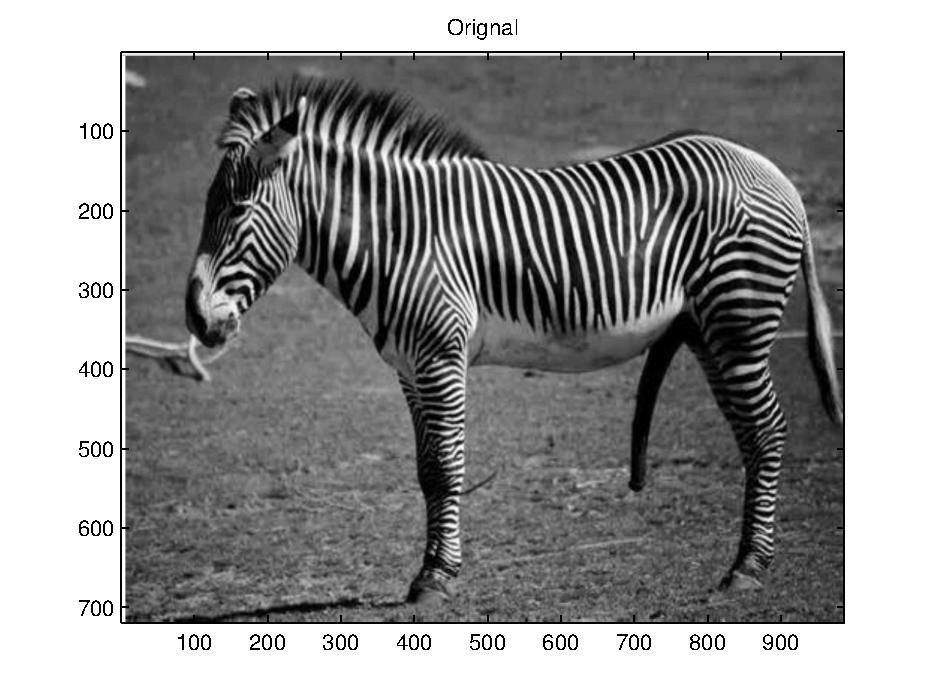}\hfill{}
 \caption{\label{fig:NMF1} Original image in Example 1}
\end{figurehere}
\end{center}

\begin{figurehere}
\hfill{}\\
\hskip -5cm
\includegraphics[clip,width=0.35\textwidth]{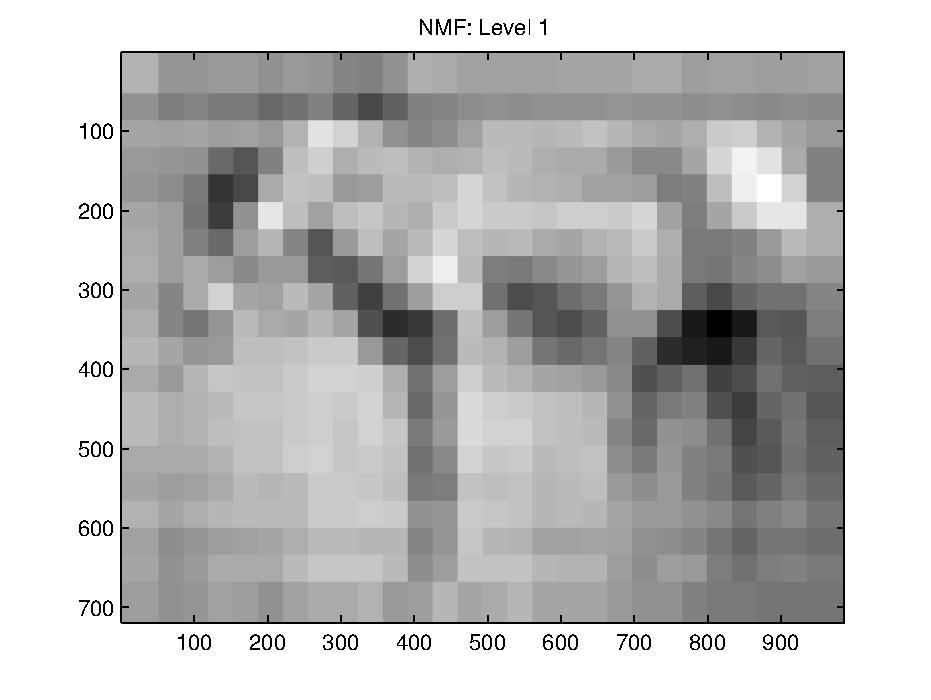}
\hskip -0.5cm
\includegraphics[clip,width=0.35\textwidth]{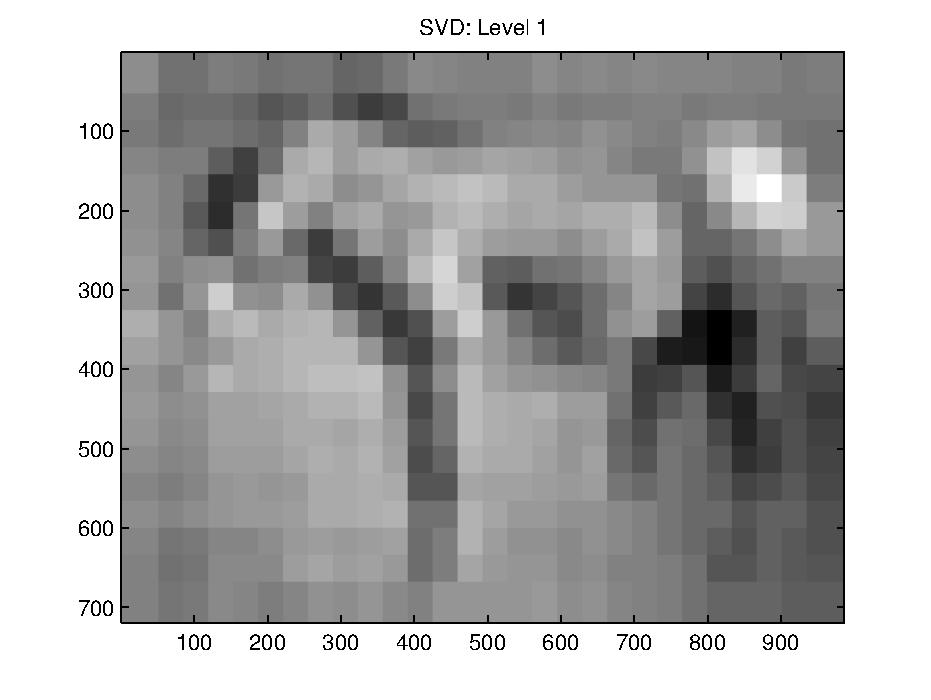}
\hskip -0.5cm
\includegraphics[clip,width=0.35\textwidth]{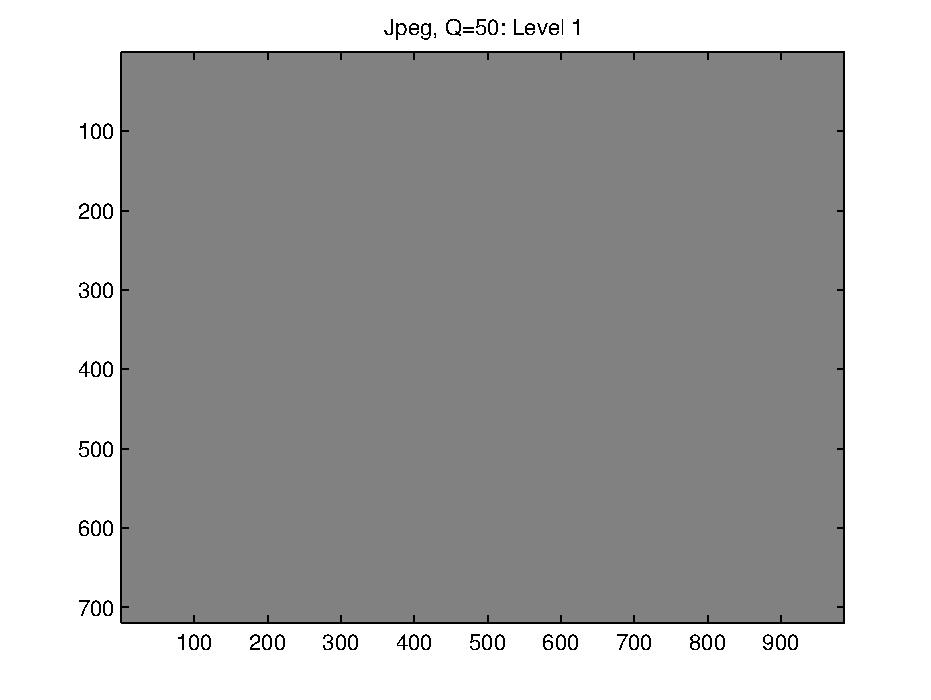}\\
\hskip -2cm
\includegraphics[clip,width=0.35\textwidth]{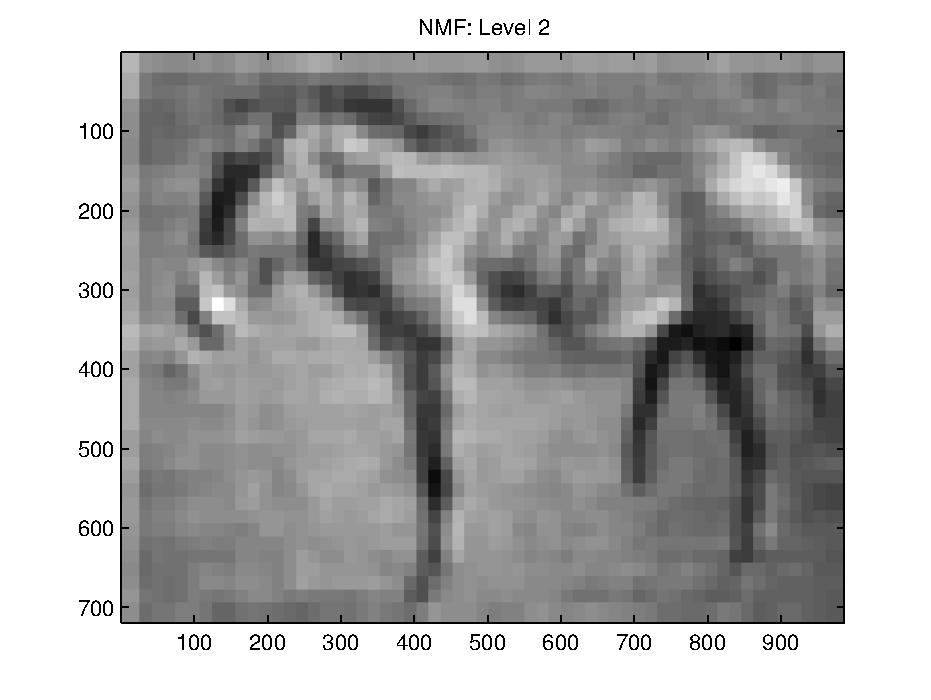}
\hskip -0.5cm
\includegraphics[clip,width=0.35\textwidth]{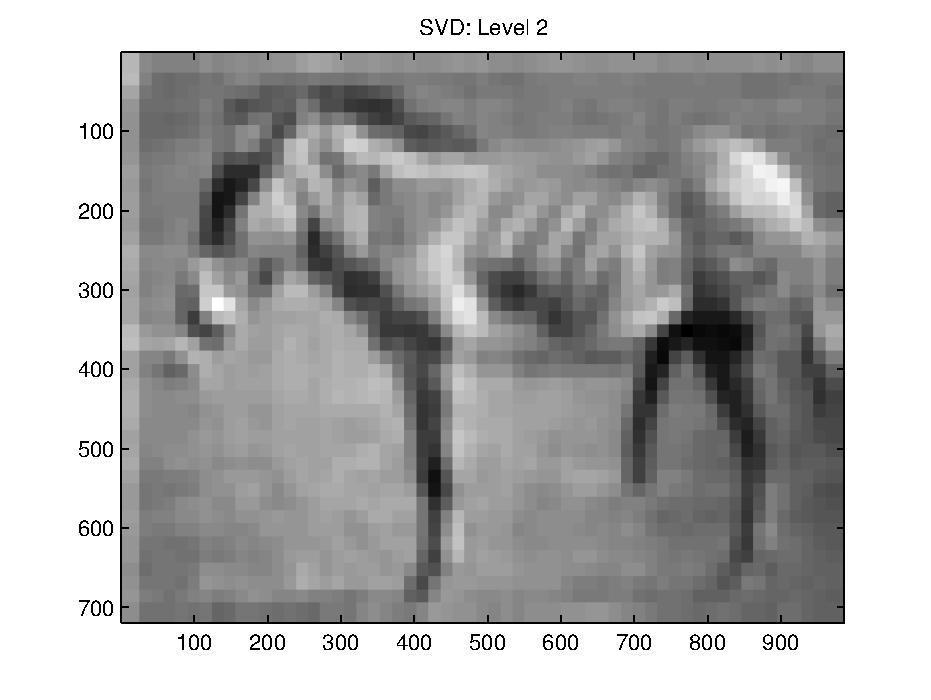}
\hskip -0.5cm
\includegraphics[clip,width=0.35\textwidth]{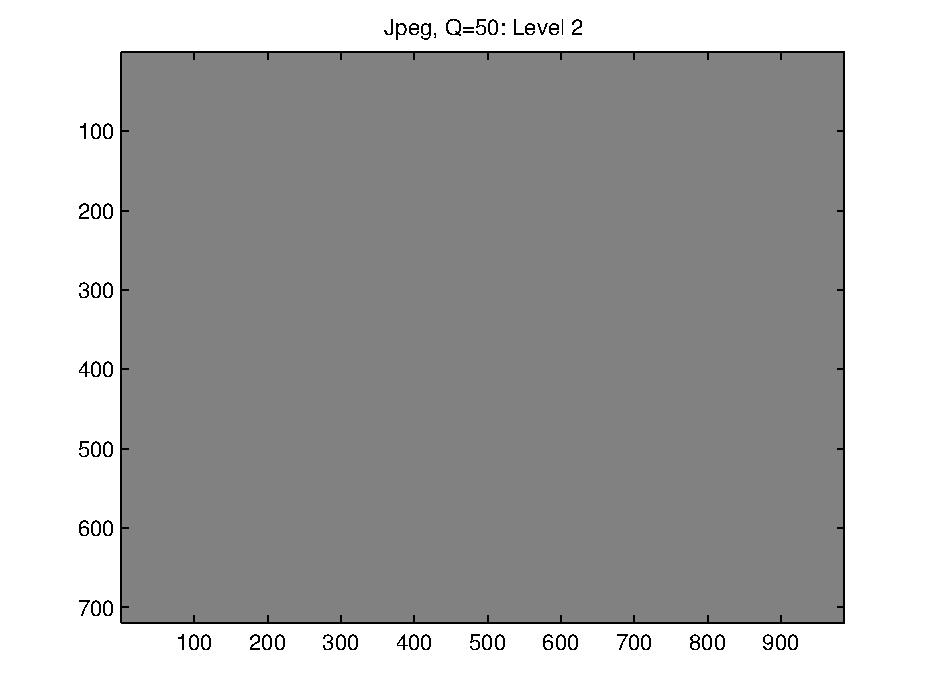} \\
\hskip -2cm
\includegraphics[clip,width=0.35\textwidth]{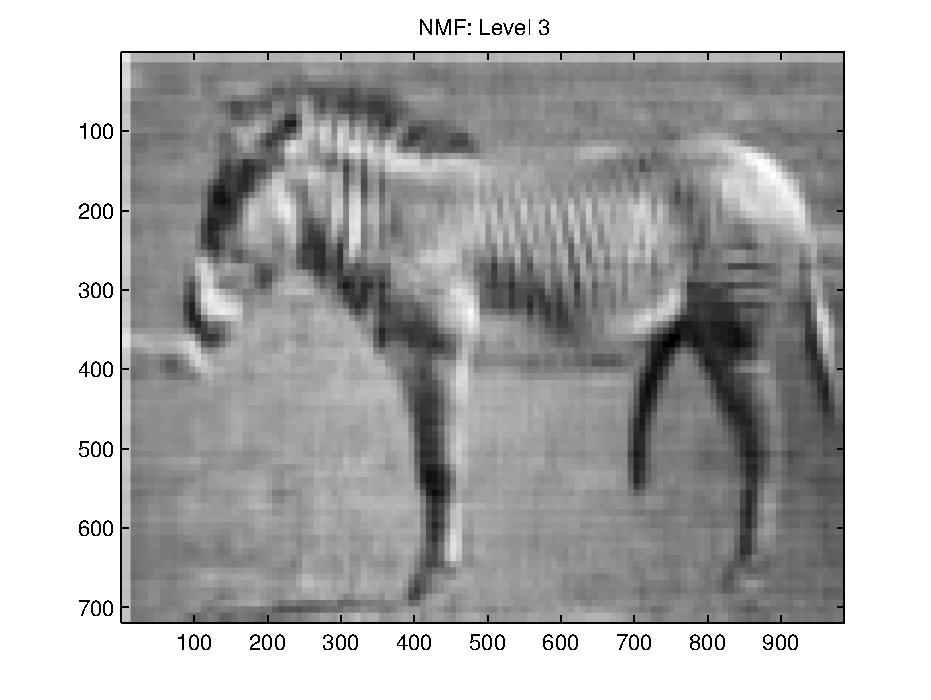}
\hskip -0.5cm
\includegraphics[clip,width=0.35\textwidth]{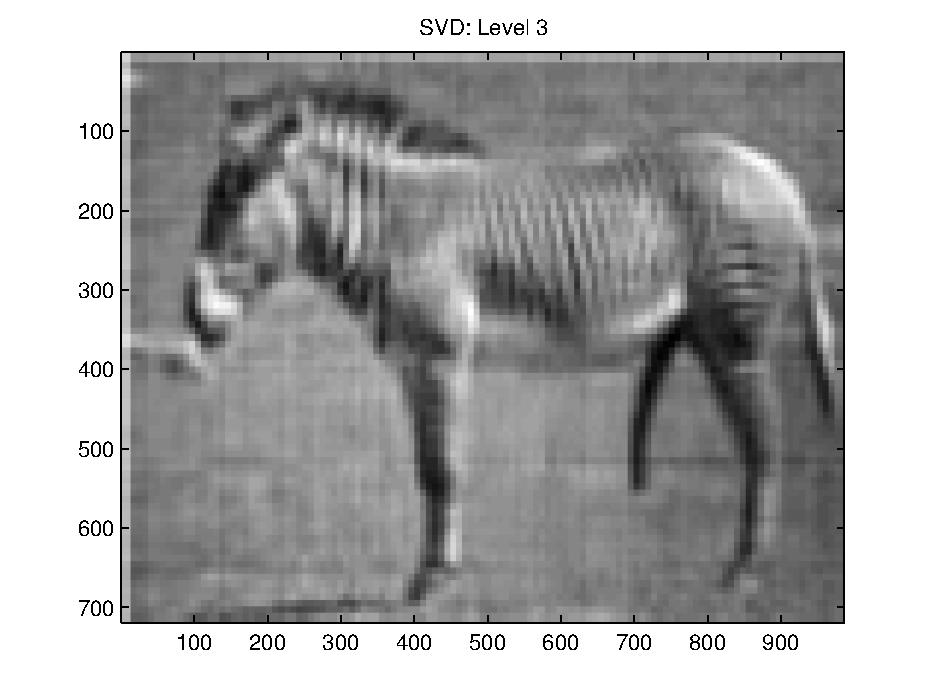}
\hskip -0.5cm
\includegraphics[clip,width=0.35\textwidth]{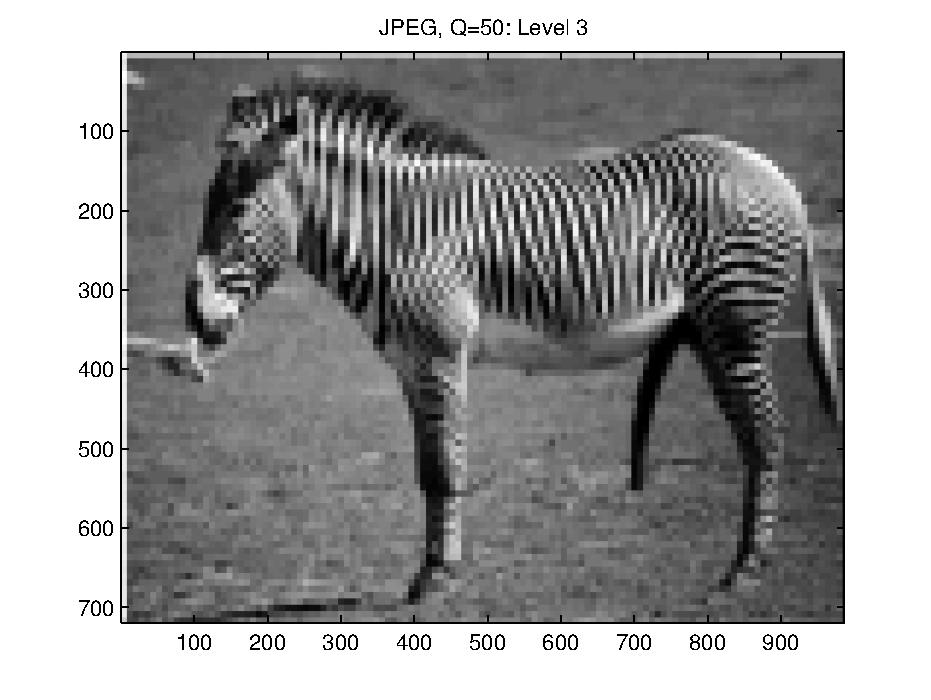} \\
\hskip -2cm
\includegraphics[clip,width=0.35\textwidth]{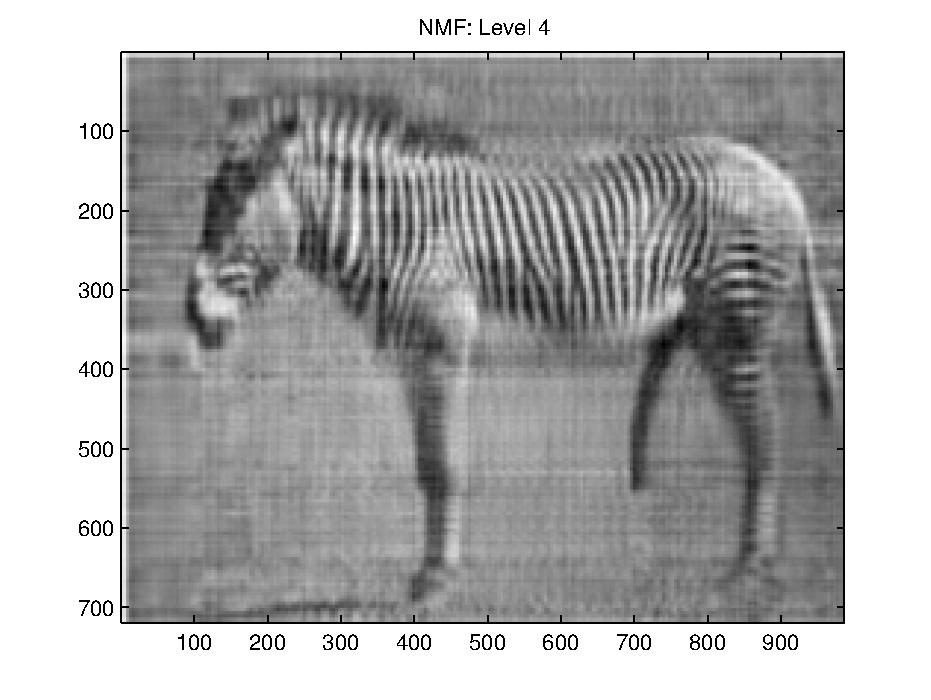}
\hskip -0.5cm
\includegraphics[clip,width=0.35\textwidth]{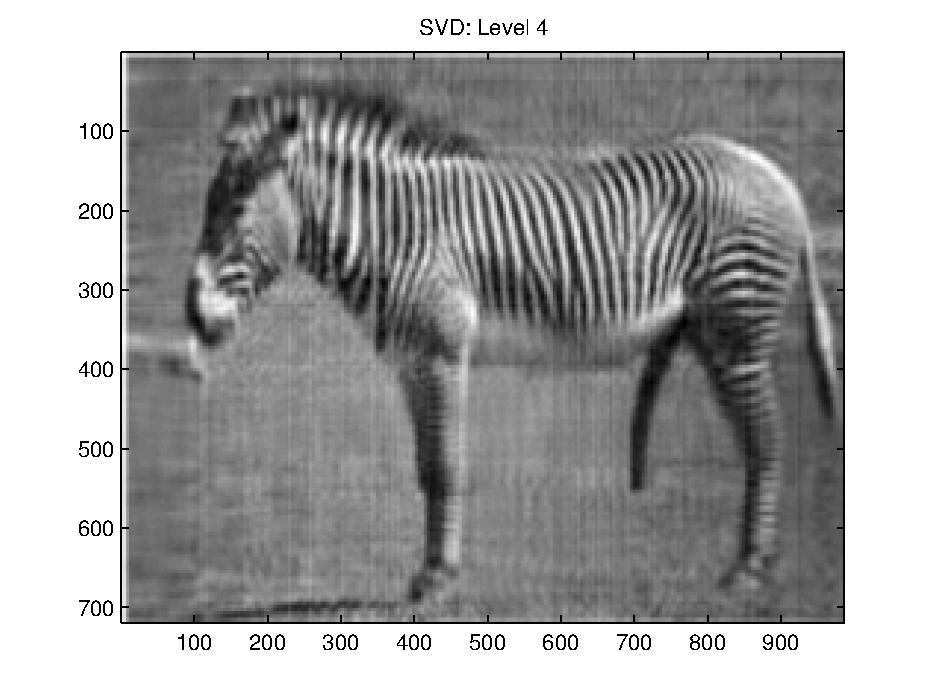}
\hskip -0.5cm
\includegraphics[clip,width=0.35\textwidth]{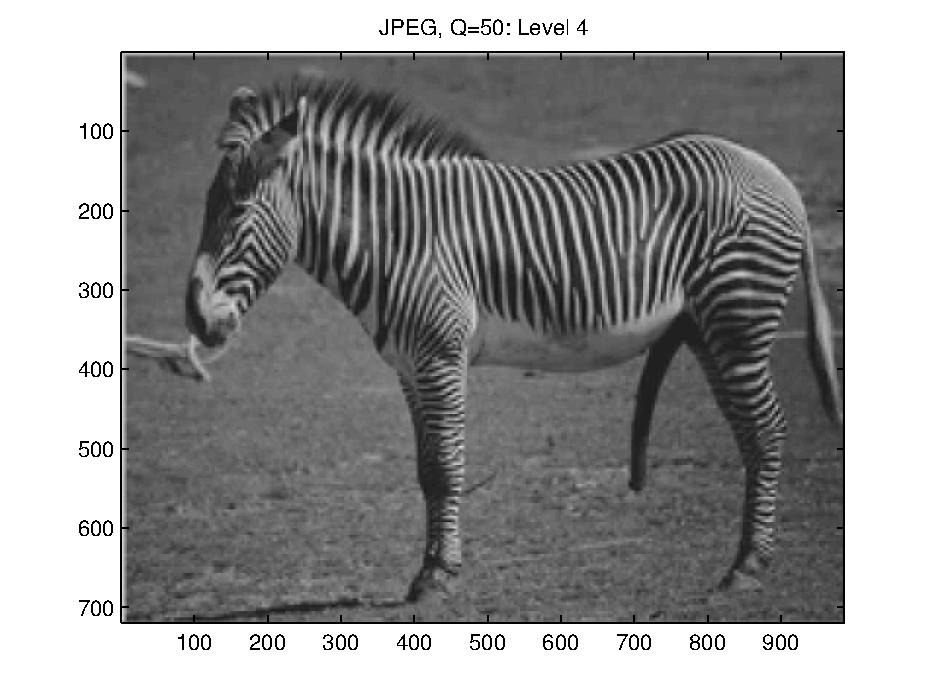} \\
\hskip -2cm
\includegraphics[clip,width=0.35\textwidth]{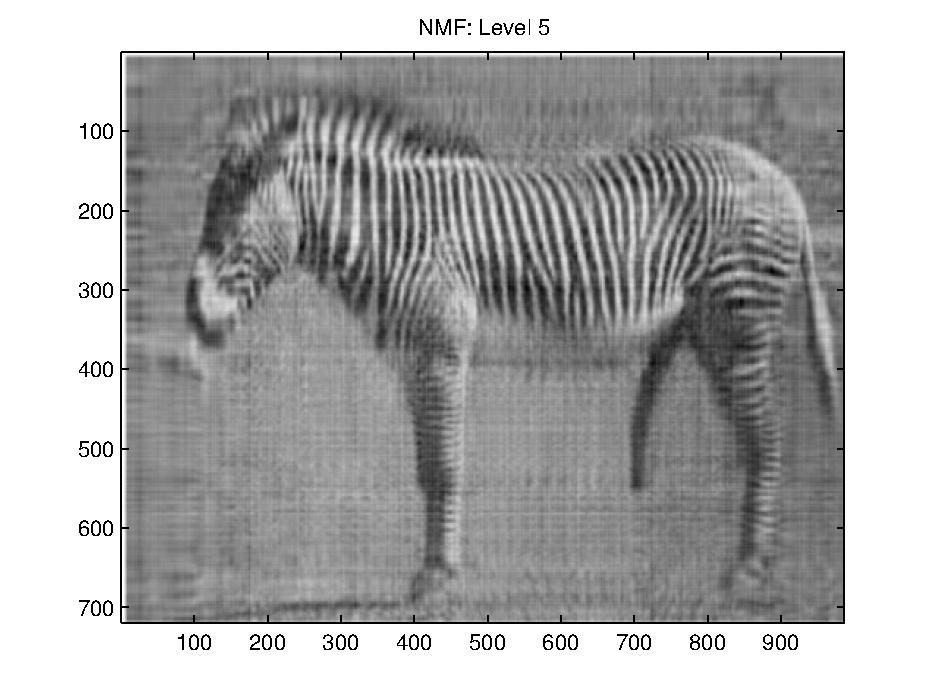}
\hskip -0.5cm
\includegraphics[clip,width=0.35\textwidth]{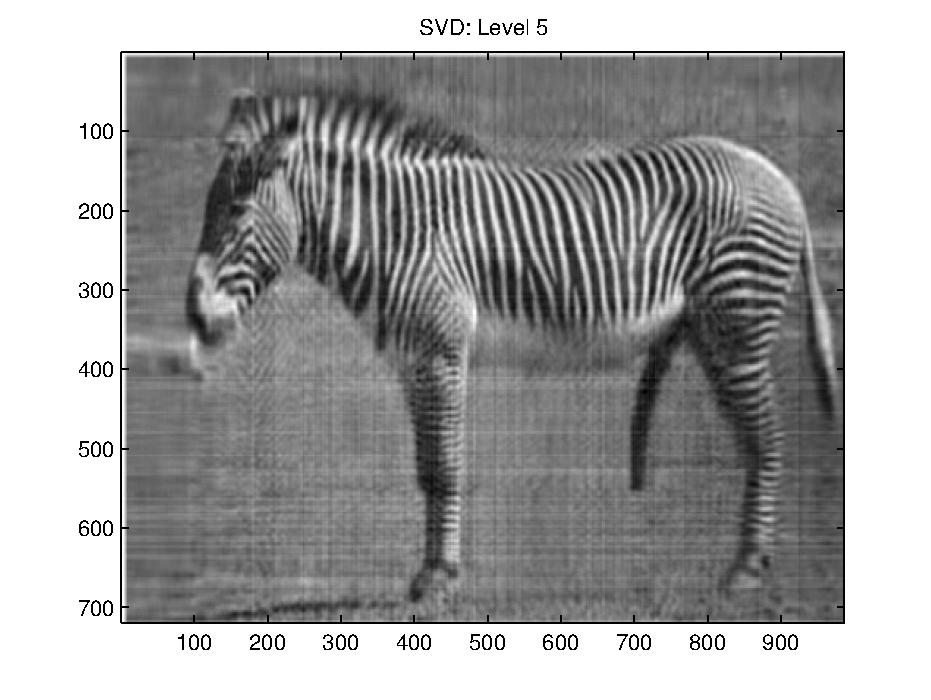}
\hskip -0.5cm
\includegraphics[clip,width=0.35\textwidth]{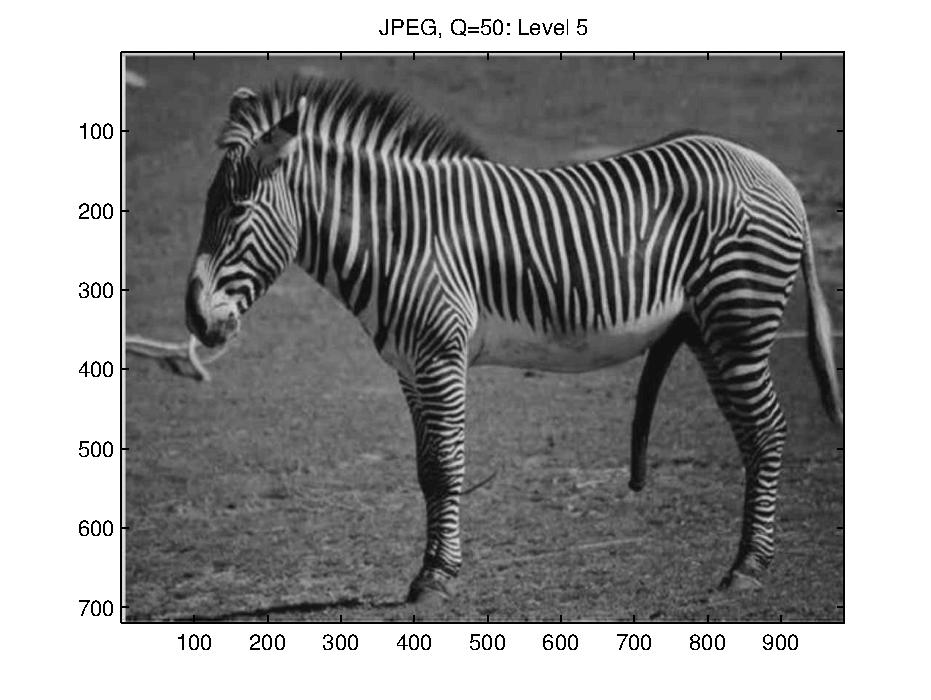} \\
\hskip -2cm
\includegraphics[clip,width=0.35\textwidth]{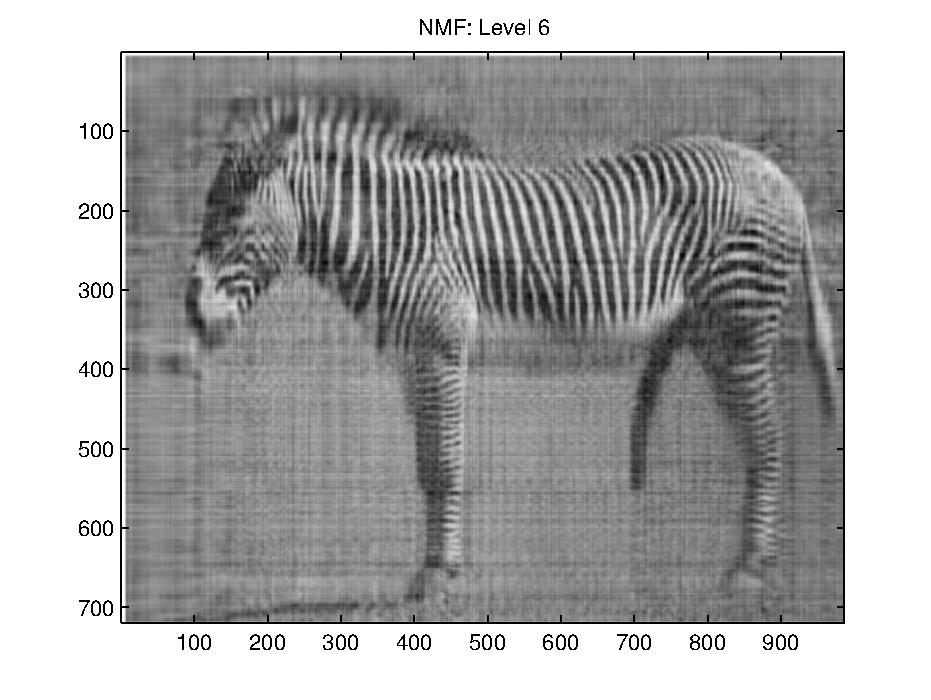}
\hskip -0.5cm
\includegraphics[clip,width=0.35\textwidth]{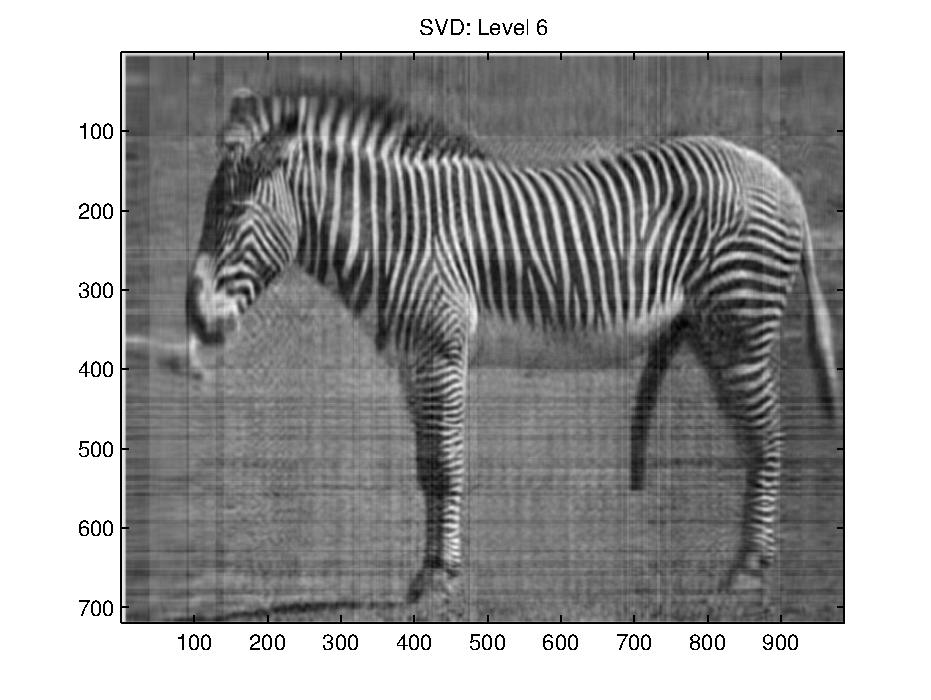}
\hskip -0.5cm
\includegraphics[clip,width=0.35\textwidth]{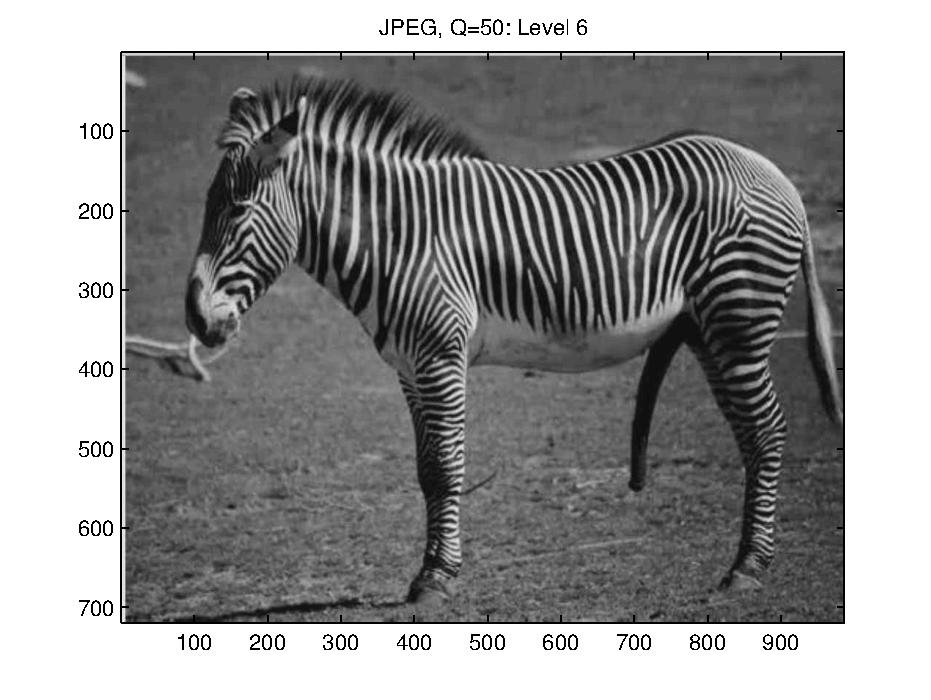}
 \caption{\label{fig:NMF1a} MLA for the image in Example 1 using NMF without noise }
\end{figurehere}

\begin{figurehere}
\hfill{}\\
\hskip -5cm
\includegraphics[clip,width=0.35\textwidth]{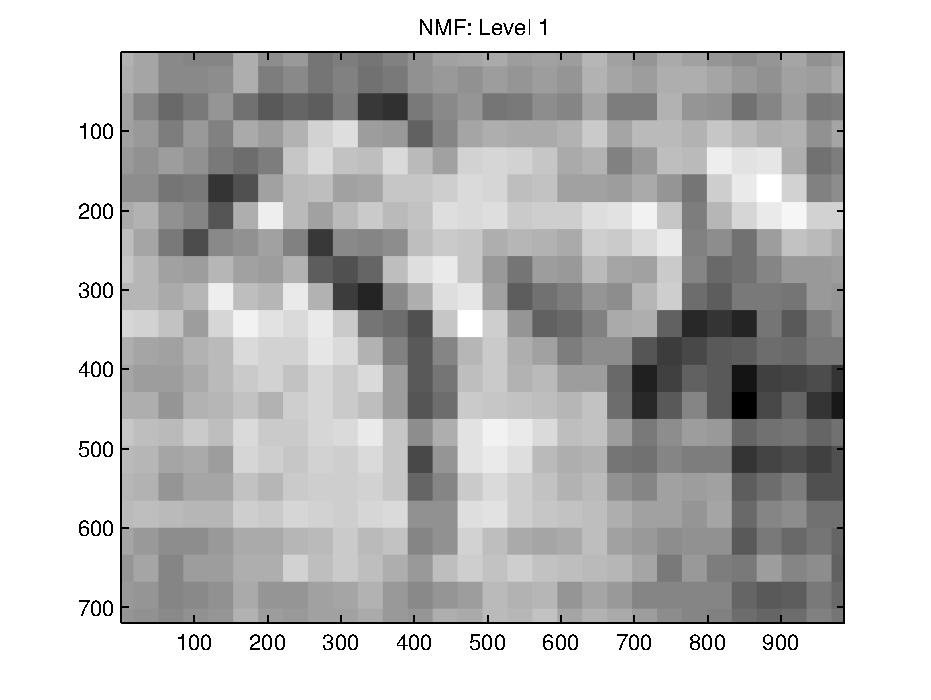}
\hskip -0.5cm
\includegraphics[clip,width=0.35\textwidth]{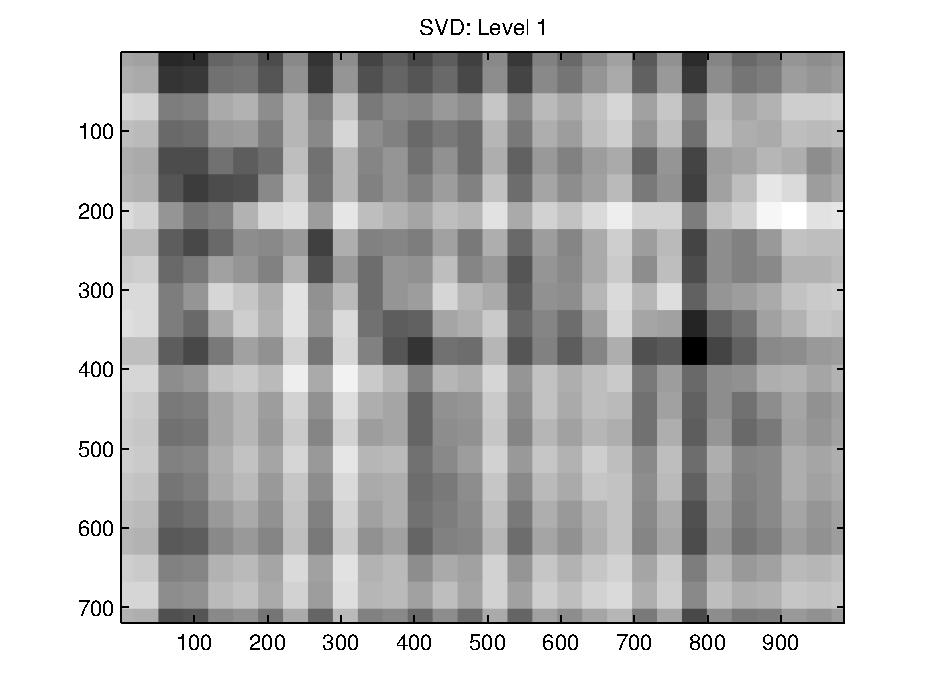}
\hskip -0.5cm
\includegraphics[clip,width=0.35\textwidth]{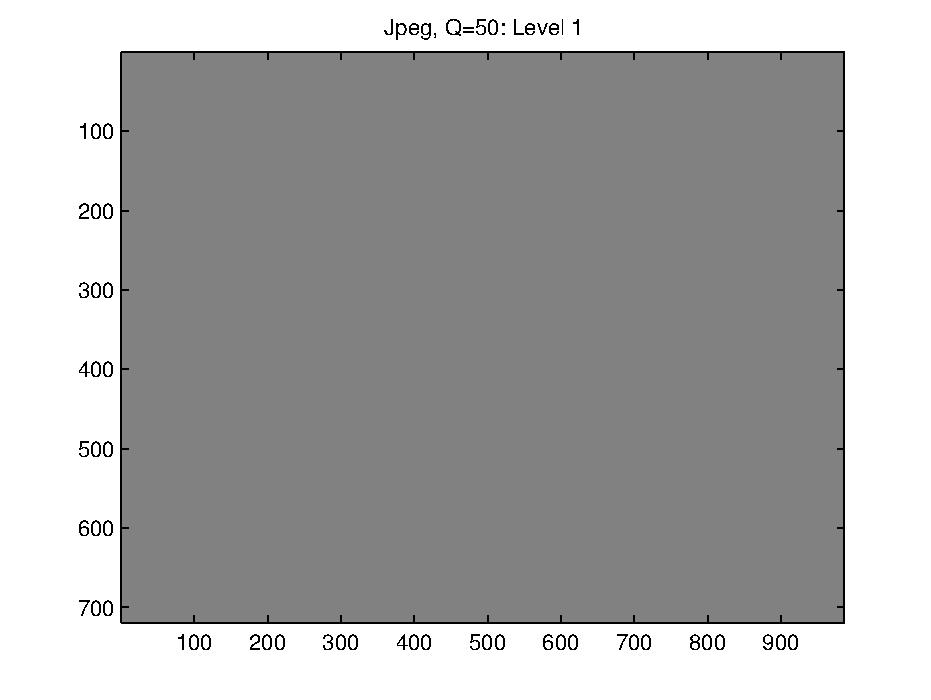}\\
\hskip -2cm
\includegraphics[clip,width=0.35\textwidth]{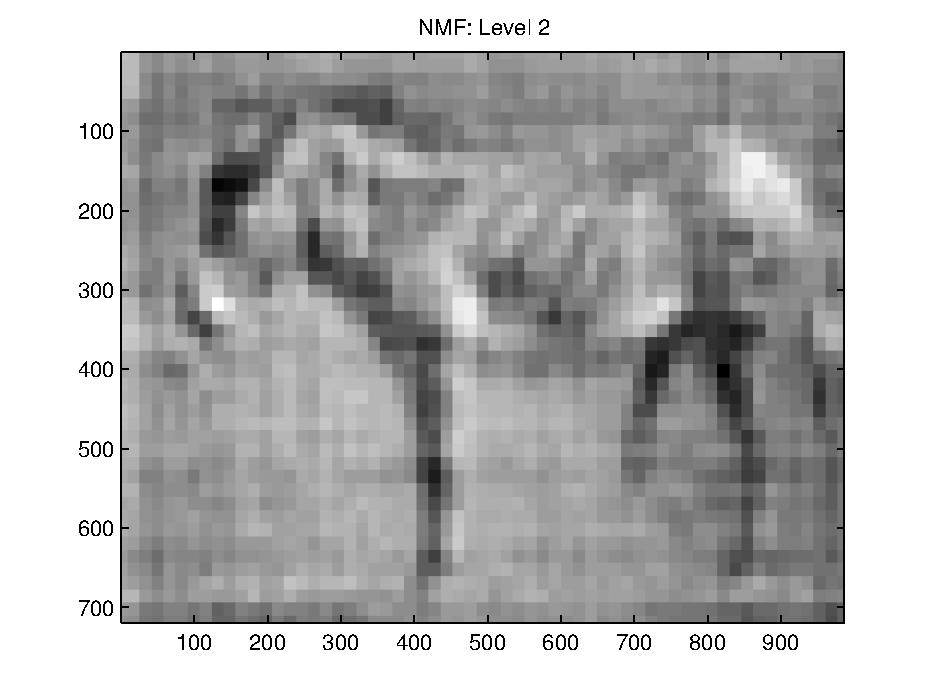}
\hskip -0.5cm
\includegraphics[clip,width=0.35\textwidth]{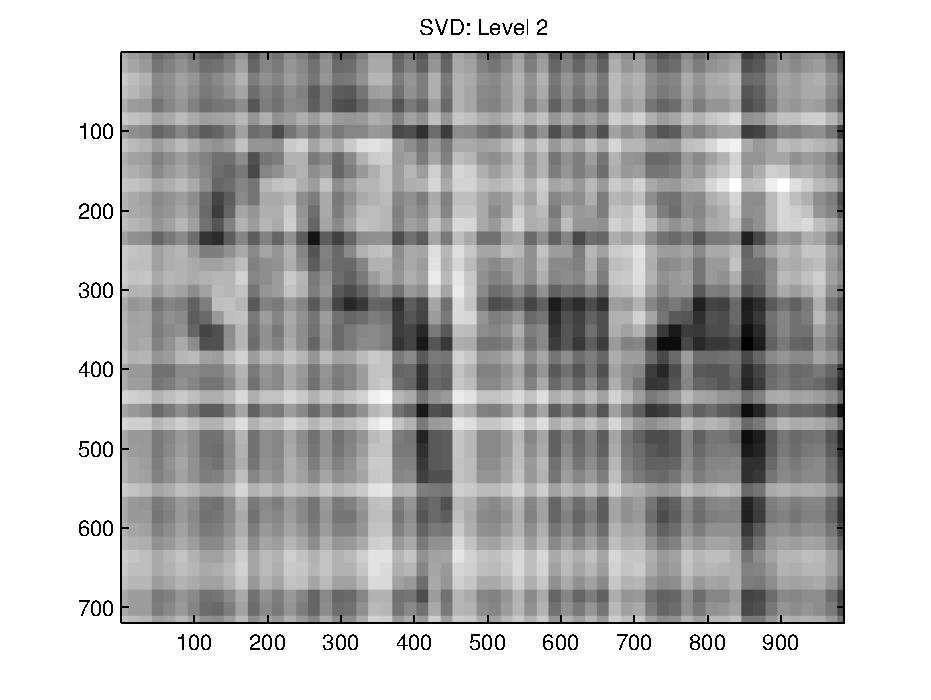}
\hskip -0.5cm
\includegraphics[clip,width=0.35\textwidth]{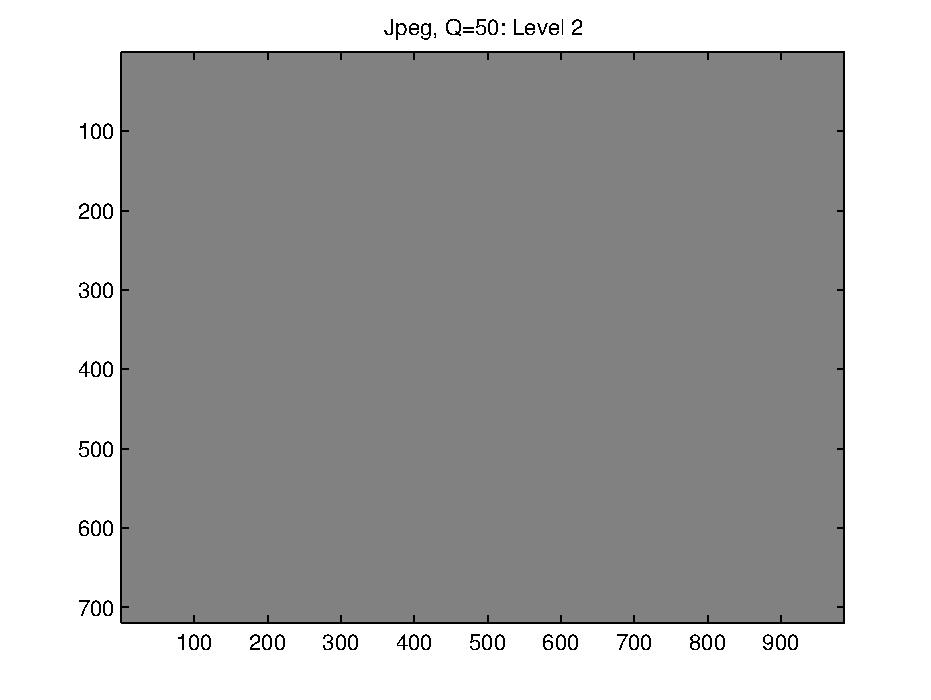} \\
\hskip -2cm
\includegraphics[clip,width=0.35\textwidth]{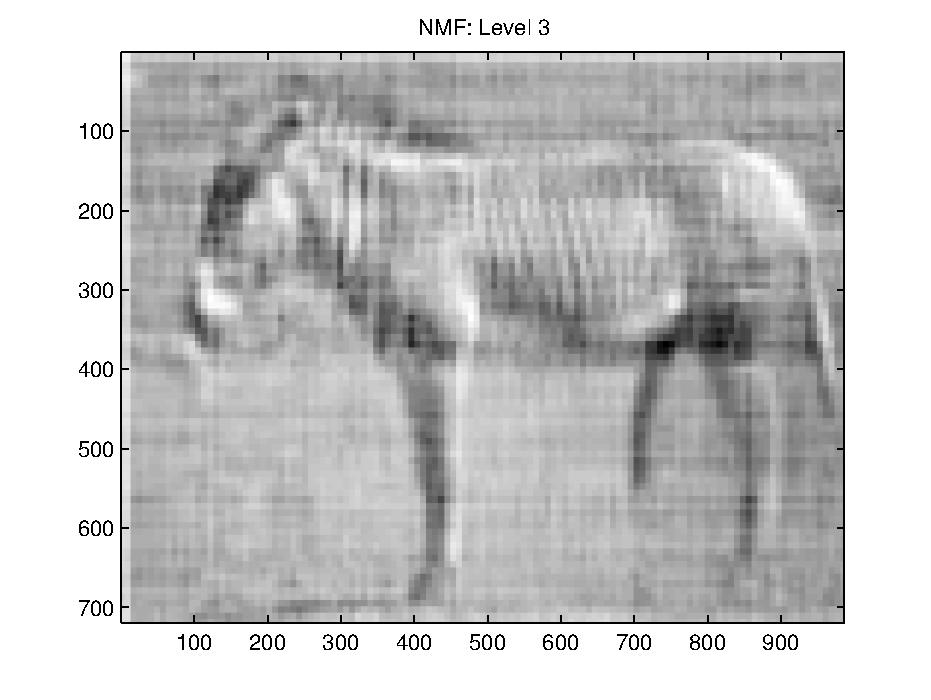}
\hskip -0.5cm
\includegraphics[clip,width=0.35\textwidth]{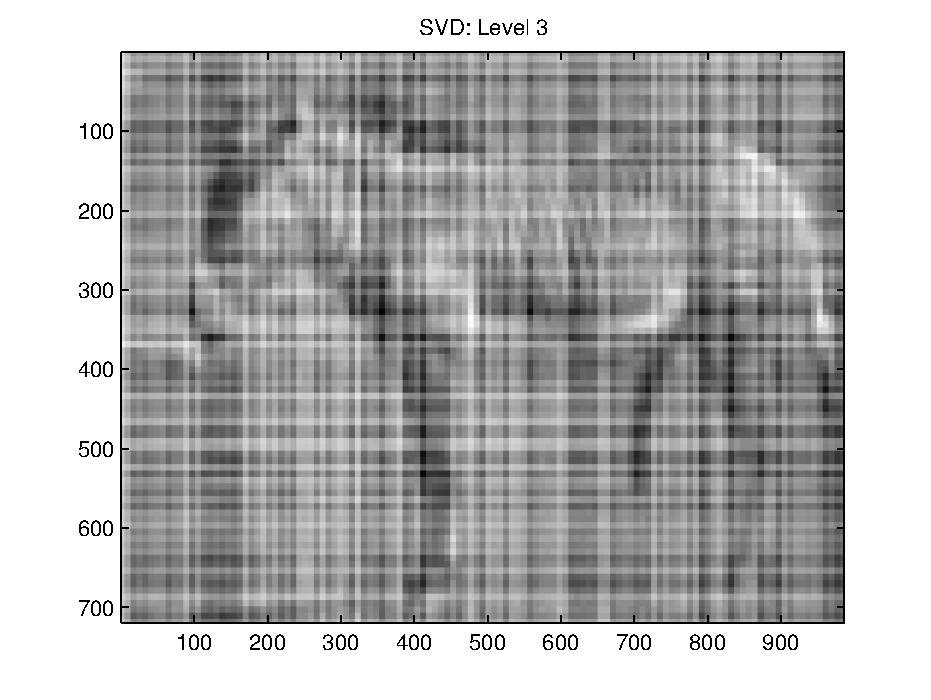}
\hskip -0.5cm
\includegraphics[clip,width=0.35\textwidth]{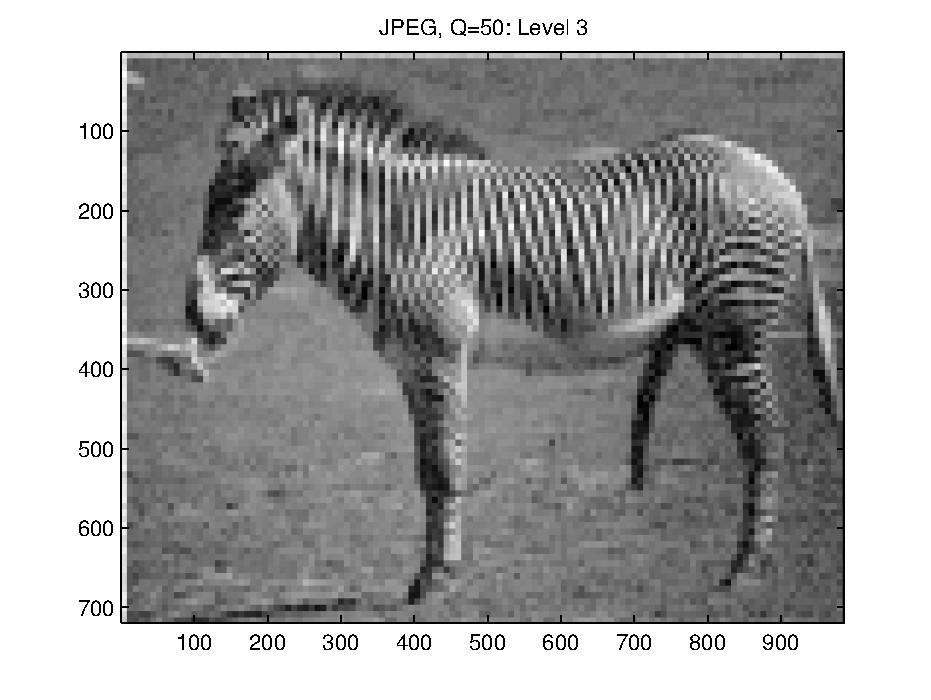} \\
\hskip -2cm
\includegraphics[clip,width=0.35\textwidth]{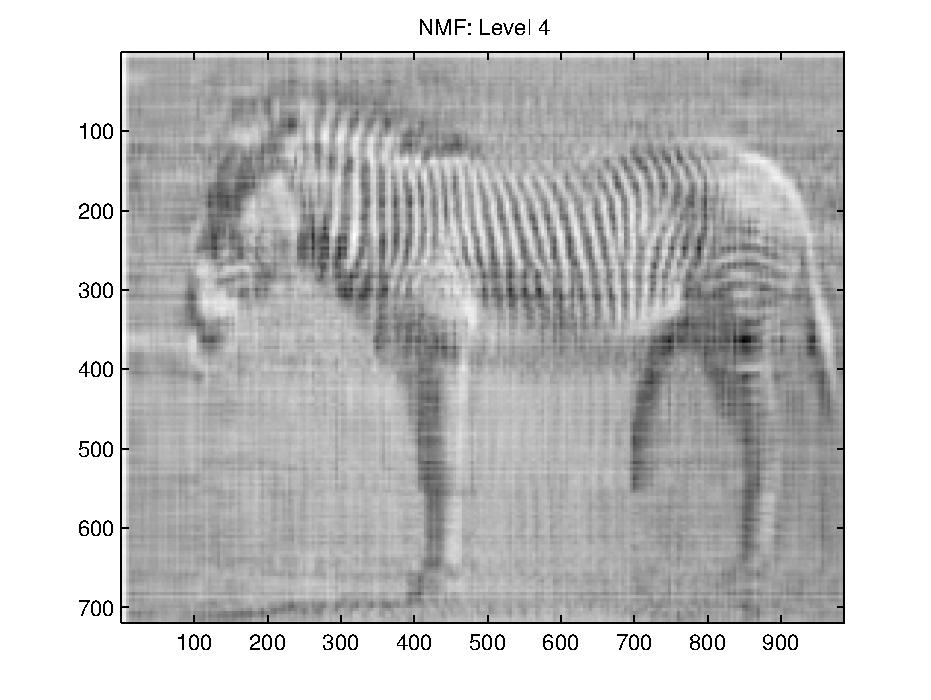}
\hskip -0.5cm
\includegraphics[clip,width=0.35\textwidth]{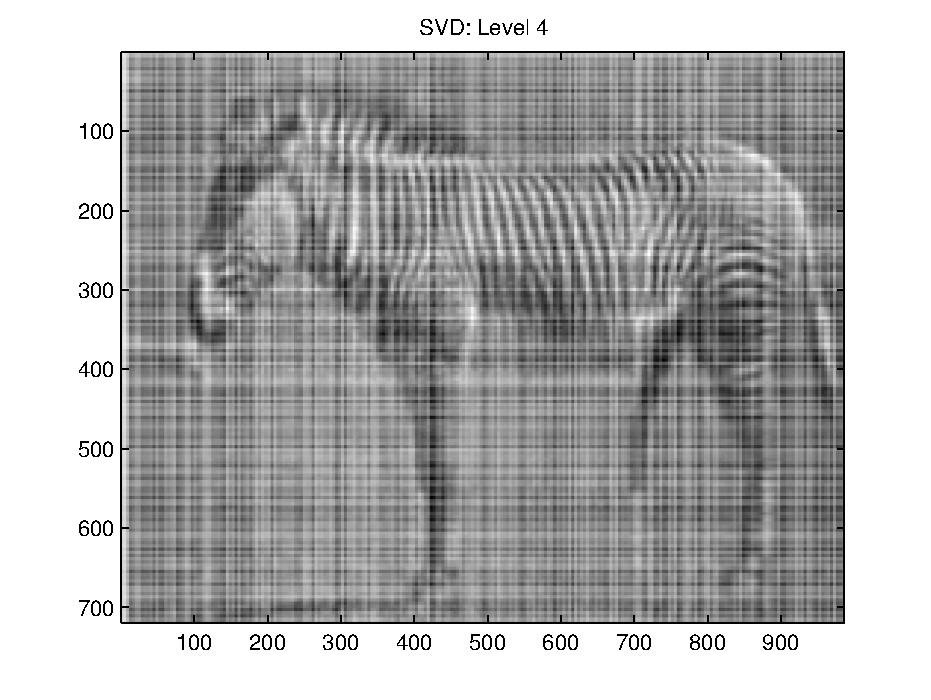}
\hskip -0.5cm
\includegraphics[clip,width=0.35\textwidth]{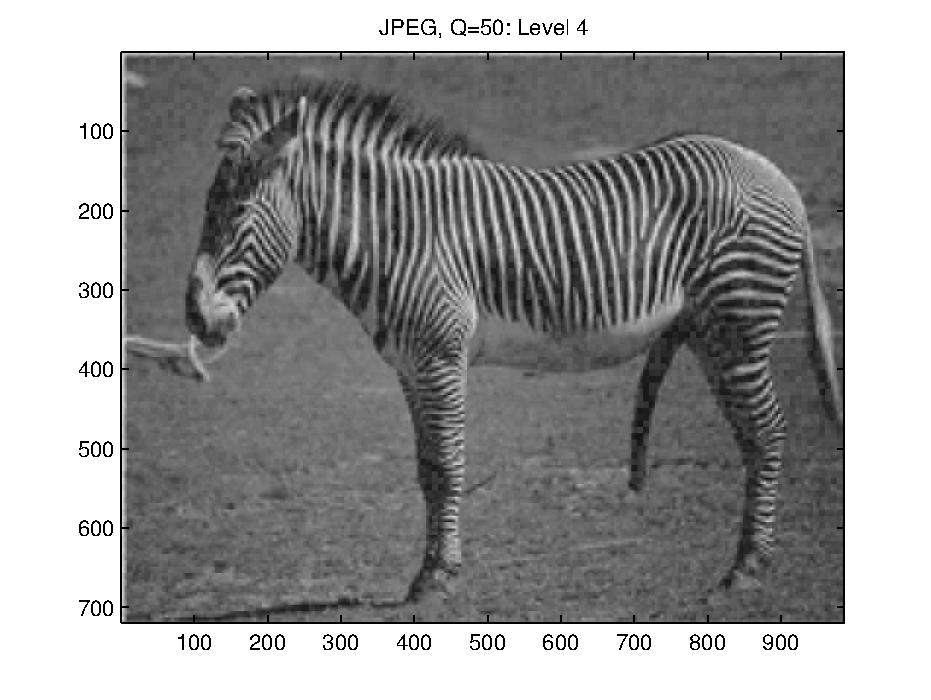} \\
\hskip -2cm
\includegraphics[clip,width=0.35\textwidth]{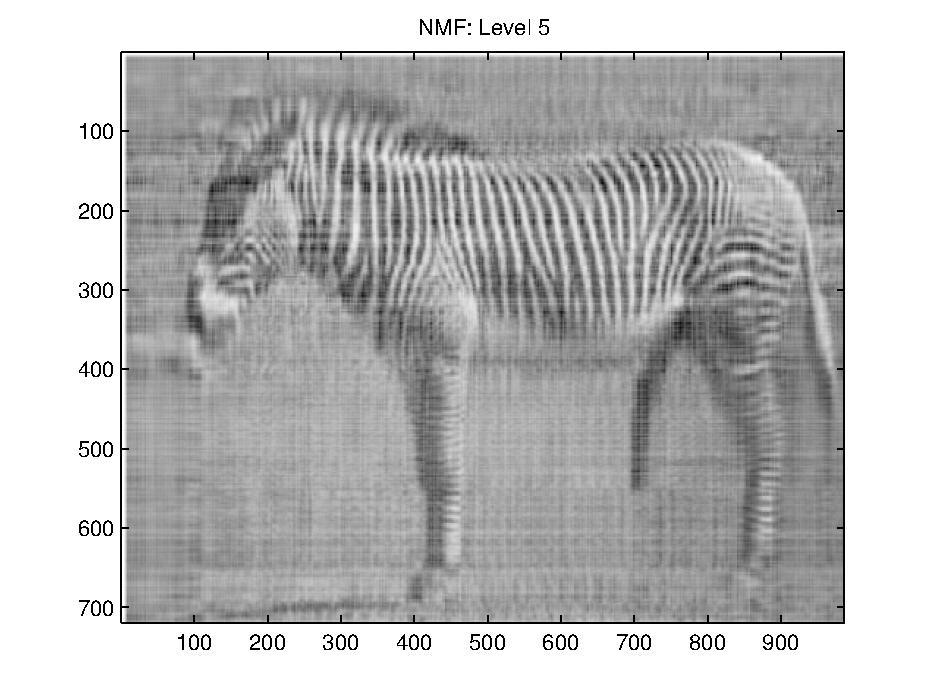}
\hskip -0.5cm
\includegraphics[clip,width=0.35\textwidth]{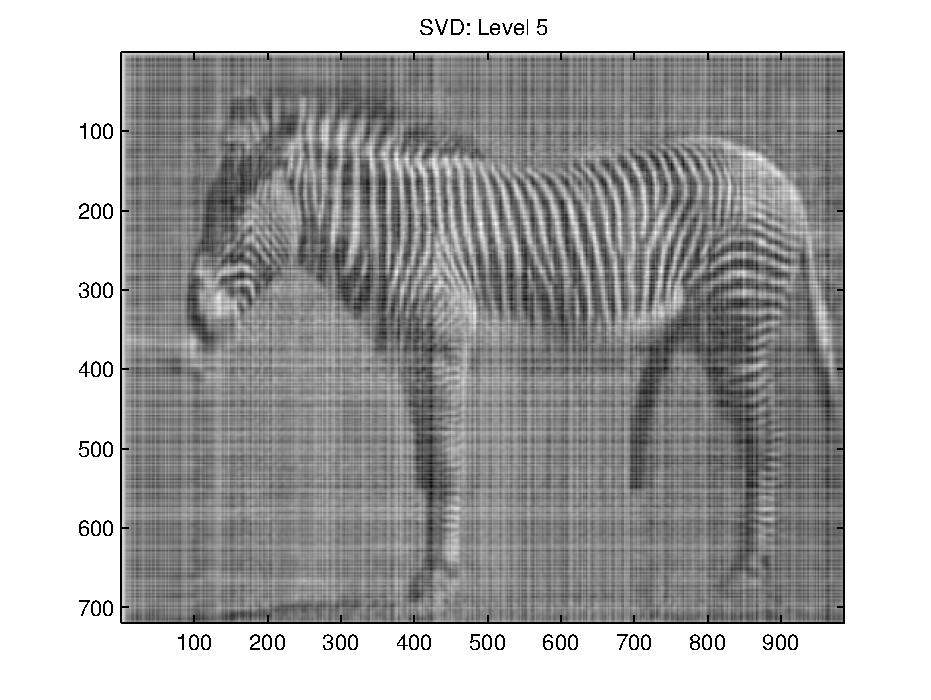}
\hskip -0.5cm
\includegraphics[clip,width=0.35\textwidth]{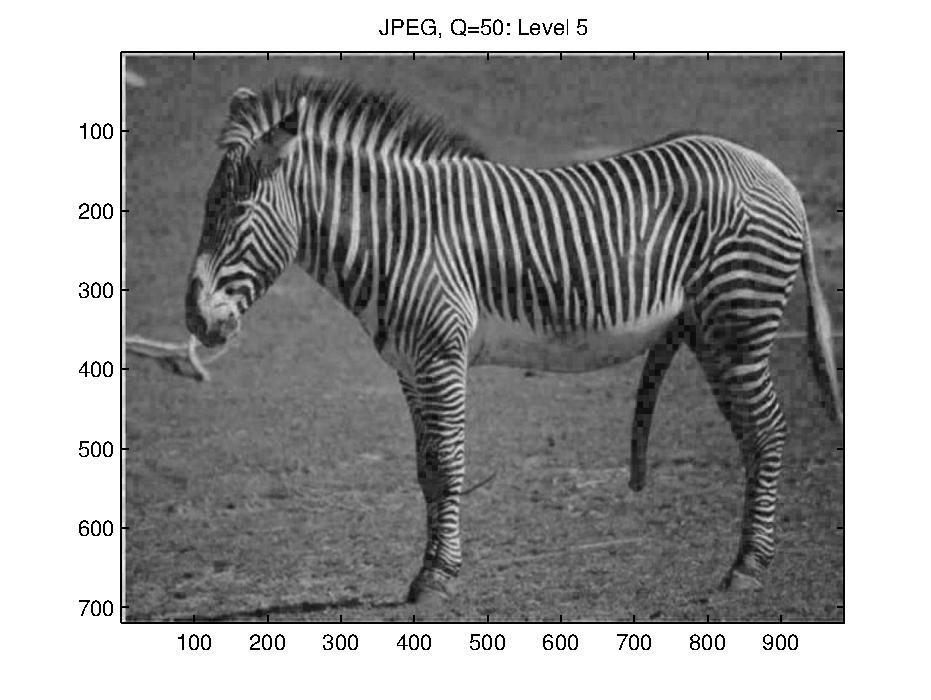} \\
\hskip -2cm
\includegraphics[clip,width=0.35\textwidth]{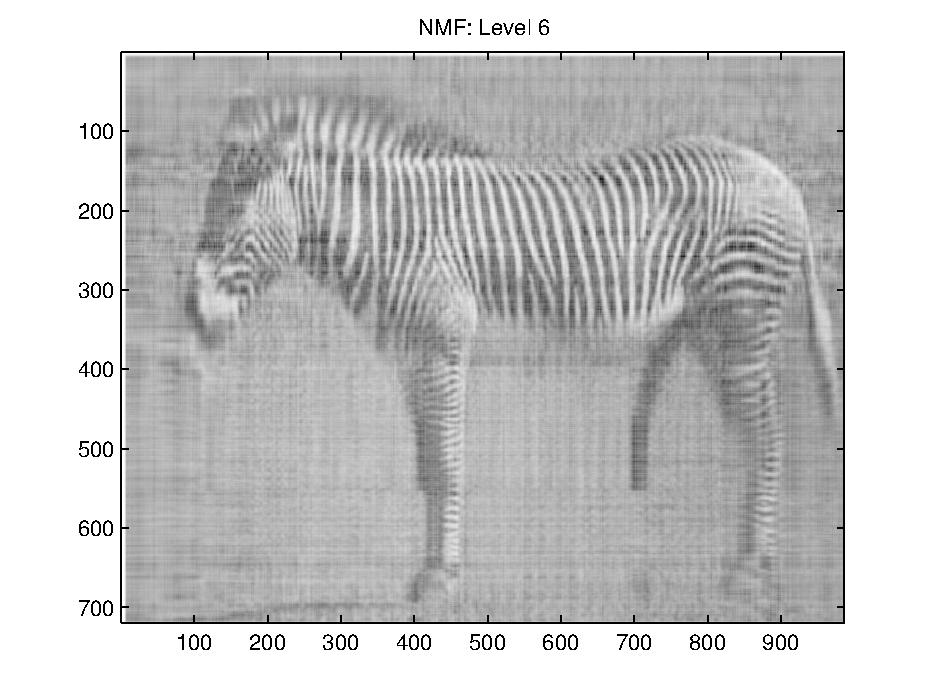}
\hskip -0.5cm
\includegraphics[clip,width=0.35\textwidth]{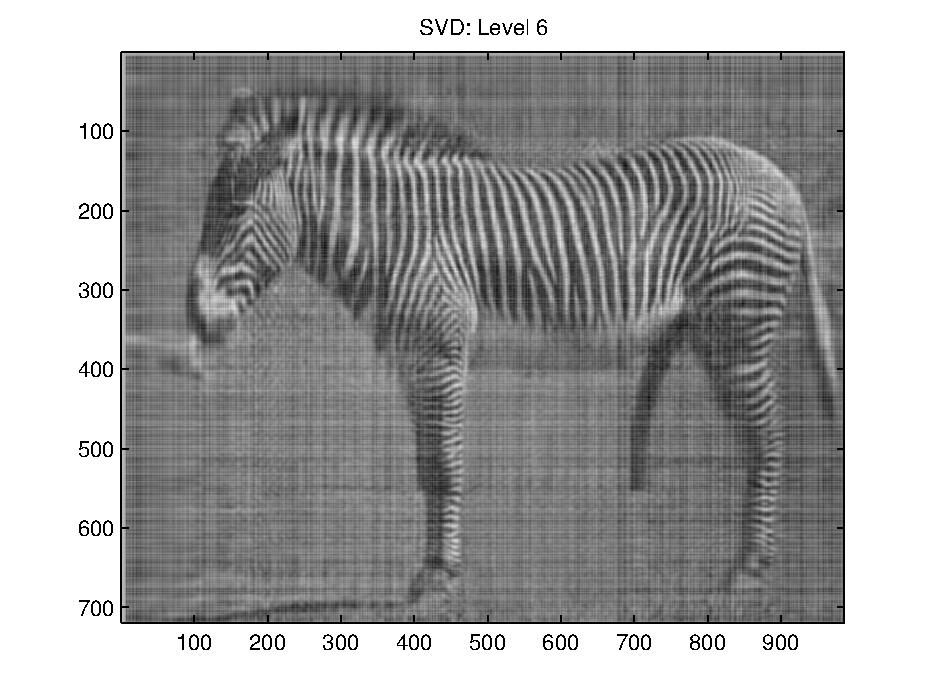}
\hskip -0.5cm
\includegraphics[clip,width=0.35\textwidth]{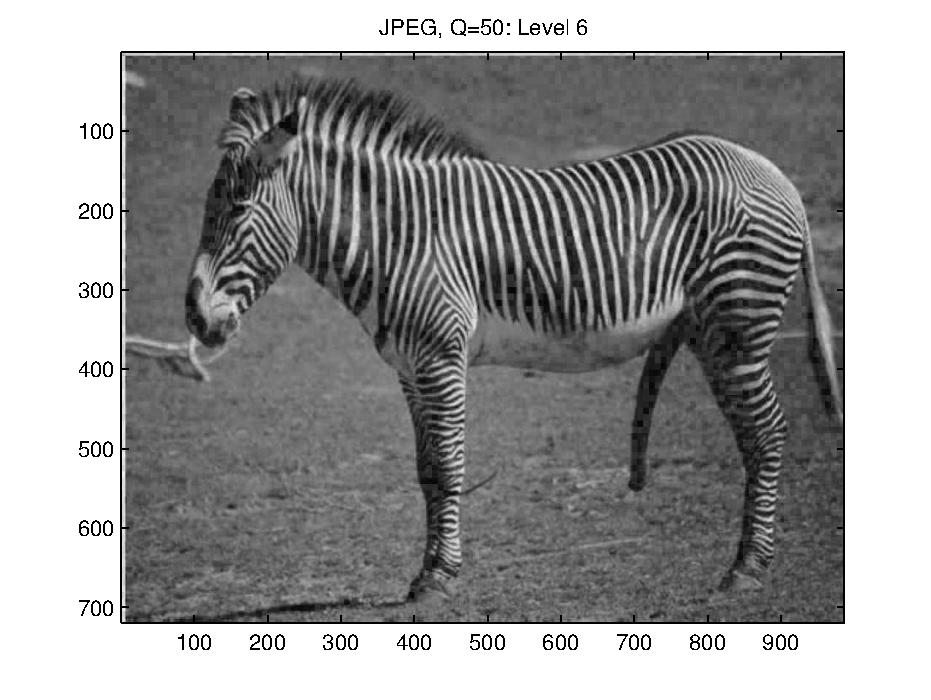}
 \caption{\label{fig:NMF1a2} MLA for the image in Example 1 using NMF with $25\%$ noise }
\end{figurehere}

\textbf{Example 2}.
In this example, we set $Y$ as the image presented in Figure \ref{fig:NMF2}. The parameters are the same as in the previous example.  The resulting images are shown in Figure \ref{fig:NMF2a}.  The memory complexity ratios for the $(s_{max}-s)$-th level of the three methods and their respective relative $L^2$ errors with and without noise are shown as follows:
\beqnx
\begin{matrix}
s_{max}-s &:& 1 & 2 & 3 &4 & 5  \\
p &:& 22  &  21  &  23  &  28  &  34  \\
\tilde{p} &:&  143  & 113  & 118  & 146  & 208  \\
\text{memory complexity ratio of NMF} &:&   0.0047  &  0.0075  &  0.0140  &  0.0309   & 0.0711 \\
\text{memory complexity ratio of SVD}&:&  0.0053  &  0.0103  &  0.0225   & 0.0548  &  0.1330  \\
\text{memory complexity ratio of JPEG} &:&  \text{NA}  &  0.0155  &  0.0364 &    0.0619  &  0.0716 \\
\text{Relative $L^2$ error in NMF (with $0\%$ noise)} &:&   0.2945 &   0.2749  &  0.2503 &   0.1966 &   0.1693 \\
\text{Relative $L^2$ error in SVD (with $0\%$ noise)} &:&   0.3024  &  0.2770  &  0.2553  &  0.2075 &   0.1717   \\
\text{Relative $L^2$ error in JPEG (with $0\%$ noise)} &:&  \text{NA}  &   0.2487  &  0.1827  &  0.0884 &   0.0663  \\
\text{Relative $L^2$ error in NMF (with $25\%$ noise)} &:&  0.3104  &  0.2842 &   0.2614  &  0.2225 &   0.1899 \\
\text{Relative $L^2$ error in SVD (with $25\%$ noise)} &:&  0.3040  &  0.2867 &   0.2536   & 0.2298  &  0.2024\\
\text{Relative $L^2$ error in JPEG (with $25\%$ noise)} &:&   \text{NA}  &   0.2664  &  0.2025   & 0.1250  &  0.1082
\end{matrix}
\eqnx
From Figures \ref{fig:NMF2a} and \ref{fig:NMF2a2}, finer and finer details are reasonably captured and  present as the level number of the NMF layers increases, while a reasonably low compression ratio is attained. This time the memory complexity of JPEG becomes comparable to NMF.
In each level, JPEG still gives the best image of the three on the same layer, however, we notice that with the same level of memory complexity, some of the NMF images can provide a finer layer of detail than the other two methods.
With the presence of noise, we can see that although the relative $L^2$ errors of NMF actually outperform the ones
of the SVD in some layers, the figures of all the three methods seem to be seriously contaminated. However, to our surprise, it seems that the figures of NMF seem more robust to keep the background clean, while the figures of the SVD
are contaminated by random strips whereas the JPEG by random squares. In the coarsest levels, the SVD does not give a shape of a table, however, the NMF images still give a recognizable shape of a table.
Moreover, the most detail of the table in the finer levels is still reasonably kept by the NMF in the presence of noise.

\begin{center}
\begin{figurehere}
\hfill{}\includegraphics[clip,width=0.35\textwidth]{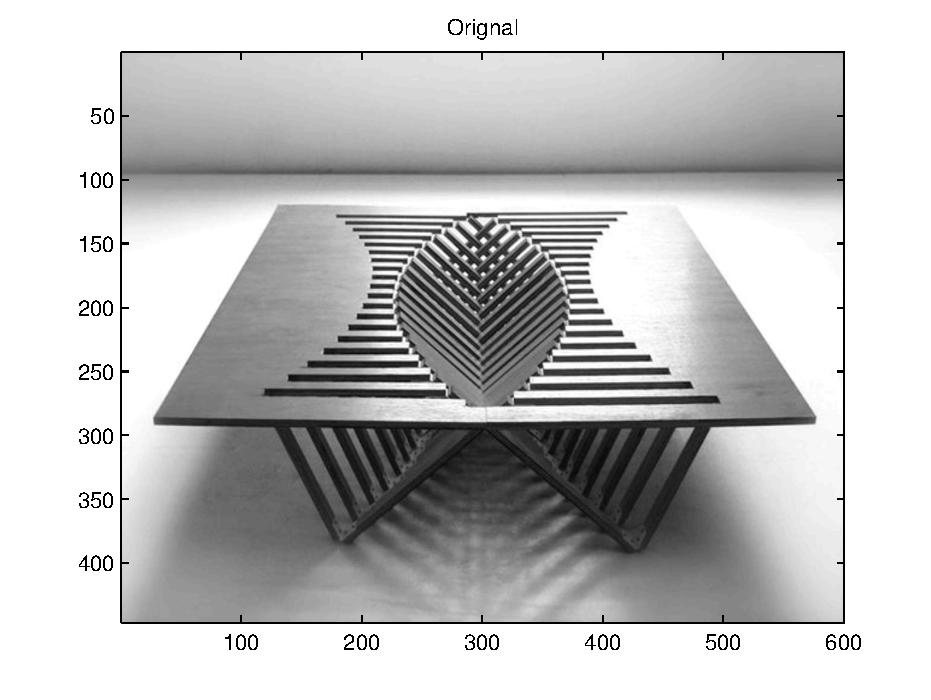}\hfill{}
 \caption{\label{fig:NMF2} Original image in Example 2}
\end{figurehere}
\end{center}

\begin{figurehere}
\hfill{}\\
\hskip -5cm
\includegraphics[clip,width=0.35\textwidth]{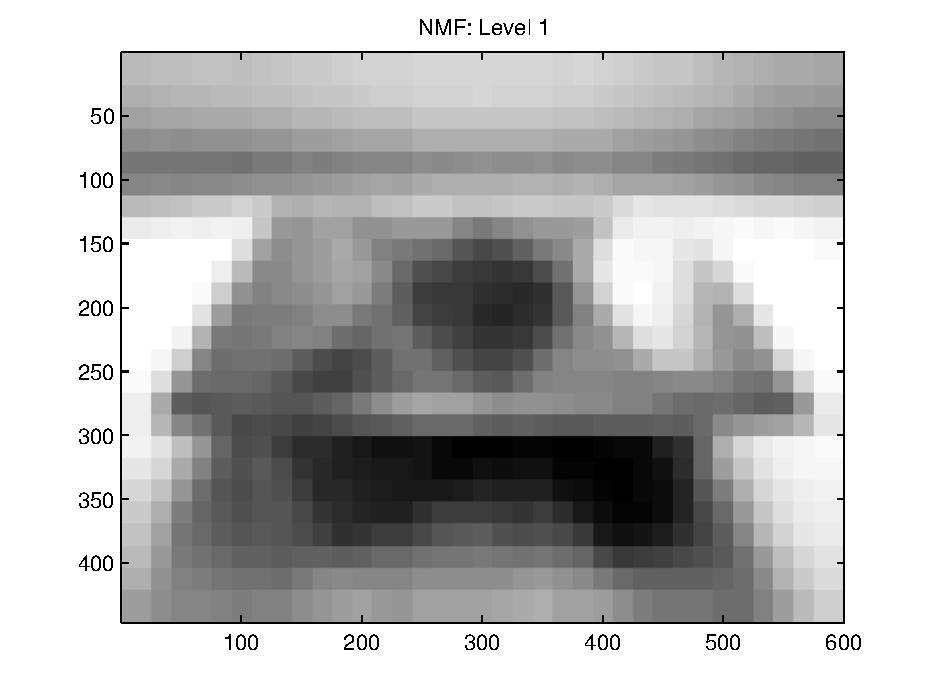}
\hskip -0.5cm
\includegraphics[clip,width=0.35\textwidth]{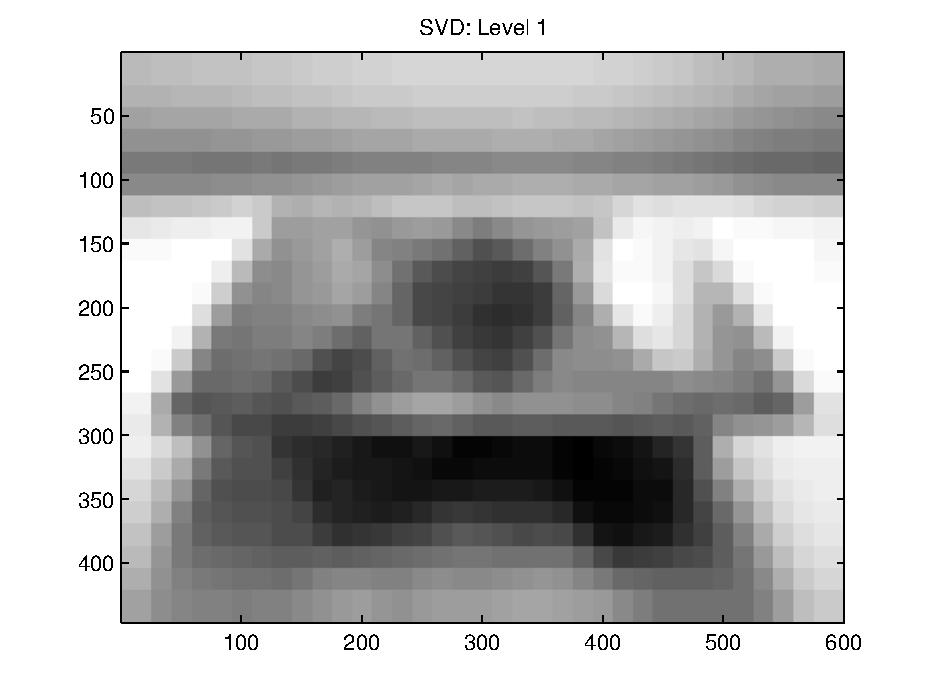}
\hskip -0.5cm
\includegraphics[clip,width=0.35\textwidth]{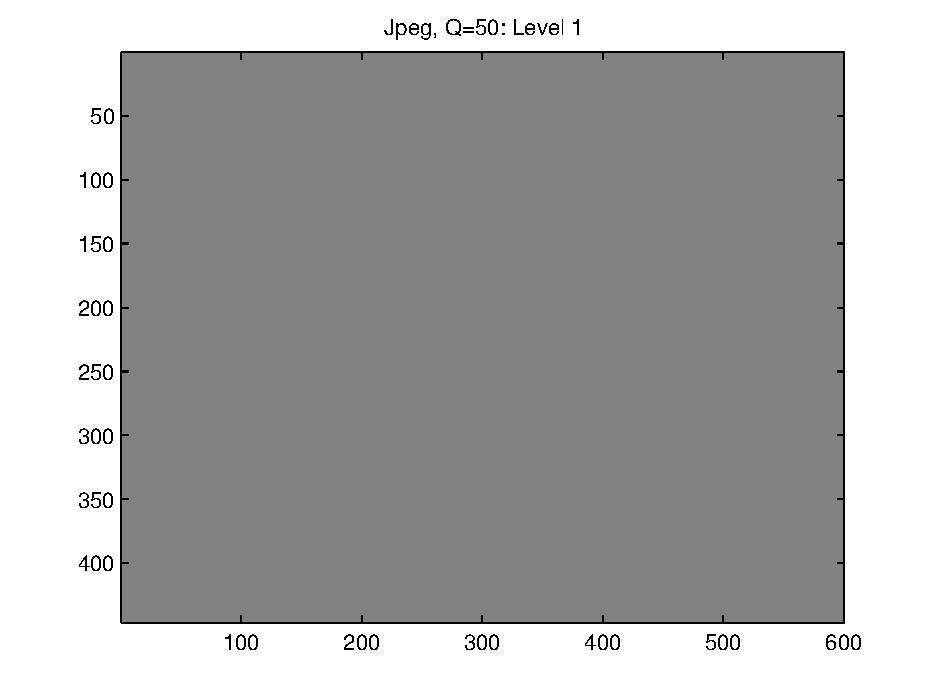}\\
\hskip -2cm
\includegraphics[clip,width=0.35\textwidth]{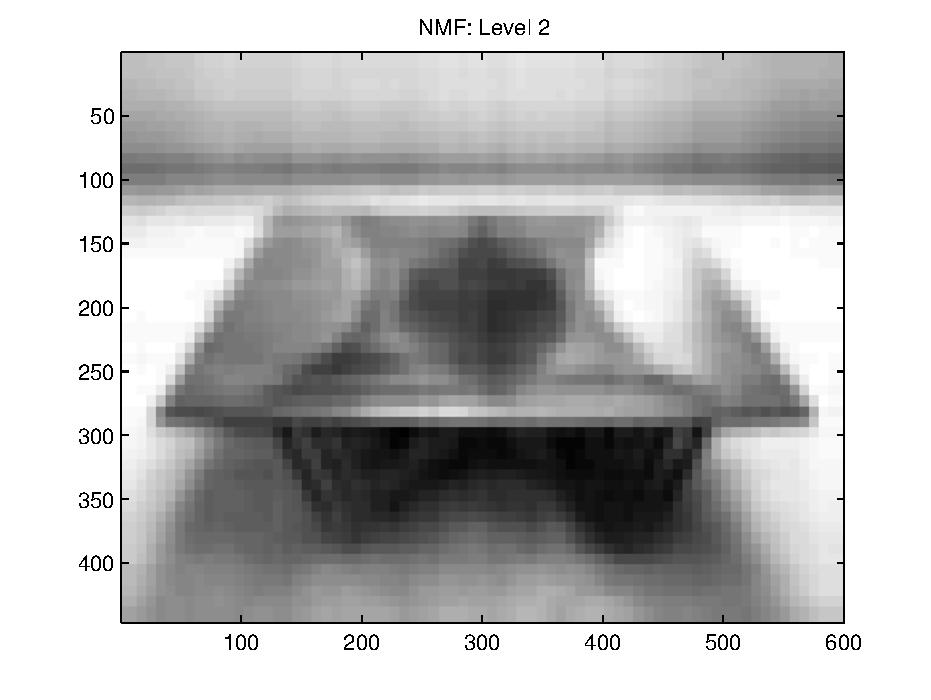}
\hskip -0.5cm
\includegraphics[clip,width=0.35\textwidth]{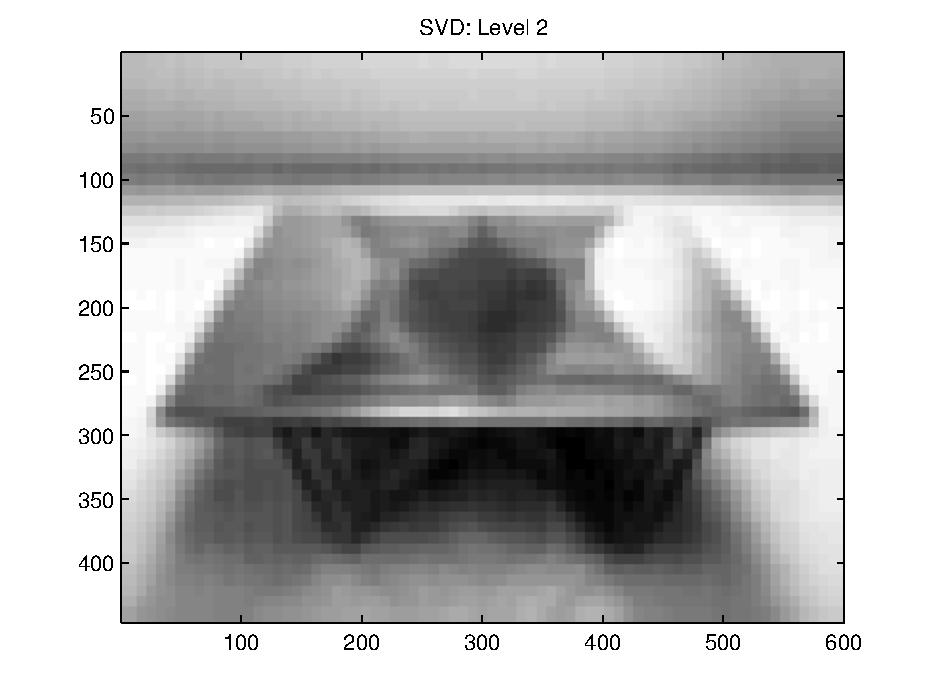}
\hskip -0.5cm
\includegraphics[clip,width=0.35\textwidth]{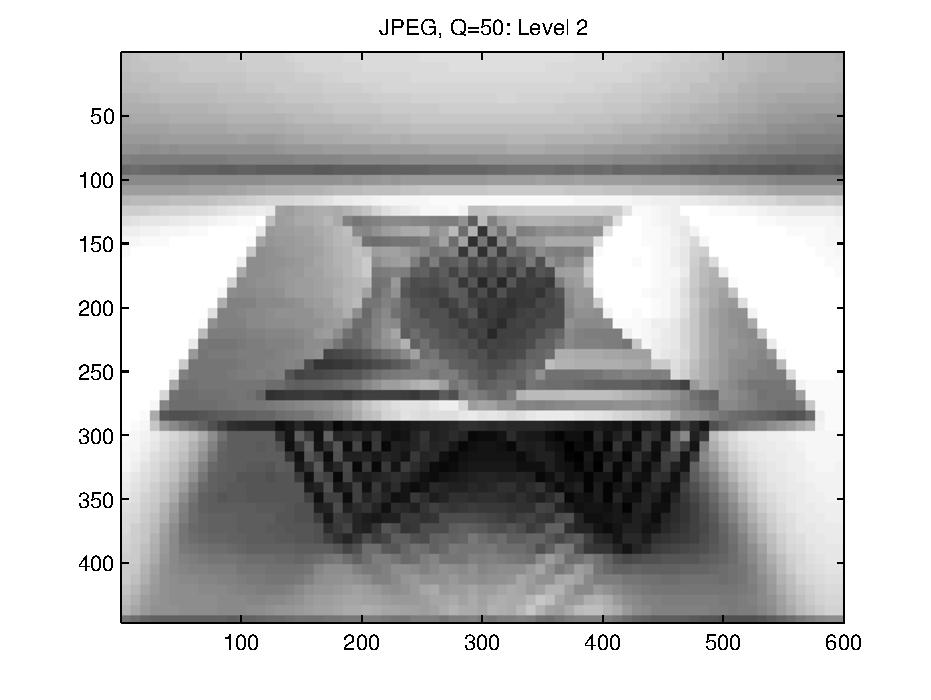} \\
\hskip -2cm
\includegraphics[clip,width=0.35\textwidth]{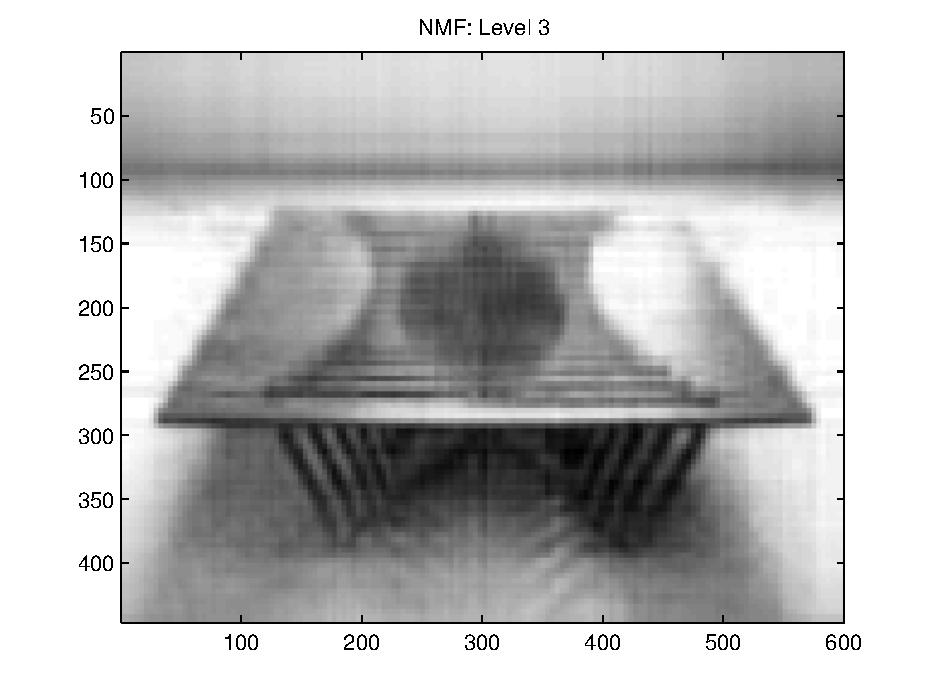}
\hskip -0.5cm
\includegraphics[clip,width=0.35\textwidth]{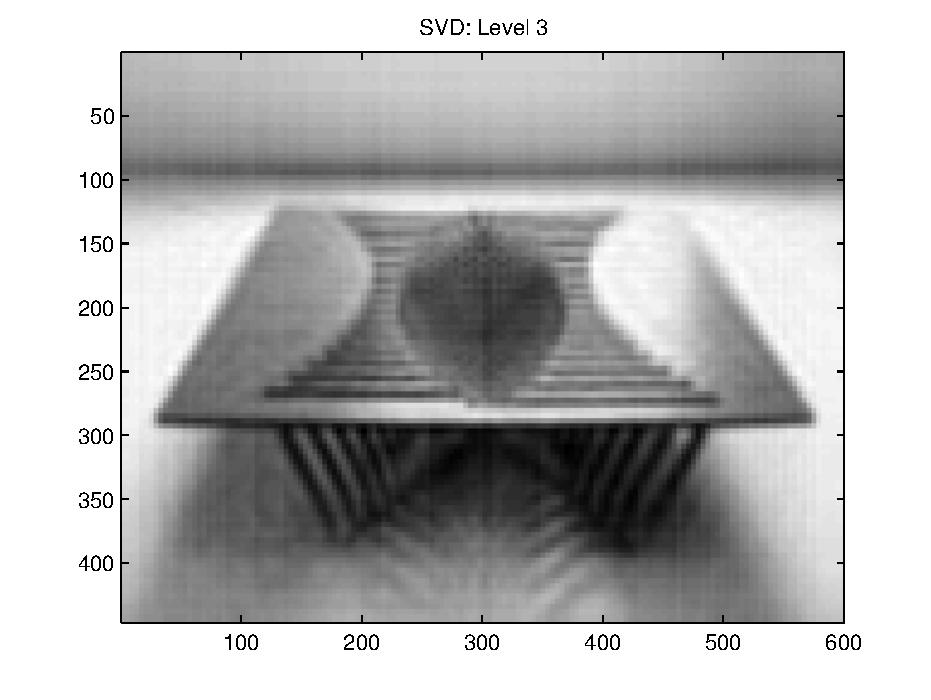}
\hskip -0.5cm
\includegraphics[clip,width=0.35\textwidth]{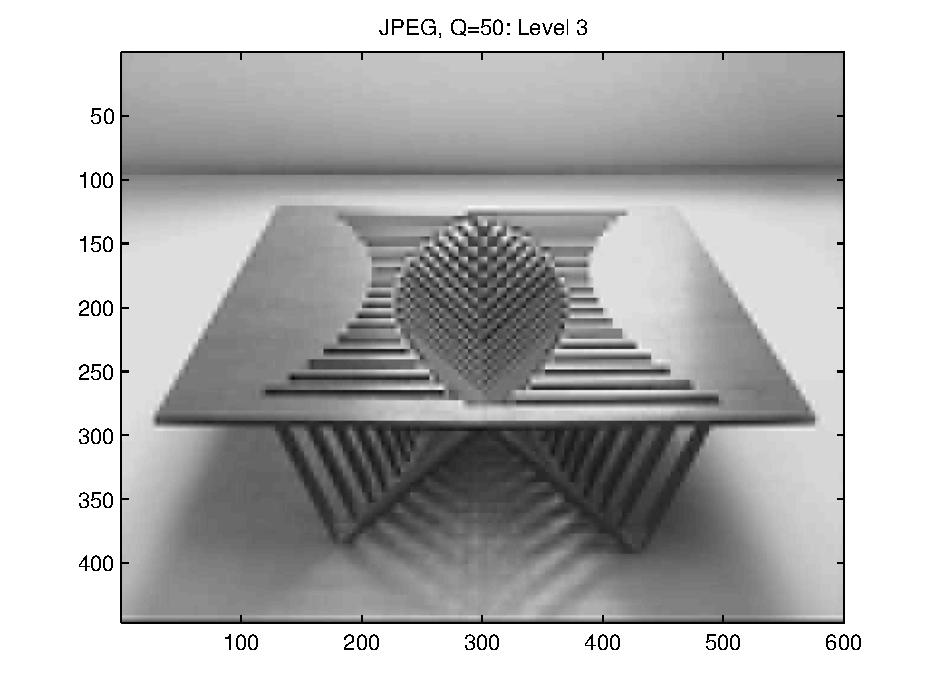} \\
\hskip -2cm
\includegraphics[clip,width=0.35\textwidth]{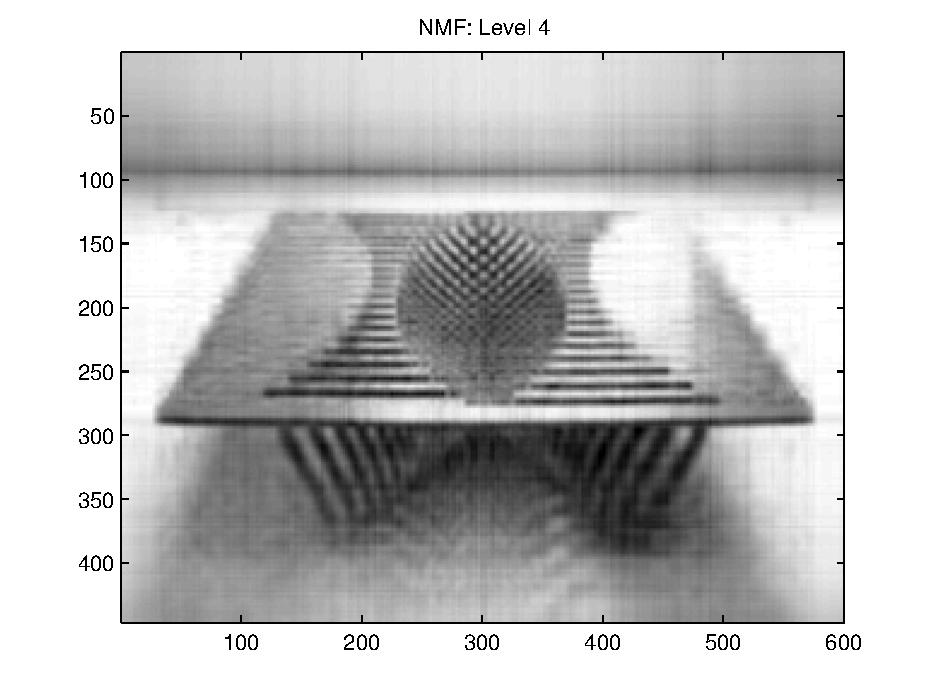}
\hskip -0.5cm
\includegraphics[clip,width=0.35\textwidth]{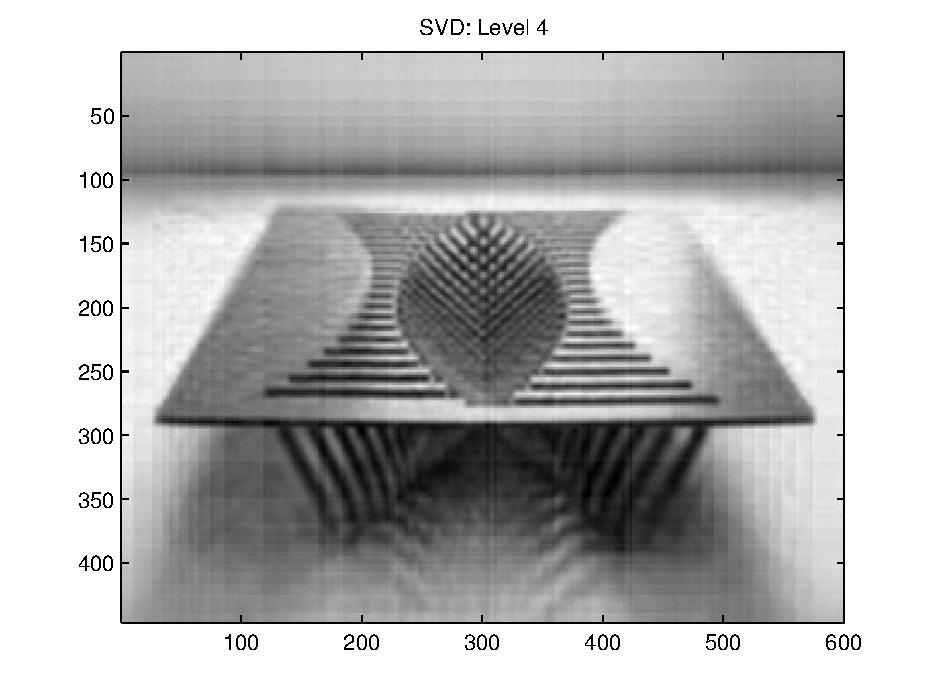}
\hskip -0.5cm
\includegraphics[clip,width=0.35\textwidth]{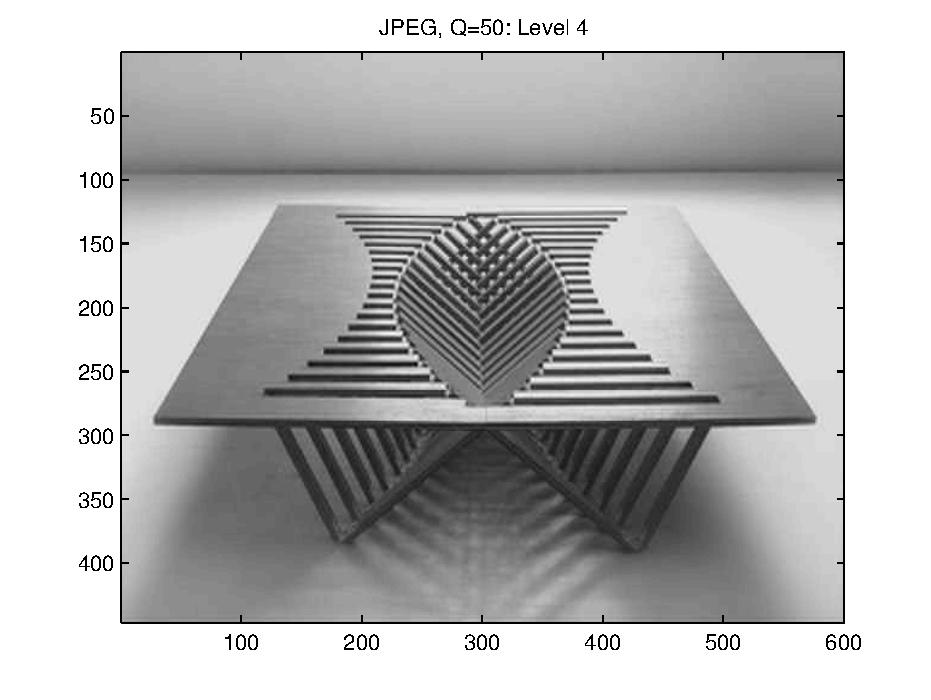} \\
\hskip -2cm
\includegraphics[clip,width=0.35\textwidth]{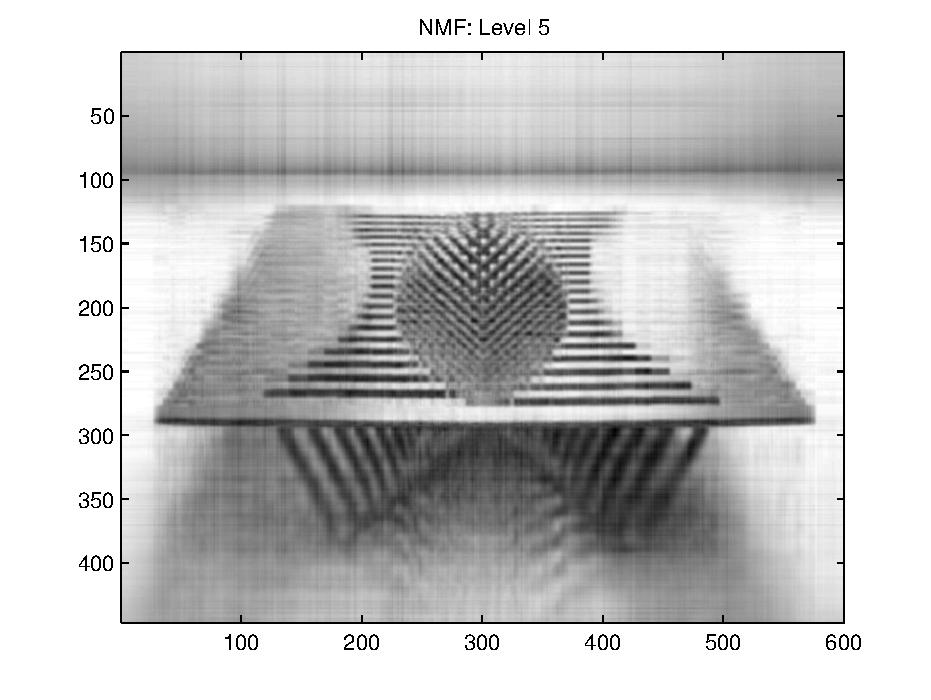}
\hskip -0.5cm
\includegraphics[clip,width=0.35\textwidth]{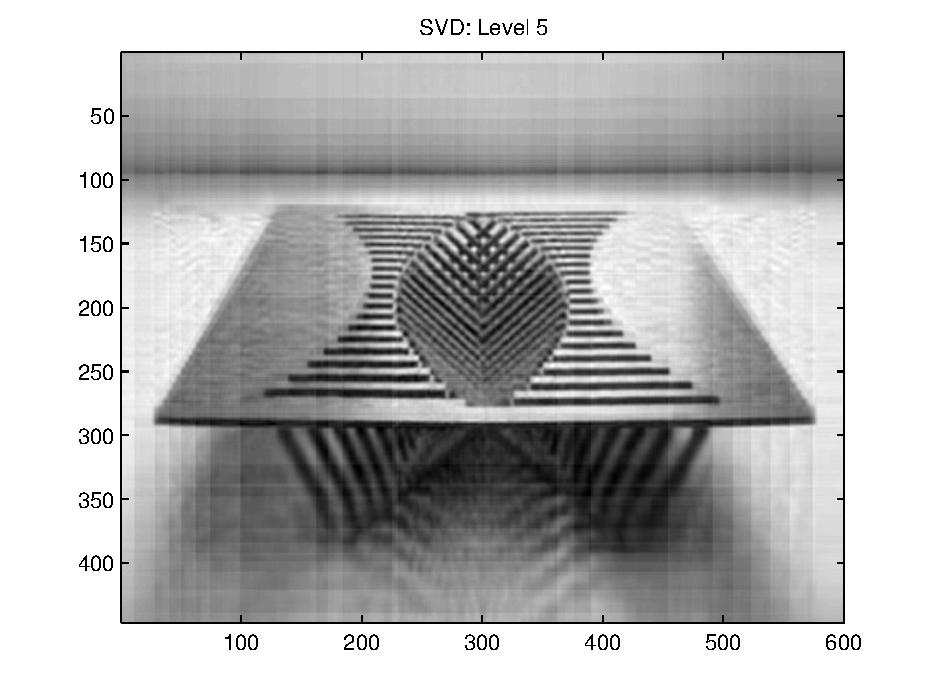}
\hskip -0.5cm
\includegraphics[clip,width=0.35\textwidth]{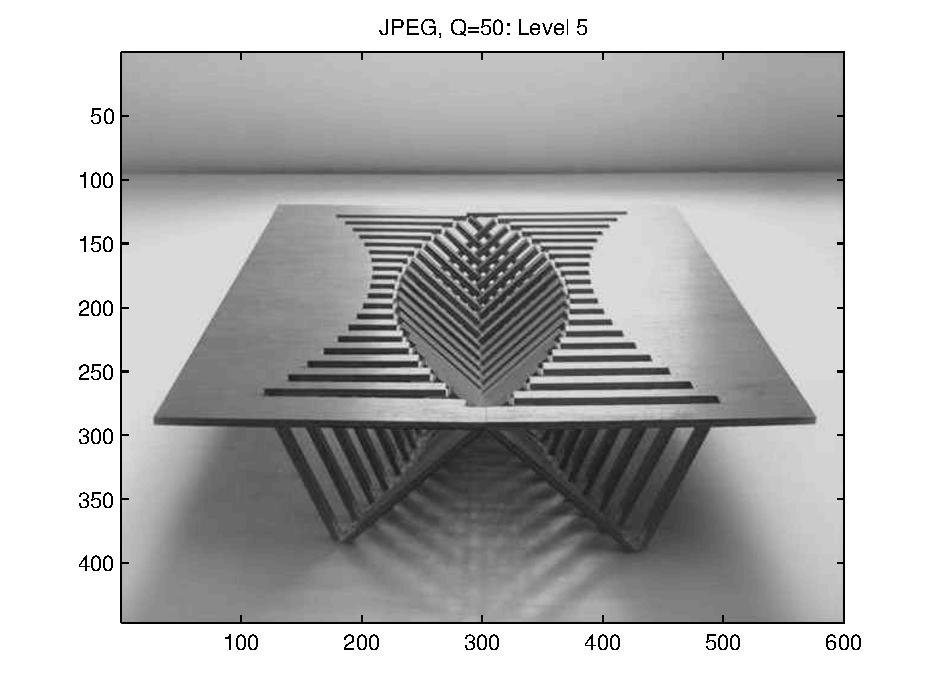}
 \caption{\label{fig:NMF2a} MLA for the image in Example 2 using NMF without noise }
\end{figurehere}

\begin{figurehere}
\hfill{}\\
\hskip -5cm
\includegraphics[clip,width=0.35\textwidth]{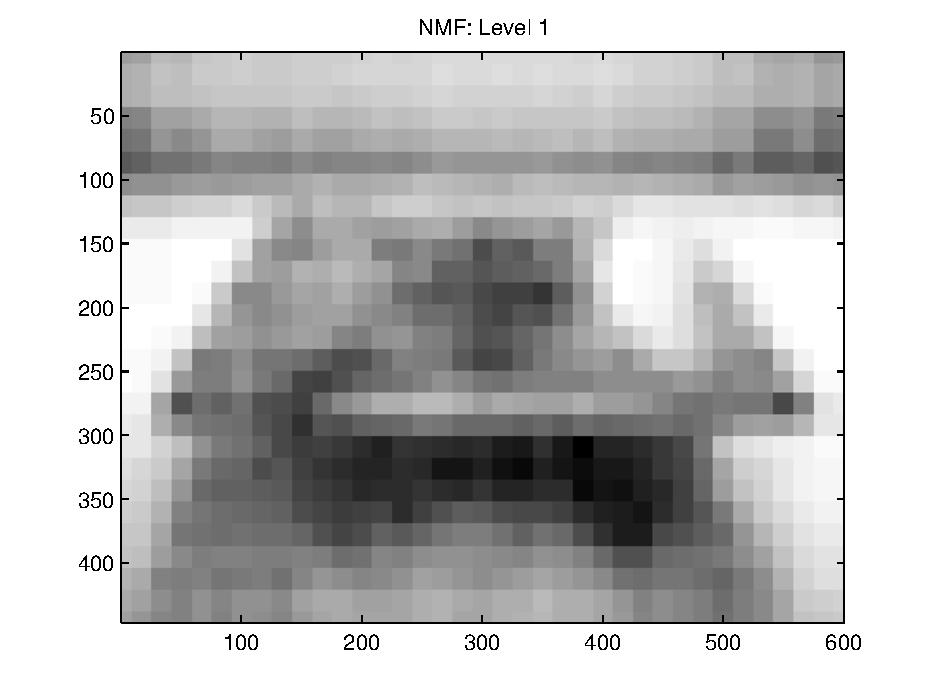}
\hskip -0.5cm
\includegraphics[clip,width=0.35\textwidth]{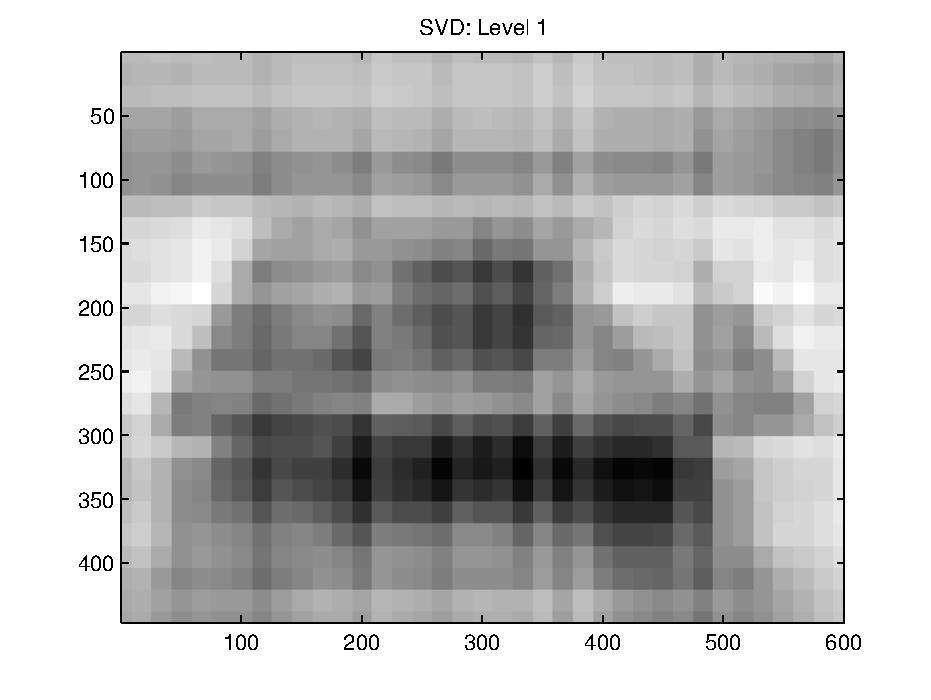}
\hskip -0.5cm
\includegraphics[clip,width=0.35\textwidth]{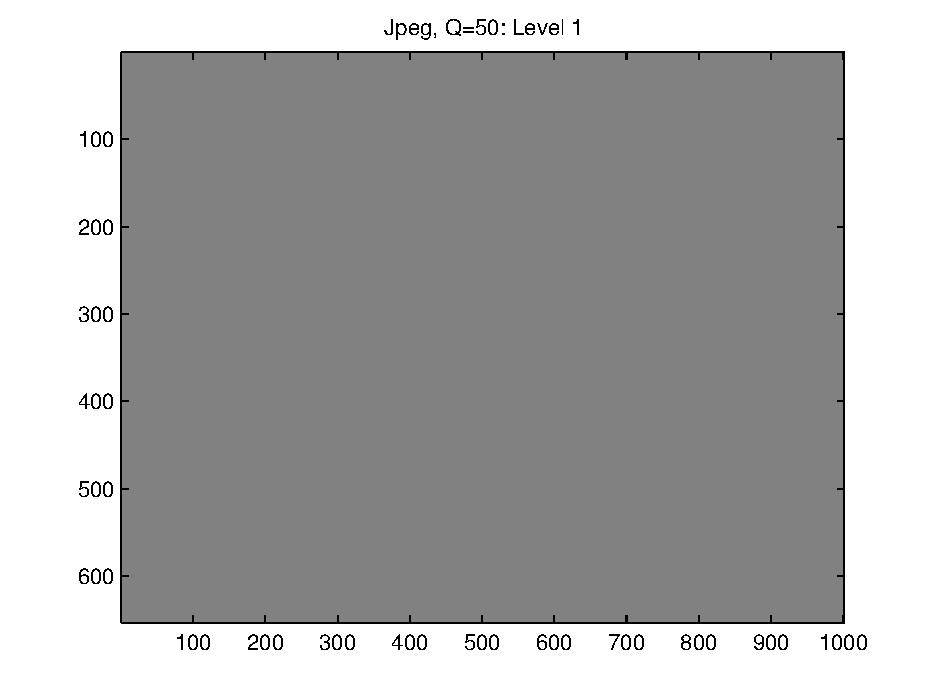}\\
\hskip -2cm
\includegraphics[clip,width=0.35\textwidth]{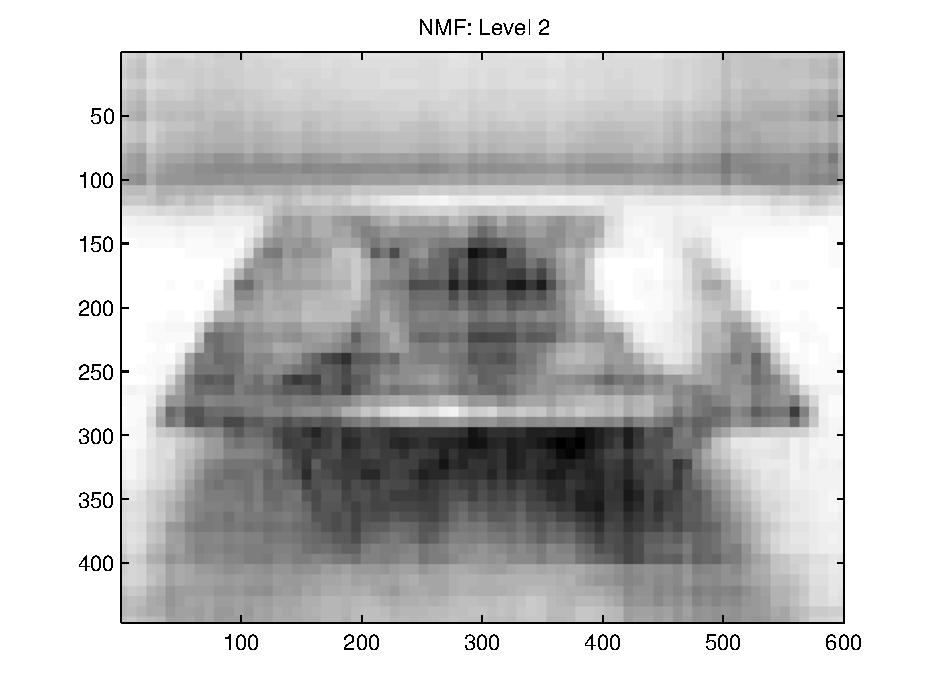}
\hskip -0.5cm
\includegraphics[clip,width=0.35\textwidth]{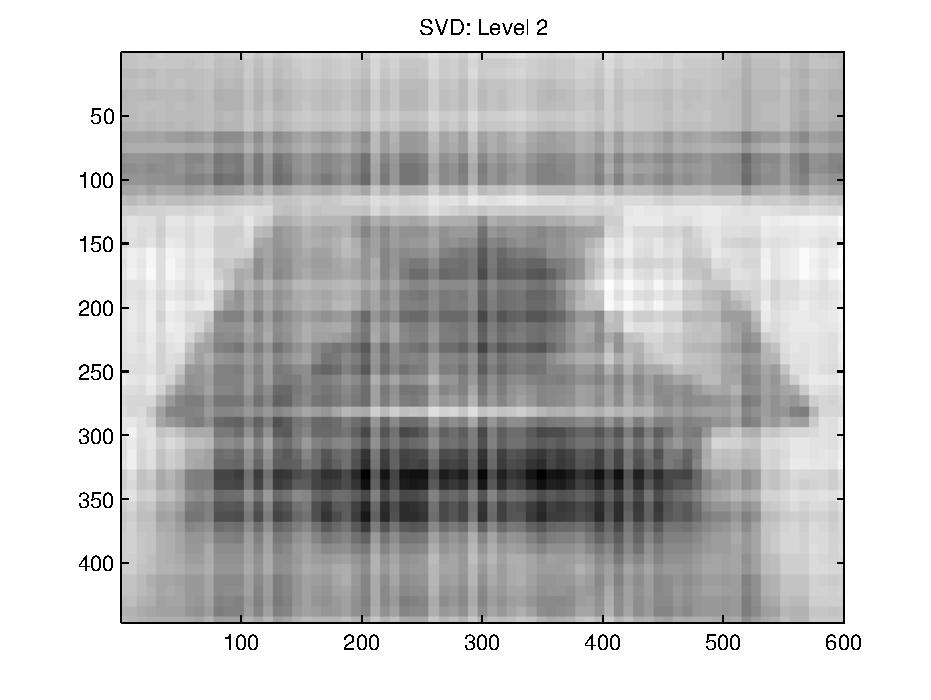}
\hskip -0.5cm
\includegraphics[clip,width=0.35\textwidth]{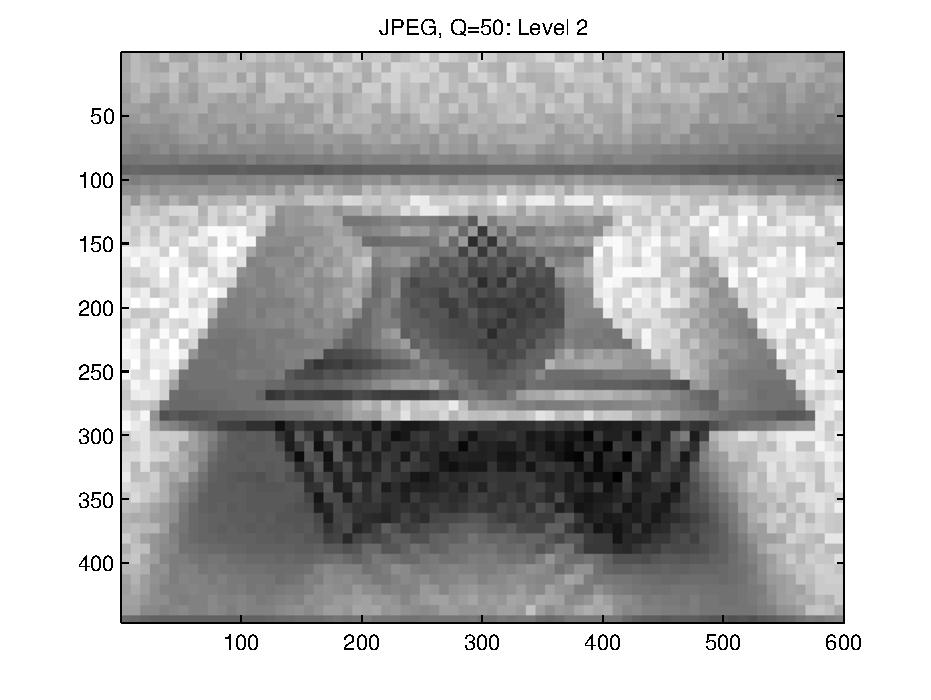} \\
\hskip -2cm
\includegraphics[clip,width=0.35\textwidth]{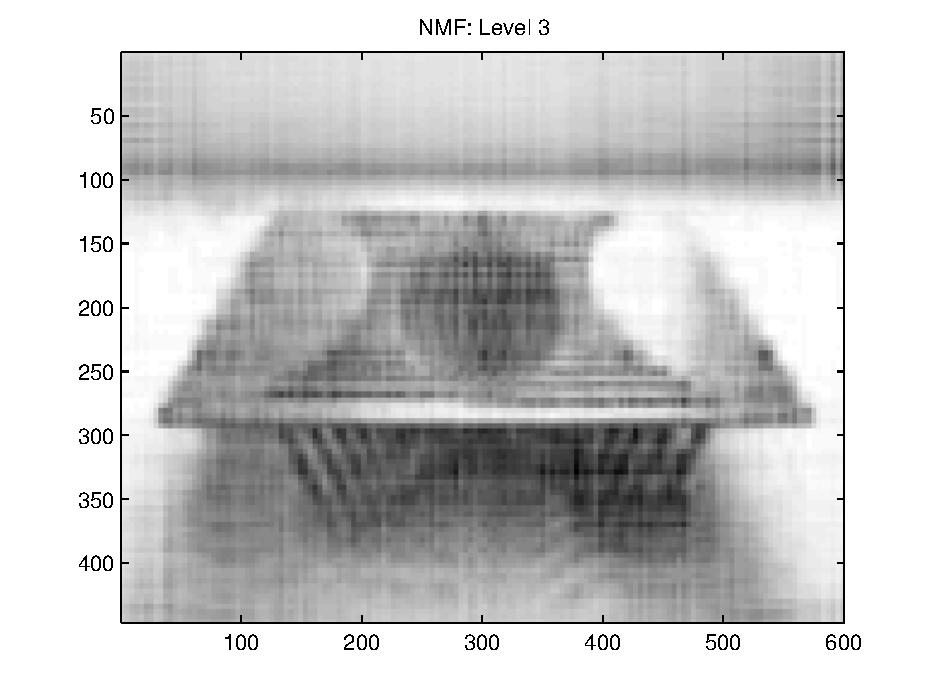}
\hskip -0.5cm
\includegraphics[clip,width=0.35\textwidth]{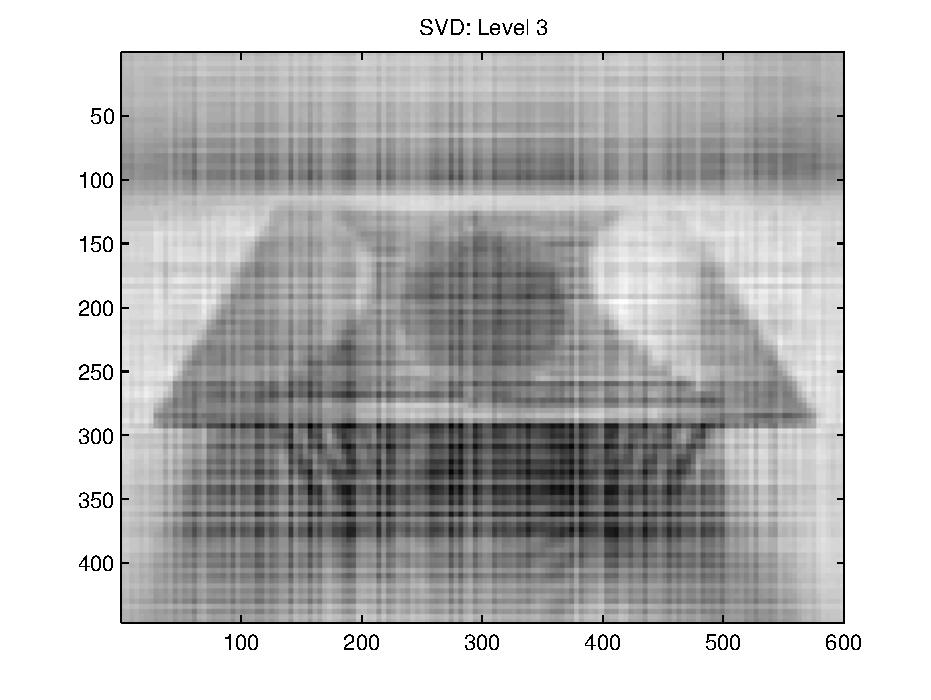}
\hskip -0.5cm
\includegraphics[clip,width=0.35\textwidth]{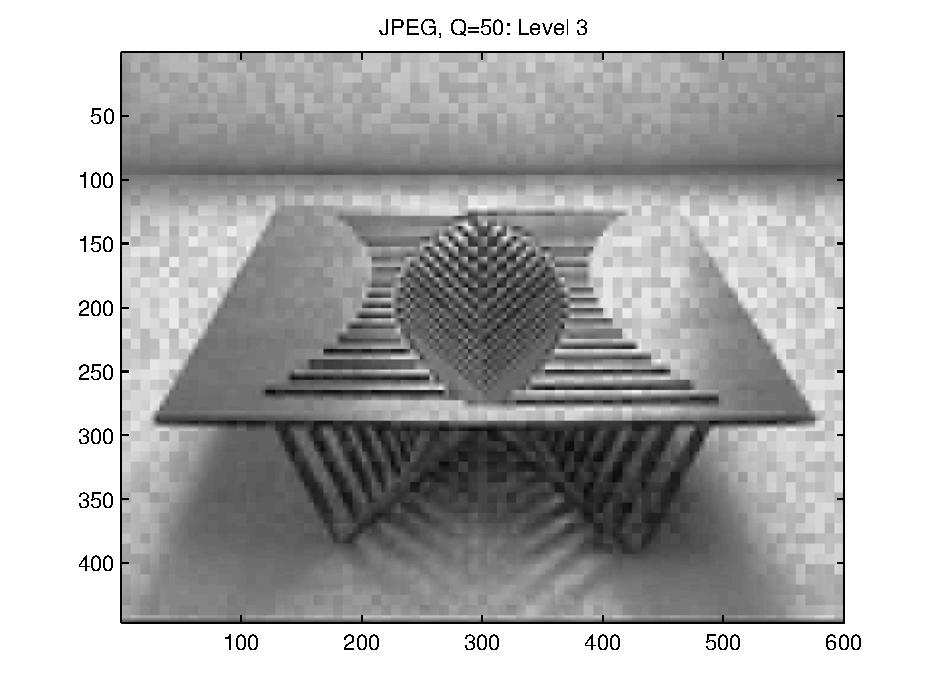} \\
\hskip -2cm
\includegraphics[clip,width=0.35\textwidth]{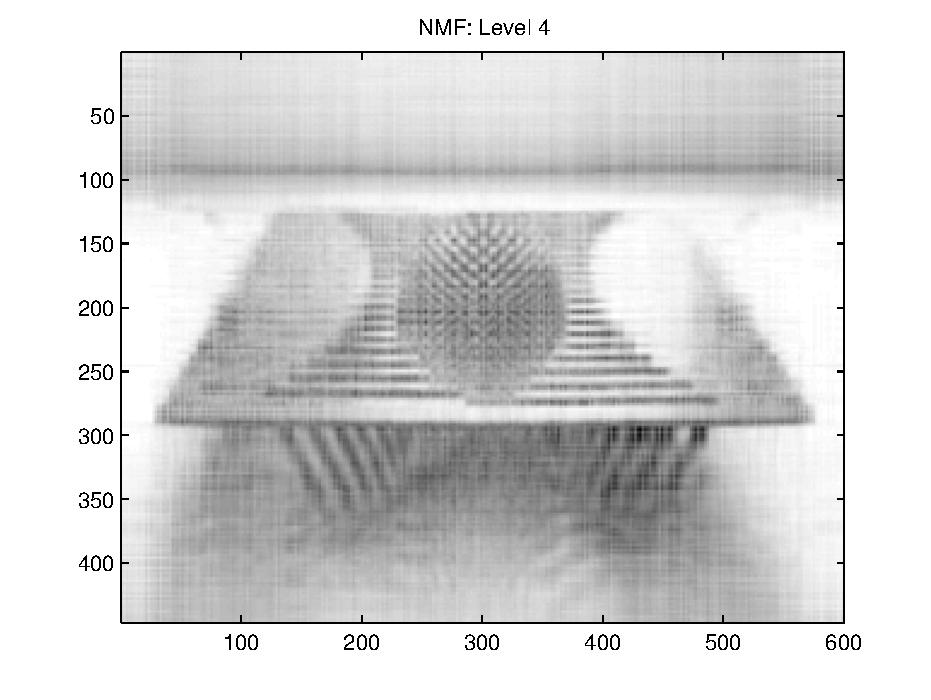}
\hskip -0.5cm
\includegraphics[clip,width=0.35\textwidth]{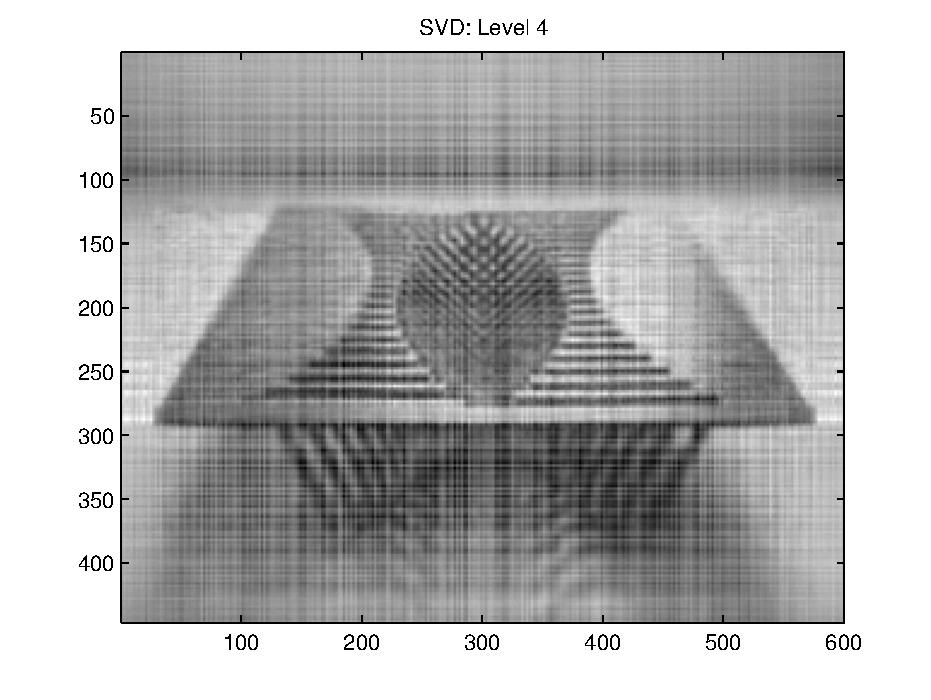}
\hskip -0.5cm
\includegraphics[clip,width=0.35\textwidth]{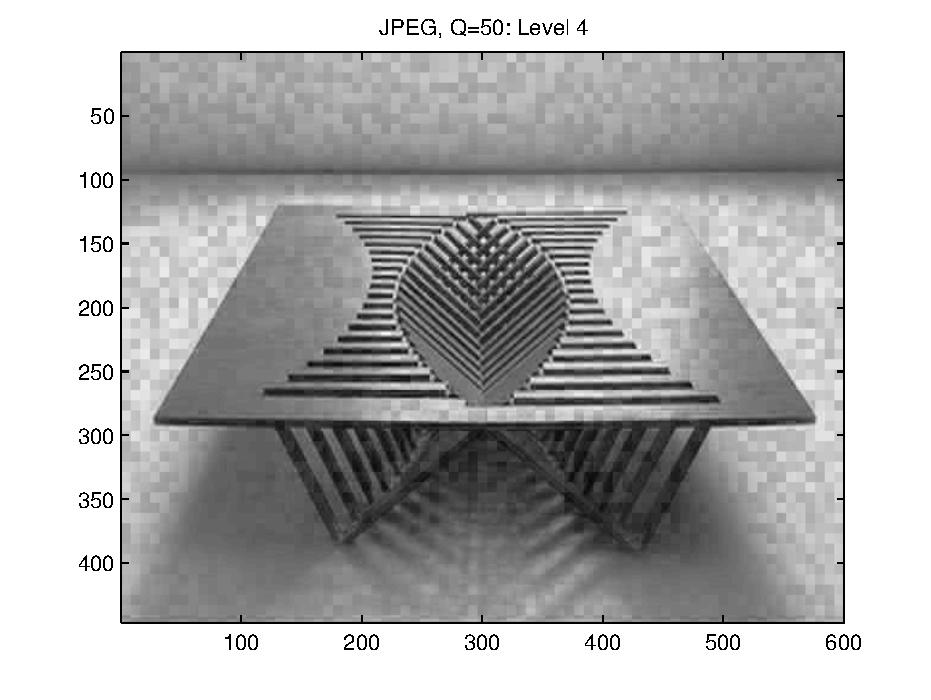} \\
\hskip -2cm
\includegraphics[clip,width=0.35\textwidth]{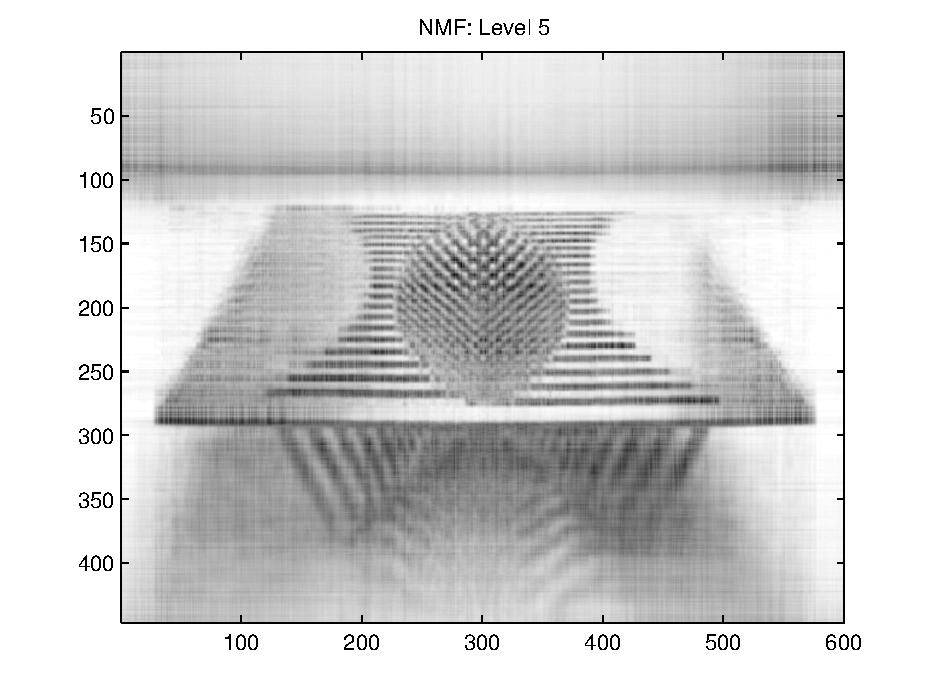}
\hskip -0.5cm
\includegraphics[clip,width=0.35\textwidth]{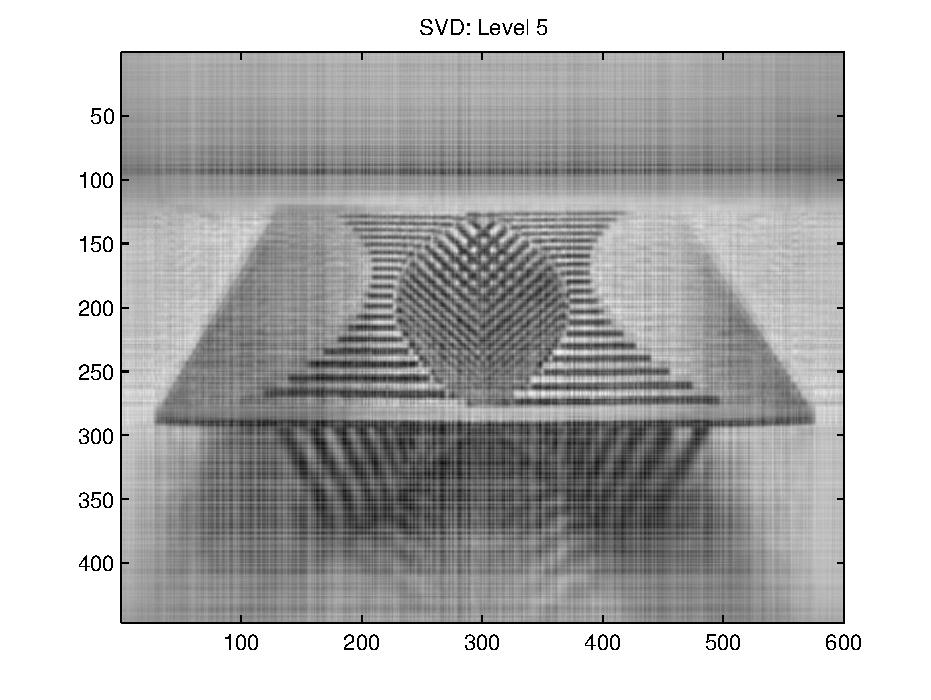}
\hskip -0.5cm
\includegraphics[clip,width=0.35\textwidth]{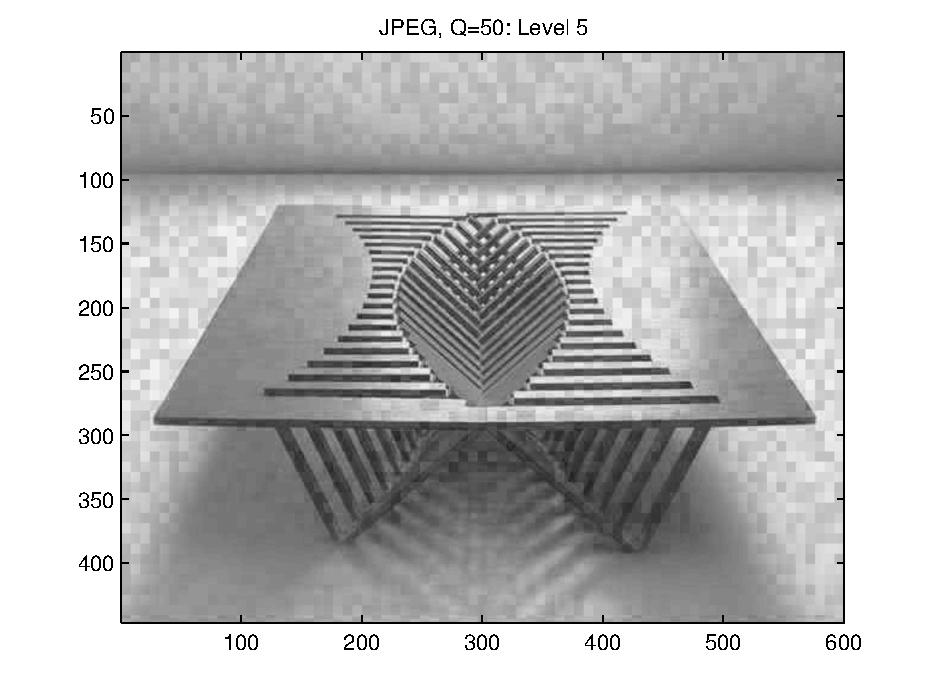}
 \caption{\label{fig:NMF2a2} MLA for the image in Example 2 using NMF with $25\%$ noise }
\end{figurehere}

\textbf{Example 3}. In this example, we use the same set of parameters as for the previous two examples
except that we set $s_{max} = [\log(\min(N,M))/\log(r)- 4]$ instead.
 $Y$ is set as the image in Figure \ref{fig:NMF3}. The resulting images are shown in Figures \ref{fig:NMF3a} and \ref{fig:NMF3a2}.  The memory complexity ratios for the $(s_{max}-s)$-th level of the three methods and their respective relative $L^2$ errors with and without noise are shown as follows:
\beqnx
\begin{matrix}
s_{max}-s &:& 1 & 2 & 3 &4 & 5  \\
p &:& 24 &   24   & 28  &  34 &   43\\
\tilde{p} &:& 134  & 129  &  89  & 120  & 187\\
\text{memory complexity ratio of NMF} &:&    0.0032  &  0.0056  &  0.0118  &  0.0267  &  0.0595 \\
\text{memory complexity ratio of SVD} &:&    0.0042  &  0.0085  &  0.0199 &   0.0483  &  0.1223 \\
\text{memory complexity ratio of JPEG} &:&  \text{NA}  &     0.0155   & 0.0480 &   0.1383 &   0.2005 \\
\text{Relative $L^2$ error in NMF (with $0\%$ noise)} &:&     0.4060 &   0.3814  &  0.3588  &  0.3391  &  0.3113 \\
\text{Relative $L^2$ error in SVD (with $0\%$ noise)} &:&    0.4075  &  0.3818  &  0.3633   & 0.3386 &   0.3112\\
\text{Relative $L^2$ error in JPEG (with $0\%$ noise)} &:&   \text{NA}  &  0.3542  &  0.3062  &  0.2120  &  0.1502 \\
\text{Relative $L^2$ error in NMF (with $25\%$ noise)} &:&  0.4164  &  0.3901  &  0.3760 &   0.3505 &   0.3310 \\
\text{Relative $L^2$ error in SVD (with $25\%$ noise)} &:&   0.4112 &   0.3912  &  0.3745  &  0.3499  &  0.3346 \\
\text{Relative $L^2$ error in JPEG (with $25\%$ noise)} &:&   \text{NA} &  0.3573  &  0.3058   & 0.2163 &   0.1605
\end{matrix}
\eqnx
We can see from Figure \ref{fig:NMF3a} that
in the absence of noise, although JPEG again performs the best among the three on the same layer, it requires usually about $4$ times of the memory than NMF due to the complexity of the figure. If we pick a memory complexity ratio of around $5$ to $6$ percent, then we can choose an NMF of the $5$-th level, while we can only choose a level $3$ among the JPEG images which provides much less finer details of the building.
With the presence of noise, the relative $L^2$ errors of the JPEG is the least among the three as shown in the above table.  Nonetheless in Figure \ref{fig:NMF3a2}, we actually notice that the several NMF layers do not seem quite different from the ones without noise, whereas the SVD and the JPEG images are obviously contaminated respectively by straight strips and random squares.

\begin{center}
\begin{figurehere}
\hfill{}\includegraphics[clip,width=0.35\textwidth]{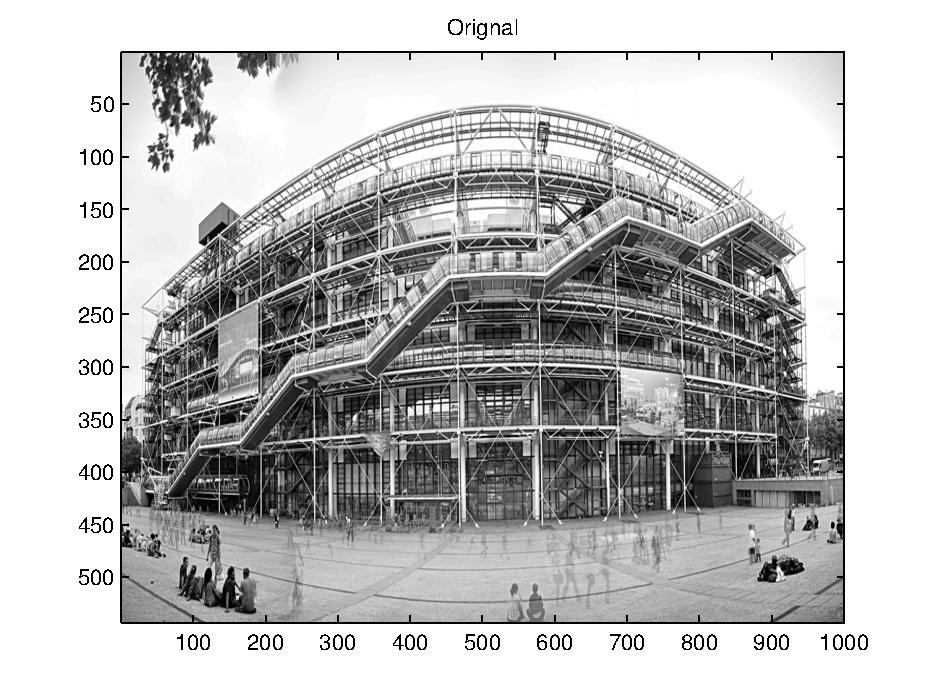}\hfill{}
 \caption{\label{fig:NMF3} Original image in Example 3}
\end{figurehere}
\end{center}

\begin{figurehere}
\hfill{}\\
\hskip -5cm
\includegraphics[clip,width=0.35\textwidth]{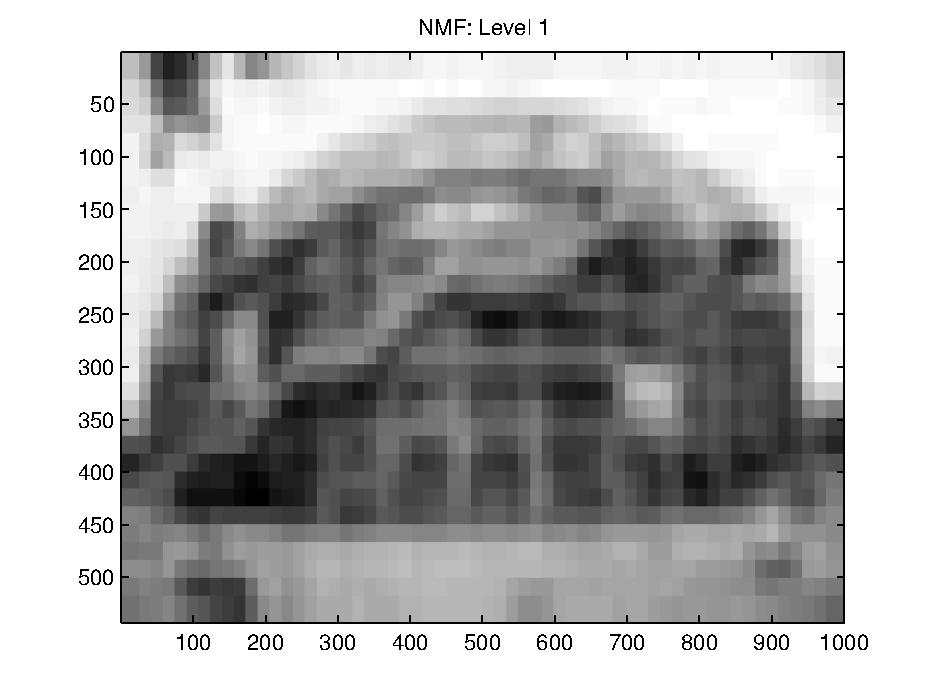}
\hskip -0.5cm
\includegraphics[clip,width=0.35\textwidth]{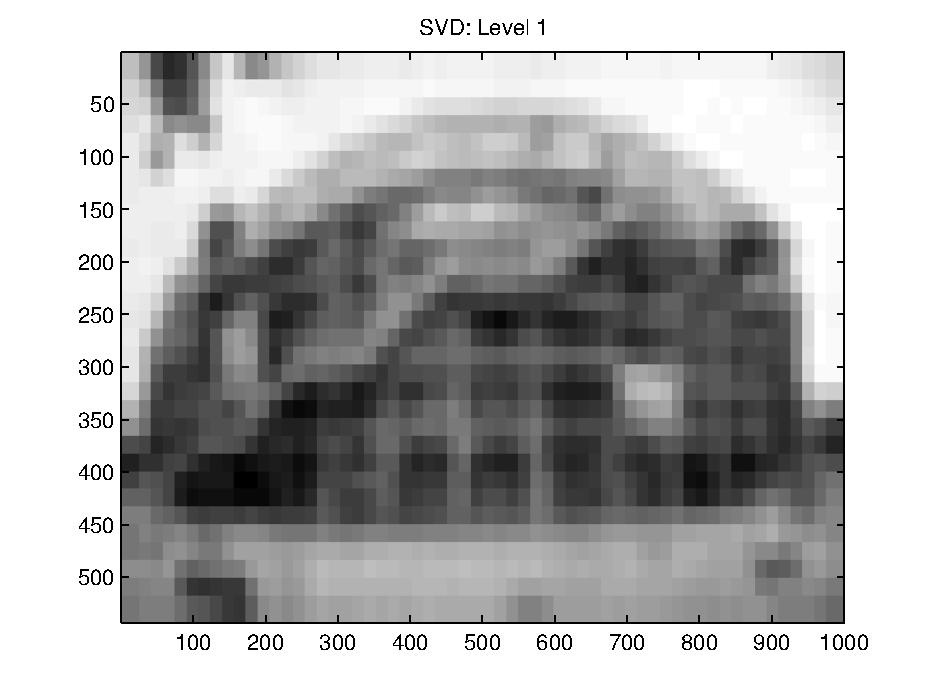}
\hskip -0.5cm
\includegraphics[clip,width=0.35\textwidth]{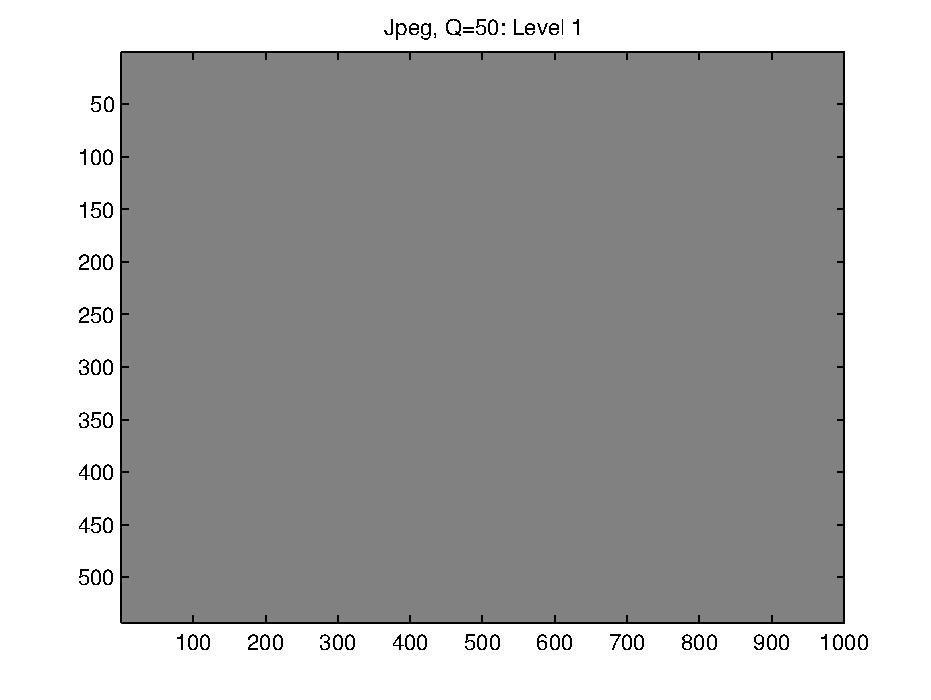}\\
\hskip -2cm
\includegraphics[clip,width=0.35\textwidth]{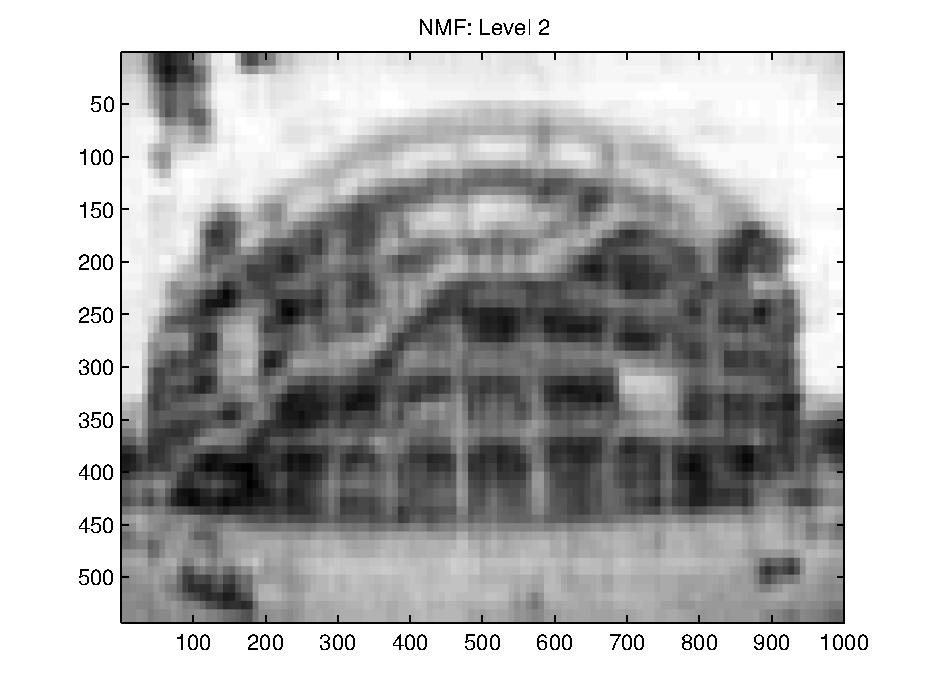}
\hskip -0.5cm
\includegraphics[clip,width=0.35\textwidth]{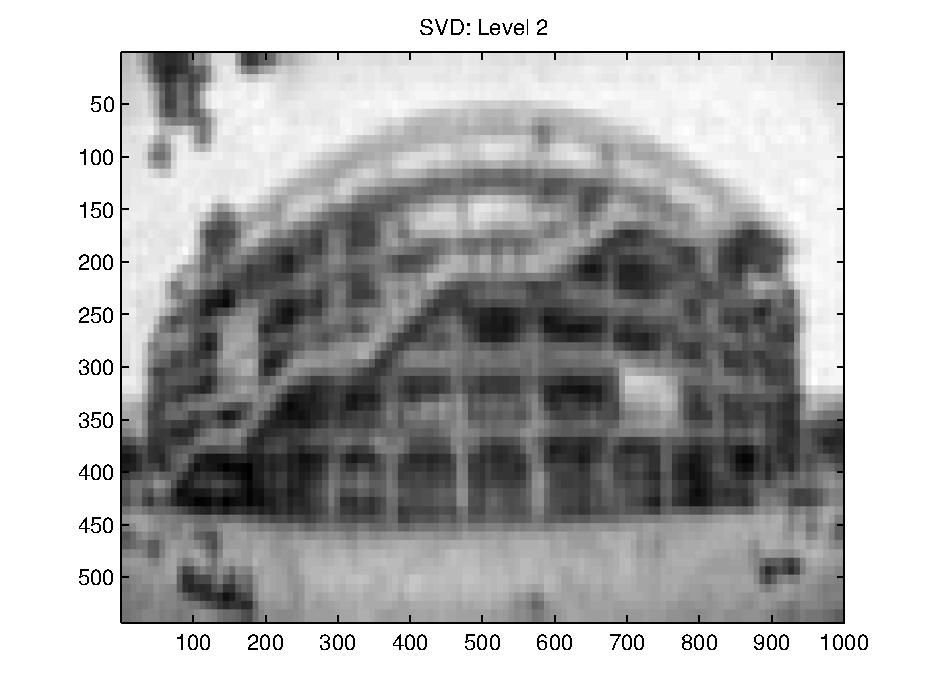}
\hskip -0.5cm
\includegraphics[clip,width=0.35\textwidth]{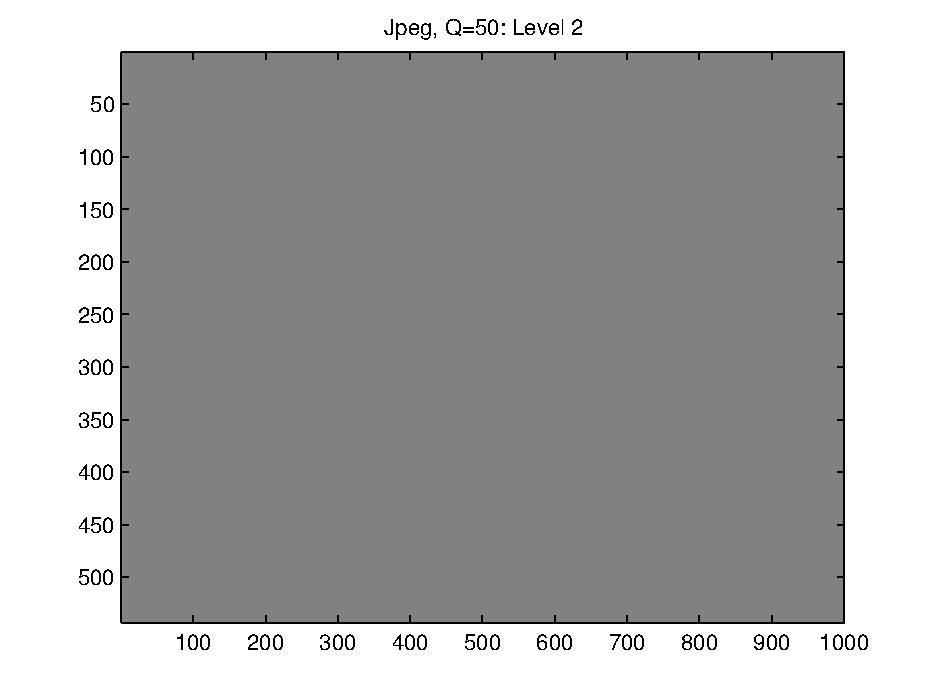} \\
\hskip -2cm
\includegraphics[clip,width=0.35\textwidth]{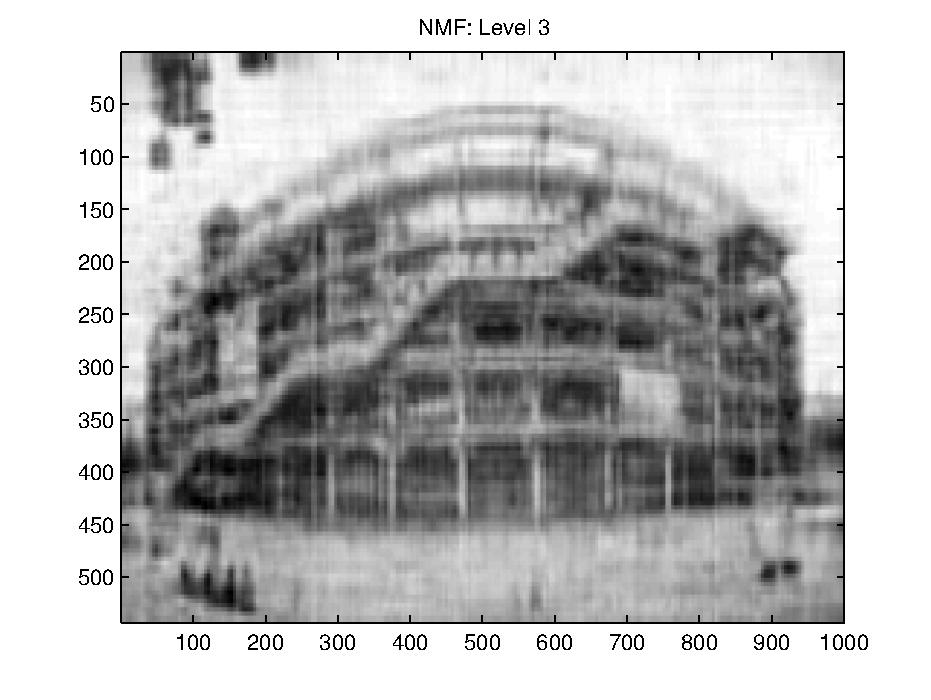}
\hskip -0.5cm
\includegraphics[clip,width=0.35\textwidth]{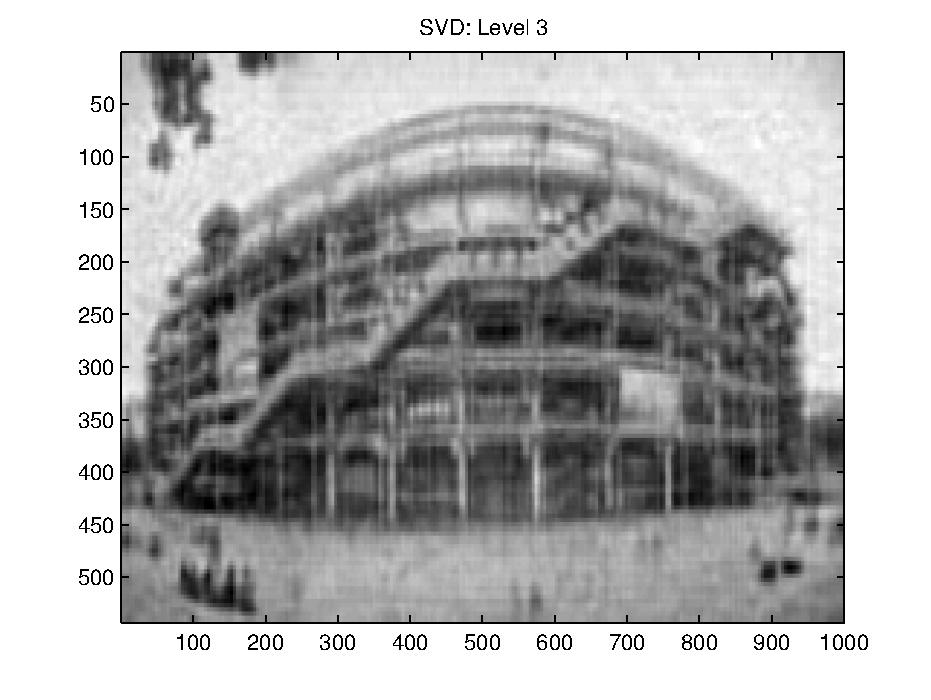}
\hskip -0.5cm
\includegraphics[clip,width=0.35\textwidth]{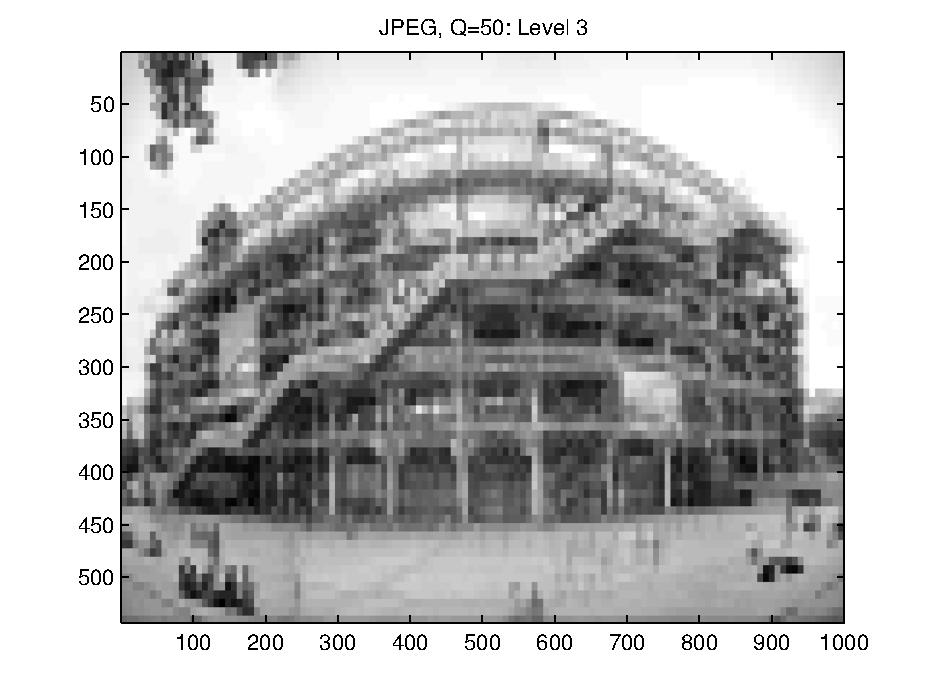} \\
\hskip -2cm
\includegraphics[clip,width=0.35\textwidth]{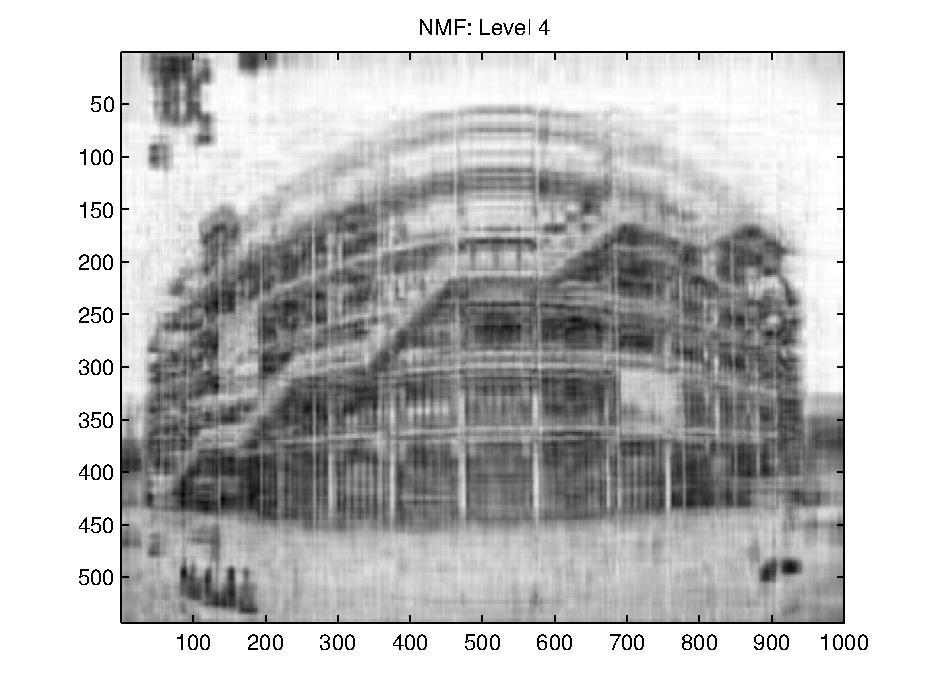}
\hskip -0.5cm
\includegraphics[clip,width=0.35\textwidth]{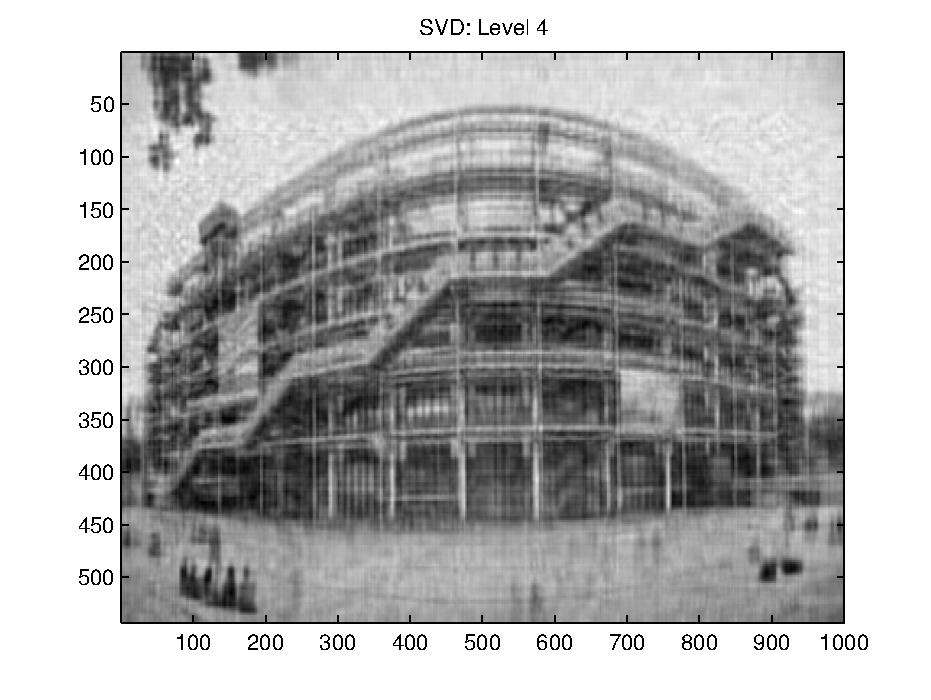}
\hskip -0.5cm
\includegraphics[clip,width=0.35\textwidth]{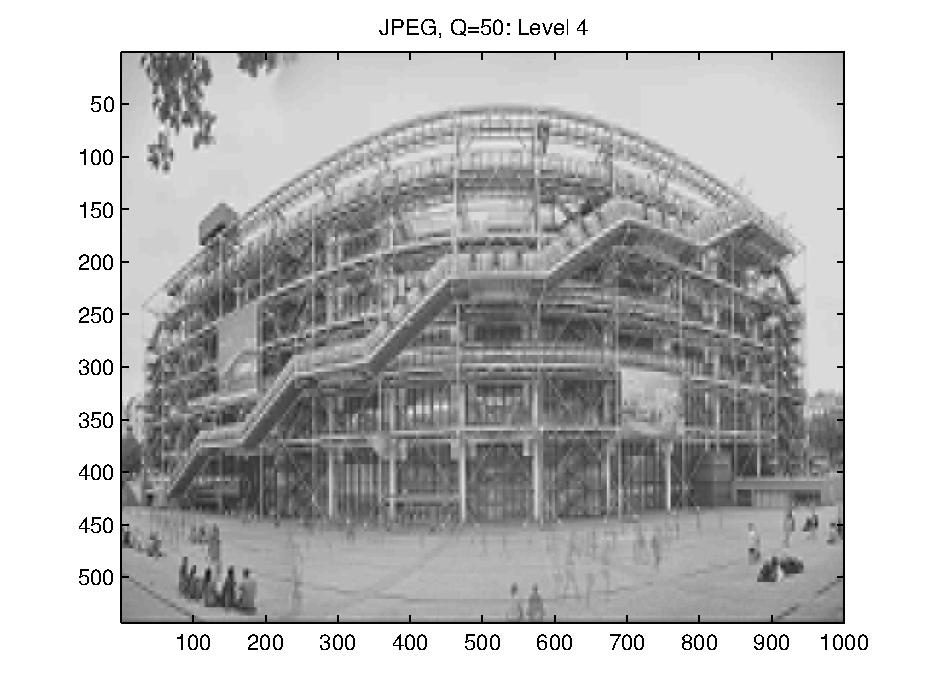} \\
\hskip -2cm
\includegraphics[clip,width=0.35\textwidth]{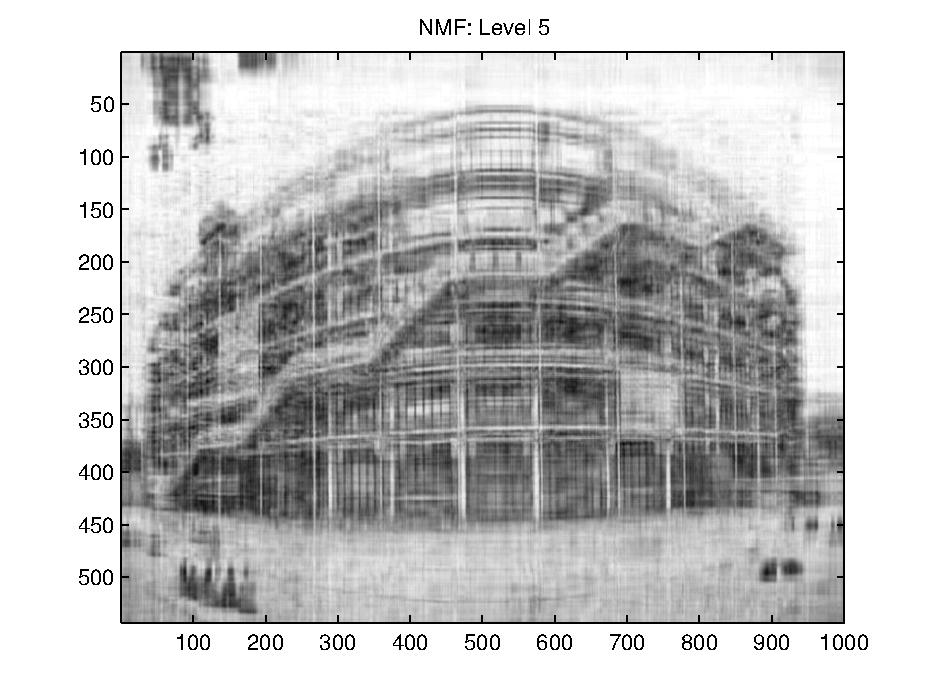}
\hskip -0.5cm
\includegraphics[clip,width=0.35\textwidth]{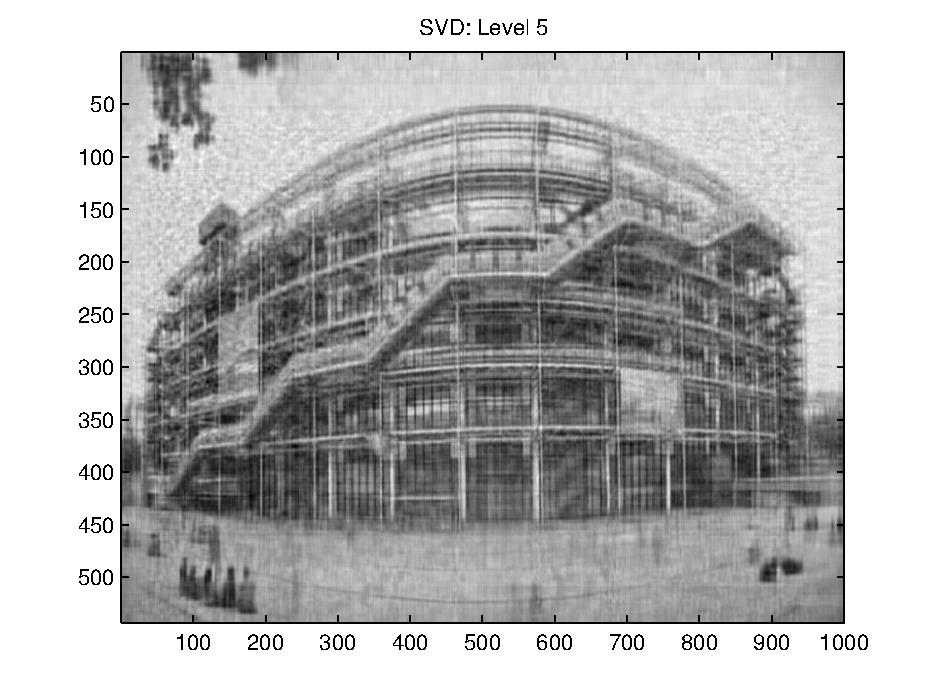}
\hskip -0.5cm
\includegraphics[clip,width=0.35\textwidth]{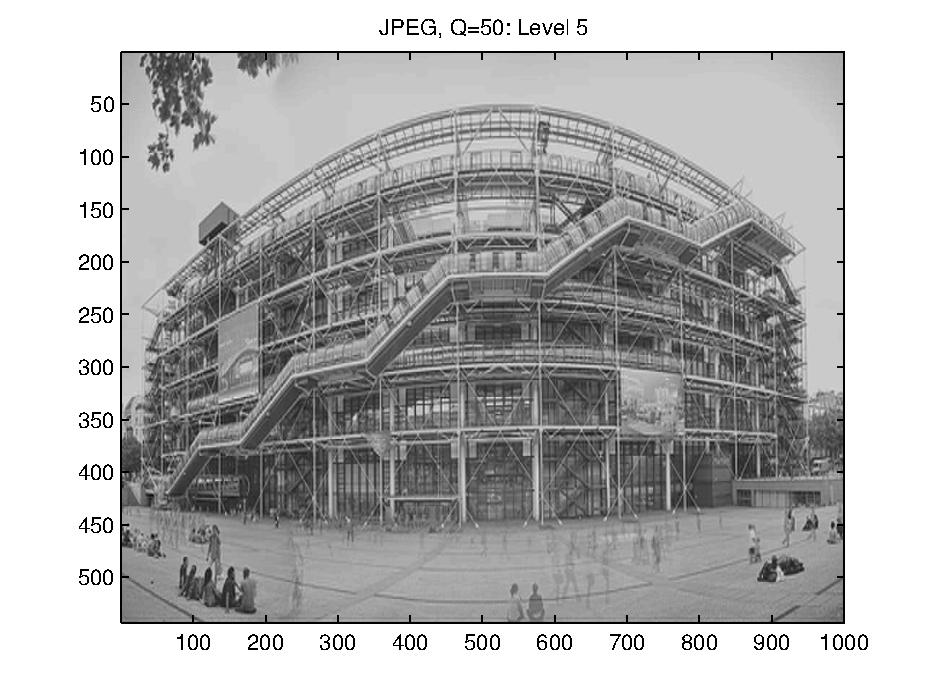}
 \caption{\label{fig:NMF3a} MLA for the image in Example 3 using NMF without noise }
\end{figurehere}

\begin{figurehere}
\hfill{}\\
\hskip -5cm
\includegraphics[clip,width=0.35\textwidth]{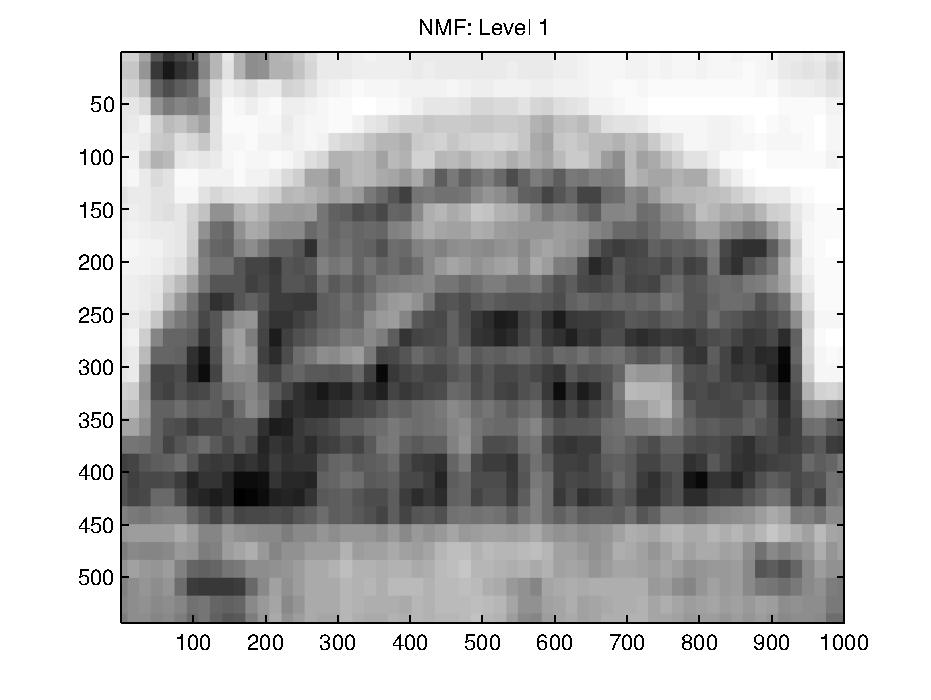}
\hskip -0.5cm
\includegraphics[clip,width=0.35\textwidth]{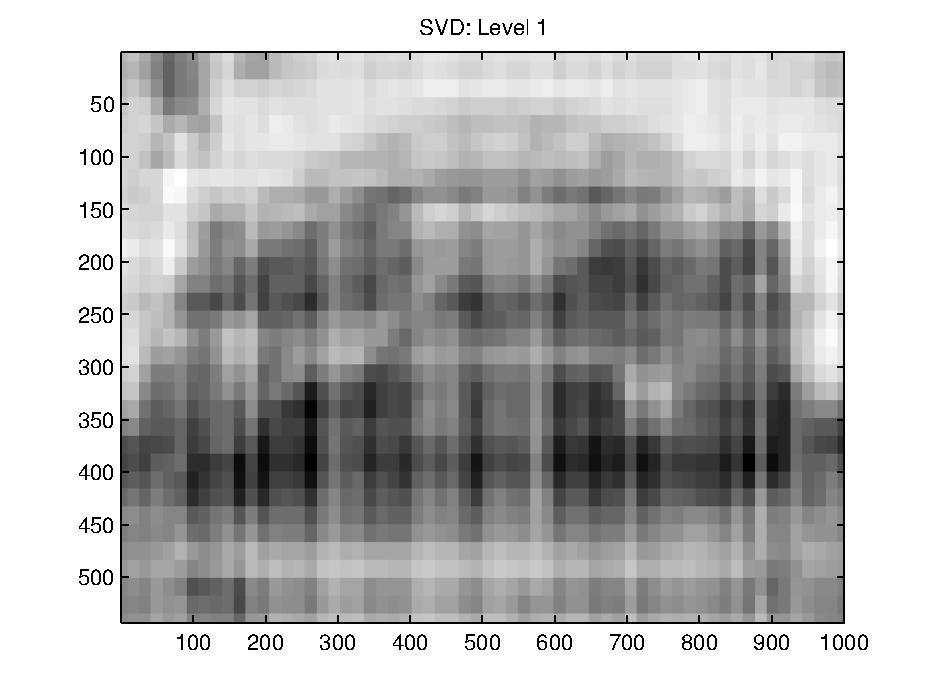}
\hskip -0.5cm
\includegraphics[clip,width=0.35\textwidth]{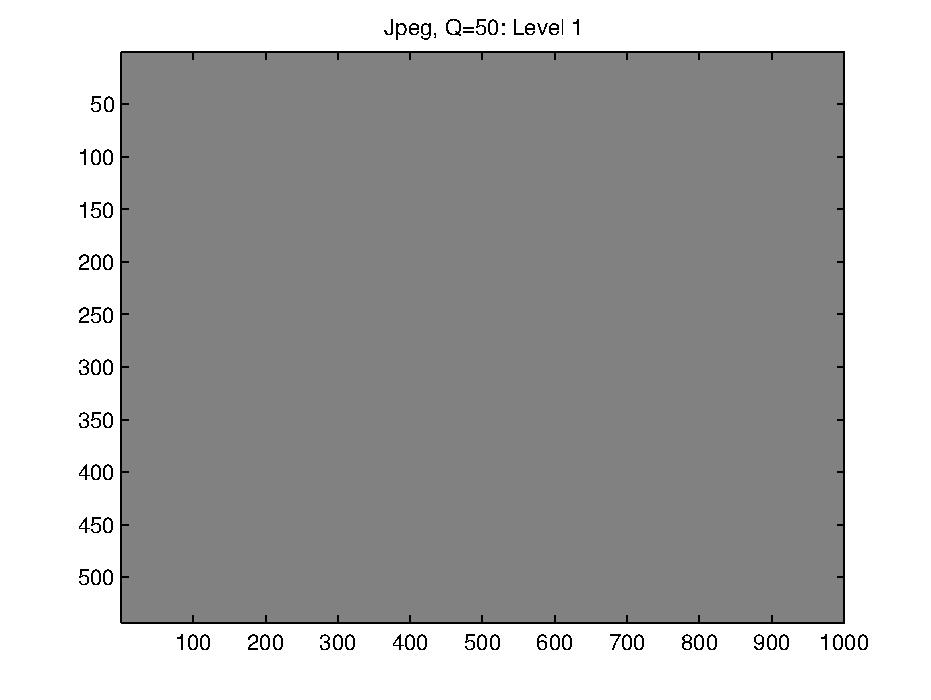}\\
\hskip -2cm
\includegraphics[clip,width=0.35\textwidth]{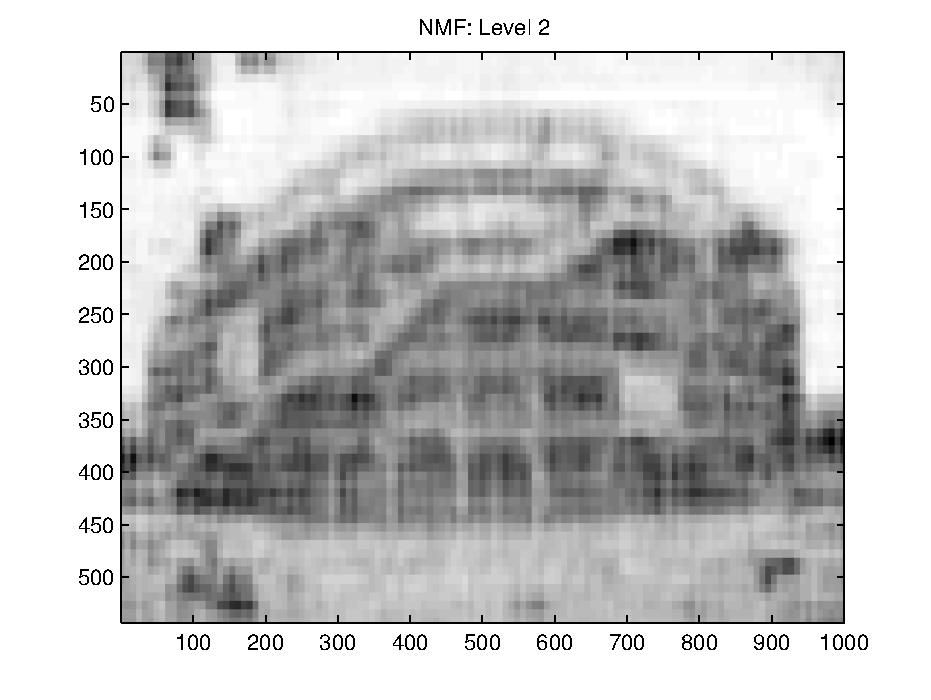}
\hskip -0.5cm
\includegraphics[clip,width=0.35\textwidth]{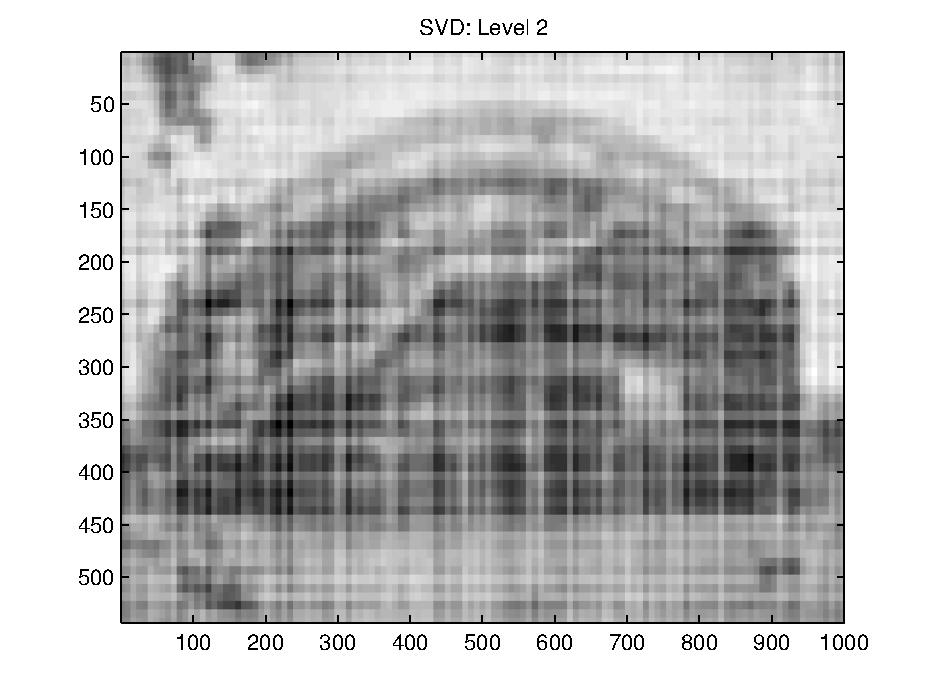}
\hskip -0.5cm
\includegraphics[clip,width=0.35\textwidth]{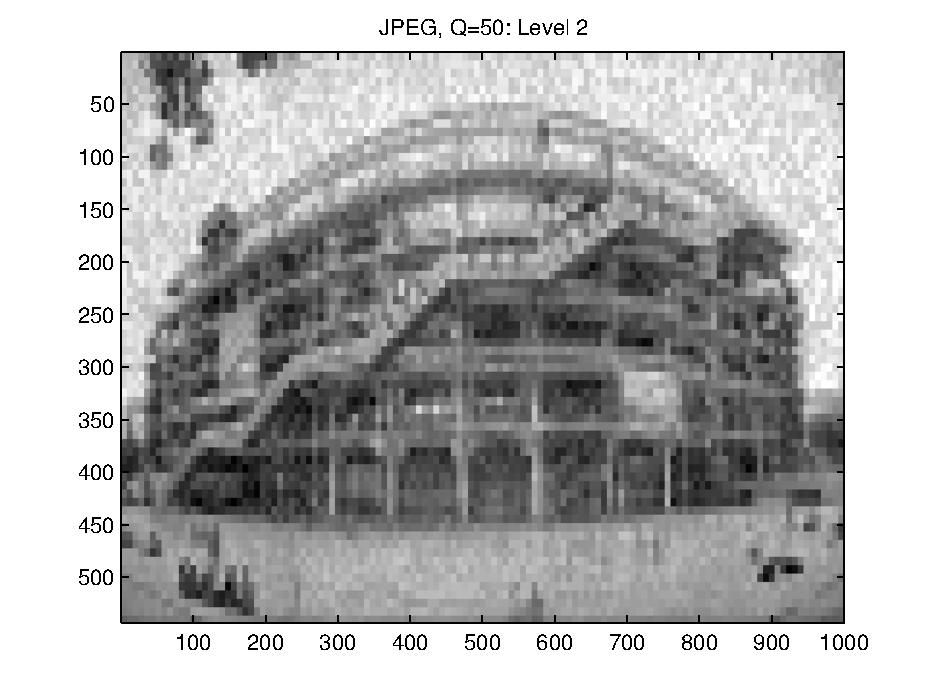} \\
\hskip -2cm
\includegraphics[clip,width=0.35\textwidth]{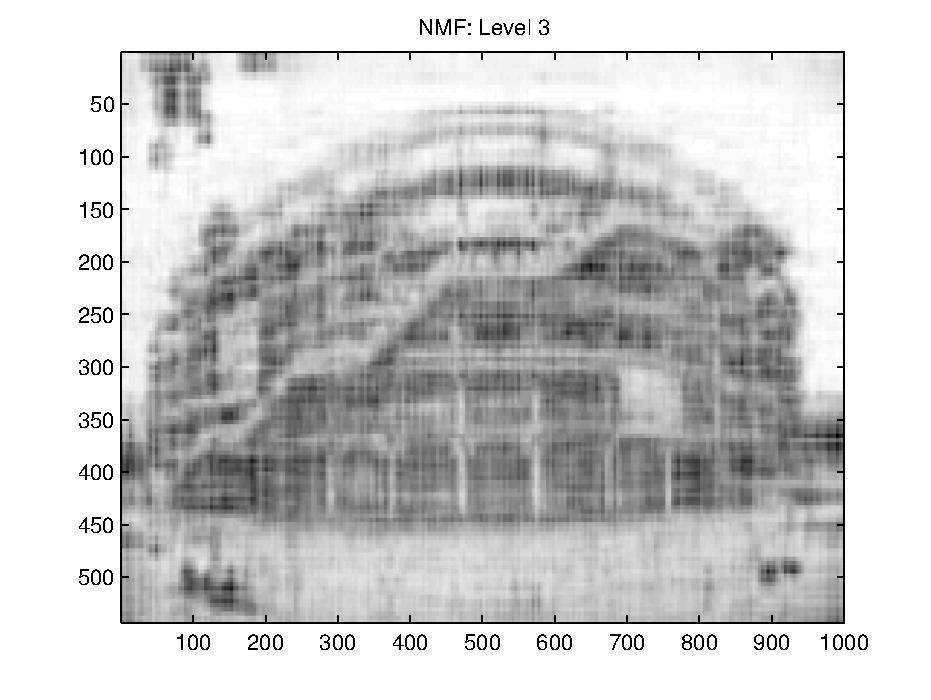}
\hskip -0.5cm
\includegraphics[clip,width=0.35\textwidth]{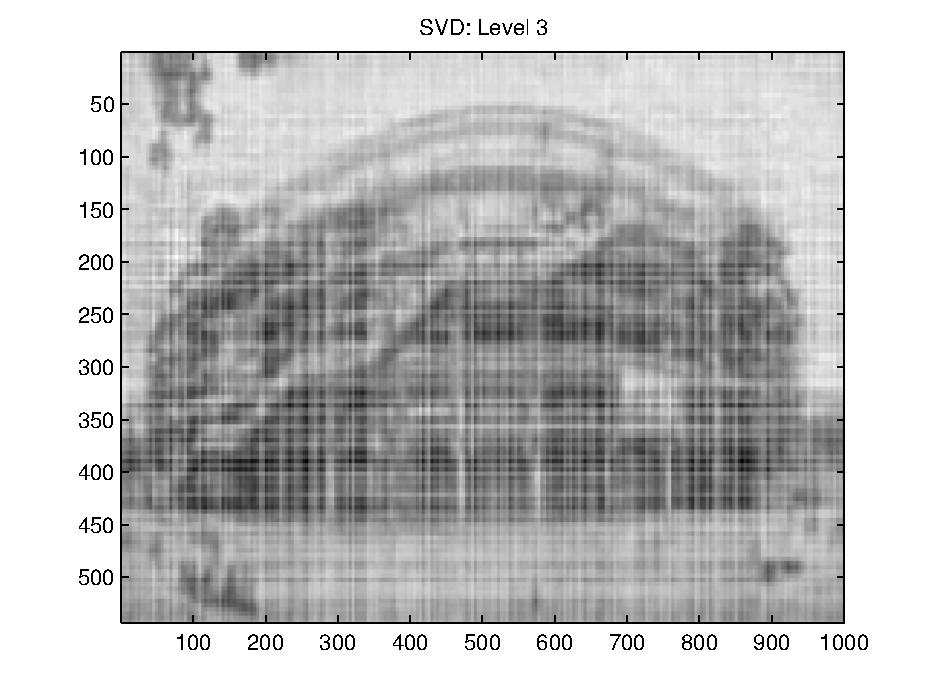}
\hskip -0.5cm
\includegraphics[clip,width=0.35\textwidth]{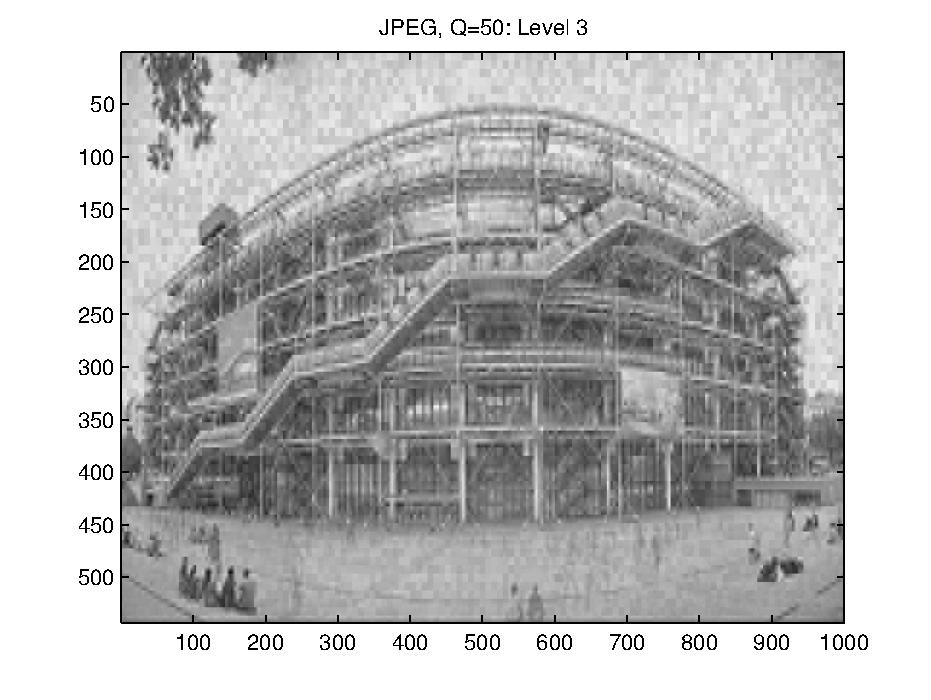} \\
\hskip -2cm
\includegraphics[clip,width=0.35\textwidth]{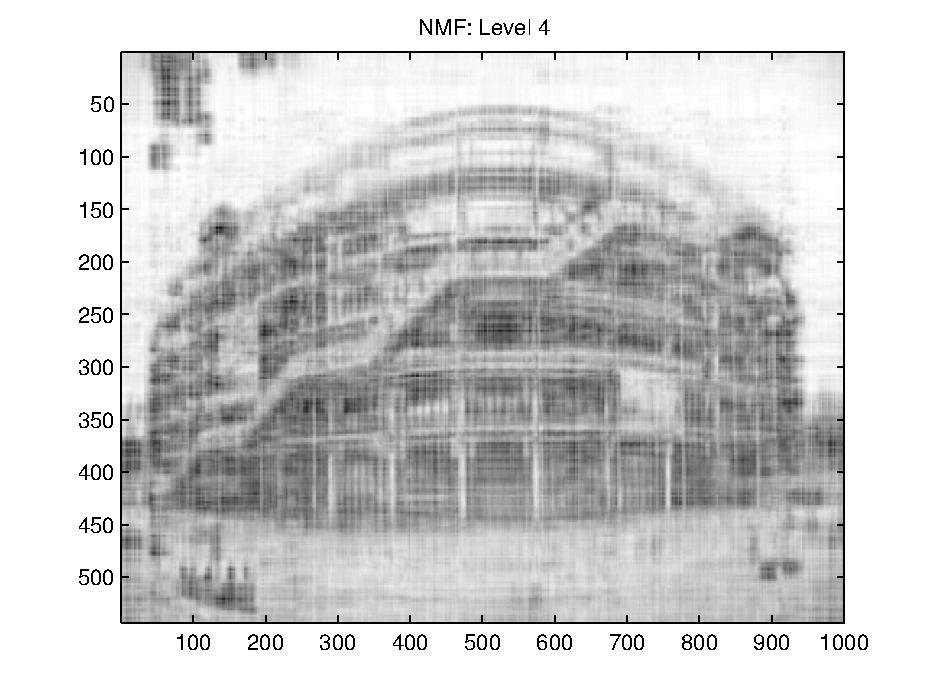}
\hskip -0.5cm
\includegraphics[clip,width=0.35\textwidth]{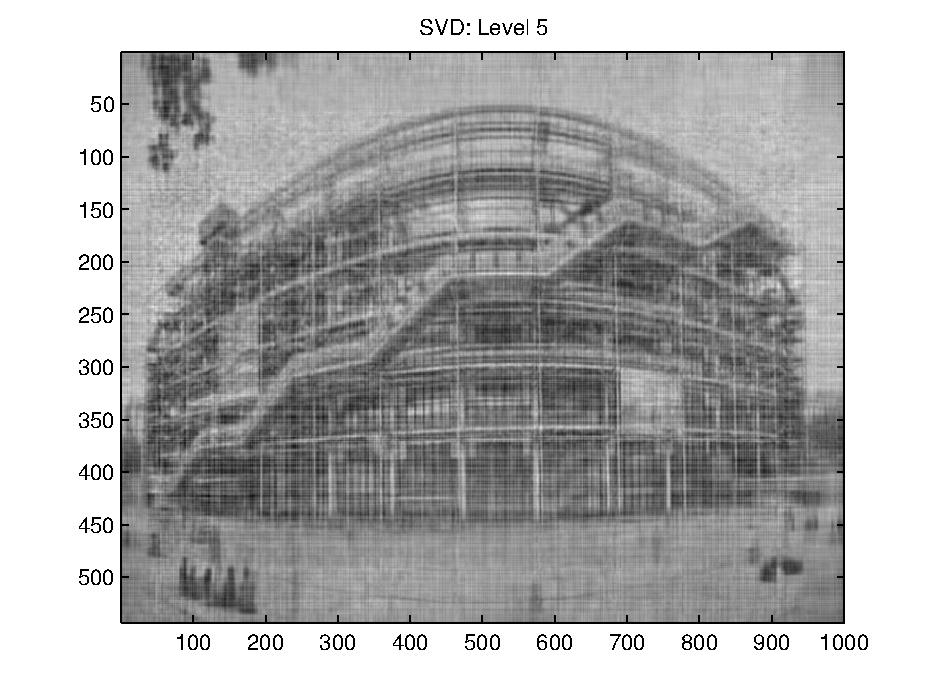}
\hskip -0.5cm
\includegraphics[clip,width=0.35\textwidth]{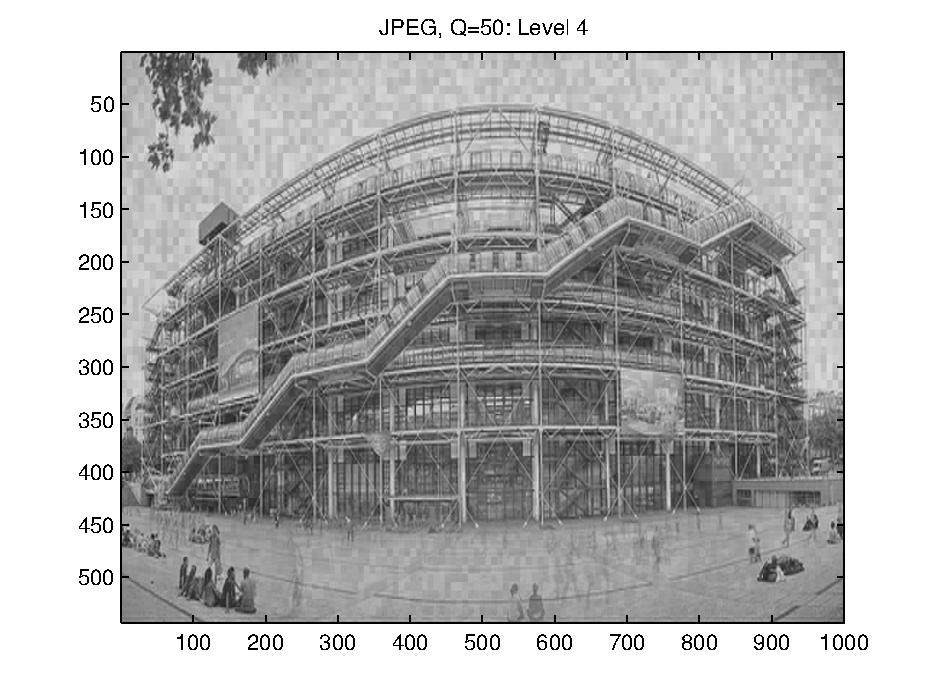} \\
\hskip -2cm
\includegraphics[clip,width=0.35\textwidth]{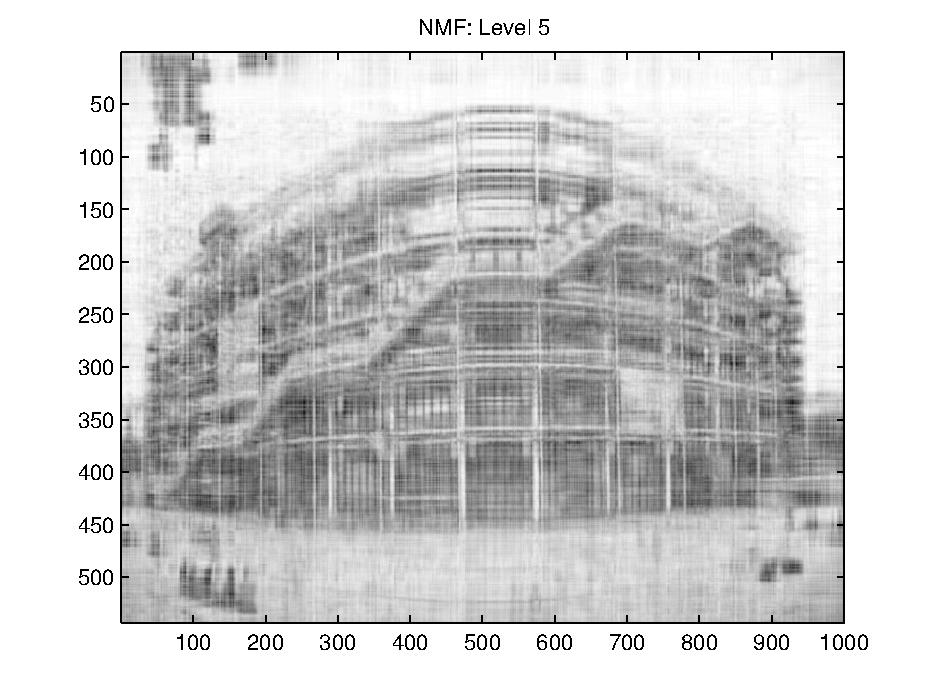}
\hskip -0.5cm
\includegraphics[clip,width=0.35\textwidth]{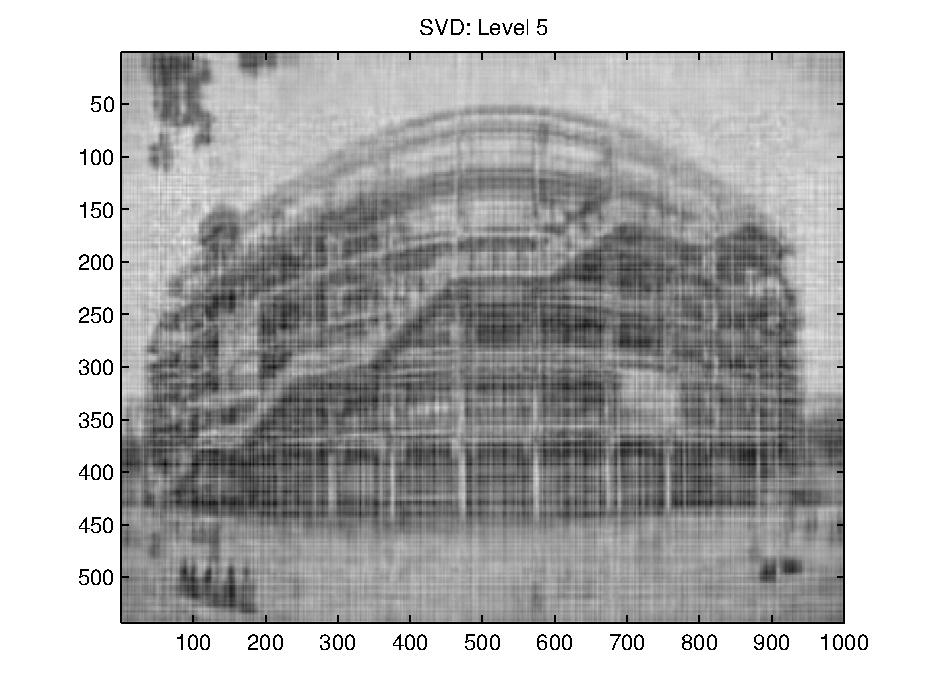}
\hskip -0.5cm
\includegraphics[clip,width=0.35\textwidth]{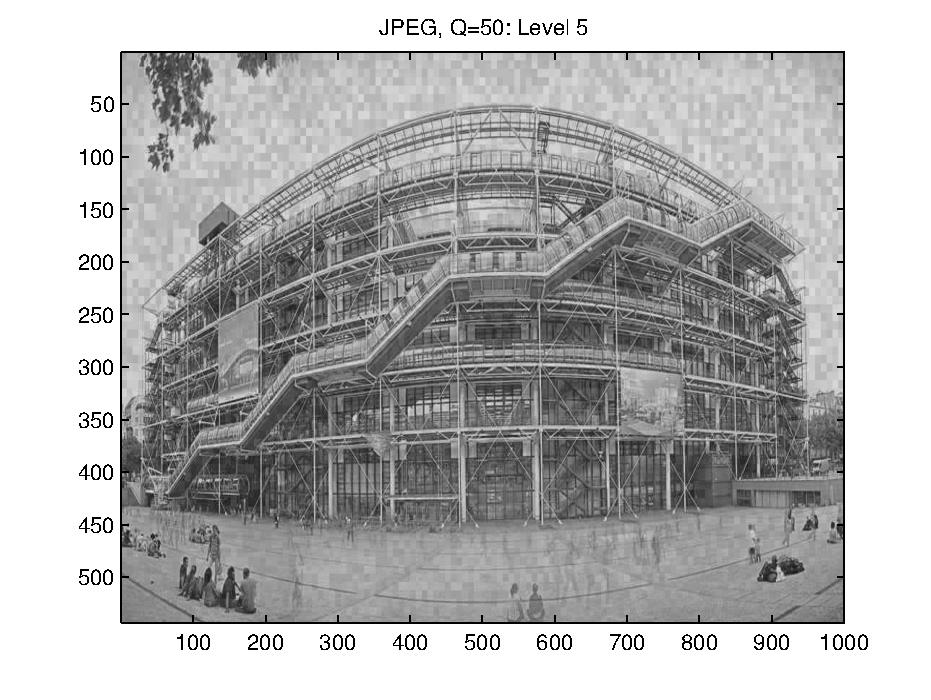}
 \caption{\label{fig:NMF3a2} MLA for the image in Example 3 using NMF with $25\%$ noise }
\end{figurehere}

\textbf{Example 4}.  In this last imaging example, the parameters are the same as in \textbf{Example 3}. $Y$ is set as the image in Figure \ref{fig:NMF4}, and the resulting images are shown in Figure \ref{fig:NMF4a}.  The memory complexity ratios for the $(s_{max}-s)$-th level of the three methods and their respective relative $L^2$ errors with and without noise are shown as follows:
\beqnx
\begin{matrix}
s_{max}-s &:& 1 & 2 & 3 &4 & 5  \\
p &:&   24  &  24 &   28  &  34  &  43 \\
\tilde{p} &:&  173  &  74 &   52&    34 &   73\\
\text{memory complexity ratio of NMF} &:& 0.0033  &  0.0059   & 0.0116  &  0.0283 &   0.0600 \\
\text{memory complexity ratio of SVD}&:&   0.0038  &  0.0076  &  0.0177 &   0.0430 &   0.1089  \\
\text{memory complexity ratio of JPEG} &:&  \text{NA}  &   0.0156   & 0.0297 &   0.0609 &   0.0753 \\
\text{Relative $L^2$ error in NMF (with $0\%$ noise)} &:&  0.4766   & 0.4283   & 0.3673 &   0.3109 &   0.2808 \\
\text{Relative $L^2$ error in SVD (with $0\%$ noise)} &:&    0.4763  &  0.4239  &  0.3638  &  0.3133 &   0.2813  \\
\text{Relative $L^2$ error in JPEG (with $0\%$ noise)} &:&  \text{NA}  &   0.3994  &  0.2753  &  0.1472   & 0.1079  \\
\text{Relative $L^2$ error in NMF (with $25\%$ noise)} &:&   0.4813 &   0.4344  &  0.3769   & 0.3311 &   0.3011  \\
\text{Relative $L^2$ error in SVD (with $25\%$ noise)} &:&   0.4832  &  0.4362 &   0.3791 &   0.3311 &   0.3038 \\
\text{Relative $L^2$ error in JPEG (with $25\%$ noise)} &:&   \text{NA}  &  0.4221  &  0.3172 &   0.2203  &  0.1967
\end{matrix}
\eqnx

From this table we can see that, on the same layer, SVD always needs about a double of the memory than the NMF to just
have a similar performance.
Again, from Figure \ref{fig:NMF4a}, we infer that JPEG outperforms the other two methods at the same layer in the absence of noise.
Nevertheless, if we choose a same memory complexity ratio e.g., $1.5$ percent, we can actually get a $3$rd layer of
the NMF but only a $2$nd layer of JPEG, and the relative error of the smaller-sized $3$rd layer of
NMF is actually smaller than the larger-sized $2$nd layer of JPEG.
Moreover, as we can see from Figures \ref{fig:NMF4a} and \ref{fig:NMF4a2}, when the layers increase and finer
details reveal, a level $4$ of NMF is enough to read the Chinese characters which requires less than $0.03$ percent of memory complexity.
With the presence of noise, the relative error of the $4$th layer of NMF where the Chinese characters are recognizable becomes comparable with the $3$rd layer of JPEG, while their memory complexity is the same. Many of the NMF figures have less errors than the SVD figures on the same layers while the memory complexities of SVD are actually larger.
Again, in Figure \ref{fig:NMF4a2}, the SVD and the JPEG images are obviously contaminated respectively by straight strips and random squares, whereas the noise contamination in the NMF layers seem less obvious.

\begin{center}
\begin{figurehere}
\hfill{}\includegraphics[clip,width=0.35\textwidth]{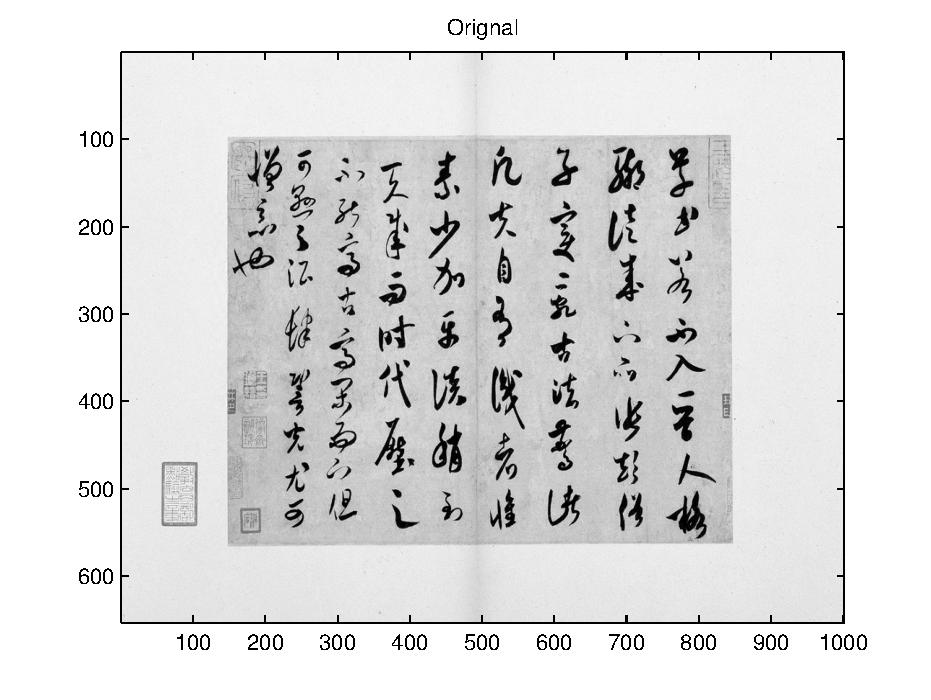}\hfill{}
 \caption{\label{fig:NMF4} Original image in Example 4}
\end{figurehere}
\end{center}

\begin{figurehere}
\hfill{}\\
\hskip -5cm
\includegraphics[clip,width=0.35\textwidth]{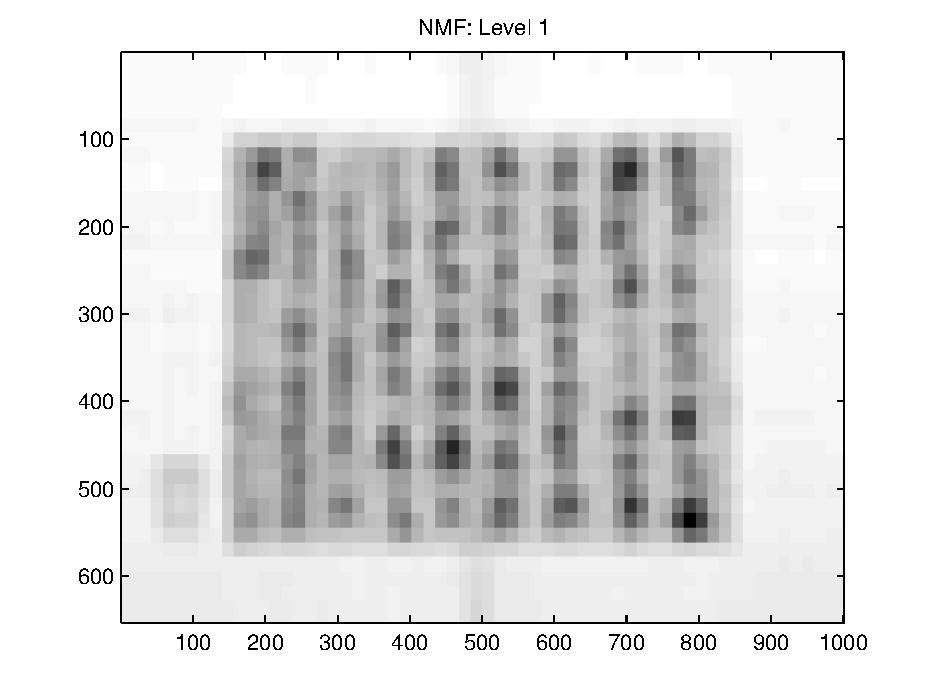}
\hskip -0.5cm
\includegraphics[clip,width=0.35\textwidth]{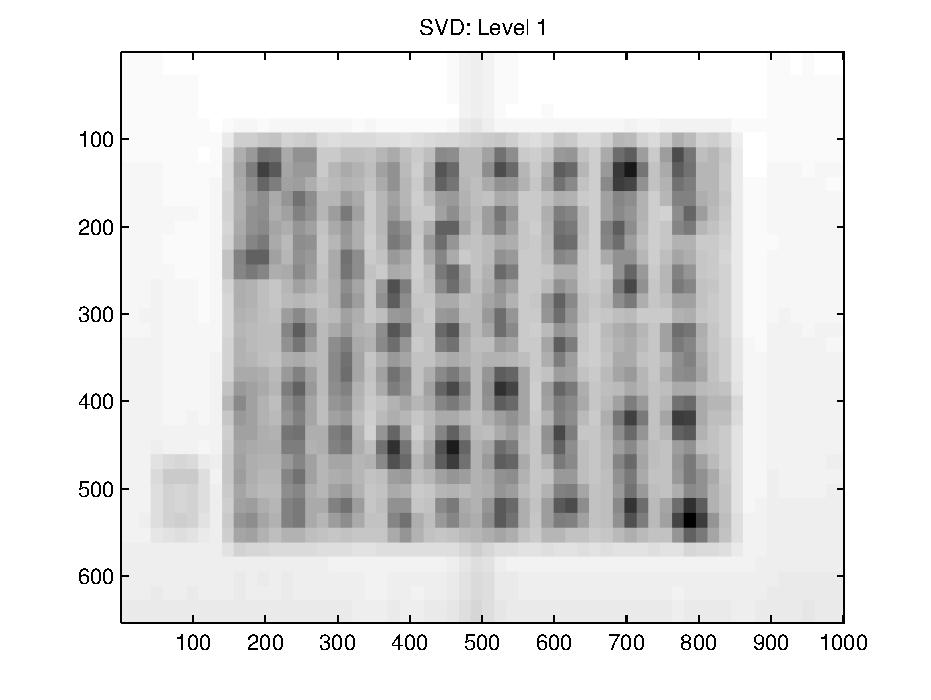}
\hskip -0.5cm
\includegraphics[clip,width=0.35\textwidth]{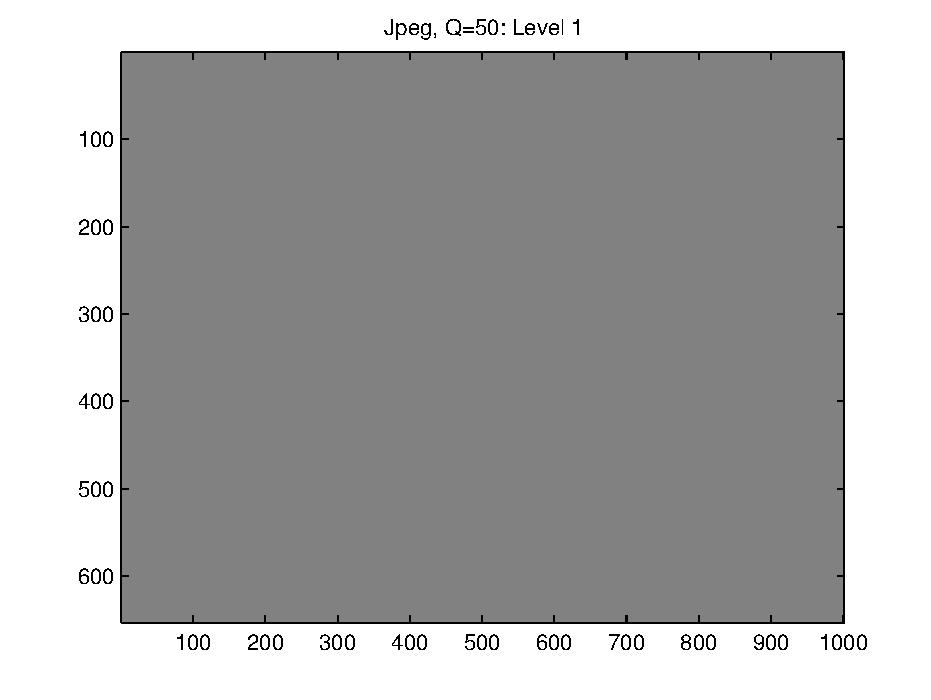}\\
\hskip -2cm
\includegraphics[clip,width=0.35\textwidth]{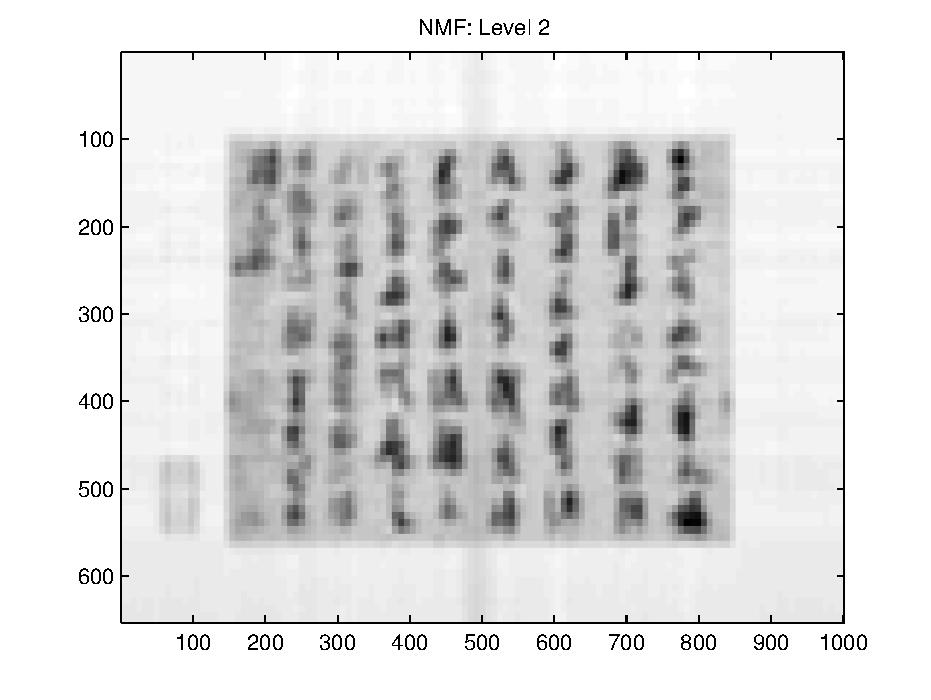}
\hskip -0.5cm
\includegraphics[clip,width=0.35\textwidth]{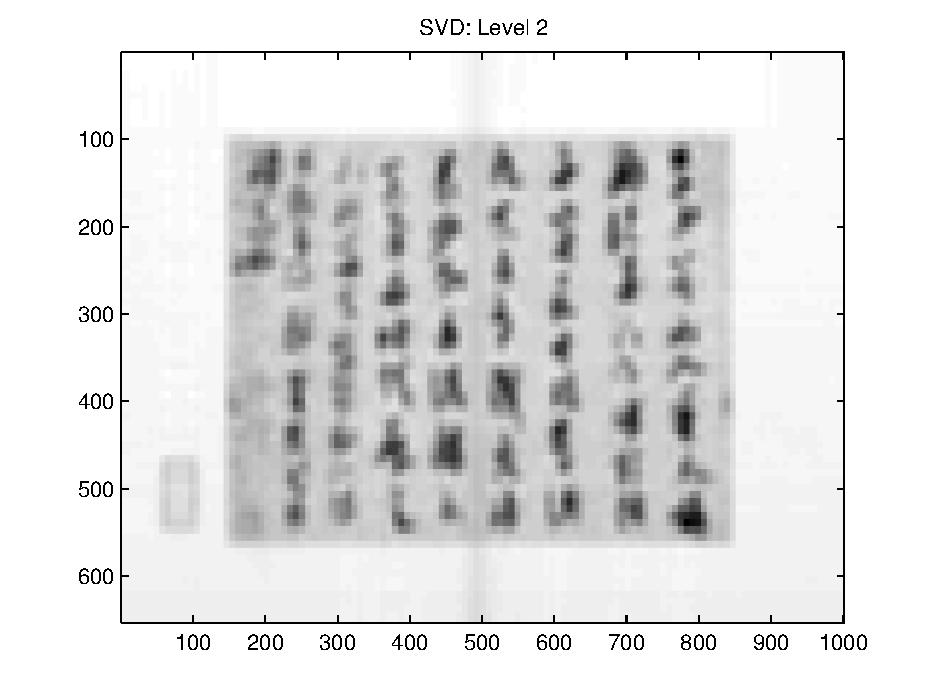}
\hskip -0.5cm
\includegraphics[clip,width=0.35\textwidth]{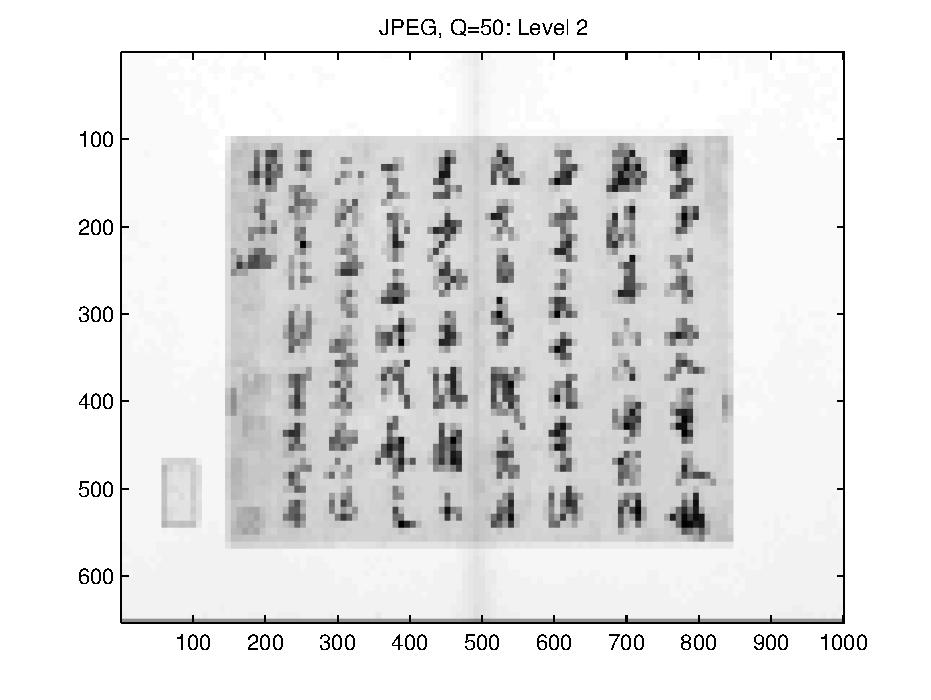} \\
\hskip -2cm
\includegraphics[clip,width=0.35\textwidth]{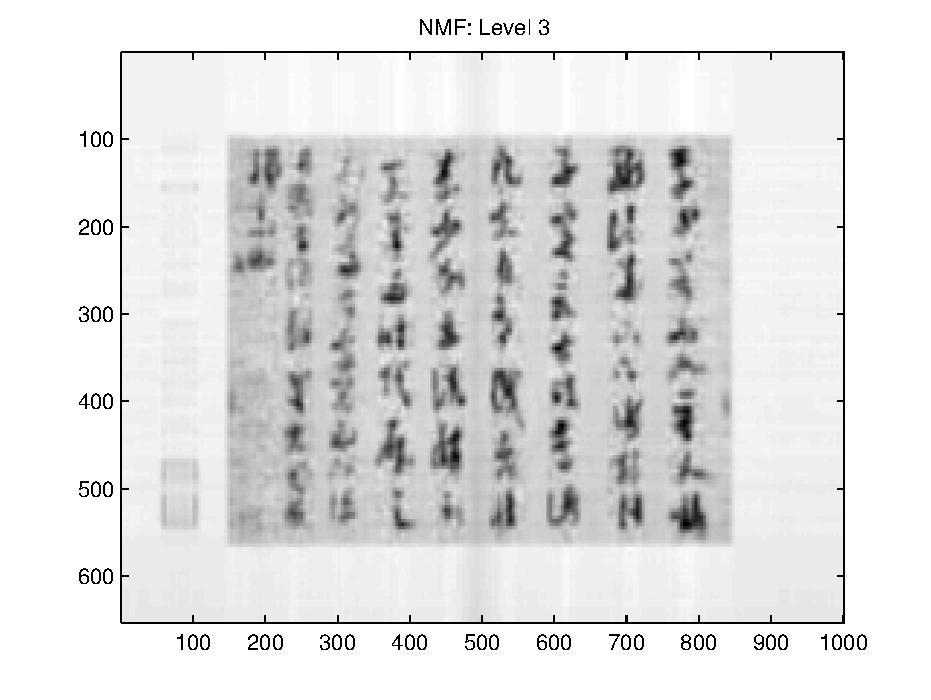}
\hskip -0.5cm
\includegraphics[clip,width=0.35\textwidth]{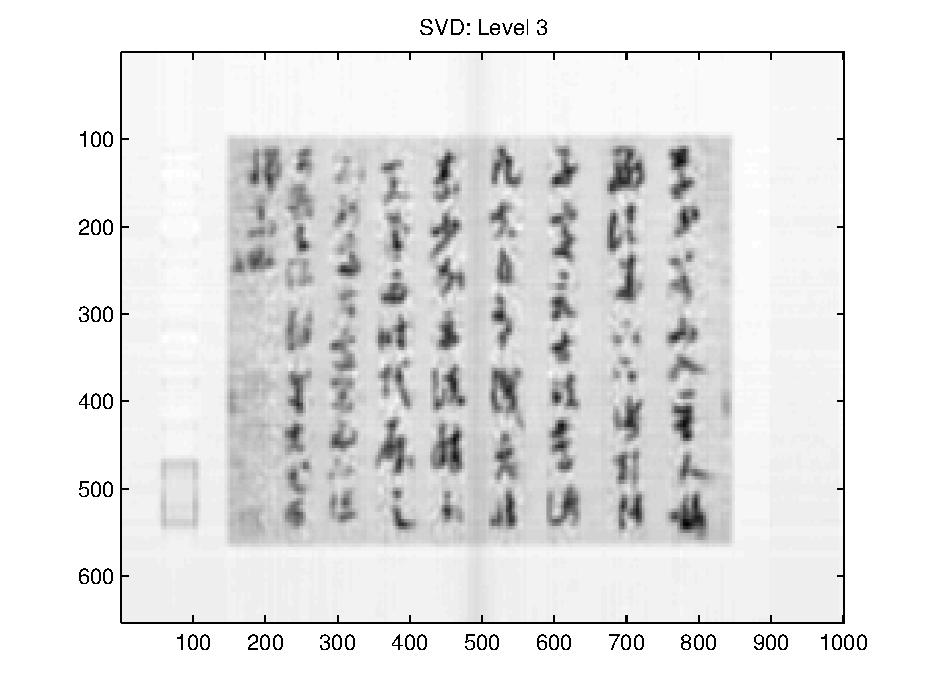}
\hskip -0.5cm
\includegraphics[clip,width=0.35\textwidth]{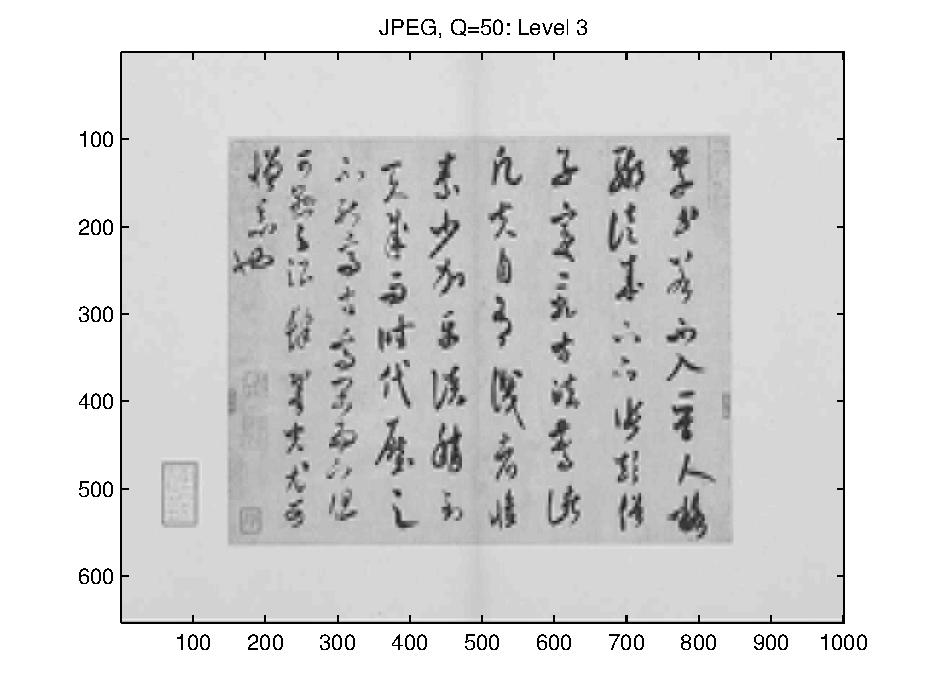} \\
\hskip -2cm
\includegraphics[clip,width=0.35\textwidth]{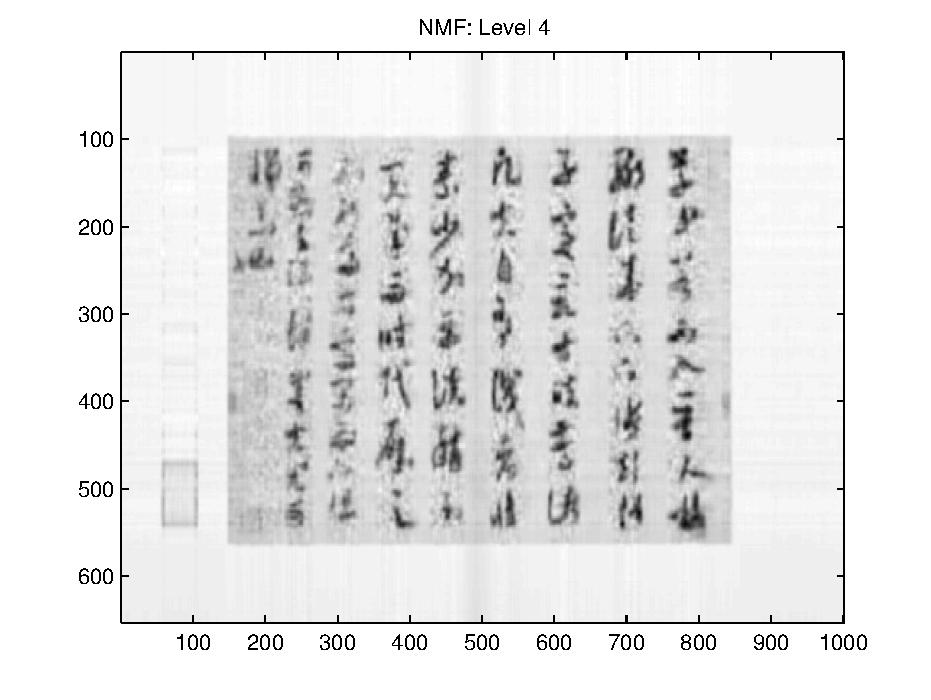}
\hskip -0.5cm
\includegraphics[clip,width=0.35\textwidth]{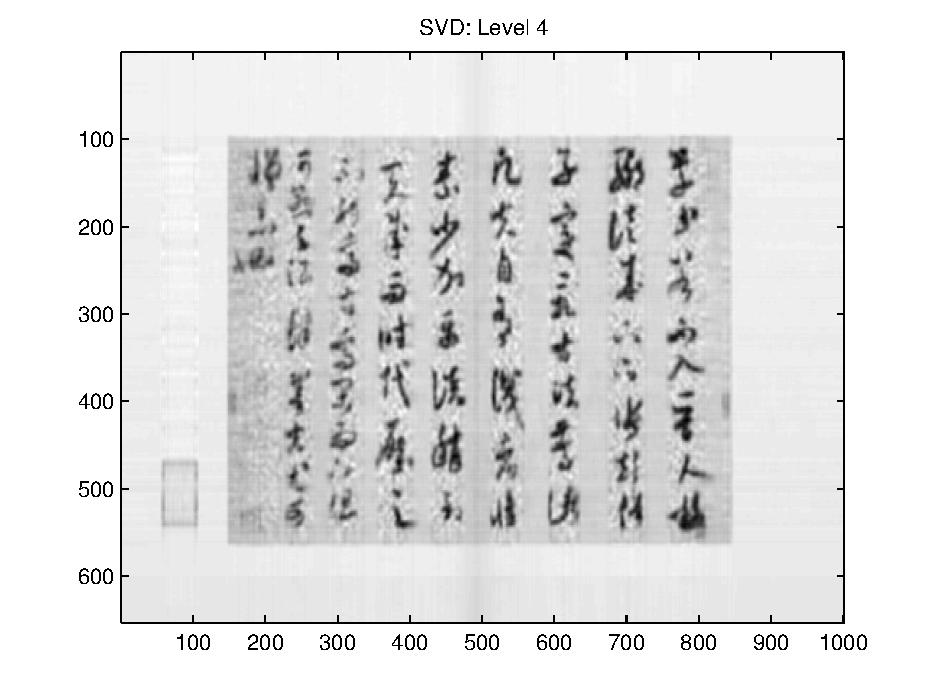}
\hskip -0.5cm
\includegraphics[clip,width=0.35\textwidth]{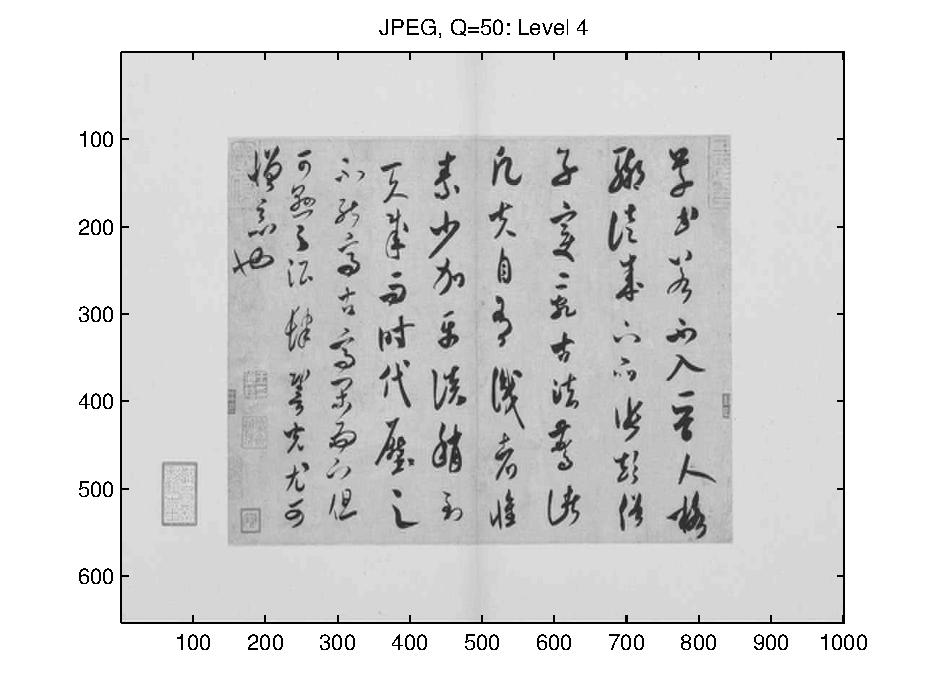} \\
\hskip -2cm
\includegraphics[clip,width=0.35\textwidth]{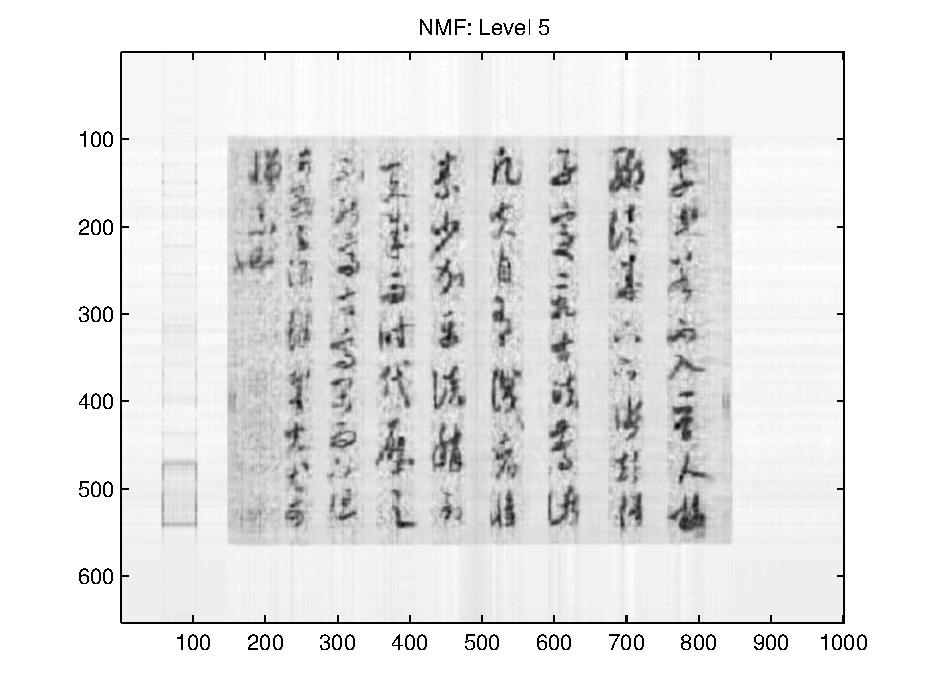}
\hskip -0.5cm
\includegraphics[clip,width=0.35\textwidth]{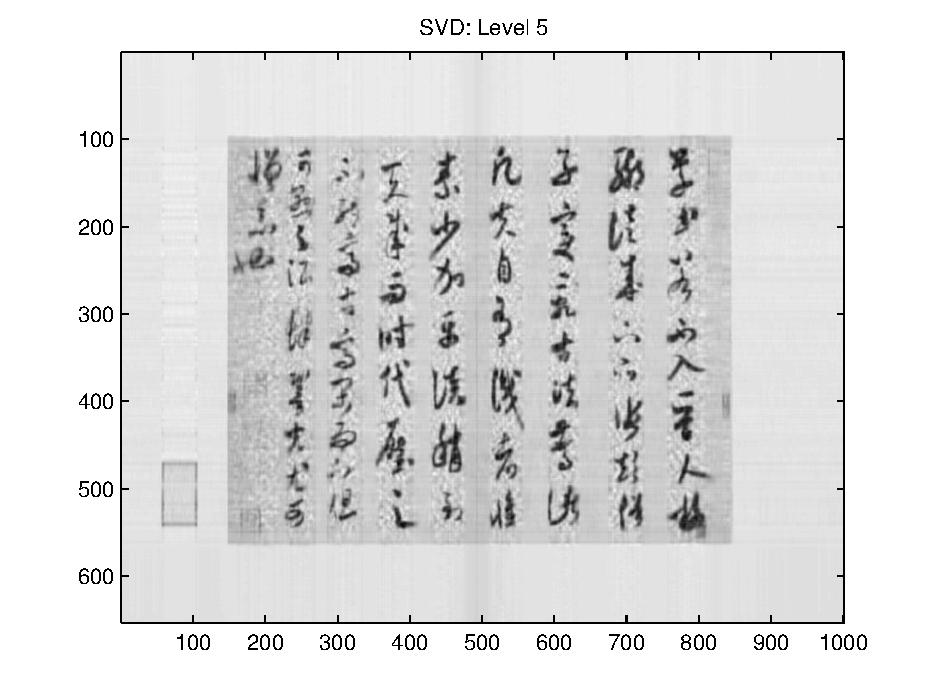}
\hskip -0.5cm
\includegraphics[clip,width=0.35\textwidth]{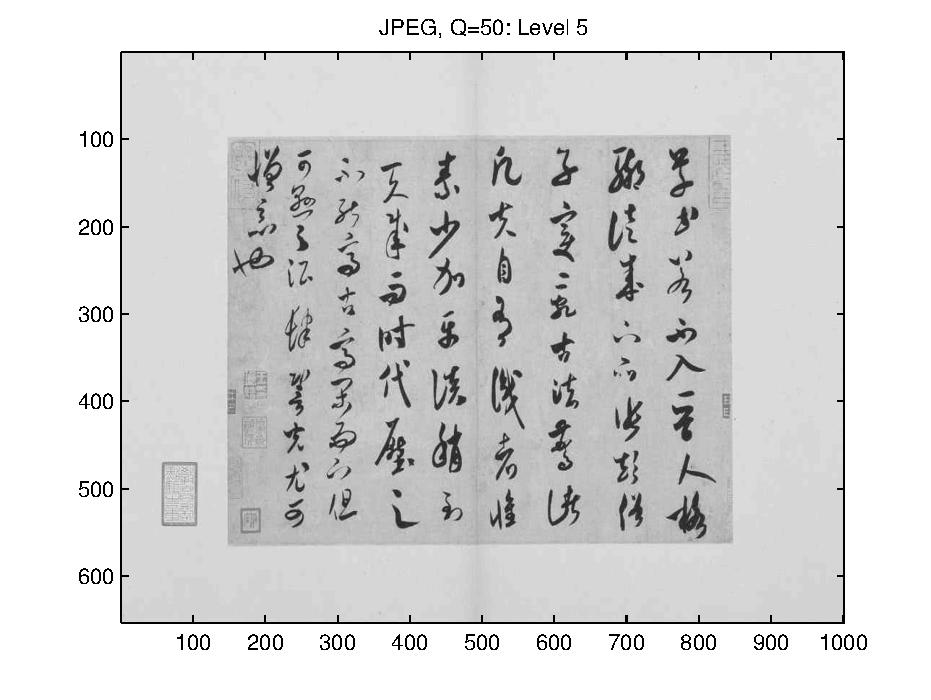}
 \caption{\label{fig:NMF4a} MLA for the image in Example 4 using NMF without noise }
\end{figurehere}

\begin{figurehere}
\hfill{}\\
\hskip -5cm
\includegraphics[clip,width=0.35\textwidth]{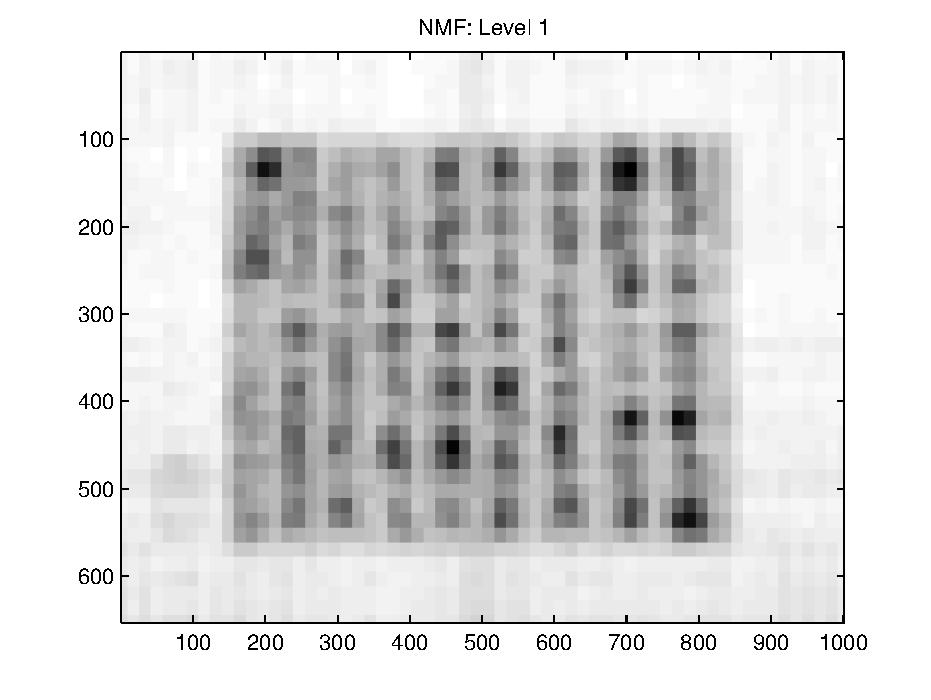}
\hskip -0.5cm
\includegraphics[clip,width=0.35\textwidth]{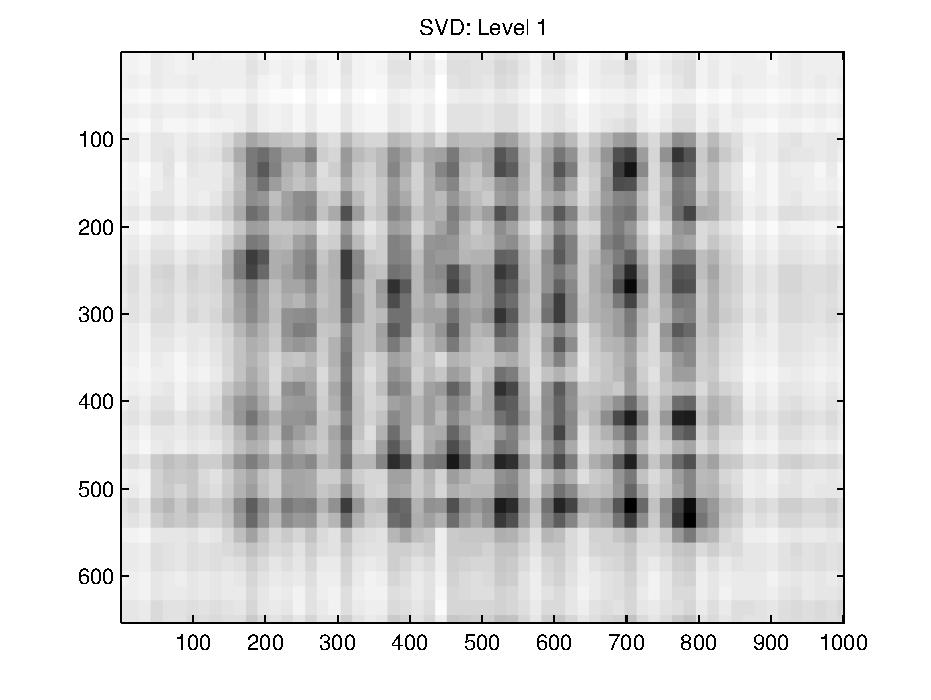}
\hskip -0.5cm
\includegraphics[clip,width=0.35\textwidth]{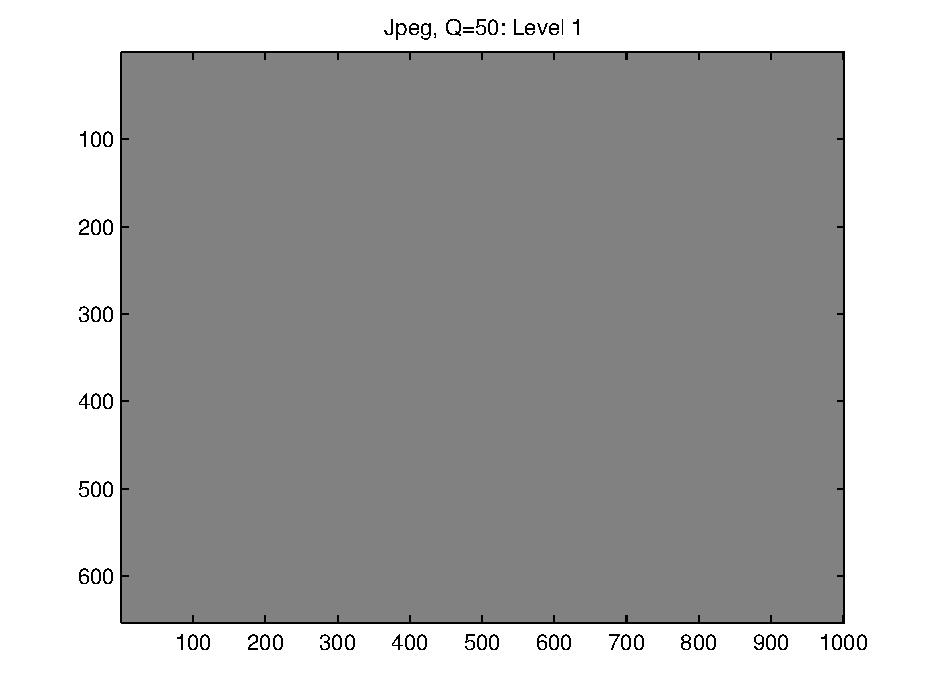}\\
\hskip -2cm
\includegraphics[clip,width=0.35\textwidth]{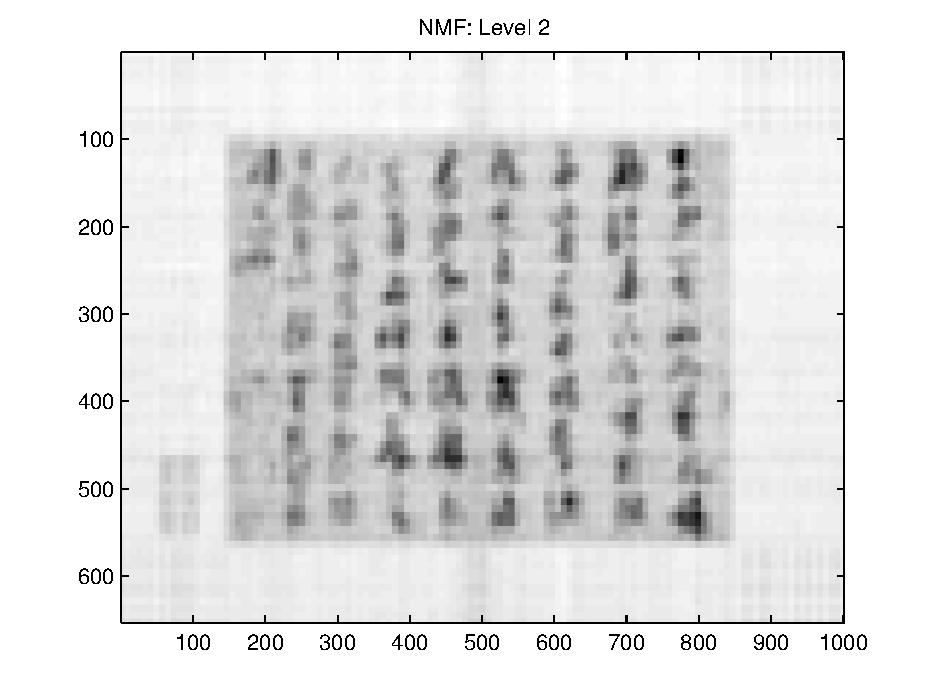}
\hskip -0.5cm
\includegraphics[clip,width=0.35\textwidth]{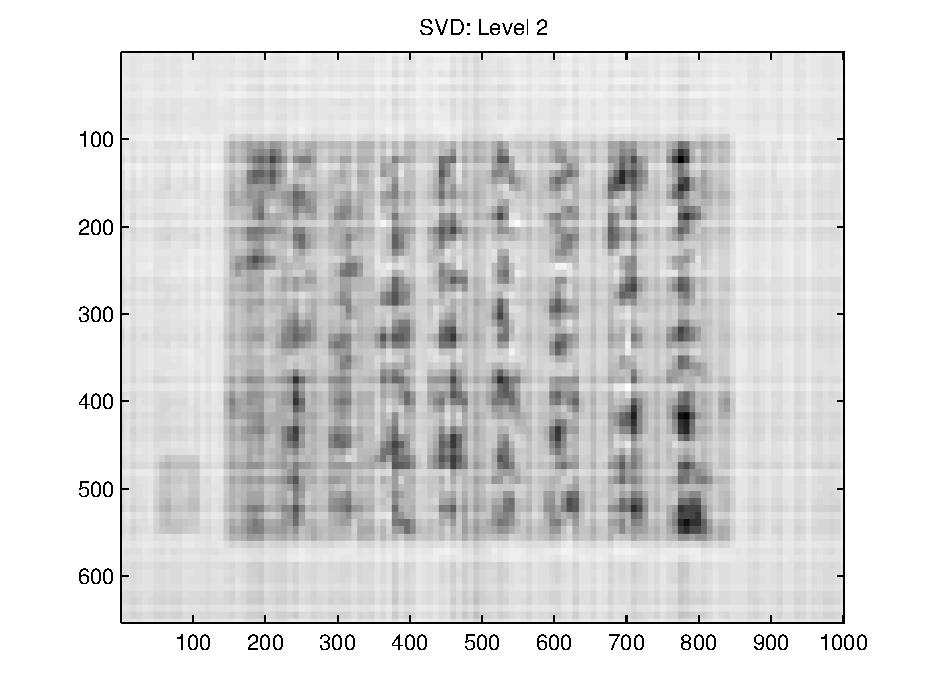}
\hskip -0.5cm
\includegraphics[clip,width=0.35\textwidth]{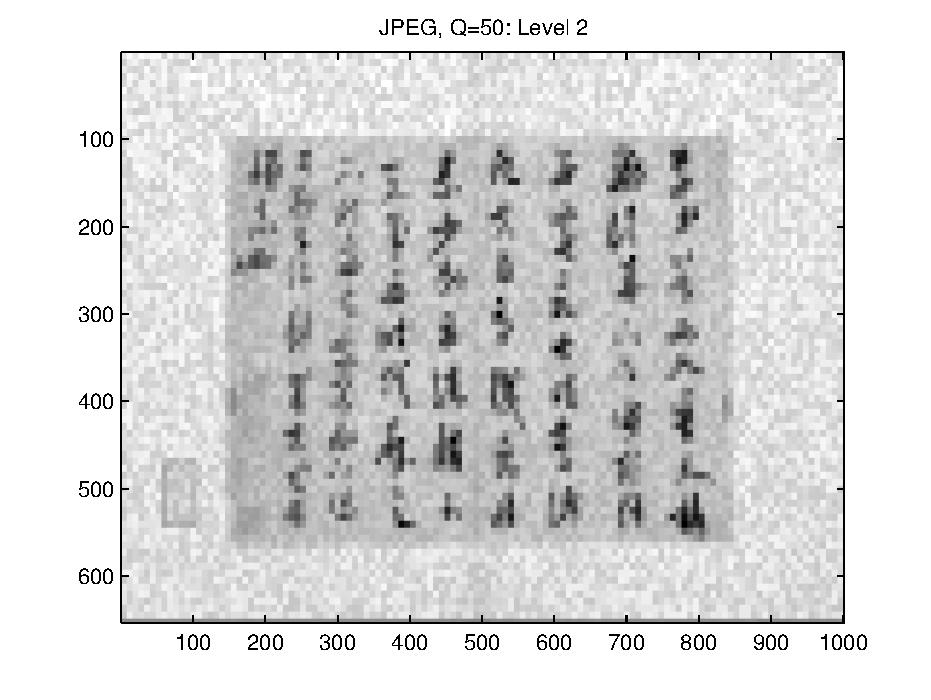} \\
\hskip -2cm
\includegraphics[clip,width=0.35\textwidth]{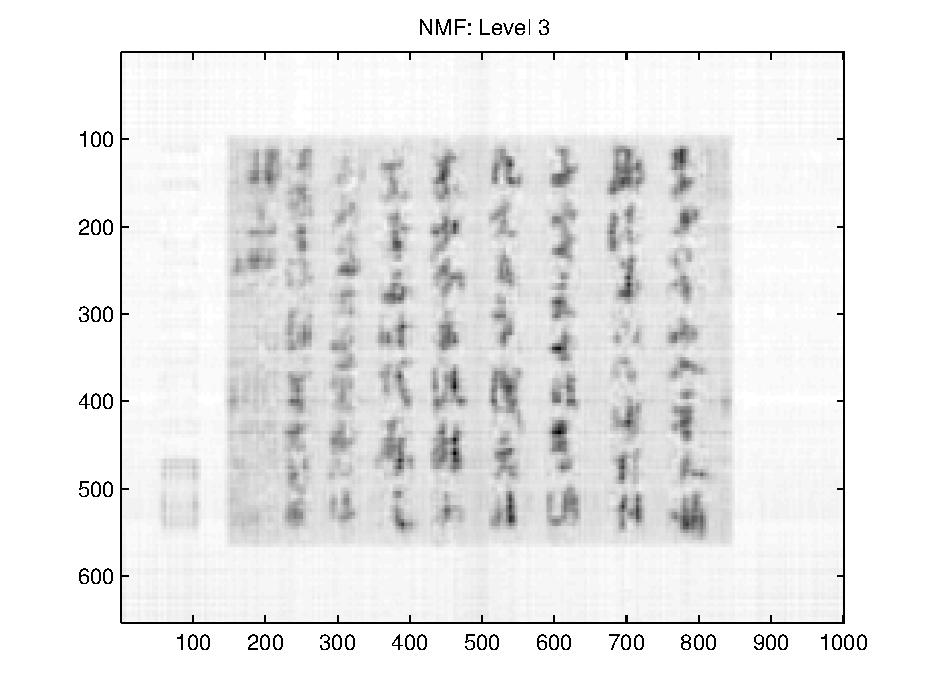}
\hskip -0.5cm
\includegraphics[clip,width=0.35\textwidth]{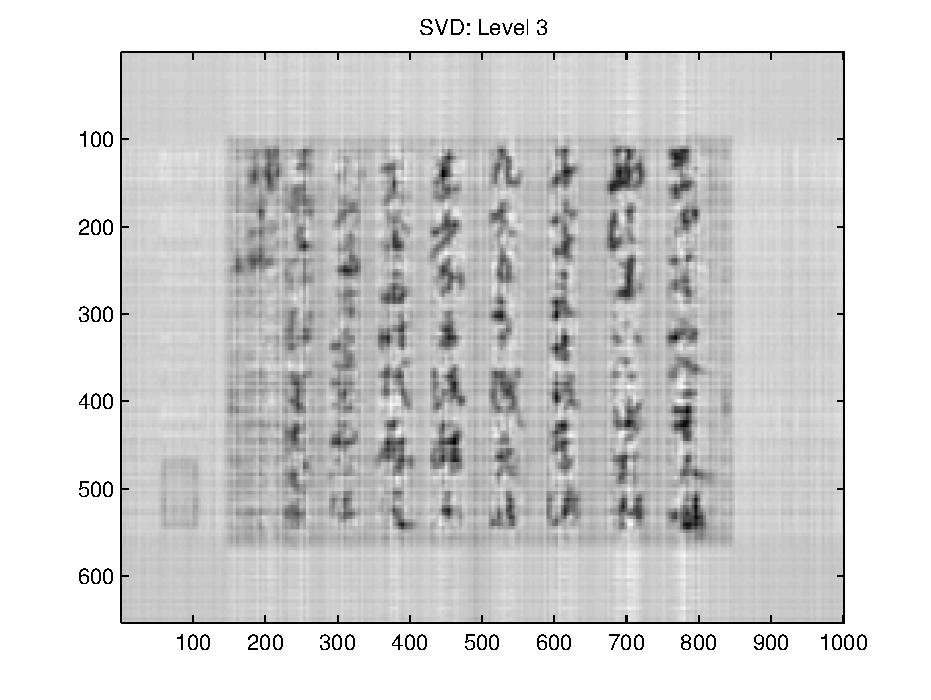}
\hskip -0.5cm
\includegraphics[clip,width=0.35\textwidth]{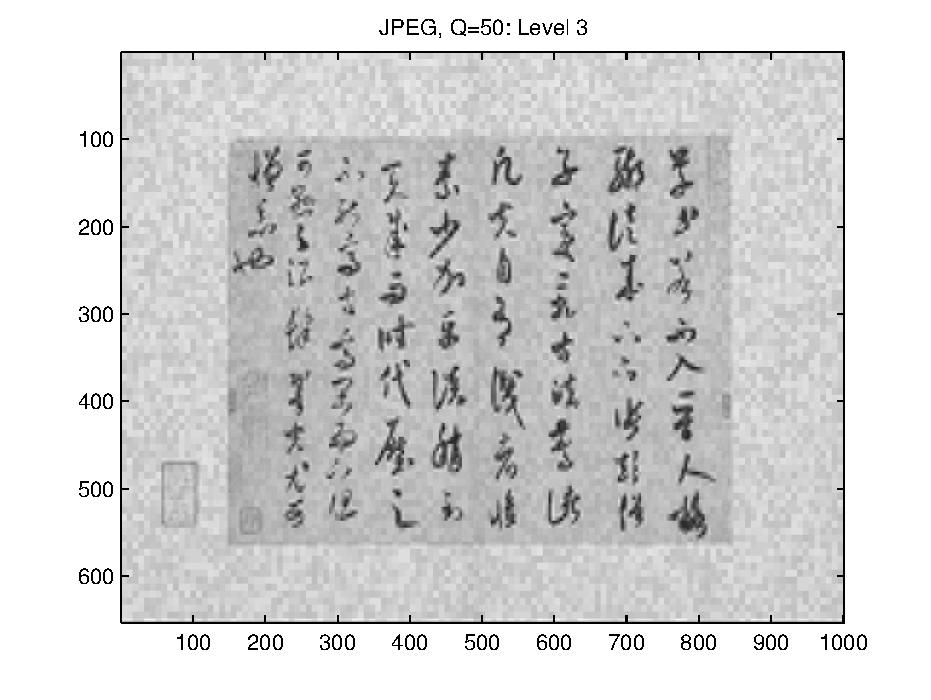} \\
\hskip -2cm
\includegraphics[clip,width=0.35\textwidth]{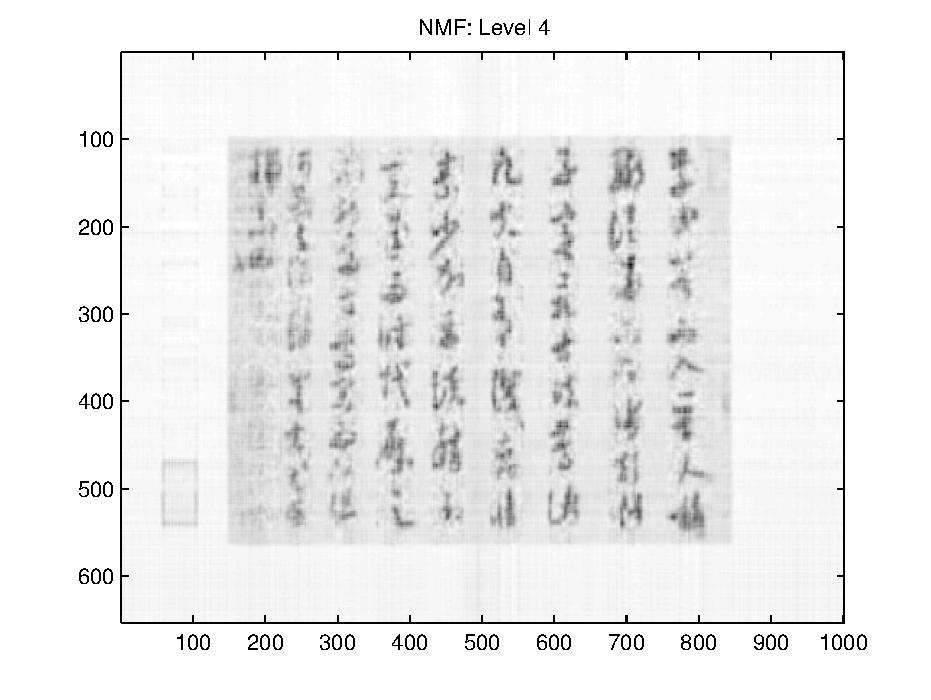}
\hskip -0.5cm
\includegraphics[clip,width=0.35\textwidth]{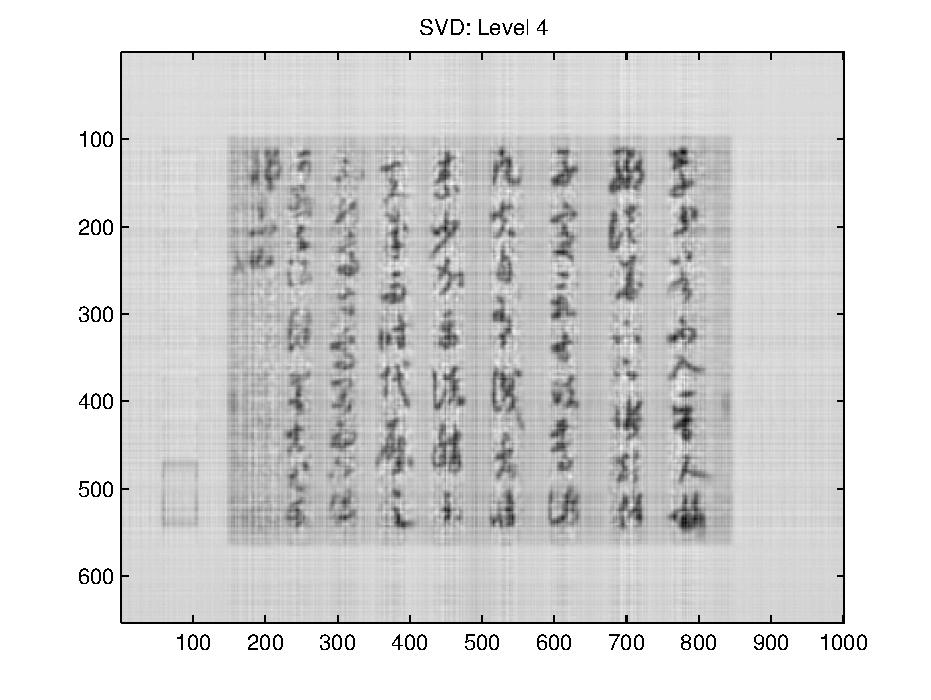}
\hskip -0.5cm
\includegraphics[clip,width=0.35\textwidth]{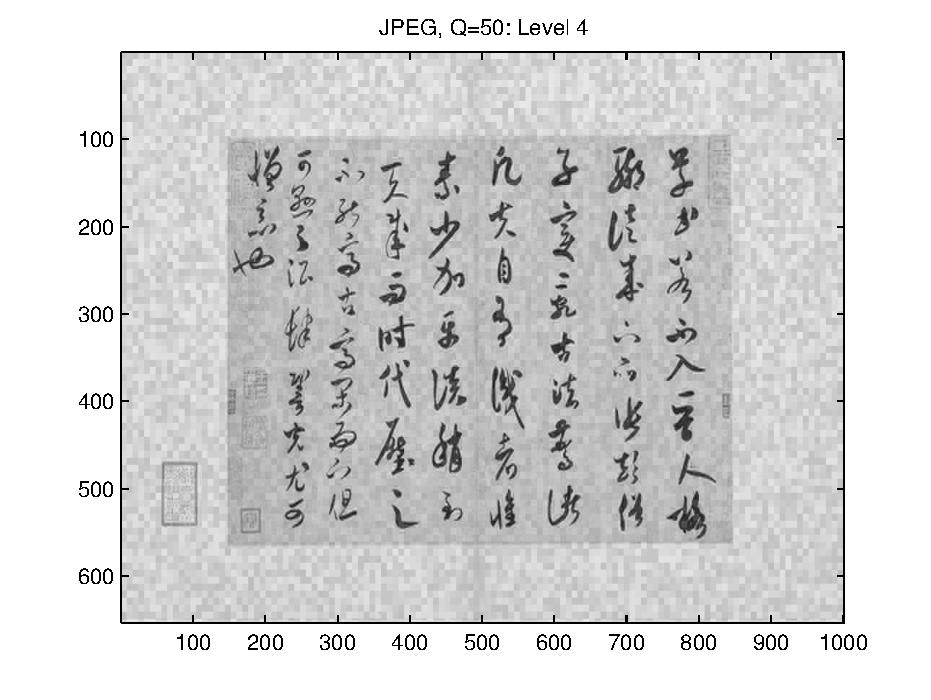} \\
\hskip -2cm
\includegraphics[clip,width=0.35\textwidth]{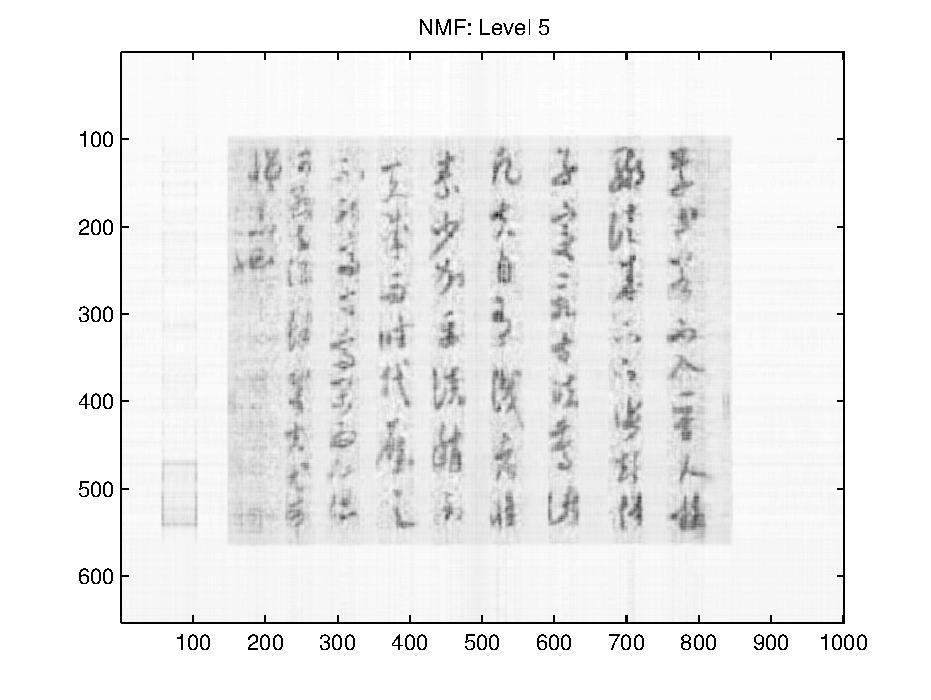}
\hskip -0.5cm
\includegraphics[clip,width=0.35\textwidth]{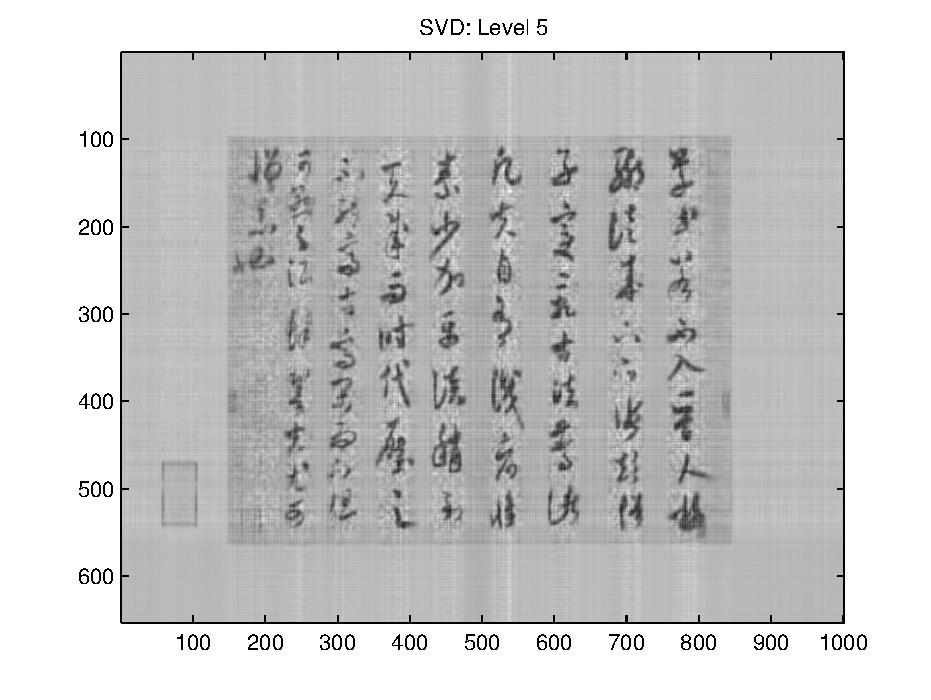}
\hskip -0.5cm
\includegraphics[clip,width=0.35\textwidth]{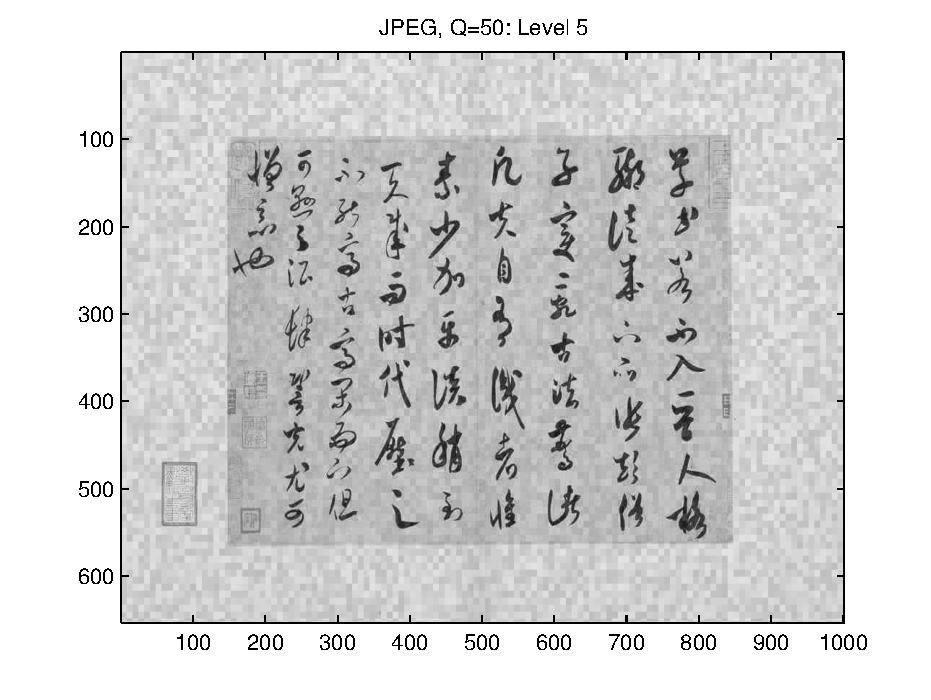}
 \caption{\label{fig:NMF4a2} MLA for the image in Example 4 using NMF with $25\%$ noise }
\end{figurehere}

\subsection{Images reconstructed by DSMs}
In this subsection, we shall present the application of NMF to the images reconstructed by some recently developed
inversion algorithms, namely the direct sampling methods (DSMs). The DSMs are a family of simple and efficient inversion methods which aim at providing a good estimate of the locations of inhomogeneities inside a homogeneous background representing various physical media from a single or a small number of boundary data in both full and limited aperture cases. They were studied in \cite{lizou13}  \cite{pott10} using far-field data and in \cite{IJZ} using near-field data for locating inhomogeneities in inverse acoustic medium scattering, and was later extended to various other coefficient determination inverse problems, such as the electrical impendence tomography (EIT) \cite{EIT}, the diffusive optical tomography (DOT) \cite{DOT} and the electromagnetic inverse scattering problem \cite{Ito13}. In each of the aforementioned tomographies, a family of probing functions is introduced and an indicator function is defined as a duality product between the observed data and the probing function under an appropriate choice of Sobolev scale.
The index function, which we shall denote as a general image $Y$, represents the likelihood of whether a given sampling point sits inside an inhomogeneous inclusion.
The evaluation of the index function is very inexpensive and works with quite limited measurement data,
and the images obtained from the index functions are proven to be effective in locating
abnormalities, especially robust against noise in the data.

However, from our numerical experiments in the aforementioned references, we notice that, in exchange for the robustness of the DSM method and the cost-effectiveness of its evaluation, the DSM images usually contain some minor artifacts.
These artefacts mainly come from the fact that the DSM image is actually the result of applying a kernel
on a function with its support sitting inside the inclusions that we aim to locate.
%
%
The DSM images we obtain are therefore usually quite diffusive and may consist of shadows and tails coming from the non-diagonal part of the kernel.

Henceforth, a DSM image $Y$ shall consist of $3$ parts: the first part coming from the signals of the inhomogeneous inclusions, the second as the contamination of the image by the non-diagonal part of the kernel, and the third part coming from noise in the measurement data. In view of the fact that both the DSM image that we obtain and a likelihood function are both positive, we shall therefore apply the NMF to the DSM images
in the hope of identifying the principal components of the image corresponding to the signal from the inhomogeneous inclusions.
As a remark, we would like to emphasize again that we are not aiming to reconstruct the original DSM image from all the components (in terms of tensor products) that we obtain from NMF, but only to look for principal components of the image containing signals from inhomogeneous inclusions and aim at reconstructing the inclusions themselves.

In what follows, we shall apply the NMF to DSM images from two tomographies, namely the DOT and EIT.
DOT is a popular non-invasive imaging technique that measures the optical properties of a medium and creates images which show the distribution of absorption coefficient inside the body. It is very useful for medical imaging, e.g. breast cancer imaging, brain functional imaging, stroke detection, muscle functional studies, photodynamic therapy, and radiation therapy monitoring; see \cite{DOT}.
In our subsequent discussion, we consider the numerical experiments of the DOT using DSM as in Section 6 of \cite{DOT}, and the same numerical setting described therein. The medium coefficient inside all the inhomogeneous inclusions are set as $\mu = 50$.
The images generated from the scattered potential using the DSM algorithm described in that work are then put into \textbf{Algorithm 2} for NMF, with parameters set to $\alpha=0.2, \nu = 0, \gamma =0.02, p=5, \tilde{p} =3$ and $c_1=c_2=1$ in all the following examples.

\textbf{Example 5}. In this example, we consider the case of two circular inclusions of radius $0.065$, which are respectively centered at $(-0.5, 0.25)$ and $(0.25, 0.15)$; see Figure \ref{fig:DOT1} (top).
The squared reconstructed images from the index $\widetilde{I}^2$ described in \cite{DOT} is presented in Figure \ref{fig:DOT1} (second).  The three images $\sigma_{i_l j_l} \, (\tilde{u}_p)_{i_l} \otimes (\tilde{v}_p)_{j_l}$, for $l = 1,2,3$ after NMF obtained in \textbf{Algorithm 2} are shown in Figure \ref{fig:DOT1} (third to fifth).
The generalized eigenvalues are respectively given as $\{\sigma_{i_l j_l} \}_{l=1}^3 = \{ 32.5522, 21.1686, 12.8299 \}$ in this example.
The squared image of the final approximation to $\mathcal{I}_{p,\tilde{p}}^{\alpha, \nu, \gamma} (\widetilde{I})$ after normalization is given in Figure \ref{fig:DOT1} (last). From the figure, we can see that with an appropriate cutoff, e.g. a $50\%$ cutoff, both the sizes and locations of inhomogeneities obtained from the image are reasonable accurate.

\textbf{Example 6}. This example tests a medium with $4$ circular inclusions of radius $0.065$ with their corresponding positions: $(-0.5, 0.3)$, $(-0.3,-0.1)$, $(0.3, 0.1)$ and $(0.5, 0.3)$; see Figure \ref{fig:DOT2} (top). Figure \ref{fig:DOT2} (second) shows the squared reconstructed images from the index $\widetilde{I}^2$ described in \cite{DOT}.  Components $\sigma_{i_l j_l} \, (\tilde{u}_p)_{i_l} \otimes (\tilde{v}_p)_{j_l}$, for $l = 1,2,3$ after NMF are shown in Figure \ref{fig:DOT2} (third to fifth).
The generalized eigenvalues are respectively given as $\{\sigma_{i_l j_l} \}_{l=1}^3 = \{ 22.2455, 16.7153, 8.9511 \}$ in this example.
Figure \ref{fig:DOT2} (last) gives the squared image of the final approximation to $\mathcal{I}_{p,\tilde{p}}^{\alpha, \nu ,\gamma} (\widetilde{I})$ after normalization.
The principal components of the image coming from signals from the inclusions can be well obtained,
with an observation
that the first two components decomposed from NMF actually represent the inhomogeneous inclusions inside the original medium,

\begin{figurehere}
\vskip -0.0cm
\hfill{}\includegraphics[clip,width=0.6\textwidth]{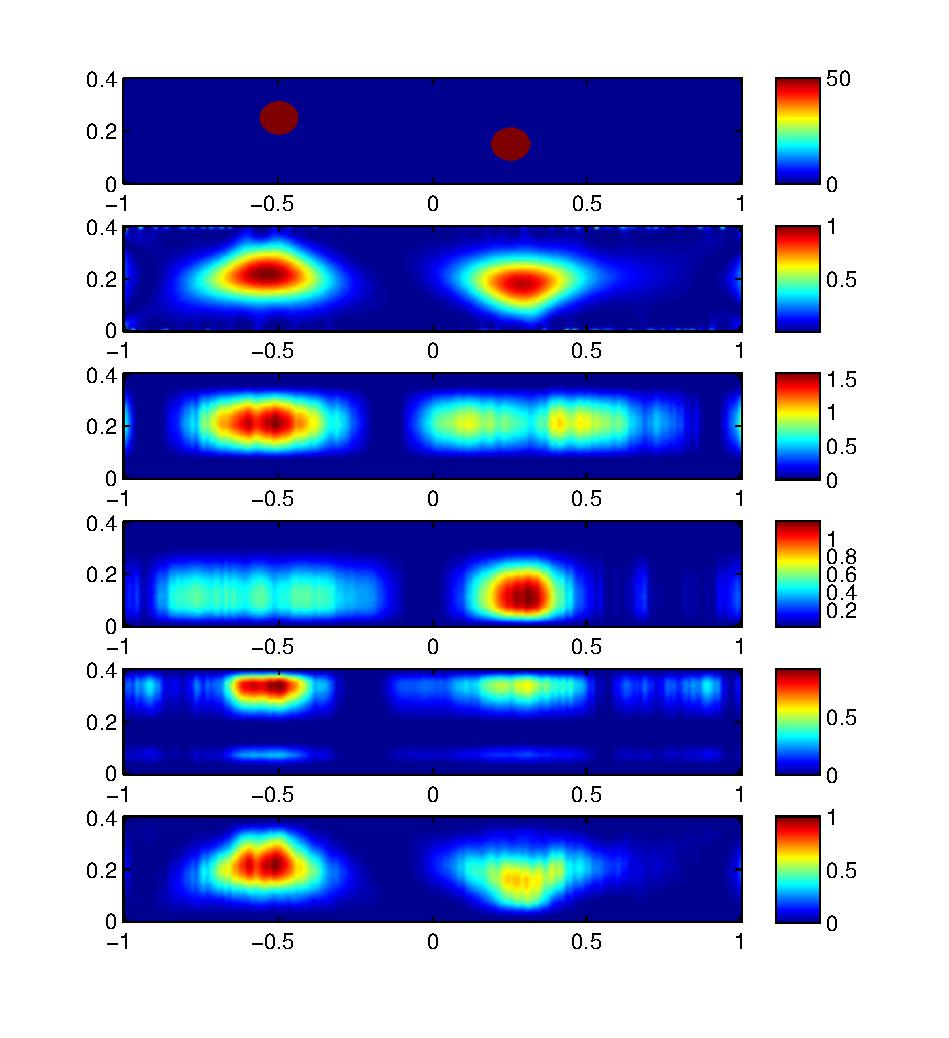}\hfill{}
\vskip -1cm
 \caption{\label{fig:DOT1}\small{NMF decomposition of DSM image from DOT in Example 5, with $\{\sigma_{i_l j_l} \}_{l=1}^3 = \{ 32.5522, 21.1686, 12.8299 \}$}}
\end{figurehere}

\begin{figurehere}
\vskip -0.5cm
\hfill{}\includegraphics[clip,width=0.6\textwidth]{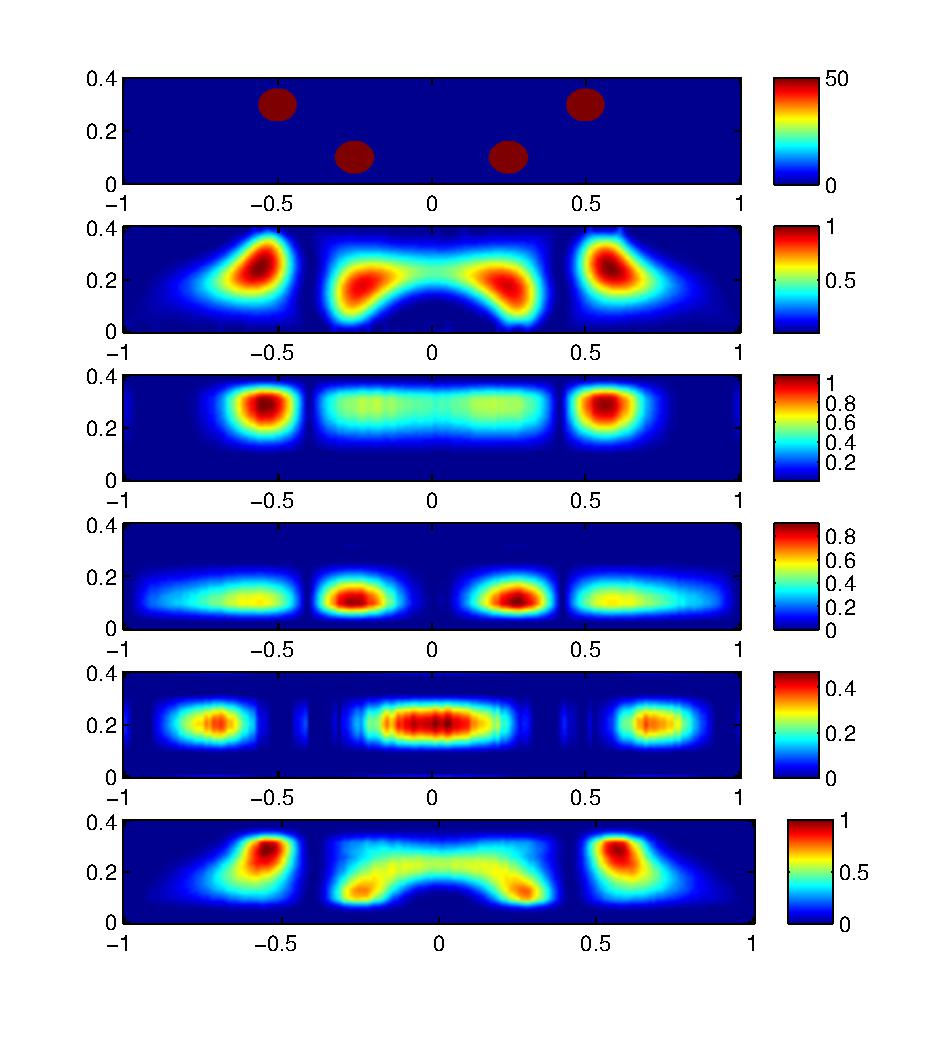}\hfill{}
\vskip -1cm
 \caption{\label{fig:DOT2}\small{NMF decomposition of DSM image from DOT in Example 6, with $\{\sigma_{i_l j_l} \}_{l=1}^3 = \{ 22.2455, 16.7153, 8.9511 \}$}}
\end{figurehere}

\ss
Next, we shall apply the NMF to the DSM images from EIT, which is an effective noninvasive evaluation method that creates images
of the electrical conductivity of an inhomogeneous medium by applying currents at a number of electrodes
on the boundary and measuring the corresponding voltages. It has found applications in many areas, such as oil and geophysical prospection, medical
imaging, physiological measurement, early diagnosis of breast cancer, monitoring of pulmonary functions and detection of leaks from buried pipes, etc; see ref. in \cite{EIT}.
In what follows, we consider the same numerical setting as in the numerical experiments of EIT for a circular domain using DSM described in Section 6 in \cite{EIT}. The physical coefficient of the inhomogeneous inclusions are all set to $\sigma = 5$.
The images generated from the scattered potential field using the DSM algorithm are then put into \textbf{Algorithm 2} for NMF, with parameters set to $\alpha=0.2, \nu = 0 , \gamma =0.02, p=5, \tilde{p} =3$ and $c_1=c_2=1$ in all the following examples.

\textbf{Example 7}. We now investigate an example with $2$ inclusions of size $0.1\times0.1$ respectively at the positions $(-0.44,0.36)$ and $(0.36,-0.44)$; see Figure \ref{fig:EITnew1} (a). The squared reconstructed images from the indices $I^2$ after normalization as described in \cite{EIT} is presented in Figure \ref{fig:EITnew1} (b).  The components $\sigma_{i_l j_l} \, (\tilde{u}_p)_{i_l} \otimes (\tilde{v}_p)_{j_l}$, for $l = 1,2,3$ obtained from NMF using \textbf{Algorithm 2} over the image $\widetilde{I}$ are shown in Figure \ref{fig:EITnew1} (c-e).
The generalized eigenvalues are respectively given as $\{\sigma_{i_l j_l} \}_{l=1}^3 = \{ 2.3712, 2.3548, 2.2904 \}$ in this example.
The squared image of the approximation to $\mathcal{I}_{p,\tilde{p}}^{\alpha, \nu, \gamma}(I)$ after normalization is in Figure \ref{fig:EITnew1} (f). The components of inhomogeneous inclusions sitting inside the original medium are decomposed into different components from the NMF.

\begin{figurehere}
\hfill{}\includegraphics[clip,width=0.3\textwidth]{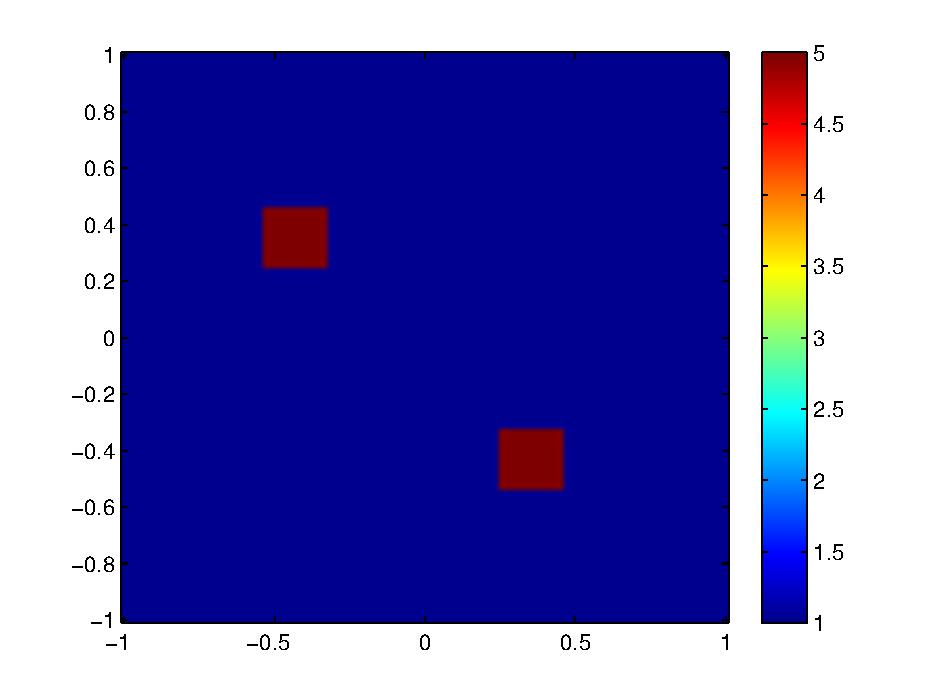}\hfill{}
\hfill{}\includegraphics[clip,width=0.3\textwidth]{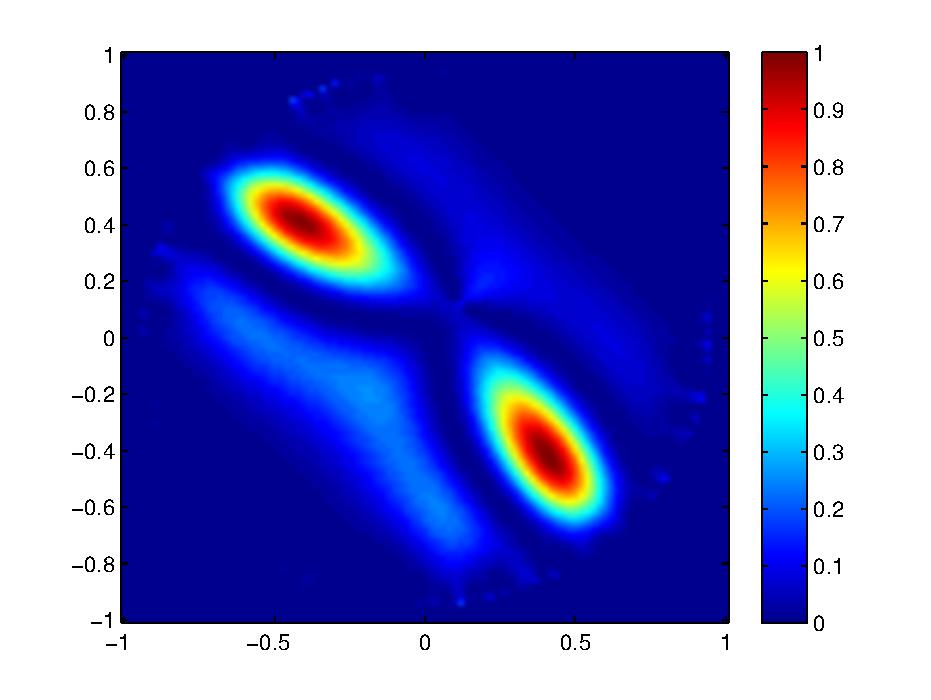}\hfill{}
\hfill{}\includegraphics[clip,width=0.3\textwidth]{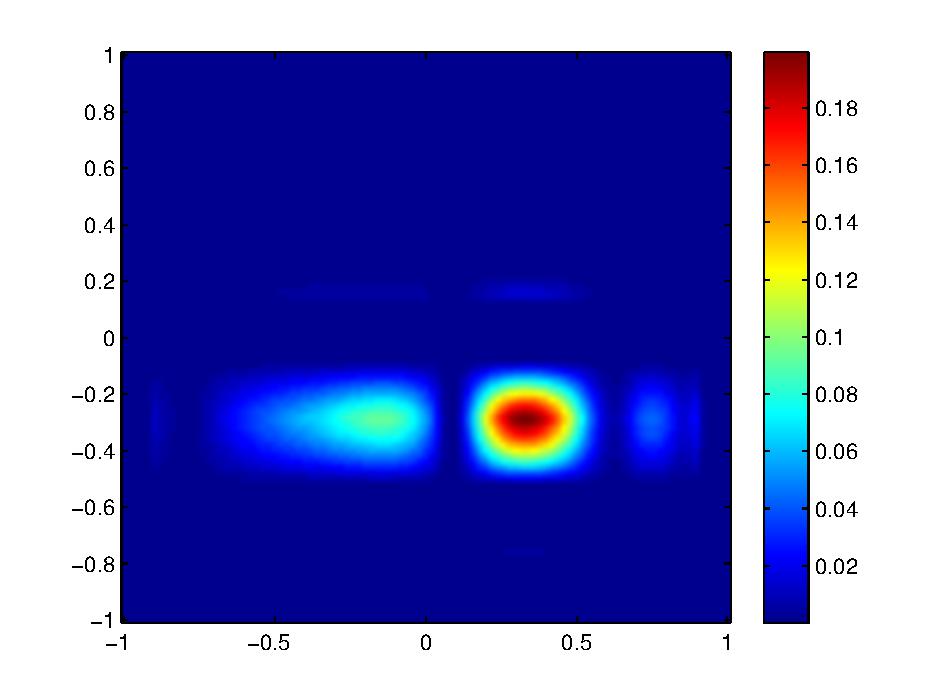}\hfill{}

 \hfill{}(a)\hfill{} \hfill{}(b)\hfill{} \hfill{}(c)\hfill{}

\hfill{}\includegraphics[clip,width=0.3\textwidth]{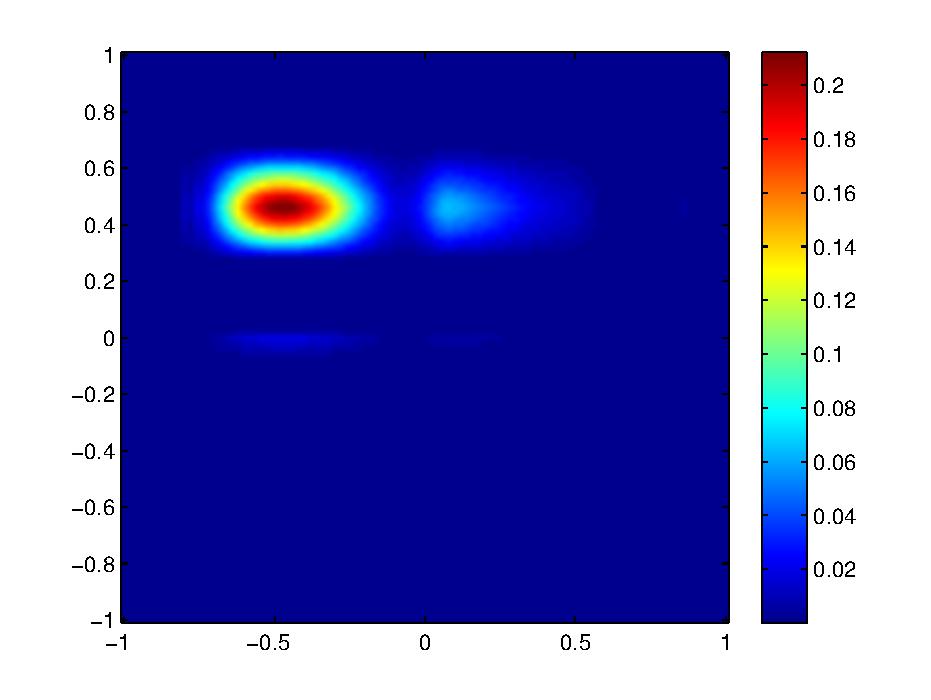}\hfill{}
\hfill{}\includegraphics[clip,width=0.3\textwidth]{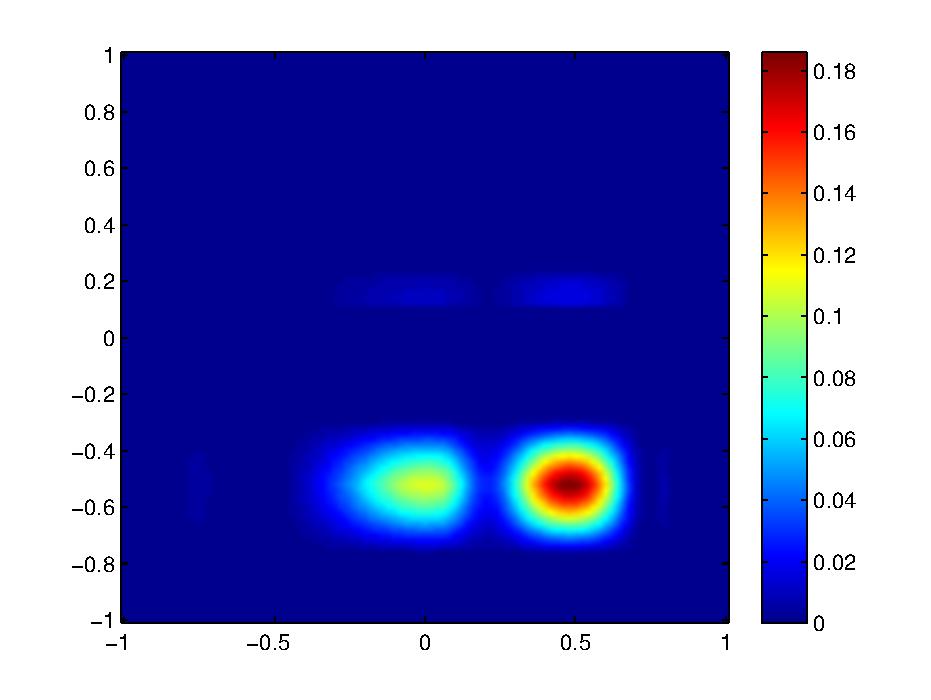}\hfill{}
\hfill{}\includegraphics[clip,width=0.3\textwidth]{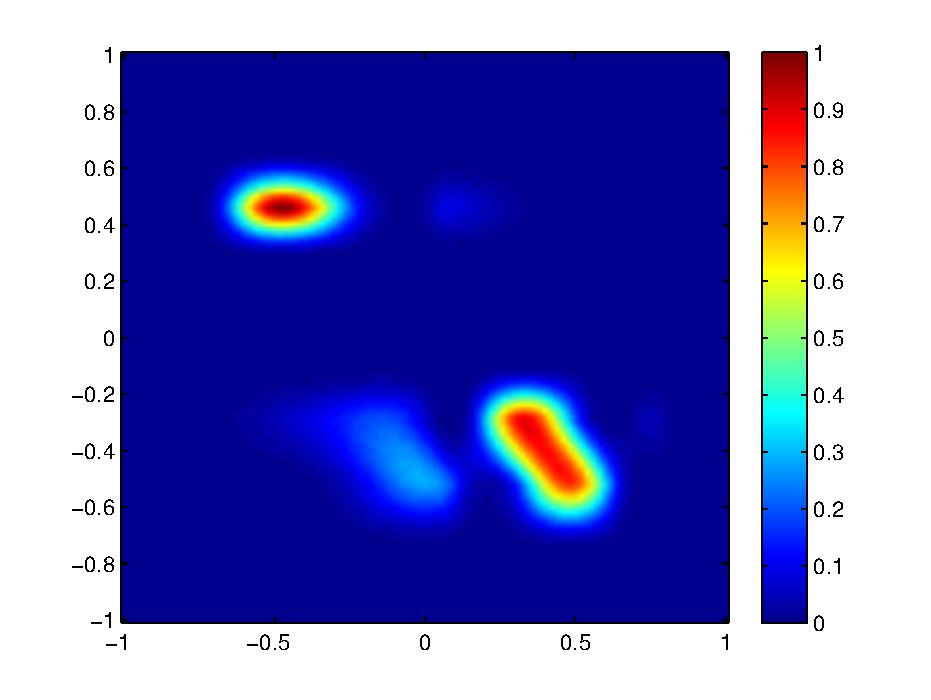}\hfill{}

 \hfill{}(d)\hfill{} \hfill{}(e)\hfill{} \hfill{}(f)\hfill{}

 \caption{\label{fig:EITnew1}\small{NMF decomposition of the DSM images from EIT in Example 7, with $\{\sigma_{i_l j_l} \}_{l=1}^3 = \{ 2.3712, 2.3548, 2.2904 \}$}}
 \end{figurehere}

\ss
\textbf{Example 8}. In this example, we consider the case of $4$ inclusions with same size as in Example $7$ sitting inside the sampling region, which are placed at positions of $(0.36,0.36)$, $(0.36,-0.44)$, $(-0.44,0.36)$ and $(-0.44,-0.44)$; see Figure \ref{fig:EITnew2} (a). The squared reconstructed images from the indices $\widetilde{I}^2$ after normalization is shown in Figure \ref{fig:EITnew2} (b).  Figure \ref{fig:EITnew2} (c-e) presents the images of $\sigma_{i_l j_l} \, (\tilde{u}_p)_{i_l} \otimes (\tilde{v}_p)_{j_l}$, for $l = 1,2,3$ after NMF over the image $I$.
The generalized eigenvalues are respectively given as $\{\sigma_{i_l j_l} \}_{l=1}^3 = \{ 5.9647, 4.2460, 3.8970 \}$ in this example.
The squared image of the approximation to $\mathcal{I}_{p,\tilde{p}}^{\alpha, \nu,\gamma}(I)$ after normalization is in Figure \ref{fig:EITnew2} (f).
We can see that we can obtain fairly nicely the principal components of the image coming from signals from the inclusions.

\begin{figurehere}
\hfill{}\includegraphics[clip,width=0.3\textwidth]{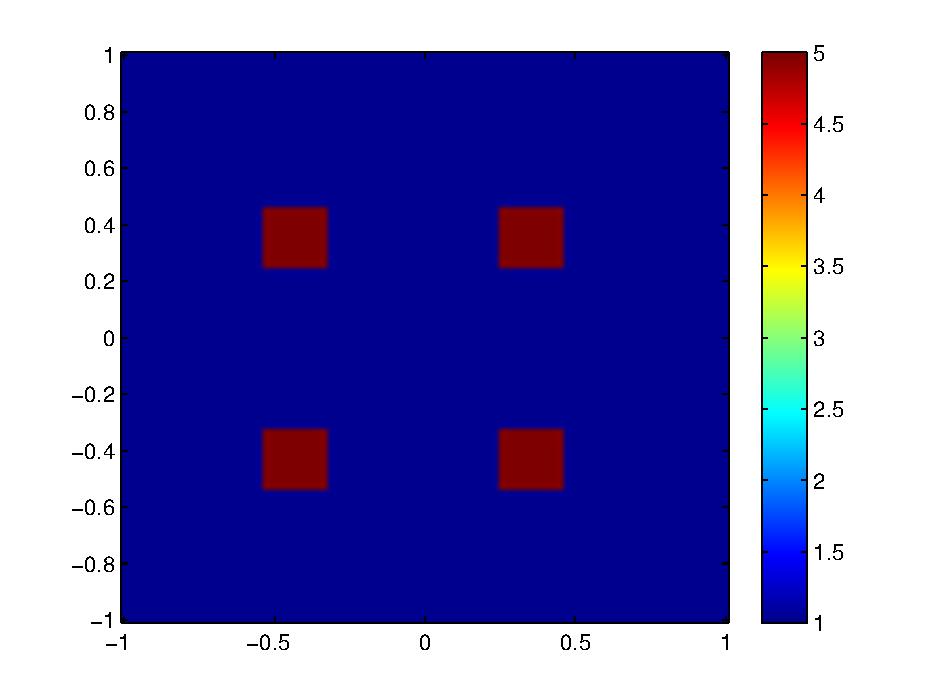}\hfill{}
\hfill{}\includegraphics[clip,width=0.3\textwidth]{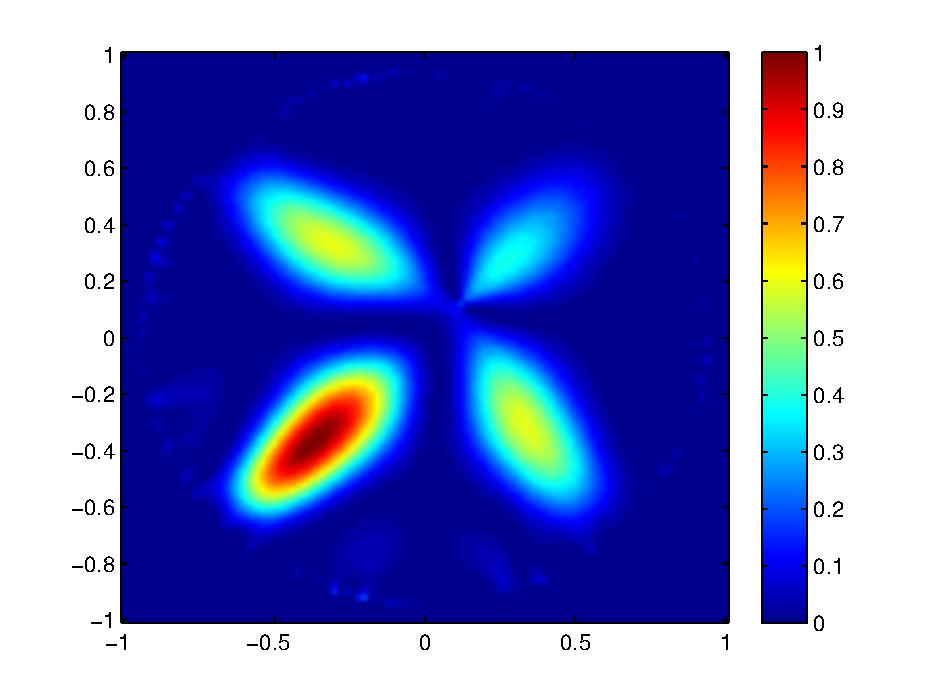}\hfill{}
\hfill{}\includegraphics[clip,width=0.3\textwidth]{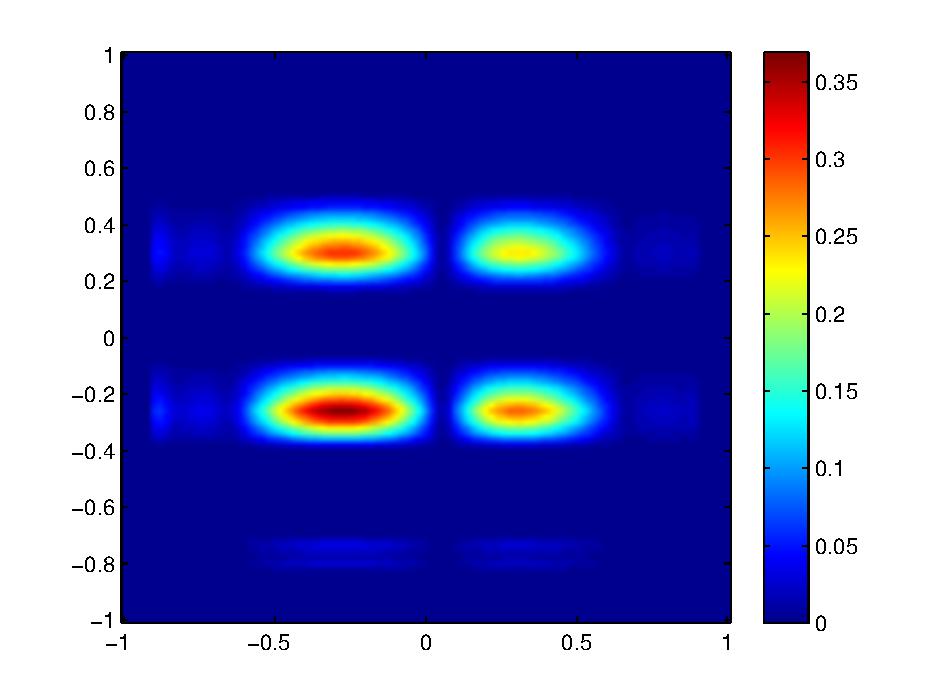}\hfill{}

 \hfill{}(a)\hfill{} \hfill{}(b)\hfill{} \hfill{}(c)\hfill{}

\hfill{}\includegraphics[clip,width=0.3\textwidth]{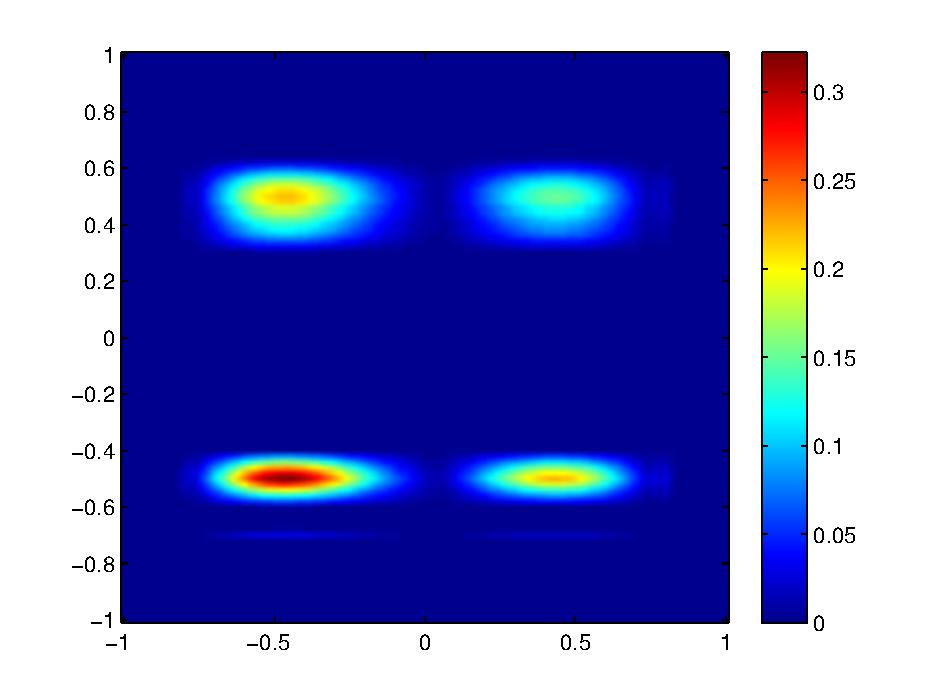}\hfill{}
\hfill{}\includegraphics[clip,width=0.3\textwidth]{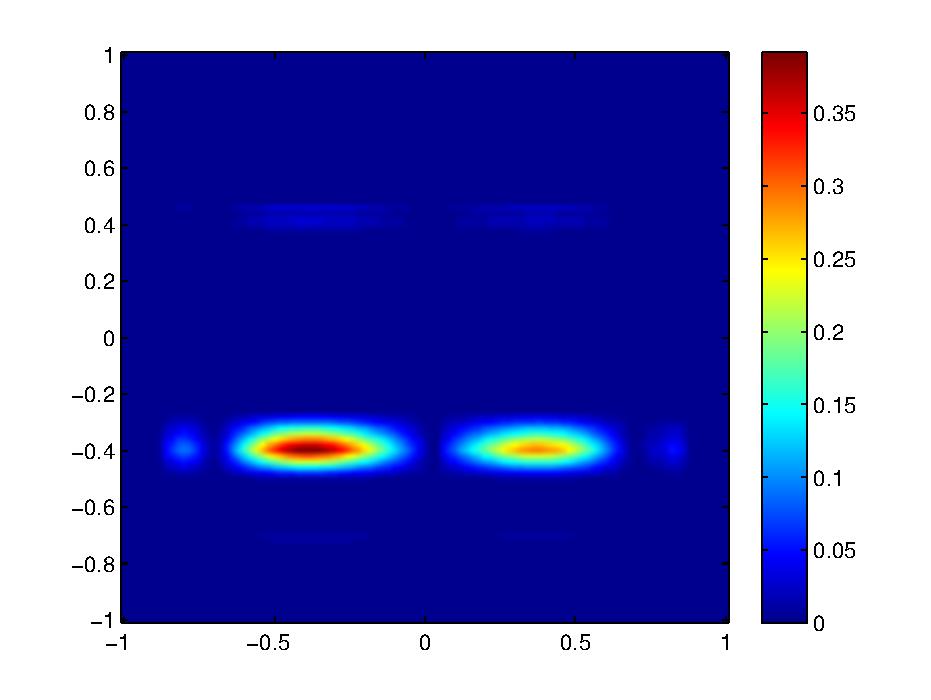}\hfill{}
\hfill{}\includegraphics[clip,width=0.3\textwidth]{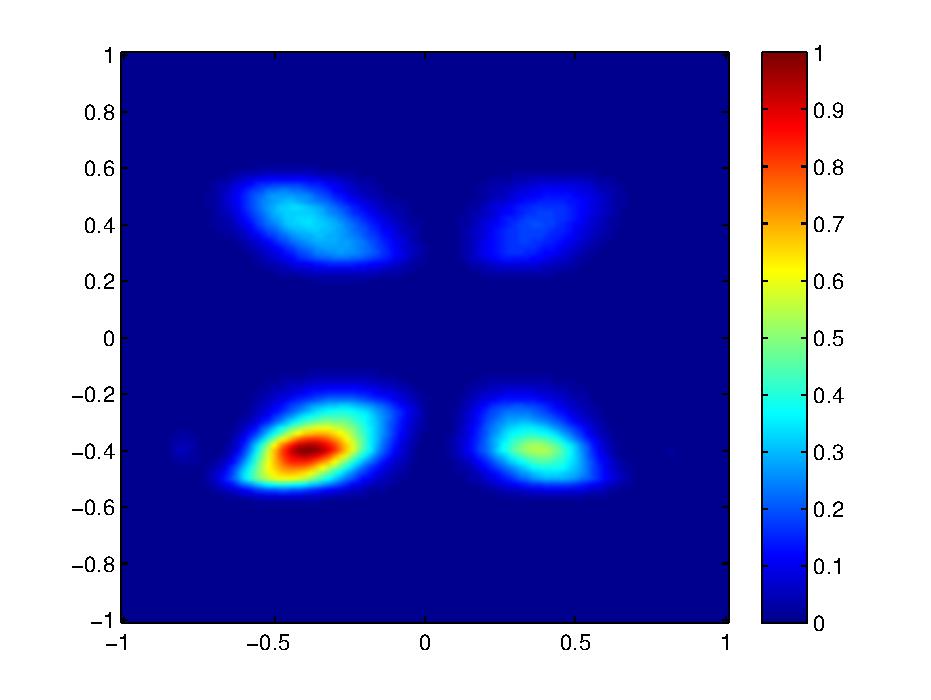}\hfill{}

 \hfill{}(d)\hfill{} \hfill{}(e)\hfill{} \hfill{}(f)\hfill{}

 \caption{\label{fig:EITnew2}\small{NMF decomposition of the DSM images from EIT in Example 8, with $\{\sigma_{i_l j_l} \}_{l=1}^3 = \{ 5.9647, 4.2460, 3.8970 \}$}}
 \end{figurehere}

\ss
\textbf{Example 9}. In this example, $2$ inclusions of the same size as in Example $7$ are introduced in the homogeneous background, and they are respectively placed at the positions $(-0.36,0.36)$ and $(0.36,0.36)$ inside the domain; see Figure \ref{fig:EITnew3} (a). The squared reconstructed images from the indices $I^2$ after normalization is given in Figure \ref{fig:EITnew3} (b).  The images of $\sigma_{i_l j_l} \, (\tilde{u}_p)_{i_l} \otimes (\tilde{v}_p)_{j_l}$, for $l = 1,2,3$ after NMF over the image $\widetilde{I}$ are shown in Figure \ref{fig:EITnew3} (c-e).
The generalized eigenvalues are respectively given as $\{\sigma_{i_l j_l} \}_{l=1}^3 = \{ 3.9194, 0, 0 \}$ in this example.
Figure \ref{fig:EITnew3} (f) presents the squared image of the approximation to $\mathcal{I}_{p,\tilde{p}}^{\alpha, \nu, \gamma} (I)$ after normalization.
From the figures, we can see that the principal components coming from the inclusions can be nicely obtained, and both the sizes and locations of inhomogeneities can be reasonably obtained from the NMF image after the introduction of a  appropriate cutoff.

\begin{figurehere}
\hfill{}\includegraphics[clip,width=0.3\textwidth]{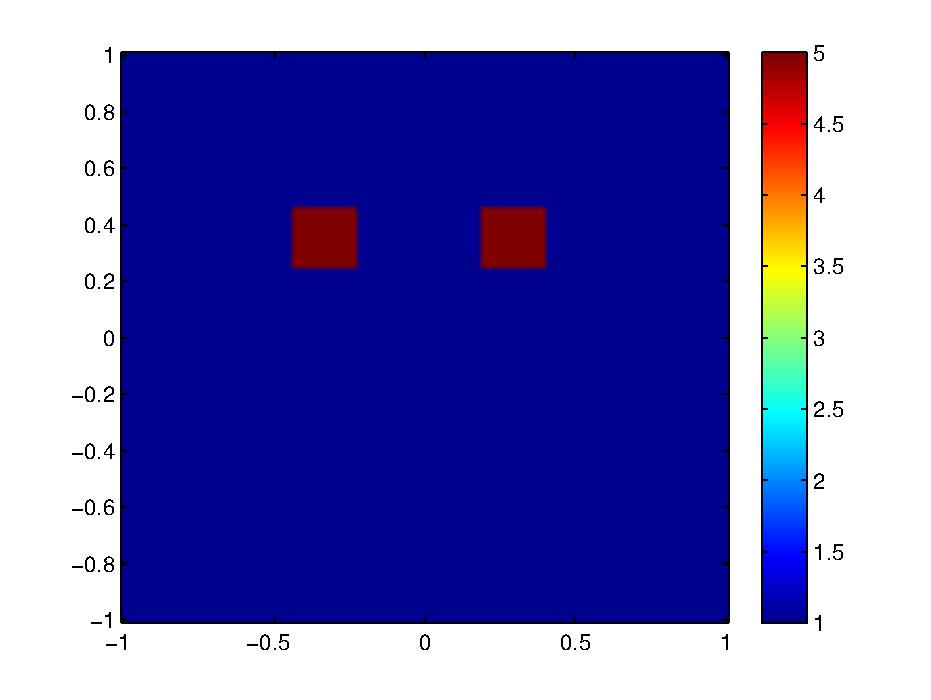}\hfill{}
\hfill{}\includegraphics[clip,width=0.3\textwidth]{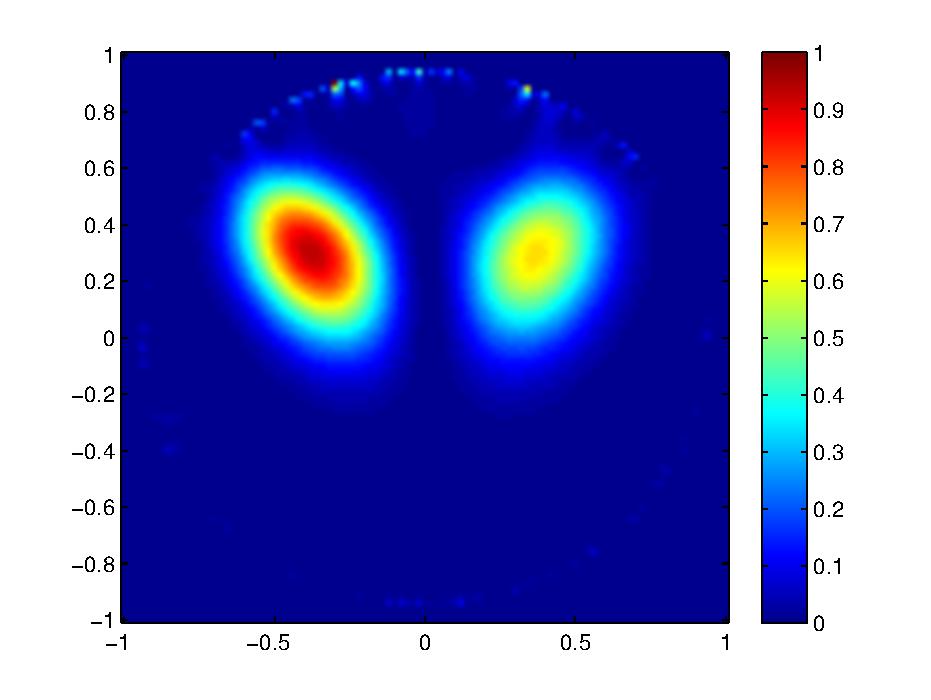}\hfill{}
\hfill{}\includegraphics[clip,width=0.3\textwidth]{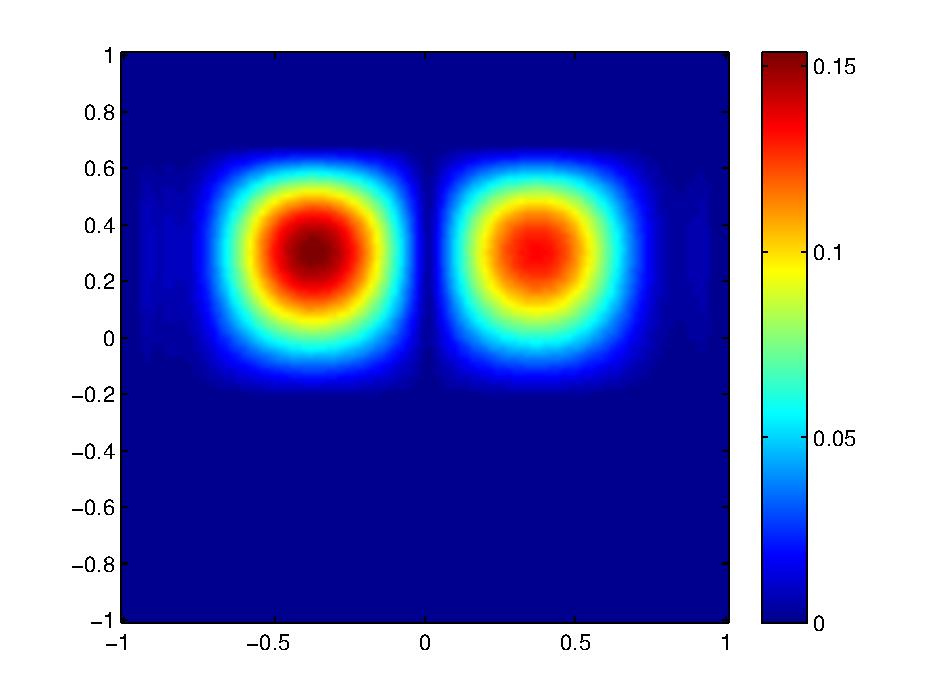}\hfill{}

 \hfill{}(a)\hfill{} \hfill{}(b)\hfill{} \hfill{}(c)\hfill{}

\hfill{}\includegraphics[clip,width=0.3\textwidth]{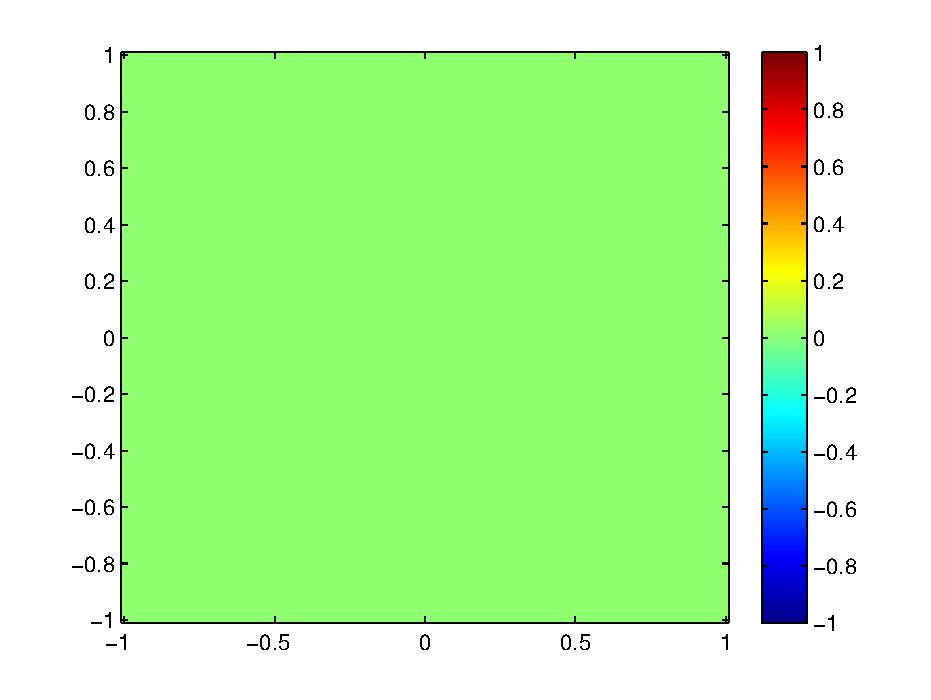}\hfill{}
\hfill{}\includegraphics[clip,width=0.3\textwidth]{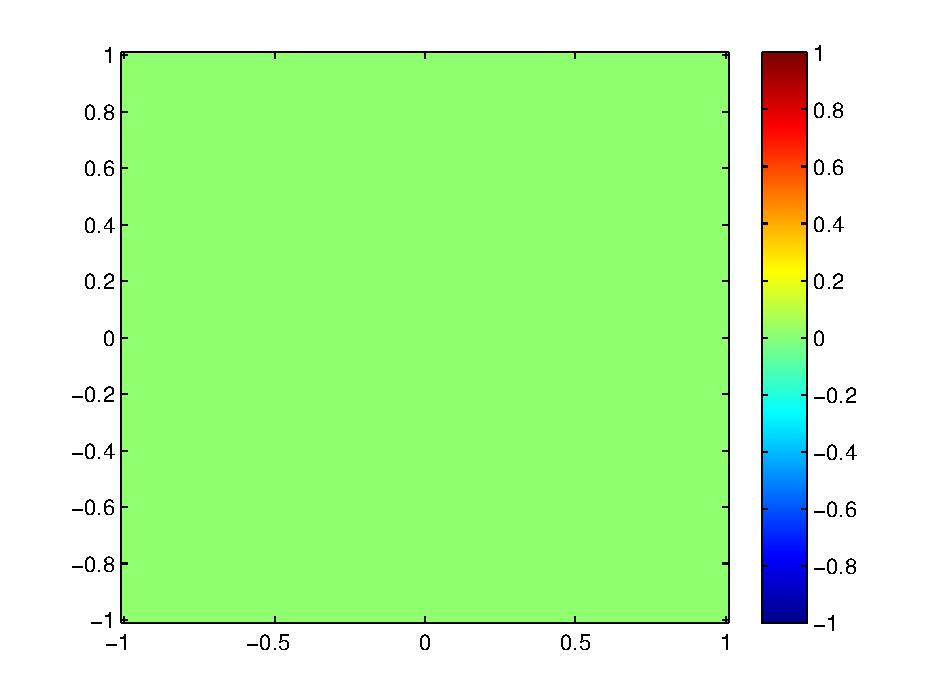}\hfill{}
\hfill{}\includegraphics[clip,width=0.3\textwidth]{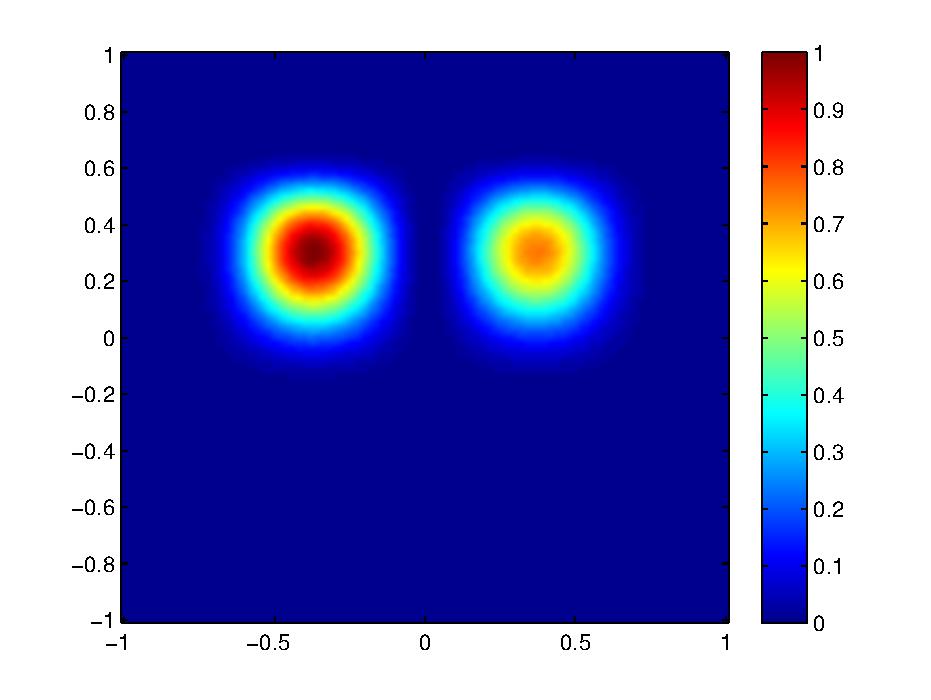}\hfill{}

 \hfill{}(d)\hfill{} \hfill{}(e)\hfill{} \hfill{}(f)\hfill{}

 \caption{\label{fig:EITnew3}\small{NMF decomposition of the DSM images from EIT in Example 9, with $\{\sigma_{i_l j_l} \}_{l=1}^3 = \{ 3.9194, 0, 0 \}$}}
 \end{figurehere}

\end{document}